\documentclass[12pt, reqno]{amsart}
\usepackage[foot]{amsaddr}

\usepackage[mathscr]{eucal}
\usepackage[all]{xy} 
\usepackage{graphicx}
\usepackage{caption}
\usepackage{subcaption}
\usepackage{multicol}
\usepackage{amssymb,amsfonts,amsthm}    
\usepackage{color}                              
\usepackage[centertags]{amsmath}
\usepackage{empheq}
\usepackage{amsfonts,amssymb,amsthm,newlfont}
\usepackage{geometry}
\usepackage{algorithm}
\usepackage{algpseudocode}
\usepackage{hyperref}
\usepackage{mathrsfs}
\usepackage{enumerate}
\usepackage{esint}
\usepackage{IEEEtrantools}  

\usepackage{nomencl}

\theoremstyle{plain}
\newtheorem{thm}{Theorem}[section]
\newtheorem{prop}[thm]{Proposition}
\newtheorem{lem}[thm]{Lemma}

\newtheorem{remark}[thm]{Remark}

\numberwithin{equation}{section}
\numberwithin{figure}{section}
\numberwithin{algorithm}{section}

\def\bv{\mathbf{v}}

\newcommand{\cT}{\mathcal{T}}

\title[A structure-preserving numerical scheme for optimal control of mixing]{A Structure-Preserving Numerical Scheme for Optimal Control and Design of Mixing in Incompressible Flows}

\author[Weiwei Hu, Ziqian Li, Yubiao Zhang and Enrique Zuazua]{Weiwei Hu$^1$}
\address{$^1$Department of Mathematics, University of Georgia, Athens, GA 30602, USA}

\author{Ziqian Li$^2$$^3$}
\address{$^2$Chair for Dynamics, Control, Machine Learning and Numerics, Department of Mathematics, Friedrich-Alexander-Universit\"{a}t Erlangen-N\"{u}rnberg, 91058 Erlangen, Germany}
\address{$^3$School of Mathematics, Jilin University, Changchun, Jilin 130012, China}

\author{Yubiao Zhang$^3$}

\author{Enrique Zuazua$^2$$^4$$^5$}
\address{$^4$Chair of Computational Mathematics, Deusto University, 48007 Bilbao, Basque Country, Spain}
\address{$^5$Departamento de Matem\'aticas, Universidad Aut\'onoma de Madrid, 28049 Madrid, Spain}

\email{Weiwei.Hu@uga.edu, ziqian.li@fau.de, yubiao\_zhang@yeah.net}
\email{enrique.zuazua@fau.de}

\begin{document}


\begin{abstract}
We develop a structure-preserving computational framework for optimal mixing control in incompressible flows. Our approach exactly conserves the continuous system's key invariants (mass and $L^2$-energy), while also maintaining discrete state--adjoint duality at every time step. These properties are achieved by integrating a centered finite-volume discretization in space with a time-symmetric Crank--Nicolson integrator for both the forward advection and its adjoint, all inside a gradient-based optimization loop. The result is a numerical solver that is faithful to the continuous optimality conditions and efficiently computes mixing-enhancing controls. In our numerical tests, the optimized time-dependent stirring produces a nearly exponential decay of a chosen mix-norm, achieving orders-of-magnitude faster mixing than any single steady flow. To our knowledge, this work provides the first evidence that enforcing physical structure at the discrete level can lead to both exact conservation and highly effective mixing outcomes in optimal flow design.
\end{abstract}

\maketitle

\section{Introduction}
\subsection{Problem statement}\label{sec:problem}
Understanding mass transport and fluid mixing by advective flows poses fundamental yet long-standing challenges in both analysis and applications. Examples where such problems arise when studying the oceanic circulation \cite{wang2023doubling}, the spreading of environmental pollutants \cite{zhang2025longrange} and chemical stimuli \cite{vergassola2007infotaxis}, as well as microfluidic mixing \cite{stroock2002chaotic}, to name but a few. Developing efficient and operational protocols for fluid transport and mixing has been widely investigated in both academic and industrial communities \cite{kelley2011separating}.

This paper contributes to this field by developing a structure-preserving numerical scheme for optimal control in mixing, or, in other words, for the optimal design of optimal mixers to maximize the mixing rate.

We consider  passive scalar fields $\theta= \theta(t, x)$, advected by incompressible velocity fields $v=v(t,x)$, in a bounded domain  $\Omega \subset \mathbb R^d$ (with, typically, $d=2,3$), whose dynamics is described by the model (with $T>0$)
\begin{align}\label{eq:state}
    \partial_t \theta  + \text{div}\,( \mathbf{v} \theta) =0
    \text{ in }  (0,T) \times \Omega;
    ~~\theta|_{t=0} = \theta^0.
\end{align}

Assume that the velocity field $\mathbf{v}$ satisfies the incompressibility and no-penetration boundary conditions in $\Omega$, i.e.,
  \begin{align}\label{cond:v}
   \text{div}\, \mathbf{v}=0 \text{ in }  \Omega 
   \text{; }~
   \mathbf{v}\cdot \mathbf{n}_{\Omega} = 0 \text{ on }  \partial\Omega,
    \end{align}
where $\mathbf{n}_{\Omega}$ denotes the outward normal vector of the domain $\Omega$. 
Under \eqref{cond:v}, the conservative form \eqref{eq:state} is equivalent to
the nonconservative form
\[
    \partial_t \theta + \mathbf{v}\cdot\nabla\theta = 0
    \quad\text{in } (0,T)\times\Omega,
\]
and we shall use these two formulations interchangeably.

For each $\theta^0 \in L^2(\Omega)$ and $\mathbf{v} \in L^2((0,T); C^1(\overline{\Omega}; \mathbb R^d) )$ verifying \eqref{cond:v}, equation \eqref{eq:state}  has a unique solution in $C([0,T];L^2(\Omega))$ (see Lemma \ref{20251209-lemma-WellPosednessOfTransport}). 

Roughly,  mixing refers to the process under which the solution $\theta$ approaches its space-like average as time $t$ increases. 
For simplicity and without loss of generality (given the conservation of the solution's space-like mean as time evolves), it is sufficient to consider initial data with zero mean. 

In the sequel, we assume
\begin{align}\label{20260105-AssumptionsOnInitialData}
    \theta^0 \in L^{\infty}(\Omega)
    \text{ and } 
    \int_{\Omega} \theta^0(x)\,dx = 0,
\end{align}
and study how $\theta(t)$ converges to $0$ as a measure of mixing efficiency. Under these assumptions, the solution $\theta(t,x)$ remains uniformly bounded, which in turn simplifies the analysis and allows the optimal control problem to be posed and treated rigorously within an appropriate functional framework.   

To quantify mixing, we adopt negative Sobolev norms,  which are widely used in the literature and initiated by Mathew, Mezi\'{c}, and Petzold in \cite{mathew2005multiscale}  (with a negative Sobolev index $-1/2$). In fact, any norm that quantifies the weak convergence in $L^2$ can be used as a mix-norm (e.g.~\cite{lin2011optimal, lunasin2012optimal}). Specifically, we use the homogeneous Sobolev norm $\dot{H}^{-1}(\Omega)$, the dual norm of zero-mean functions in $H^1(\Omega)$, which, when restricted to the subspace $L^2(\Omega)$, takes the following form:
 \begin{align}\label{neg_norm}
\|f\|_{ \dot{H}^{-1}(\Omega) } = \sup
\Big\{
    \int_{\Omega}  f \phi dx
    ~:~  \int_{\Omega} \phi dx = 0, ~\int_{\Omega} | \nabla \phi |^2 dx  \leq 1
\Big\},
~\forall\, f \in  L^2(\Omega).
\end{align}
A smaller value  $\Vert \theta \Vert_{\dot{H}^{-1}(\Omega)}$ corresponds to better mixing, as it corresponds to a transfer of energy from low to high frequencies. 

Other measures, such as the geometric mixing scale,  the measure of the interface of the mixtures, as well as entropy, are employed to quantify the mixing process and the related control problems (e.g.~\cite{chakravarthy1996mixing, d1999control,vikhansky2002enhancement, vikhansky2002control, alberti2016exponential,yao2017mixing,crippa2019polynomial, elgindi2019universal}). A more detailed review can be found in \cite{thiffeault2012using}. But, in the present work, we focus on the $\dot{H}^{-1}(\Omega)$-norm above since it is better suited for an optimal control treatment.

Admissible incompressible velocity fields live in an infinite-dimensional function space. In many practical settings, however, particularly in actuator-constrained mixer design, the flow cannot be chosen freely: the mixing velocity must be realized as an optimal (typically time-dependent) combination of a finite set of prescribed stirring profiles. This reduced-actuation viewpoint is standard in the fluid mechanics literature; see, for instance, work on optimized stirring flows \cite{lin2011optimal} and related finite-mode mixing constructions such as random cellular flows \cite{navarro2025exponential}.

Thus, in the sequel, we assume the velocity field $\mathbf{v}$ to be generated by a finite set of basis flows. More precisely, define the finite-dimensional space of velocity fields 
\begin{align}\label{20251004-ControlConstraint}
     U := \text{span}\, \{ \mathbf{b}_1, \ldots, \mathbf{b}_m \},
\end{align}
where $\{ \mathbf{b}_i \}_{i=1}^m\subset C^1( \overline{\Omega}; \mathbb R^d )$ are orthonormal in $L^2(\Omega; \mathbb R^d)$ and satisfy \eqref{cond:v} for $i=1,...,m.$ 
Such a finite-dimensional restriction of the velocity field is standard in optimal control of mixing studies, reflecting practical actuator limitations (see, e.g., \cite{zheng2023numerical}). 

We consider the class of velocity fields $\mathbf{v}$ in $L^2((0,T);U)$ taking the form
\begin{align}
\label{eq:v_finite}
    \mathbf{v}(t) = \sum_{i=1}^m v_i(t) \mathbf{b}_i, ~t\in(0,T)
\end{align}
for some coefficients  $(v_1,\ldots,v_m)$ in $ L^2((0,T); \mathbb R^m)$. 
Our goal is to derive efficient numerical solvers to compute the optimal mixer within this class. 

Although formulated as an optimal control problem, the control $\mathbf v(\cdot)\in L^2((0,T); U)$ can be interpreted as the design of an efficient incompressible stirring field within a prescribed actuation subspace $U$. 

There are multiple possible choices for the basis functions $\mathbf{b}_i$. In 2D cases, a typical example, often employed in the analysis of mixing, consists of cellular flows \cite{Fannjiang2006Quenching, crippa2017cellular, hu2025cellular} in square domains: 
\begin{align}
\label{eq:cellular}
&\mathbf{b}_i(x_1,x_2)= 
 \begin{bmatrix}
      -i\pi\sin (i\pi x_1)\cos(i\pi x_2)          \\[0.3em]
        i\pi\cos (i\pi x_1)\sin( i\pi x_2) 
     \end{bmatrix}^T,
    ~i=1,2, \ldots, m.
\end{align}
Another example is Doswell frontogenesis \cite{doswell1984kinematic, Morales2019adjoint} in circular domains, whose first basis mode can be written as
\begin{align}
\label{eq:v_Doswell}
&\mathbf{b}_1(x_1,x_2)= 
 \begin{bmatrix}
      -x_2g(r(x_1,x_2))          \\[0.3em]
        x_1g(r(x_1,x_2)) 
     \end{bmatrix}^T,
\end{align}
where
\begin{equation*}
    g(r(x_1,x_2)):=\frac{1}{r(x_1,x_2)}\ \overline{v}\ {\text{sech}}^2{(r(x_1,x_2))} \tanh{(r(x_1,x_2))},
\end{equation*}
with $r(x_1,x_2):=\sqrt{x_1^2+x_2^2}$ and $\overline{v}$ a constant parameter ($\overline{v}:=2.59807$ in \cite{doswell1984kinematic, Morales2019adjoint}). We construct a multi-scale version of this Doswell flow basis as the second basis $\mathbf{b}_2$ for our experiments (see Section~\ref{sec:Doswell} for details). 

These two prototypical families of basis flows (cellular flows and Doswell frontogenesis flows) will serve as illustrative test cases in the numerical experiments of Section~\ref{sec:experiment}. However, both the numerical scheme and the accompanying analysis developed in this paper are fully general and are not tied to these particular examples.

We now introduce the following optimal control problem, aiming to achieve effective mixing
\begin{align}\label{eq:cost}
\min_{ \mathbf{v} \in  L^2( (0,T); U)} J(\bv) 
:=  \min_{ \mathbf{v} \in  L^2( (0,T); U)}
\bigg[
\frac{1}{2} \| \theta(T; \mathbf{v}) \|^2_{ \dot{H}^{-1}(\Omega) } 
+ \frac{\gamma}{2}\int_0^T \|\bv(t)\|^2_{L^2(\Omega; \mathbb R^d)}\,dt
\bigg],
\end{align}
subject to  \eqref{eq:state}--\eqref{cond:v} and \eqref{eq:v_finite}, where $\gamma>0$ is a regularization parameter.

Note that in this setting, the initial datum $\theta^0$ has been fixed. Therefore, the optimal mixer depends on the initial datum. We could also consider the problem of building the optimal mixer for all possible initial data. This would lead to a more complex min-max problem:
\begin{align}\label{eq:costminmax}
 \min_{ \mathbf{v} \in  L^2( (0,T); U)} \, \max_{||\theta^0||_{L^2(\Omega)} =1}
\bigg[
\frac{1}{2} \| \theta(T; \mathbf{v}) \|^2_{ \dot{H}^{-1}(\Omega) } 
+ \frac{\gamma}{2}\int_0^T \|\bv(t)\|^2_{L^2(\Omega; \mathbb R^d)}\,dt
\bigg].
\end{align}
The latter corresponds to an optimal functional mixer, in the sense of being the best possible, for all possible configurations of the initial data. Rather different and more complex from an analytical perspective, we will focus on the first problem, \eqref{eq:cost}, in which the initial datum is assumed to be given and fixed. We refer to \cite{privat2015optimal} for an analysis of the problem of the optimal observers for wave and heat processes, in which the analog of \eqref{eq:costminmax} is considered, and the significant differences to the case where the initial data are fixed are discussed.

To better interpret the negative Sobolev norm, we observe that the functional $J$ admits the following equivalent representation  (see Appendix~\ref{Appendix-NewFormOfJ} for the proof):
\begin{align}\label{20251010-AnotherFormOfCostFunctional}
\text{(OP)}  \quad
 J(\bv)  = & 
\frac{1}{2} \big\langle \theta(T; \mathbf{v}), \eta( \theta (T; \mathbf{v}) )
\big\rangle_{L^2(\Omega)} 
\nonumber\\
& \quad + \frac{\gamma}{2}\int_0^T \|\bv(t)\|^2_{L^2(\Omega; \mathbb R^d)}\,dt,
~\mathbf{v} \in  L^2( (0,T); U), \hspace{0.3in} 
\end{align}
where $\eta = \eta( \theta (T; \mathbf{v}) ) \in H^1(\Omega)$ is the solution of the elliptic equation (see Appendix~\ref{Appendix-WellPosenessOfEllpiticEquation} for the well-posedness)
\begin{align}\label{eq:eta}
-\Delta \eta=\theta(T; \mathbf{v}) 
\text{ in }  \Omega; 
~~  \partial_{ \mathbf{n}_{\Omega}} \eta = 0
\text{ on }  \partial\Omega;
~~ \int_{\Omega} \eta dx = 0.
\end{align}

A control $\mathbf{v}^* \in L^2((0,T);U)$ is said to be optimal for (OP) if it is a minimizer, namely, when $$J(\mathbf{v}^*) = \min_{\mathbf{v} \in L^2((0,T);U)} J(\mathbf{v}).$$ The proof of existence of an optimal control for (OP) is straightforward, and is presented in Appendix~\ref{Appendix-ExistenceOfOptimalControl}, for the sake of completeness. 

To solve the problem (OP) numerically, we aim to develop a numerical scheme based on the gradient of $J$ which can be characterized as (see Appendix~\ref{Appendix-Gradient} for the proof): 
\begin{align}\label{20251006-GraidentOfJ}
    J'( \mathbf{v} ) = P_U  (\gamma \mathbf{v} + \theta \nabla \rho),
    \text{ in }  L^2( (0,T); U),
    ~\mathbf{v}  \in   L^2( (0,T); U),
\end{align}
where $P_U$ is the orthogonal projection of $L^2(\Omega; \mathbb R^d)$ onto $U$, and $\rho \in C([0,T]; H^1(\Omega))$ satisfies the adjoint equation
\begin{align}\label{eq:adjoint}
    \partial_t \rho +  \text{div}\, (\bv \rho) = 0 
    \text{ in } (0,T) \times \Omega;
    ~~  \rho|_{t=T} = \eta
    \text{ in } \Omega,
\end{align}
with the Lagrange multiplier $\eta \in H^1(\Omega)$ given by \eqref{eq:eta}. This gradient is well-defined thanks to the uniform boundedness of $\theta(t, x)$.

An optimal control $\mathbf{v}^*$ to problem (OP) satisfies the optimality condition $J'(\mathbf{v}^*)  = 0$. Together with the state equation \eqref{eq:state} and the adjoint equation \eqref{eq:adjoint}, this condition yields the first-order necessary optimality conditions (FONCs) (in agreement with Pontryagin's maximum principle) for problem (OP).

\subsection{Computational challenges and motivations.}

A primary challenge in numerical simulations of scalar transport and mixing is the need to preserve the \emph{global conservation} properties of the scalar field under incompressible flows. Beyond this, several coupled difficulties arise:
\begin{itemize}
  \item \textbf{Mixing-driven multiscale formation.} The mixing dynamics generate increasingly fine-scale structures and sharp gradients in the scalar distribution.
  \item \textbf{Frequency transfer and norm imbalance.} As the mix-norm $\|\theta\|_{\dot H^{-1}(\Omega)}$ decreases, by duality $\|\theta\|_{H^{1}(\Omega)}$ generally increases, reflecting the transfer of energy from low to high frequencies.
  \item \textbf{Resolution requirements.} The appearance of thin filaments during the evolution forces the computational mesh to be sufficiently refined to resolve these small spatial scales.
  \item \textbf{Algorithmic cost of the optimal control loop.} Introducing an optimal control problem adds substantial cost: one must solve the state equation forward in time together with the adjoint equation backward in time, coupled through a nonlinear optimality condition. A direct enforcement of these conditions typically yields prohibitively expensive computations.
  \item \textbf{Structure- and conservation-preserving discretization.} The numerical me\-thod must also maintain the conservation laws and the structural consistency inherent in both the state and adjoint systems \eqref{eq:state} and \eqref{eq:adjoint}.
\end{itemize}

Specifically, the following conservation laws need to be taken into consideration:
\begin{itemize}
\item \textbf{$L^p$-norm conservation.} 
$L^p$-norms are conserved along the dynamics generated by \eqref{eq:state}. More precisely, when  $\theta^0 \in L^\infty(\Omega)$,  
    \begin{align*}
    \|\theta(t)\|_{L^{p}(\Omega)}=\|\theta^{0}\|_{L^{p}(\Omega)}
    \text{ for each } t \geq 0, ~p \in [1,+\infty],
    \end{align*} 
    as a consequence of incompressibility and no-penetration conditions. 
In particular, the total mass 
\[
    M(t) := \int_{\Omega} \theta(t)\,dx, ~ t \geq 0
\]
and the energy
\[
E(t) := \int_{\Omega} \theta^2(t)\,dx, ~ t \geq 0.
\]
remain constant in time.

\item \textbf{State--adjoint consistency.} The numerical scheme must preserve the duality relation between the state and adjoint equations (see \eqref{20251209-ForwardBackwardPair}) so that
\begin{align}\label{eq:consist}
    \int_{\Omega} \theta(t)\,\rho(t)\,dx \equiv \int_{\Omega} \theta(T)\,\rho(T)\,dx,
    ~ t \in [0,T].
\end{align}
This identity is fundamental for accurate gradient computation in PDE-constrained optimization. If the discrete analog \eqref{eq:consist} is violated, the discrete adjoint becomes inconsistent with the discrete forward solver. This inconsistency produces an inexact reduced gradient, which can stall convergence in gradient-based optimization.
	
\end{itemize} 

However, structure-preserving numerical schemes specifically tailored to mixing-control problems remain scarce. 
The majority of PDE-constrained mixing studies employ conventional spatial discretizations, for instance, discontinuous Galerkin (DG) methods (e.g.~\cite{hu2023feedback, zheng2023numerical}) without explicitly targeting geometric or conservation properties inherent to the continuous system. To the best of our knowledge, no existing work systematically enforces state--adjoint consistency within the optimization loop, a property crucial for ensuring numerical fidelity and stability in optimal control solutions. This gap motivates the development of structure-preserving numerical methods specifically adapted to fluid mixing control.

\subsection{Main results and contributions}

Our main results are as follows:

\begin{itemize}
  \item \textbf{Numerical scheme.}
  We develop a solver  (stated in Algorithm~\ref{alg:openloop}) for the problem (OP) based on a conservative finite-volume spatial discretization and a second-order, time-symmetric Crank--Nicolson time integrator applied to both the state equation \eqref{eq:state} and the adjoint equation \eqref{eq:adjoint}. The gradient  \eqref{20251006-GraidentOfJ} is enforced by a conjugate gradient method with line search for step selection. We show the convergence of discrete solutions of state equation in Appendix \ref{Appendix-state_convergence} and the convergence of the discrete cost functional in Appendix \ref{Appendix-ConvDiscreteObjective}.

\vskip 5pt
  \item \textbf{Structure-preserving properties.}
  The fully discrete scheme preserves the key invariants and dualities of the problem (OP), (see Theorem~\ref{20251008-theorem-PerservedStructures}).
  \begin{itemize}
        \vskip 3pt
    \item \emph{Discrete mass conservation} (see \eqref{20251008-MassConservation}): the flux-difference finite-volume method with no-penetration boundaries preserves discrete mass at every time step.
        \vskip 3pt
    \item \emph{Discrete $L^2$-energy conservation} (see \eqref{20251008-EnergyConservation}): the semi-discrete advection operator is skew-symmetric, and the Crank--Nicolson time integrator is an isometry in the discrete inner product, hence the discrete energy is invariant.
        \vskip 3pt
    \item \emph{State--adjoint consistency} (see \eqref{20251008-ForwardBackwardConsistency}): the Crank--Nicolson time integrator is applied to both the state and adjoint equations,  to preserve the discrete analog of the pairing \eqref{eq:consist}.
  \end{itemize}
  To the best of our knowledge, this is the first mixing-control scheme that simultaneously enforces all three properties in a single fully discrete framework.

\vskip 5pt
  \item \textbf{Equivalence between the  Optimize-Then-Discretize and Discre\-tize\--Then-Optimize approaches.} The numerical scheme developed in this paper is designed by the Optimize-Then-Discretize approach. We construct a discrete gradient of $J$ from the continuous gradient \eqref{20251006-GraidentOfJ} by discretizing a coupled system consisting of the state and adjoint equations, together with \eqref{eq:eta}. Theorem~\ref{202501008-theorem-BasicEquivalence} states that this discrete gradient can also be established by the Discre\-tize-Then-Optimize approach, where only the functional $J$ and the state equation \eqref{eq:state} are discretized. Thus, our numerical scheme owes the advantages of both approaches.

\vskip 5pt
  \item \textbf{Numerical experiments.}
Our numerical experiments provide consistent evidence that the proposed discretization delivers the intended structure-preserving behavior and robust optimization performance. Across the tested configurations, the fully discrete scheme preserves the targeted invariants to machine precision and maintains a tight state--adjoint consistency throughout the time-stepping. In addition, the computed optimal controls produce a marked faster decay of the chosen mix-norm than standard steady benchmarks, with the time-dependent optimized strategies exhibiting nearly exponential decay over significant time horizons.

\vskip 5pt
  \item \textbf{Practical implications for optimal mixing}
From an actuator-design perspective, these results point to a clear practical lesson: within a realistic low-dimensional actuation space, the decisive factor is not the sophistication of individual steady stirring modes, but the ability to combine them time-dependently. Each steady mode in isolation tends to lead, at best, to polynomial decay of the mix-norm. By contrast, an optimized time-dependent combination of the same limited set of modes can systematically break these barriers and yield near-exponential mixing rates. This highlights time modulation as a key design lever for effective mixing under actuation constraints.
\end{itemize}

\subsection{Related work}
\begin{itemize}
\item \textbf{Control in fluid mixing.} 
Control of transport and mixing of passive scalars via pure advection or stirring strategies has been extensively studied in the existing literature. In real-life applications, the velocity field is usually subject to certain constraints of physical interest, such as fixed flow energy, enstrophy, or fixed action, and the research very often leads to the design of open-loop mixing control inputs (e.g.,~\cite{franjione1992symmetry, sharma1997control, d1999control, vikhansky2002enhancement, vikhansky2002control, mathew2007optimal, cortelezzi2008feasibility, gubanov2010towards, couchman2010control, lin2011optimal,lunasin2012optimal, balasuriya2015dynamical, eggl2020mixing, eggl2022mixing}).

As discussed in Section~\ref{sec:problem}, it is common to restrict attention to finite-dimensio\-nal velocity models generated by a few actuators or reference mixers (e.g.\ \cite{aref1984stirring,mathew2007optimal,couchman2010control, gubanov2010towards,rodrigo2003optimization,paul2004handbook,aref2002development, wang2003closed,navarro2025exponential}). A classical example is Aref's ``batch stirring device'' (or ``blinking vortex'') model \cite{aref1984stirring,aref2002development}, widely used in chemical engineering, in which a small number of vortices are switched on and off to generate chaotic advection.

In our framework for mixing problems, the velocity field itself serves as the direct control input. Another approach of controlling its evolution through dynamical equations like the controlled Stokes and Navier--Stokes equations can be found in \cite{hu2018boundarycontrol, hu2018boundary}.

However, due to the nonlinear nature of this problem, deriving a closed-loop feedback law for effective flow advection is completely nontrivial.

It is also worth noting that the research integrating the controlled dynamics into the velocity field to optimize the transport and mixing processes has been investigated in (e.g.~\cite{balogh2004optimal, liu2008mixing, foures2014optimal, hu2018boundary, hu2018boundarycontrol,  eggl2022mixing, hu2023feedback, hu2019approximating, hu2020approximating}).

A recent work \cite{hu2025cellular} links the mixing problem to the Least Action Principle (LAP) for incompressible flows (e.g.~\cite{brenier1989least, brenier1993dual}), which is formally analogous to the Benamou--Brenier formulation of optimal transport \cite{Benamou2000monge}, but imposes an incompressibility constraint.

\item \textbf{Structure-preserving schemes.} In computational mathematics, there is broad recognition that numerical schemes should respect the intrinsic structures and invariants of physical systems. Schemes that exactly conserve quantities like mass, energy, or symplectic form can prevent unphysical drift and improve long-term stability.

For pure transport equations such as advection, numerous structure-preserving methods have been developed to maintain fundamental conservation laws. The finite-volume method (FVM) discretizes the integral form of a conservation law over control volumes, which guarantees strict local conservation of mass by construction \cite{LeVeque2002FVM}. High-order DG methods provide conservation together with sharp resolution of advective features; classical analyses include the unified DG framework \cite{ArnoldBrezziCockburnMarini2002SINUM, Shu1998DG} and hp-DG theory for advection--diffusion--reaction \cite{HoustonSchwabSuli2002SINUM}.

In PDE-constrained optimal control problems, state--adjoint consistency reflects the standard duality relation between the state and adjoint variables; see, for example, \cite{lions1971, diperna1989ordinary}. The numerical scheme must preserve this consistency at the discrete level.

Adjoint consistency means that the discrete adjoint is truly the ``transpose'' (with respect to the chosen inner product) of the discrete forward operator, so that duality pairings are preserved, and the computed gradient is exact for the discrete problem \cite{Lewis1998adjoint}.

Time-integration schemes that are time-symmetric (self-adjoint), such as Crank--Nicolson or the implicit midpoint rule, are particularly attractive in this regard. Related consistency-preserving discretizations can be found, for instance, in \cite{condi2010,stuck2017}.
\end{itemize}

\subsection{Organization of the paper}
The remainder of the paper is structured as follows.

In Section~\ref{sec:scheme}, we present our numerical method: first, a finite-volume spatial discretization combined with a second-order Crank--Nicolson time-stepping scheme for both the state equation \eqref{eq:state} and the adjoint equation \eqref{eq:adjoint}; then, a conjugate gradient method enhanced with a line-search strategy is applied to problem (OP), based on the implementation of the gradient \eqref{20251006-GraidentOfJ}.

Section~\ref{sec:property} is devoted to the theoretical analysis of the scheme's structure-preserving properties, where we rigorously prove discrete analogues of mass conservation, energy conservation, and state--adjoint consistency.

Section~\ref{sec:OPtimization-Discretization} establishes the equivalence between the Optimize-Then-Discretize and Dis\-cretize-Then-Optimize approaches for problem (OP).

In Section~\ref{sec:experiment}, we present numerical experiments demonstrating the effectiveness and accuracy of our scheme.

In Section~\ref{conclusions}, we summarize the main findings and propose some future research directions.

Finally, the Appendix~\ref{appendix} gathers some technical proofs.

\section{Numerical scheme}\label{sec:scheme}

This section is devoted to the numerical scheme (stated in Algorithm~\ref{alg:openloop}) for solving problem (OP), which preserves several structural properties (presented in Theorem~\ref{20251008-theorem-PerservedStructures}). This scheme is designed according to the following four-component workflow: (i) discretization of the state equation \eqref{eq:state}, (ii) computation of the Lagrange multiplier  \eqref{eq:eta}, (iii) discretization of the adjoint equation \eqref{eq:adjoint}, and (iv) computation of the derivative of the functional with respect to the velocity field, which, combined with the conjugate gradient method, is used to update this field and optimize its design. These four components are introduced in detail in the following subsections. 

\subsection{Full discretization of the state equation}\label{sec:space_discrete} 
We construct a numerical method for the state equation \eqref{eq:state} in two stages. First, we develop a finite-volume method (FVM) for spatial discretization. Second, we apply the Crank--Nicolson time discretization to obtain a fully discrete numerical method for equation \eqref{eq:state}. The fully discrete quantities depend on both the spatial mesh size $h$ and the time step $\Delta t$.
To keep the notation light, we write $\theta_h$ for the fully discrete density, with the dependence on $\Delta t$ understood implicitly. The same convention applies to all other discrete variables.

\textbf{(1) Finite-volume space discretization.} 
 The FVM is based on enforcing the conservation law on each \emph{control volume}. Let $\mathcal T_h$ be a partition of the physical domain $\Omega$ into cells $K$. Integrating the transport equation over an arbitrary cell $K$ and applying the divergence theorem yields the local balance relation
\begin{align}\label{20250930-yb-IntegralEquationOverControlVolume}
\frac{d}{dt}   \int_{ K } \theta(t)\,dx  \;+\;
    \int_{ \partial K }  (\mathbf{v}(t)\cdot \mathbf{n}_K)\,\theta(t)\, d \sigma
    \;=\;  0,
\end{align}
where $\mathbf{n}_K$ denotes the unit outward normal on $\partial K$. 
This identity states that the rate of change of the \emph{cell mass} $\int_K \theta$ is exactly balanced by the net advective flux through the cell boundary. 
The spatial discretization then proceeds by representing $\theta$ with cell averages $$\theta_K(t):=  \fint_{K} \theta (x) dx, ~K \in \mathcal T_h$$ and replacing the boundary integral in \eqref{20250930-yb-IntegralEquationOverControlVolume} by a sum of numerical fluxes across the faces of $K$. This leads to a conservative semi-discrete ODE system for the cell averages.

\vskip 5pt
\textit{Spatial mesh and its faces.}
Let $h>0$. 
Consider an admissible mesh $\cT_h $ of the mesh size $h$ for the physical domain $\Omega$, that is, a finite family of open, convex polygonal control volumes $\{K\}_{K\in\mathcal T_h}$ forming a conforming partition of $\Omega$. So that there exist constants $c_1,c_2 > 0$ such that 
\begin{align}\label{20251111-PropertiesOfMesh}
    c_1 h^d \leq |K|
    ~\text{and}~
    |\partial K | \leq  c_2  h^{d-1}
    ~\text{for each}~
    K  \in  \mathcal T_h
\end{align}
(see \cite[Definition\@ 24.1]{EymardFVM} for further details on admissible meshes). 
Let $|\mathcal T_h|$ be the total number of elements in $\mathcal T_h$. We define the sets of faces of a control volume $K \in \mathcal T_h$ and the mesh $\mathcal T_h$, respectively:
\begin{align*}
    \mathcal F_K := \big\{ 
        \Gamma \subset \partial K   ~:~  \Gamma \text{ is a face of } K  
    \big\}
    ~~\text{and}~~
     \mathcal{F}(\mathcal T_h) :=  \cup_{ K \in \mathcal T_h }  \mathcal F_K.
\end{align*}
In what follows, we will repeatedly use these definitions for the discrete velocity and the corresponding numerical flux.

\vskip 5pt
\textit{Discrete unknowns.}  According to  \eqref{20250930-yb-IntegralEquationOverControlVolume}, a simple way to approximate the solution of \eqref{eq:state} is to discretize it as a piecewise constant function over the elements of the mesh $\mathcal T_h$.  In line with this approach, we introduce the following state spaces for solutions of \eqref{eq:state} and for their discrete approximations:
\begin{align}\label{20251008-StateSpaces}
   X :=  L^2(\Omega)  ~\text{and}~    X_h := \mathbb R^{ |\mathcal T_h| },  
\end{align}
where $X_h$ is equipped with the inner product
\begin{align}\label{20251008-InnerProductOfXh}
\big\langle 
     (a_K)_{K\in \mathcal T_h}, (b_K)_{K\in \mathcal T_h}
\big\rangle_{X_h}
 := \sum_{K \in \mathcal T_h}   a_K b_K  |K|.
\end{align}
The space $X_h$ is regarded as a discretization of $X$ in the sense that each function in $X$ can be approximated by a vector in $X_h$. 
 
To approximate the solution $\theta(t)$ of \eqref{eq:state}, we represent it by an element of $X_h$, namely
\begin{align*}
    \theta_h(t) := ( \theta_K(t) )_{ K \in \mathcal T_h },
\end{align*}
which will be computed via a discretized transport equation. In particular, the discretization $( \theta_K^0 )_{ K \in \mathcal T_h } \in X_h$ of the initial data $\theta^0$ is given by
\begin{align}\label{20250930-DiscreteIntialData}
    \theta_K^0 :=  \fint_{K} \theta^0 (x) dx,  
    ~K \in \mathcal T_h.      
\end{align}

\vskip 5pt
\textit{Numerical flux.}  
We introduce the discretization of the flux (i.e., the boundary integral in \eqref{20250930-yb-IntegralEquationOverControlVolume}). By the no-penetration condition  $\mathbf{v} \cdot \mathbf {n}_{\Omega} =0$ on $\partial\Omega$, the boundary integral in \eqref{20250930-yb-IntegralEquationOverControlVolume} for a control volume $K$ is the sum of the fluxes across its interior faces, i.e.,
\begin{align*} 
    \int_{ \partial K }  (\mathbf{v}(t) \cdot \mathbf{n}_K) \theta(t) d \sigma
    =  \sum_{ \Gamma \in \mathcal F_K, \,\Gamma \not\subset \partial\Omega }  \int_{ \Gamma }  (\mathbf{v}(t) \cdot \mathbf{n}_K) \theta(t) d \sigma.
\end{align*}
At the same time, $\theta(t)$ is approximated by a vector $\theta_h(t)$ whose coordinates are correspondingly assigned to control volumes. Thus, we discretize the above boundary integral in the following way: 
\begin{align*}
\int_{ \partial K }  (\mathbf{v}(t) \cdot \mathbf{n}_K) \theta(t) d \sigma
=  \sum_{ \Gamma \in \mathcal F_K, \,\Gamma \not\subset \partial\Omega }   F_{K,\Gamma}( \mathbf{v}(t),  \theta_h(t) ), 
\end{align*}
where for each $\mathbf{u} \in U$ and $\mathbf{y}_h:=(y_K)_{K \in \mathcal T_h} \in X_h$, 
\begin{align}\label{20250930-FluxOnFace}
    F_{K,\Gamma} ( \mathbf{u},  \mathbf{y}_h  ) := 
    \frac{y_K + y_{L(K,\Gamma)} }{2} 
    \int_{ \Gamma }  \mathbf{u} \cdot \mathbf{n}_K d \sigma,
\end{align}
and $L(K,\Gamma)$ is the unique control volume in $\mathcal T_h$ sharing the same face $\Gamma$ with $K$. Here, we used the central flux for the approximation of the boundary integral. 
Furthermore, we note that when two control volumes $K$ and $L$ share the same face $\Gamma$, the numerical fluxes from one to the other cancel each other, i.e.,
\begin{align*}
    F_{K,\Gamma}  +  F_{L,\Gamma}  = 0. 
\end{align*}
This property is crucial for ensuring mass conservation of the numerical scheme.

\vskip 5pt
\textit{Semi-discrete state equation.}
For each $\mathbf{u} \in U$, we introduce the following numerical 
 form of $\text{div}\,(\mathbf{u}y)$ as an operator from $\mathbb R^{ | \mathcal T_h |}  $  to $ \mathbb R^{ | \mathcal T_h |}$: for each $\mathbf{y}_h \in \mathbb R^{ | \mathcal T_h |}$, 
\begin{align}\label{20251007-DiscreteDivergence}
    D_{ \mathbf{u} } (  \mathbf{y}_h ) := 
    \bigg(
        \frac{1}{ |K| }
         \sum_{ \Gamma \in \mathcal F_K,  \,\Gamma \not\subset \partial\Omega }   F_{K,\Gamma}( \mathbf{u},  \mathbf{y}_h  )
     \bigg)_{ K \in \mathcal T_h },
\end{align}
where $F_{K, \Gamma}$ is given by \eqref{20250930-FluxOnFace}. Note that the velocity fields are taken from $L^2((0,T);U)$ with $U$ given by \eqref{20251004-ControlConstraint}. The state equation \eqref{eq:state} is discretized as the following semi-discrete ODE with the discrete unknown $\theta_h$: 
\begin{align}\label{eq:semi_discrete}
  \frac{d}{dt}  \theta_h(t)  +  \sum_{i=1}^m v_i(t)  D_{ \mathbf{b}_i }  \theta_h(t)
  = 0,~ t \in (0,T);
  ~~ \theta_h(0) = (\theta^0_K)_{K \in \mathcal T_h}   \in   X_h,
\end{align}
where the velocity coefficients $(v_1,\ldots,v_m) \in  L^2((0,T); \mathbb R^m)$ and the discrete initial data $(\theta^0_K)_{K \in \mathcal T_h}$ are given by \eqref{20250930-DiscreteIntialData}. 

\begin{remark}[Exact flux conditions for structure preservation]\label{rem:exact_flux_conditions}
In the statements below, ``exact conservation'' means exact at the algebraic (machine-precision) level and holds under the following discrete conditions:
(i) \emph{single-valued interior fluxes:} for every interior face $\Gamma=\partial K\cap\partial L$,
$F_{K,\Gamma}(\mathbf u,\mathbf y_h)=-F_{L,\Gamma}(\mathbf u,\mathbf y_h)$;
(ii) \emph{discrete incompressibility:} $D_{\mathbf u}1=0$ (equivalently, zero net flux in each cell);
(iii) \emph{exact face flux evaluation:} the face integrals used to build $F_{K,\Gamma}$ are evaluated exactly (or enforced so that (i)-(ii) hold exactly).
If face integrals are approximated by quadrature or $\mathbf u$ is only approximately divergence-free, then (i)-(ii) hold up to a residual and the invariants below are preserved up to the corresponding quadrature/discrete-divergence error.
\end{remark}

\vskip 5pt
\textbf{(2) Crank--Nicolson time discretization.} 
Let $N_t$ be a positive integer. Consider the following time partition
\begin{align*}
    t_n := n \Delta t  ~(n = 0, 1, \ldots, N_t)
    \text{ with }
    \Delta t : = T / N_t. 
\end{align*} 
Applying the Crank--Nicolson (trapezoidal) rule to the semi-discrete equation \eqref{eq:semi_discrete}, we obtain the following fully discrete scheme for the state equation \eqref{eq:state} with unknowns $\{ \theta^{(n)}_h \}_{ n=0}^{N_t}$:
\begin{align}\label{eq:CN_cell}
\left\{
    \begin{array}{r@{~}l}
      \dfrac{\theta^{(n+1)}_h-\theta^{(n)}_h}{\Delta t}  
      + \sum_{i=1}^m v_i^{(n)}  D_{ \mathbf{b}_i } \dfrac{ \theta^{(n+1)}_h  +  \theta^{(n)}_h }{2}  =&  0,
      ~n = 0,  \ldots, N_t-1, 
      \vspace{0.5em}\\
      \theta^{(0)}_h =& (\theta^0_K)_{ K \in \mathcal T_h }  \in  X_h,
    \end{array}
\right.
\end{align}
where the discrete initial data $(\theta^0_K)_{K \in \mathcal T_h}$ is given by \eqref{20250930-DiscreteIntialData}, and the discrete velocity coefficients $\{ ( v_i^{(n)})_{i=1}^m \}_{n=0}^{N_t-1}  \subset  \mathbb R^m$.

Under standard mesh-regularity and consistency assumptions, this cell-centered finite-volume discretization combined with a Crank-Nicolson time integrator yields a convergent approximation of equation \eqref{eq:state} (see Appendix~\ref{Appendix-state_convergence}). Moreover, for sufficiently regular solutions, the underlying cell-centered finite-volume discretization is first-order accurate in the mesh size $h$ (see \cite[Remark\@ 6.2]{EymardFVM}), while the Crank--Nicolson integrator is second-order accurate in the time step $\Delta t$ (see \cite[Eq.\@ (2.85)]{MortonMayersPDE}).

It is worth noting that the Crank--Nicolson finite-volume scheme \eqref{eq:CN_cell} applied to the state equation \eqref{eq:state} is \emph{symmetric} (time-reversible): \eqref{eq:CN_cell} remains unchanged if we interchange $\theta_h^{(n+1)}$ and $\theta_h^{(n)}$ and replace $\Delta t$ by $-\Delta t$.
This observation motivates our choice to discretize the adjoint equation \eqref{eq:adjoint} in the same manner, as introduced in Subsection~\ref{subsection:DiscreteAdjointEquation} below.

\begin{remark}[Why not an upwind form]
Since our goal is a \emph{structure-preserving} discretization for a PDE-constrained mixing-control problem, we deliberately avoid upwind fluxes and instead employ a centered finite-volume flux coupled with Crank--Nicolson time stepping. The main reasons are as follows.

(i) \emph{$L^2$ conservation.} With a centered FVM discretization combined with the Crank--Nicolson scheme, our method preserves both the discrete energy and the state--adjoint pairing at each time step. In contrast, upwind fluxes introduce numerical diffusion, destroy skew-adjointness, and lead to $L^2$-energy decay; see, e.g., \cite[Sec.\@ 4.4.2.2]{Durranbook}.

(ii) \emph{Scheme complexity.} For an upwind spatial discretization, the algebraic discrete adjoint requires a different non-symmetric flux and special boundary treatment, so the forward and adjoint solvers cannot share a single simple scheme. In our approach, we use the same centered Crank--Nicolson discretization for both state and adjoint equations, retaining the conservative structure needed for accurate gradient computation while keeping the overall algorithmic design considerably simpler.
\end{remark}

\begin{remark}[Spurious oscillations]
Centered advection schemes are non-dissipative and can, in principle, generate Gibbs-like oscillations near sharp gradients. In our simulations, however, we did not observe problematic oscillations. The initial scalar profiles we consider in the numerical experiments are resolved on a fine grid, which minimizes any overshoot. Moreover, the  ``mix-norm'' is relatively insensitive to small-scale wiggles in $\theta$, since it emphasizes large-scale concentration differences. Thus, any minor oscillations at the grid scale do not noticeably pollute the mix-norm or its decay rate in our results.
\end{remark}

\subsection{Computation of the Lagrange multiplier}

In this subsection, we aim to solve \eqref{eq:eta} numerically by the FVM discretizations introduced in Section~\ref{sec:space_discrete} and adapted from \cite[Section\@ 10]{EymardFVM}. For this purpose, we integrate \eqref{eq:eta} over each control volume to get the following integral form
\begin{align*}
    - \int_{\partial K} \nabla \eta  \cdot \mathbf{n}_K  d\sigma
= \int_K  \theta(T;\mathbf{v}) dx,
~ \forall\, K \in \mathcal T_h.
\end{align*}
Write $X_h^0$ for the following subspace of $X_h$:
\begin{align}\label{20251208-SubspaceOfXh}
    X_h^0 := \Big\{  (f_K)_{ K \in \mathcal T_h }   \in X_h  
        ~:~  \sum_{K \in \mathcal T_h} f_{K} |K| = 0
    \Big\}.
\end{align}
The aforementioned integral form is then discretized as the following linear equations with  unknowns $\eta_h = ( \eta_{h,K} )_{ K \in \mathcal T_h }$ in $X_h^0$:
\begin{align}\label{20251011-DiscreteLaplacian}
L_h \eta_h := \left(
     - \frac{1}{ |K| }
    \sum_{ 
     \Gamma\in\mathcal F_K, \, \Gamma\not\subset\partial\Omega
    }
    \frac{ |\Gamma| }{ d_{K,L(K,\Gamma)} } \Big( \eta_{h,L(K,\Gamma)} - \eta_{h,K} \Big)
\right)_{ K \in \mathcal T_h}
= \theta_h^{(N_t)}
\text{ in }  X_h^0,
\end{align}
where $d_{K,L}$ denotes the distance between the centers of $K$ and $L$, and  $\theta_h^{(N_t)}$ is the solution to equation \eqref{eq:CN_cell} at time step $N_t$. The operator $L_h$ is invertible and self-adjoint over $X_h^0$ (see Lemma \ref{20251209-lemma-PropertiesOfLh} in Appendix~\ref{Appendix-DiscreteEllipticEquation}). 
By solving the linear equations \eqref{20251011-DiscreteLaplacian} by a numerical solver (e.g. conjugate-gradient (CG) method), $\eta_h$ is obtained.

\subsection{Backward discretization of the adjoint equation.}
\label{subsection:DiscreteAdjointEquation}
The adjoint equation \eqref{eq:adjoint} has the same form as the state equation 
\eqref{eq:state}, except that the initial data is imposed at time $t=T$. Meanwhile, the Crank--Nicolson rule used for the state equation \eqref{eq:state}  is symmetric, as mentioned above. 
Thus, we can apply the same numerical scheme to the adjoint equation \eqref{eq:adjoint}. More precisely, by exchanging $\theta^{(n+1)}_h$ and $\theta^{(n)}_h$ and replacing $\Delta t$ by $-\Delta t$ in \eqref{eq:CN_cell}, we obtain the following fully discrete equation for the adjoint equation \eqref{eq:adjoint} with unknowns $\{ \rho^{(n)}_h \}_{ n=0}^{N_t}$: 
\begin{align}\label{20251004-DiscreteAdjointEquation}
\left\{
    \begin{array}{r@{~}l}
      \dfrac{\rho^{(n)}_h - \rho^{(n+1)}_h}{ - \Delta t}  
        + \sum_{i=1}^m v_i^{(n)}  D_{ \mathbf{b}_i } \dfrac{ \rho^{(n)}_h  +  \rho^{(n+1)}_h }{2}  
        =&  0,       ~n = N_t-1, \ldots,1,0,  
       \\
       \vspace{0.5em}
      \rho^{(N_t)}_h =&  \eta_h  \in  X_h, 
    \end{array}
\right.
\end{align}
where $\eta_h$ is obtained from \eqref{20251011-DiscreteLaplacian}. 
This equation is essentially the same as \eqref{eq:CN_cell}, except for the different notations for unknowns and the initial data.
Starting from the terminal condition $\rho^{(N_t)}$, we solve \eqref{20251004-DiscreteAdjointEquation} step-by-step for $n=N_t-1,\dots,1,0$. 

\begin{remark}[Algebraic adjoint interpretation]\label{rem:alg_adj}
Let the forward Crank--Nicolson update be written as $\theta_h^{(n+1)}=A_n\theta_h^{(n)}$, where
\[
A_n=(I_h+D^{(n)})^{-1}(I_h-D^{(n)}),
\qquad
D^{(n)}:=D_{0.5\,\Delta t\sum_{i=1}^m v_i^{(n)}\mathbf b_i}.
\]
Under the antisymmetry of $D^{(n)}$ with respect to $\langle\cdot,\cdot\rangle_{X_h}$ (Lemma~\ref{20251008-lemma-AntisymmetricDiscreteDivergence}),
$A_n$ is an isometry on $X_h$, hence $A_n^\ast=A_n^{-1}$.
Therefore, applying the same CN scheme backward in time coincides with using the algebraic adjoint (transpose) of the forward timestep map.
\end{remark}

\subsection{Numerical scheme for problem (OP)}\label{subsec:NumericalScheme}
Recall that each admissible velocity field $\mathbf{v}$ is of the finite-dimensional form \eqref{eq:v_finite}. It is therefore equivalent, and convenient, to formulate problem (OP) directly in terms of the coefficient functions relative to the basis $\{\mathbf{b}_i\}_{i=1}^m$. To this end, we introduce the coefficient-space functional
\begin{align*}
    \widetilde J\big((v_i)_{i=1}^m\big)
    \;:=\;
    J\!\left(\sum_{i=1}^m v_i(t)\,\mathbf{b}_i\right),
    \qquad
    (v_i)_{i=1}^m \in L^2\big((0,T);\mathbb{R}^m\big),
\end{align*}
and compute the gradient of $\widetilde J$ with respect to the velocity coefficients. It follows  
from \eqref{20251006-GraidentOfJ}  that for each $ \mathbf{v} = \sum_{i=1}^m v_i \mathbf{b}_i \in L^2( (0,T); U)$, 
\begin{align*}
    \nabla  \widetilde J( (v_i)_{i=1}^m )   =
    \begin{pmatrix}
         \langle \gamma \mathbf{v} + \theta \nabla \rho , \mathbf{b}_1\rangle_{L^2(\Omega; \mathbb R^d)}\\
        \vdots   \\
        \langle \gamma \mathbf{v} +  \theta \nabla \rho , \mathbf{b}_m  \rangle_{L^2(\Omega; \mathbb R^d)}
    \end{pmatrix}
    =
    \begin{pmatrix}
        \gamma v_1 
        + \langle \theta, \text{div}(\mathbf{b}_1 \rho) \rangle_{L^2(\Omega)}\\
        \vdots   \\
        \gamma v_m 
        + \langle \theta, \text{div}(\mathbf{b}_m \rho) \rangle_{L^2(\Omega)}
    \end{pmatrix}.
\end{align*}

This allows us to introduce the discrete version of the gradient as follows: for discrete velocity coefficients $\{ ( v_i^{(n)})_{i=1}^m \}_{n=0}^{N_t-1}  \subset \mathbb R^m$, 
\begin{align}\label{cost_cell}
    (\nabla \widetilde J)_h \Big( 
        \{ ( v_i^{(n)})_{i=1}^m \}_{n=0}^{N_t-1}
    \Big)=&  
    \left\{  \Delta t 
    \begin{pmatrix}
          \gamma  v_1^{(n)}  
        + \big\langle 
            \frac{ \theta^{(n+1)}_h  +  \theta^{(n)}_h }{2}, D_{\mathbf{b}_1}   \rho_h^{(n)} 
        \big\rangle_{ X_h }   \\
            \vdots   \\
         \gamma v_m^{(n)}  
        + \big\langle 
            \frac{ \theta^{(n+1)}_h  +  \theta^{(n)}_h }{2}, D_{\mathbf{b}_m}   \rho_h^{(n)} 
        \big\rangle_{ X_h }
    \end{pmatrix}
    \right\}_{n=0}^{N_t-1},
\end{align}
where $D$ is given by \eqref{20251007-DiscreteDivergence}. 

It is worth noting that $\theta$ and $\rho$ are discretized in \eqref{cost_cell} using slightly different time conventions: the state $\theta$ enters through the midpoint average $(\theta_h^{(n+1)} + \theta_h^{(n)})/2$, which is naturally associated with the time interval $(t_n,t_{n+1})$, whereas the adjoint $\rho$ appears with its nodal value $\rho_h^{(n)}$ at time $t_n$. This choice ensures that the numerical gradient defined above coincides with the gradient of the discrete functional $\widetilde{J}_h$ introduced in \eqref{20251008-DiscreteFunctional}, as stated in Theorem~\ref{202501008-theorem-BasicEquivalence}. 

In parallel with the discrete state and adjoint equations, we now introduce the discrete cost functional that defines the fully discrete optimization problem
\begin{align}\label{20251008-DiscreteFunctional}
\widetilde J_h ( \mathbf{v}_h )
:= \frac{1}{2} \langle \theta_h^{(N_t)}, \eta_h \rangle_{X_h}
+ \frac{\gamma}{2}  \Delta t  \sum_{n=0}^{N_t-1} \sum_{i=1}^m  | v_i^{(n)}|^2,
~
\mathbf{v}_h := \{ (v_i^{(n)})_{i=1}^m \}_{n=0}^{N_t-1} \subset \mathbb R^m,
\end{align}
where $\theta_h^{(N_t)}$ and $\eta_h$ are computed from equations \eqref{eq:CN_cell} and \eqref{20251011-DiscreteLaplacian}, respectively. 

The convergence of the discrete cost functional $\widetilde J_h$ toward $J$ is proved in Appendix~\ref{Appendix-ConvDiscreteObjective}. 
In the fully discrete setting, the admissible set is finite-dimensional and closed, so the existence of a global minimizer follows from standard compactness/continuity arguments. 
This provides a rigorous justification for minimizing $\widetilde J_h$ as a reliable route to approximating optimal continuous mixers.

Next, to compute the optimal discrete velocity coefficients
$\{(v_i^{(n),*})_{i=1}^m\}_{n=0}^{N_t-1}$ for problem (OP), we rely on the following components:
(i) a conjugate gradient (CG) method,
(ii) a line-search strategy detailed in our recent work \cite[Algorithm\@ 3.9]{zheng2023numerical}, and
(iii) the discrete gradient $(\nabla \widetilde J)_h$ given by \eqref{cost_cell}, which is evaluated from the discrete state and adjoint systems \eqref{eq:CN_cell} and \eqref{20251004-DiscreteAdjointEquation}, together with the discrete approximation of $\eta$ in \eqref{20251011-DiscreteLaplacian}. 

Once the optimal coefficients are obtained, the corresponding optimal velocity field $\mathbf{v}^*$ for (OP) is reconstructed via
    \begin{align*}
         v^*(t)  \approx \sum_{n=0}^{N_t-1}  \sum_{i=1}^m \chi_{[nT/N_t, \,(n+1)T/N_t)}(t) 
         ~v_i^{(n),*}  \mathbf{b}_i,
         ~t \in (0,T).
    \end{align*}

With the above-mentioned elements, the complete framework of the numerical scheme for problem (OP)  is summarized in Algorithm~\ref{alg:openloop}, which preserves good structures,  introduced in Section~\ref{sec:property}.


\begin{algorithm}
\caption{Computation of an optimal velocity field}\label{alg:openloop}
\begin{algorithmic}[1]
    \State \textbf{Input:} Initial data $\theta^0$, parameters $T,\gamma>0$, mesh size $h>0$, integer $N_t>0$, line-search parameter $c>0$, and stopping tolerance $\varepsilon$.
    \State \textbf{Output:} Optimal discrete velocity coefficients $\{\mathbf{v}^{(n),k}\}_{n=0}^{N_t-1}\subset\mathbb{R}^m$.
    \State \textbf{Initialize:} Generate a mesh $\mathcal{T}_h$ of size $h$ and initial discrete velocity coefficients $\{\mathbf{v}^{(n),0}\}_{n=0}^{N_t-1}\subset\mathbb{R}^m$. Set the initial discrete data $(\theta_K^0)_{K\in\mathcal{T}_h}$ by \eqref{20250930-DiscreteIntialData}, set $k:=0$, and define the initial search direction
    \[
        d^0 := -(\nabla \widetilde J)_h\Big(\{\mathbf{v}^{(n),0}\}_{n=0}^{N_t-1}\Big)
        \qquad\text{(according to \eqref{cost_cell}).}
    \]

    \Repeat
        \State \textbf{(1) State solve.} Solve the discrete state equation \eqref{eq:CN_cell} with $(v_i^{(n)})_{i=1}^m$ replaced by $\mathbf{v}^{(n),k}$ to obtain $\{\theta_h^{(n)}\}_{n=0}^{N_t}$.
        \State \textbf{(2) Multiplier update.} Compute the Lagrange multiplier $\eta_h^k$ for \eqref{eq:eta} (with $\theta(T;\mathbf{v})$ replaced by $\theta_h^{(N_t)}$) according to \eqref{20251011-DiscreteLaplacian}.
        \State \textbf{(3) Adjoint solve.} Solve the discrete adjoint equation \eqref{20251004-DiscreteAdjointEquation} with $((v_i^{(n)})_{i=1}^m,\eta_h)$ replaced by $(\mathbf{v}^{(n),k},\eta_h^k)$.
        \State \textbf{(4) CG update with line search.}
        \begin{enumerate}[(i)]
            \item Compute the discrete gradient
            \[
                (\nabla \widetilde J)_h\Big(\{\mathbf{v}^{(n),k}\}_{n=0}^{N_t-1}\Big)
            \]
            according to \eqref{cost_cell} using the results of Steps \textbf{(1)} and \textbf{(3)}.
            \item Apply the line-search algorithm with $\{\mathbf{v}^{(n),k}\}_{n=0}^{N_t-1}$ and $d^k$ to obtain $\alpha^k>0$ and update
            \[
                \{\mathbf{v}^{(n),k+1}\}_{n=0}^{N_t-1}
                :=
                \{\mathbf{v}^{(n),k}\}_{n=0}^{N_t-1} + \alpha^k d^k.
            \]
            Compute $\widetilde{J}_h\big(\{\mathbf{v}^{(n),k+1}\}_{n=0}^{N_t-1}\big)$ as in \eqref{20251008-DiscreteFunctional}, where $\alpha^k$ is chosen such that
            \[
                \widetilde{J}_h\Big(\{\mathbf{v}^{(n),k+1}\}_{n=0}^{N_t-1}\Big)
                \le
                \widetilde{J}_h\Big(\{\mathbf{v}^{(n),k}\}_{n=0}^{N_t-1}\Big)
                + c\,\alpha^k \Big\langle \nabla \widetilde{J}_h^k, d^k \Big\rangle.
            \]
            \item Update the search direction
            \[
                d^{k+1}
                :=
                -(\nabla \widetilde J)_h\Big(\{\mathbf{v}^{(n),k+1}\}_{n=0}^{N_t-1}\Big)
                + \beta^{k} d^{k},
            \]
            with
            \[
                \beta^{k}
                :=
                \frac{
                    \left\|(\nabla \widetilde J)_h\Big(\{\mathbf{v}^{(n),k+1}\}_{n=0}^{N_t-1}\Big)\right\|
                }{
                    \left\|(\nabla \widetilde J)_h\Big(\{\mathbf{v}^{(n),k}\}_{n=0}^{N_t-1}\Big)\right\|
                }.
            \]
        \end{enumerate}
        \State \textbf{(5)} Set $k \leftarrow k+1$.
    \Until{
        $\displaystyle
        \frac{
            \left|
            \widetilde{J}_h\big(\{\mathbf{v}^{(n),k}\}_{n=0}^{N_t-1}\big)
            -
            \widetilde{J}_h\big(\{\mathbf{v}^{(n),k-1}\}_{n=0}^{N_t-1}\big)
            \right|
        }{
            \widetilde{J}_h\big(\{\mathbf{v}^{(n),k-1}\}_{n=0}^{N_t-1}\big)
        }
        < \varepsilon.$
    }
\end{algorithmic}
\end{algorithm}

\section{Structure-preserving properties of the numerical scheme}\label{sec:property}

This section is devoted to the structural properties preserved by the numerical scheme introduced in Section~\ref{sec:scheme}.

We first define the discrete mass $M_h$ and discrete $L^2$-energy $E_h$ as follows: 
Given
    $\xi_h = (\xi_{h,K})_{K\in\mathcal T_h} \in X_h$, we set
    \begin{align}\label{eq:discrete}
        M_h(\xi_h) := \sum_{K\in\mathcal T_h} \xi_{h,K}\,|K|,
        \qquad
        E_h(\xi_h) := \|\xi_h\|_{X_h}^2.
    \end{align}

The following holds.

\begin{thm}\label{20251008-theorem-PerservedStructures}
Assume that the fluxes satisfy the exact conditions in Remark~\ref{rem:exact_flux_conditions}.
Let $\theta^{(n)}_h $ and $\rho^{(n)}_h $ be the solutions  at time level $t_n$ ($n \in \{0,1,\ldots, N_t\}$) to equations \eqref{eq:CN_cell} and \eqref{20251004-DiscreteAdjointEquation} with arbitrary initial data $\theta_h^{(0)},\rho_h^{(N_t)} \in X_h$. 

Then, for each time step $n$, the following properties hold:
\begin{align}
    &\text{(Mass conservation)} \quad 
        M_h(\theta_h^{(n)})=M_h(\theta_h^{(0)})
        \text{ and }
        M_h(\rho_h^{(n)})=M_h(\rho_h^{(N_t)})
    \label{20251008-MassConservation}
    \\
    &\text{(Energy conservation)} \quad
        E_h(\theta_h^{(n)})=E_h(\theta_h^{(0)})
        \text{ and }
        E_h(\rho_h^{(n)})=E_h(\rho_h^{(N_t)})
    \label{20251008-EnergyConservation}
    \\
    &\text{(State--adjoint consistency)}   \quad
        \big\langle \theta^{(n)}_h, \rho^{(n)}_h \big\rangle_{X_h} = 
        \big\langle \theta^{(N_t)}_h, \rho^{(N_t)}_h \big\rangle_{X_h}.
    \label{20251008-ForwardBackwardConsistency}
\end{align}

\end{thm}

The proof of Theorem~\ref{20251008-theorem-PerservedStructures} requires the following lemma, which presents a basic property of the discrete divergence $D_{ \mathbf{u} }$ in  \eqref{20251007-DiscreteDivergence} (see Appendix~\ref{sec:proof-lemma-AntisymmetricDiscreteDivergence} for the proof).

\begin{lem}\label{20251008-lemma-AntisymmetricDiscreteDivergence}
Let $X_h$ be given by \eqref{20251008-StateSpaces},  $\mathbf{u} \in U$ and $D_{ \mathbf{u} }$ be as in  \eqref{20251007-DiscreteDivergence}. 
Assume that the fluxes satisfy the exact conditions in Remark~\ref{rem:exact_flux_conditions}.
Then, $D_{ \mathbf{u} }$ is antisymmetric over $X_h$  (equipped with the inner product in \eqref{20251008-InnerProductOfXh}). 
Furthermore,  $I_h +  D_{ \mathbf{u} }$ is invertible over $X_h$, where $I_h$ is the identity operator over $X_h$. 
\end{lem}

We are now in the position to prove Theorem~\ref{20251008-theorem-PerservedStructures}. 
\begin{proof}[Proof of Theorem~\ref{20251008-theorem-PerservedStructures}]
Let $n \in \{0,1,\ldots, N_t\}$ be fixed. 
The proof is organized in the following three steps.

\vskip 5pt
\noindent\textit{
Step 1. Proof of \eqref{20251008-ForwardBackwardConsistency}
}

Let $I_h$ be the identity operator over $X_h$. Write
\begin{align*}
D^{(k)} : =  D_{0.5 \Delta t \sum_{i=1}^m  v_i^{(k)} \mathbf{b}_i}
    = 0.5 \Delta t \sum_{i=1}^m  v_i^{(k)} D_{\mathbf{b}_i},
    ~ k=0,\ldots, N_t-1.
\end{align*}
From \eqref{eq:CN_cell}, we know that for each $k \in \{0,\ldots, N_t-1\}$, 
\begin{align*}
\big( I_h +  D^{(k)} \big)
\theta_h^{(k+1)} =  
\big( I_h -  D^{(k)} \big)
\theta_h^{(k)}.
\end{align*}
Since each operator $I + D^{(k)}$ is invertible and antisymmetric over $X_h$ (by Lemma \ref{20251008-lemma-AntisymmetricDiscreteDivergence} with $\mathbf{u}$ replaced by $0.5 \Delta t \sum_{i=1}^m  v_i^{(k)} \mathbf{b}_i$), we rewrite the above equality as
\begin{align}\label{20251008-ProofOfConsistency-1}
\theta_h^{(k+1)} =
\big( I_h +  D^{(k)} \big)^{-1}
\big( I_h -  D^{(k)} \big)
\theta_h^{(k)},
~ k=0,\ldots, N_t-1.
\end{align}
Because each $D^{(k)}$ is antisymmetric (see Lemma \ref{20251008-lemma-AntisymmetricDiscreteDivergence}), the above implies that for each $k \in \{0,\ldots, N_t-1\}$, 
\begin{align*}
\langle \theta_h^{(k+1)},  \rho_h^{(k+1)}  \rangle_{X_h}  =
\Big\langle  \theta_h^{(k)}, 
\big( I_h -  D^{(k)} \big)^{-1}
\big( I_h +  D^{(k)} \big)
\rho_h^{(k+1)}
\Big\rangle_{X_h}.
\end{align*}
Meanwhile, $\{ \rho^{(k)}_h \}_{k=0}^{N_t}$ satisfies the same discrete transport equation as $\{ \theta^{(k)}_h \}_{k=0}^{N_t}$ (see \eqref{eq:CN_cell} and \eqref{20251004-DiscreteAdjointEquation}). Then, by replacing $\theta$ with $\rho$ in \eqref{20251008-ProofOfConsistency-1}, we determine from the last equality  that
\begin{align*}
\langle \theta_h^{(k+1)},  \rho_h^{(k+1)}  \rangle_{X_h}  =
\Big\langle  \theta_h^{(k)}, \rho_h^{(k)}
\Big\rangle_{X_h},
~ k=0,\ldots, N_t-1.
\end{align*}
This gives \eqref{20251008-ForwardBackwardConsistency}. 

\vskip 5pt
\noindent\textit{
Step 2. Proof of \eqref{20251008-EnergyConservation}
}

By Step 1, we know that \eqref{20251008-ForwardBackwardConsistency} holds for any solutions to equations \eqref{eq:CN_cell} and \eqref{20251004-DiscreteAdjointEquation} (with any initial data). Meanwhile, equations \eqref{eq:CN_cell} and \eqref{20251004-DiscreteAdjointEquation} are the same except for their initial data. Thus, we can take $\rho_h^{(n)} = \theta_h^{(k)}$ ($k=0,1,\cdots, N_t$) in \eqref{20251008-ForwardBackwardConsistency}. This, together with \eqref{eq:discrete}, leads to \eqref{20251008-EnergyConservation} directly.

\vskip 5pt
\noindent\textit{
Step 3. Proof of  \eqref{20251008-MassConservation}
}
Note that the constant function $y_K^{(k)}\equiv 1$ is an exact discrete solution since $D_{\mathbf u}1=0$ by Remark~\ref{rem:exact_flux_conditions}.
Meanwhile, equations \eqref{eq:CN_cell} and \eqref{20251004-DiscreteAdjointEquation} are the same except their initial data. Thus, we  take $\rho_h^{(k)} \equiv 1$  (resp. $\theta_h^{(k)} \equiv 1$)   in \eqref{20251008-ForwardBackwardConsistency} to obtain that when $k=0,1,\cdots, N_t$,
\begin{align*}
    \langle \theta_h^{(k)}, 1 \rangle_{X_h}  = \langle \theta_h^{(0)}, 1 \rangle_{X_h}
    ~(\text{resp. } \langle \rho_h^{(k)}, 1 \rangle_{X_h}  = \langle \rho_h^{(N_t)}, 1 \rangle_{X_h}).
\end{align*}
 The last equalities, along with \eqref{20251008-InnerProductOfXh}, yield  \eqref{20251008-MassConservation}. 

\vskip 5pt
In summary, the proof of Theorem~\ref{20251008-theorem-PerservedStructures} is completed. 
\end{proof}

\section{Commutativity of Optimization and Discretization}
\label{sec:OPtimization-Discretization}

The numerical scheme in Section~\ref{subsec:NumericalScheme} is designed by the Optimize-Then-Discretize approach. This means that the discrete gradient \eqref{cost_cell} of the function $J$  for problem (OP) is obtained from the original gradient  \eqref{20251006-GraidentOfJ} by discretizing a coupled system of the state and adjoint equations. The following Theorem~\ref{202501008-theorem-BasicEquivalence} states that this discrete gradient can also be established only from the discretizations of the functional $J$ and the state equation. In other words, this theorem presents the equivalence between the  Optimize-Then-Discretize and Discretize-Then-Optimize approaches for the problem (OP).

\begin{thm}\label{202501008-theorem-BasicEquivalence}
Let $(\nabla \widetilde J)_h$ and $\widetilde J_h$ be given by \eqref{cost_cell}  and \eqref{20251008-DiscreteFunctional}, respectively. Then, it holds that $(\nabla \widetilde J)_h  \equiv \nabla \widetilde J_h$  over $\mathbb R^{ m\times N_t}$. 
\end{thm}

\begin{proof}
We will compute $\nabla \widetilde{J}_h$ at any fixed velocity coefficients $v_h = \{ (v_i^{(n)})_{i=1}^m \}_{n=0}^{N_t-1}$ in $\mathbb R^{m\times N_t}$. For this purpose, we take a variational direction $u_h = \{ (u_i^{(n)})_{i=1}^m \}_{n=0}^{N_t-1}$ in $\mathbb R^{m\times N_t}$,  
and define new velocity coefficients in $\mathbb R^{m\times N_t}$ as follows:
\begin{align}\label{20251009-NewControl}
    v_h^{\varepsilon} :=  v_h +  \varepsilon u_h,
    ~\varepsilon \in \mathbb R. 
\end{align}
Write $\{ \delta\theta_h^{(n)} \}_{n=0}^{N_t} $ for the solution to the following variational equation:
\begin{align}\label{20251009-VaritalDiscreteEquation}
\left\{
    \begin{array}{r@{~}l}
      \dfrac{\delta\theta^{(n+1)}_h - \delta\theta^{(n)}_h}{\Delta t}  
      +& \sum_{i=1}^m v_i^{(n)}  D_{ \mathbf{b}_i } \dfrac{ \delta\theta^{(n+1)}_h  +  \delta\theta^{(n)}_h }{2}  
      \vspace{0.5em}\\
      \quad\quad &=  - \sum_{i=1}^m u_i^{(n)}  D_{ \mathbf{b}_i } \dfrac{ \theta^{(n+1)}_h  +  \theta^{(n)}_h }{2},
      ~n = 0,  \ldots, N_t-1, 
      \vspace{0.5em}\\
      \delta\theta^{(0)}_h &= 0  \text{ in }  X_h,
    \end{array}
\right.
\end{align}
Write $\eta_h$ (resp. $\delta\eta_h$) for the solutions to equation \eqref{20251011-DiscreteLaplacian} with $\theta_h^{(N_t)}$ (resp. $\delta\theta_h^{(N_t)}$). 
Then, by the definition of $\widetilde J_h$ in \eqref{20251008-DiscreteFunctional},  we obtain 
\begin{align}\label{20251009-VaritationOfDiscreteFunctional}
\lim_{ \varepsilon \rightarrow 0}
\frac{1}{\varepsilon} \Big( 
    \widetilde J_h(v_h^{\varepsilon})  -  \widetilde J_h(v_h)
\Big)
= \frac{1}{2} \big\langle 
    \theta_h^{(N_t)}, \delta\eta_h 
\big\rangle_{X_h}
+  \frac{1}{2} \big\langle 
    \delta\theta_h^{(N_t)}, \eta_h 
\big\rangle_{X_h}
 + \gamma \Delta t \langle u_h, v_h \rangle_{ \mathbb R^{m\times N_t} }.
\end{align}
Meanwhile, we determine from \eqref{20251011-DiscreteLaplacian} and Lemma \ref{20251209-lemma-PropertiesOfLh} in Appendix~\ref{Appendix-DiscreteEllipticEquation} that 
\begin{align*}
 \big\langle 
    \theta_h^{(N_t)}, \delta\eta_h
\big\rangle_{X_h}
=& \big\langle 
    \theta_h^{(N_t)}, L_h^{-1} \delta\theta_h^{(N_t)}
\big\rangle_{X_h}
= \big\langle 
   (L_h^{-1})^* \theta_h^{(N_t)}, \delta\theta_h^{(N_t)}
\big\rangle_{X_h}
\nonumber\\
=& \big\langle 
   L_h^{-1} \theta_h^{(N_t)}, \delta\theta_h^{(N_t)}
\big\rangle_{X_h}
=   \big\langle 
   \eta_h, \delta\theta_h^{(N_t)}
\big\rangle_{X_h}.
\end{align*}
This, along with \eqref{20251009-VaritationOfDiscreteFunctional} and \eqref{20251009-NewControl}, yields
\begin{align*}
\nabla\widetilde J_h(v_h) u_h
= \big\langle 
    \delta\theta_h^{(N_t)}, \eta_h
\big\rangle_{X_h}
+ \gamma \Delta t  \langle u_h, v_h \rangle_{ \mathbb R^{m\times N_t} }.
\end{align*}
Combined with equations \eqref{20251009-VaritalDiscreteEquation} and \eqref{20251004-DiscreteAdjointEquation} (as well as Lemma \ref{20251008-lemma-AntisymmetricDiscreteDivergence} with $\mathbf{u}$ replaced by $\mathbf{b}_i$), the last equality yields
\begin{align*}
\nabla\widetilde J_h(v_h) u_h
=& \sum_{n=0}^{N_t-1} 
\Big\langle
- \Delta t \sum_{i=1}^m u_i^{(n)}  D_{ \mathbf{b}_i } \frac{ \theta^{(n+1)}_h  +  \theta^{(n)}_h }{2},
\rho_h^{(n)}
\Big\rangle_{X_h}
+ \gamma \Delta t  \langle u_h, v_h \rangle_{ \mathbb R^{m\times N_t} }
\nonumber\\
=& \Delta t  \sum_{n=0}^{N_t-1} \sum_{i=1}^m 
u_i^{(n)}
\Big\langle
    \frac{ \theta^{(n+1)}_h  +  \theta^{(n)}_h }{2},
D_{ \mathbf{b}_i } \rho_h^{(n)}
\Big\rangle_{X_h}
+ \gamma  \Delta t \langle u_h, v_h \rangle_{ \mathbb R^{m\times N_t} }.
\end{align*}
Since $u_h = \{ (u_i^{(n)})_{i=1}^m \}_{n=0}^{N_t-1} $ was arbitrarily taken from $\mathbb R^{m\times N_t}$, the above implies
\begin{align*}
\nabla\widetilde J_h(v_h) 
= \bigg\{  \Big( \Delta t 
\Big\langle
    \frac{ \theta^{(n+1)}_h  +  \theta^{(n)}_h }{2},
D_{ \mathbf{b}_i } \rho_h^{(n)}
\Big\rangle_{X_h}
\Big)_{i=1}^m
\bigg\}_{n=0}^{N_t-1}
+ \gamma \Delta t  v_h.
\end{align*}
From this and \eqref{cost_cell}, we conclude that 
$ \nabla \widetilde J_h = (\nabla \widetilde J)_h$ over $\mathbb R^{ m\times N_t}$. 
This completes the proof. 
\end{proof}

\section{Numerical experiments}\label{sec:experiment}

We now present numerical tests to verify the structure-preserving properties of our scheme and to demonstrate its performance in enhancing mixing. The experiments are divided into two parts: (i) velocity controls using the cellular flow basis in a polygon domain, and (ii) velocity controls using the Doswell frontogenesis basis in a circular domain. In each case, we first examine the evolution of the mix-norm and scalar field under a single basis flow (without any optimization) to establish a baseline. We then present the optimal control results obtained by using combinations of the basis flows \eqref{eq:v_finite} to drive the flow. In what follows, $M_h(\theta_h(t))$ and $E_h(\theta_h(t))$ (defined in \eqref{eq:discrete}) denote the discrete mass and discrete $L^2$-energy of the scalar field $\theta_h(t)$. In all experiments, the control penalty weight in the objective is set to $\gamma = 10^{-6}$ in accordance with \cite{zheng2023numerical}.

\subsection{Cellular flow basis}
In this section, all experiments are performed on the domain $\Omega = (0,1)^2$ with a uniform Cartesian grid of spacing $h=0.002$. The time step is $\Delta t = 10^{-3}$, and the final time is $T=1$. We consider two types of zero-mean initial data. The first is a sharp jump distribution given by 
\begin{equation}
\theta^0(x_1,x_2) := \tanh\!\left(\frac{x_2-0.5}{0.01}\right),
~x_1,x_2 \in \Omega, 
\label{eq:theta_init_1}
\end{equation}
and the second initial data is a trigonometric distribution:
\begin{equation}
\theta^0(x_1,x_2):=\sin\!\left(2\pi x_2\right),
~x_1,x_2 \in \Omega. 
\label{eq:theta_init_2}
\end{equation}

Figure~\ref{fig:cellular} illustrates the two incompressible basis flows $\mathbf{b}_1$ and $\mathbf{b}_2$ (defined in \eqref{eq:cellular}) that span the control space $U$ for this case. Flow $\mathbf{b}_1$ consists of a single convection cell in the domain $\Omega$, while $\mathbf{b}_2$ consists of a four-cell pattern. We use these as representative stirring modes.

\begin{figure}
  \centering
  \begin{subfigure}[b]{0.4\textwidth}
    \includegraphics[width=\textwidth]{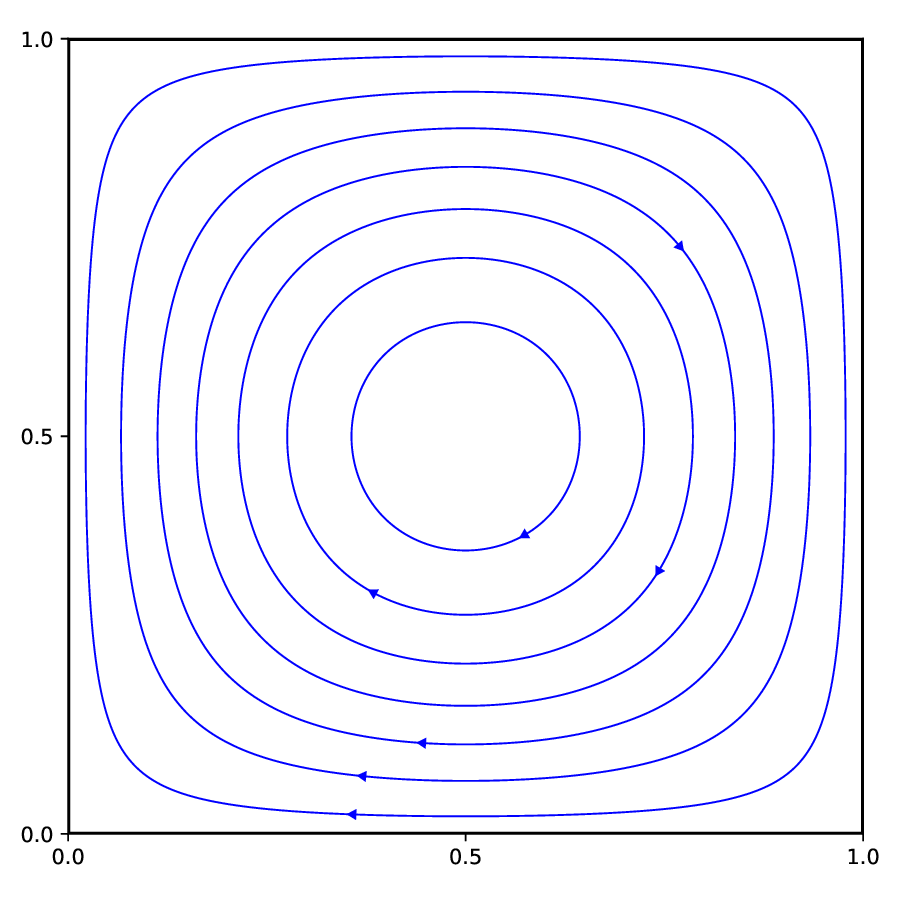}
    \caption{$\mathbf{b}_1$ cellular flow}
    \label{fig:b1}
  \end{subfigure}
  \begin{subfigure}[b]{0.4\textwidth}
    \includegraphics[width=\textwidth]{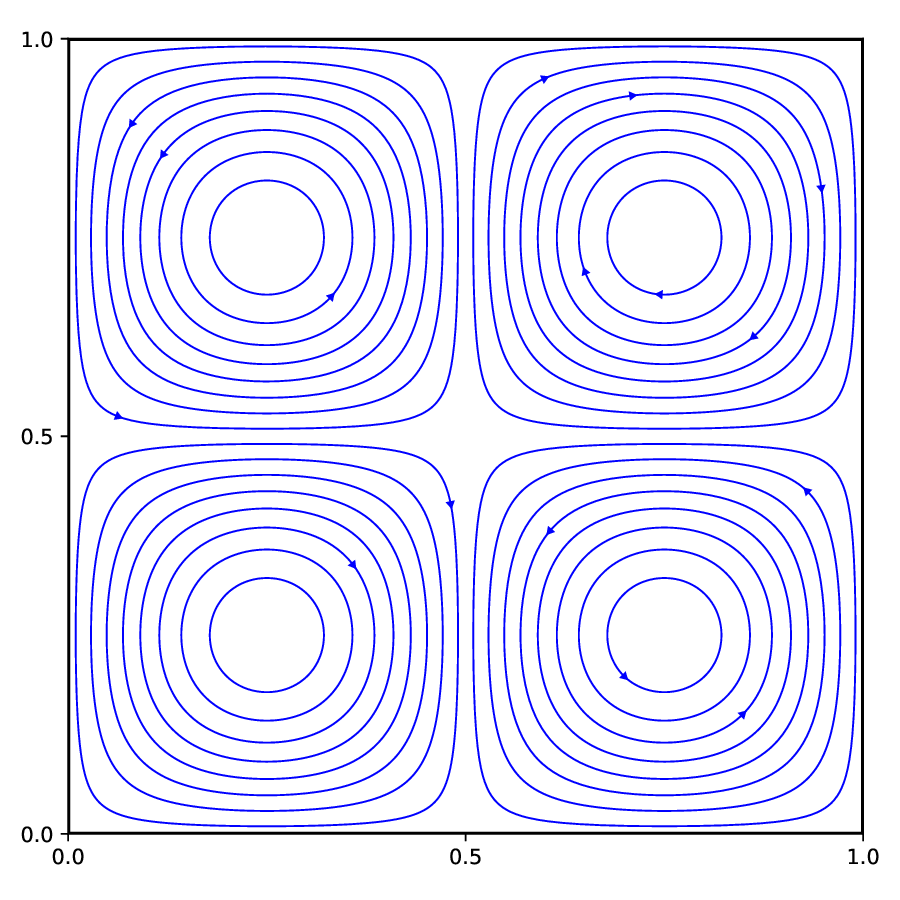}
    \caption{$\mathbf{b}_2$ cellular flow}
    \label{fig:b2}
  \end{subfigure}
  \caption{Cellular basis flows $\mathbf{b}_1$ and $\mathbf{b}_2$ defined in \eqref{eq:cellular}.}
  \label{fig:cellular}
\end{figure}

\subsubsection{Single cellular flow baseline mixing}
We first investigate the mixing achieved by a single cellular flow in the absence of any control optimization. For the basis flow $\mathbf{b}_1$, we simulate the advection of the scalar field under the velocity $\mathbf{v}(t,x) = \mathbf{b}_1(x)$, i.e., using one basis flow alone. The simulations are performed for the two initial conditions \eqref{eq:theta_init_1} and \eqref{eq:theta_init_2}. We do not consider $\mathbf{b}_2$, since its streamline structure does not promote mixing throughout the entire domain and therefore leads to comparatively ineffective stirring. As shown in Figure~\ref{fig:mixnorm_cellular}, the mix-norm $\|\theta_h(t)\|_{\dot H^{-1}(\Omega)}$ exhibits a polynomial decay for both initial data. Figures \ref{fig:theta_cellular_1} and \ref{fig:theta_cellular_2} show snapshots of the scalar field $\theta_h(t,x)$ under both initial data, respectively, at times $t=0,0.2,0.4,0.6,0.8,1.0$. We observe that each steady cellular flow produces some filamentation of the scalar field, but the overall mixing is limited. This reflects the intuitive fact that a single time-independent stirring motion cannot efficiently mix the entire domain on its own. These baseline results will serve as a point of comparison for the optimally controlled scenarios below.

\begin{figure}
  \centering
  \begin{subfigure}[b]{0.45\textwidth}
    \includegraphics[width=\textwidth]{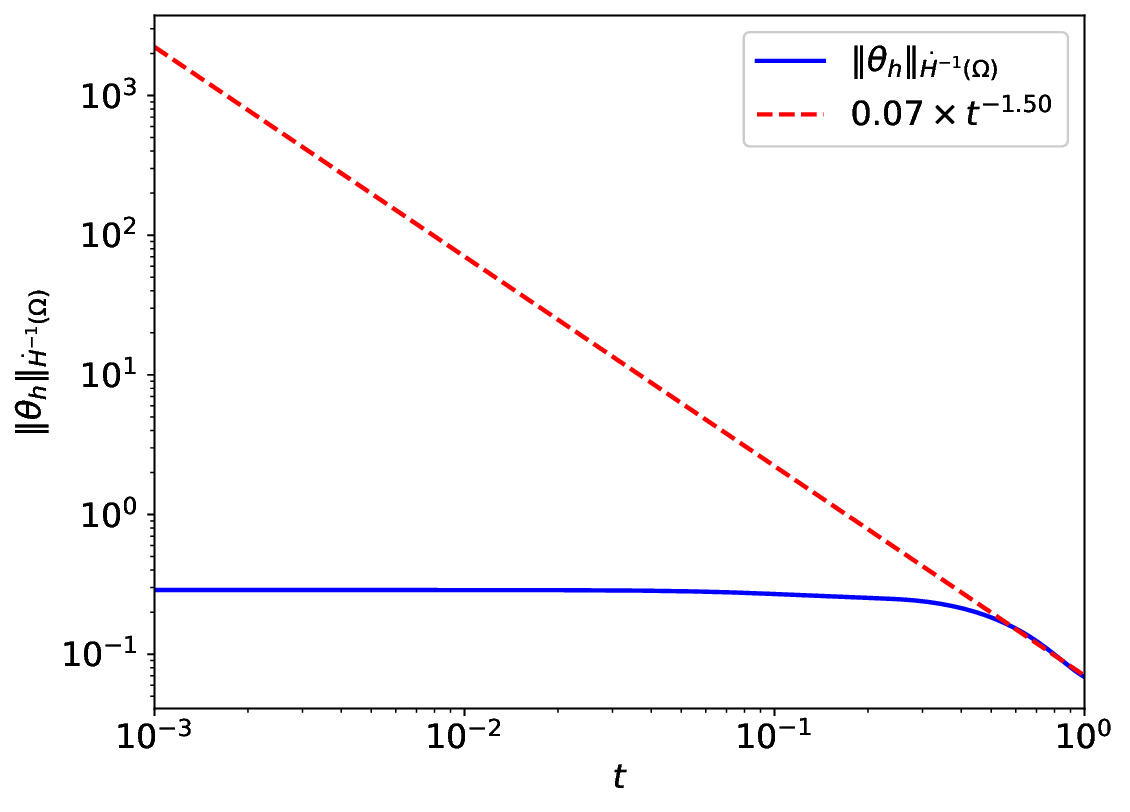}
    \caption{Evolution of $\Vert \theta_h \Vert_{\dot{H}^{-1}(\Omega)}$ under initial data \eqref{eq:theta_init_1}}
    \label{fig:mixnorm_b1}
  \end{subfigure}
  \begin{subfigure}[b]{0.45\textwidth}
    \includegraphics[width=\textwidth]{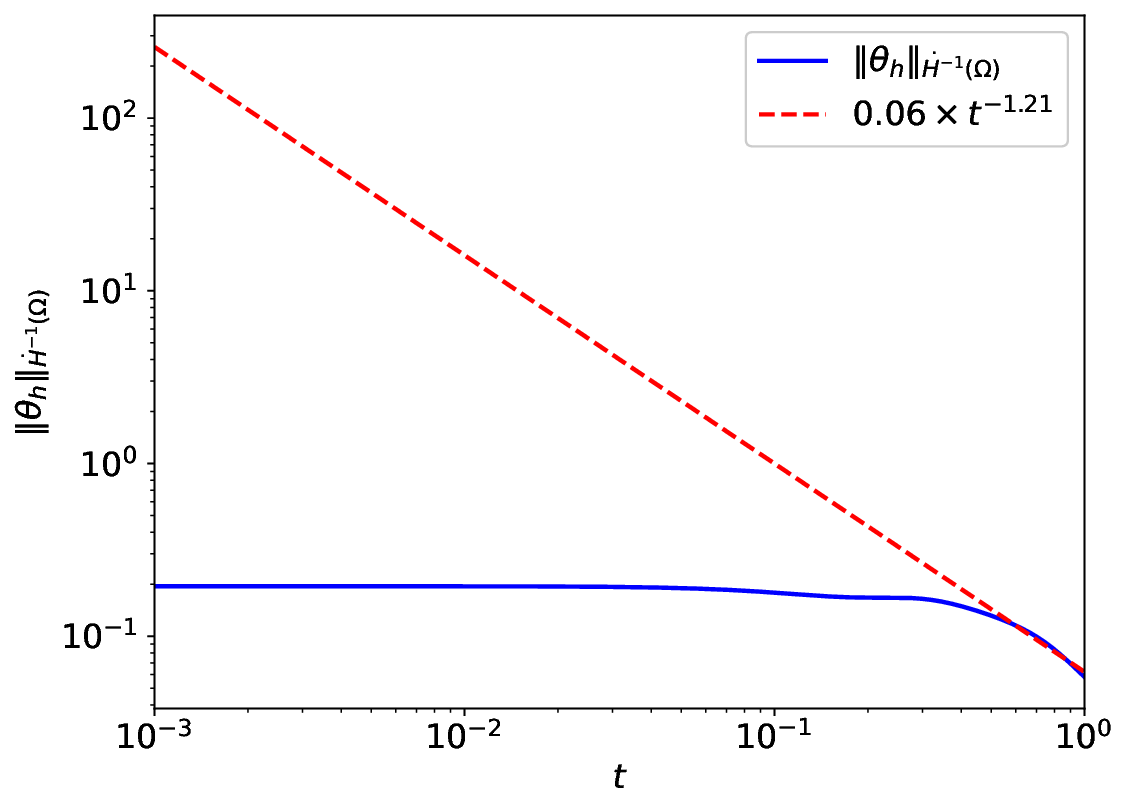}
    \caption{Evolution of $\Vert \theta_h \Vert_{\dot{H}^{-1}(\Omega)}$ under initial data \eqref{eq:theta_init_2}}
    \label{fig:mixnorm_b2}
  \end{subfigure}
  \caption{Evolution of $\Vert \theta_h \Vert_{\dot{H}^{-1}(\Omega)}$ with cellular flow $\mathbf{b}_1$ under initial data \eqref{eq:theta_init_1} and \eqref{eq:theta_init_2}.}
  \label{fig:mixnorm_cellular}
\end{figure}

\begin{figure}
  \centering
  \begin{subfigure}[b]{0.16\textwidth}
    \includegraphics[width=\textwidth]{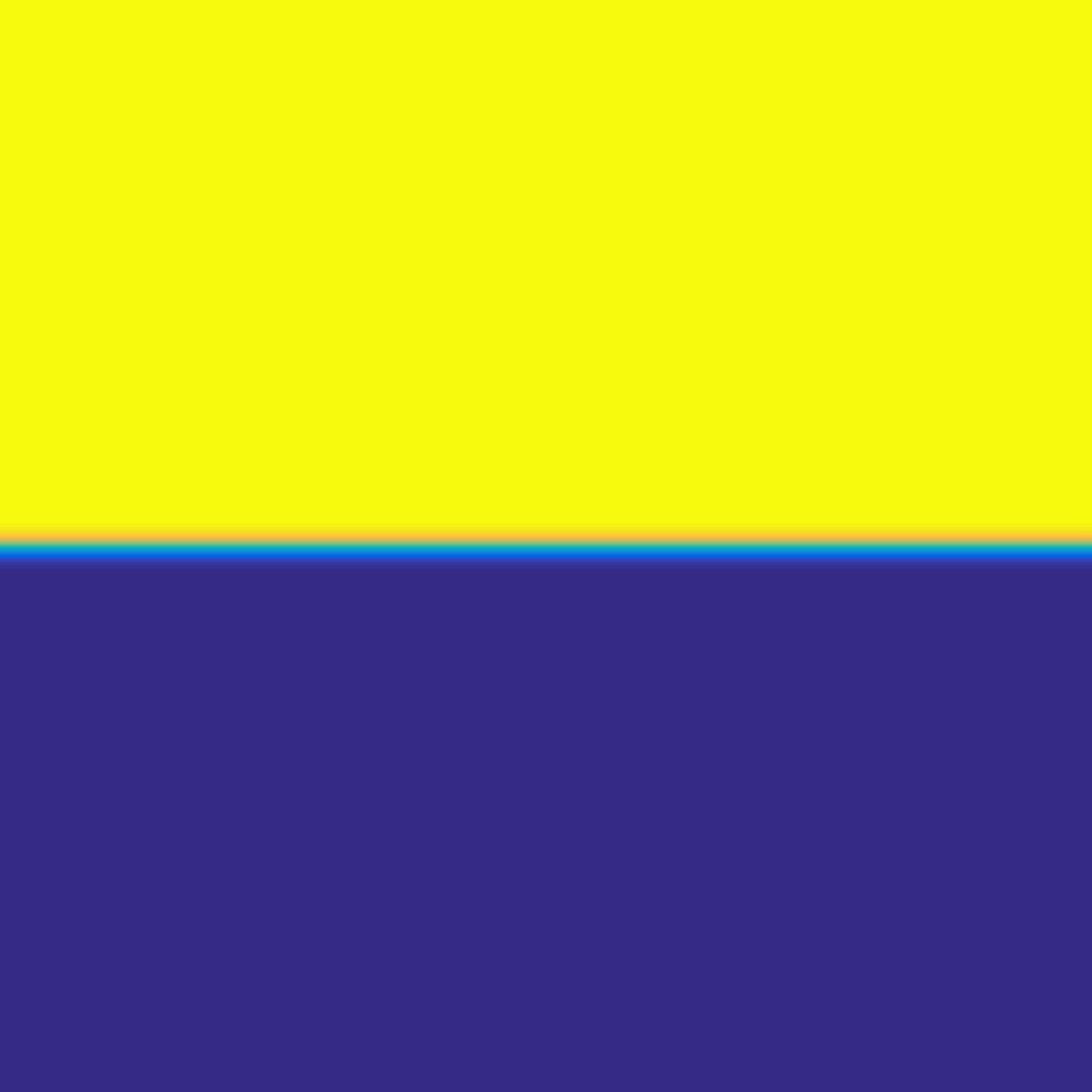}
    \caption{$t=0$}
  \end{subfigure}
  \begin{subfigure}[b]{0.16\textwidth}
    \includegraphics[width=\textwidth]{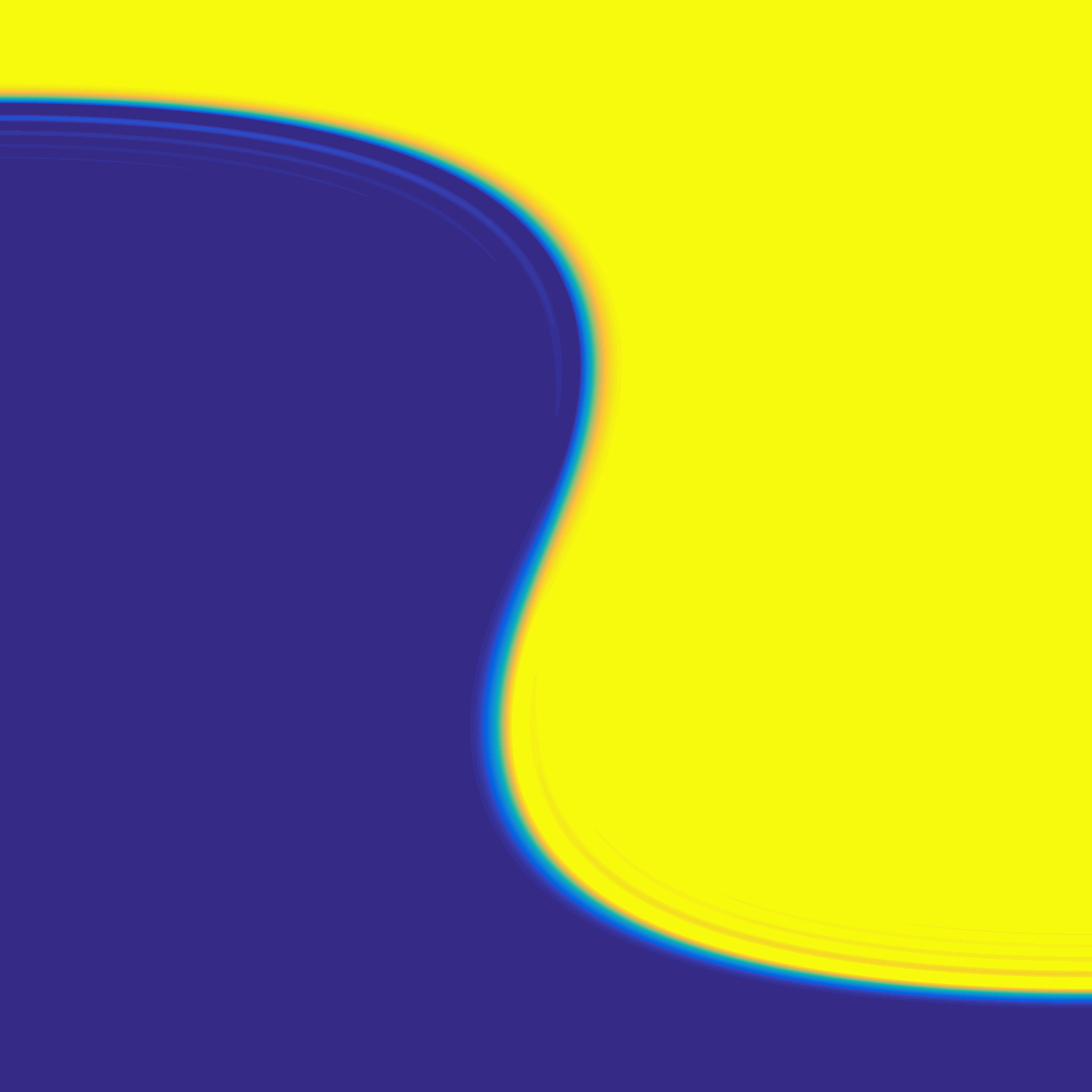}
    \caption{$t=0.2$}
  \end{subfigure}
  \begin{subfigure}[b]{0.16\textwidth}
    \includegraphics[width=\textwidth]{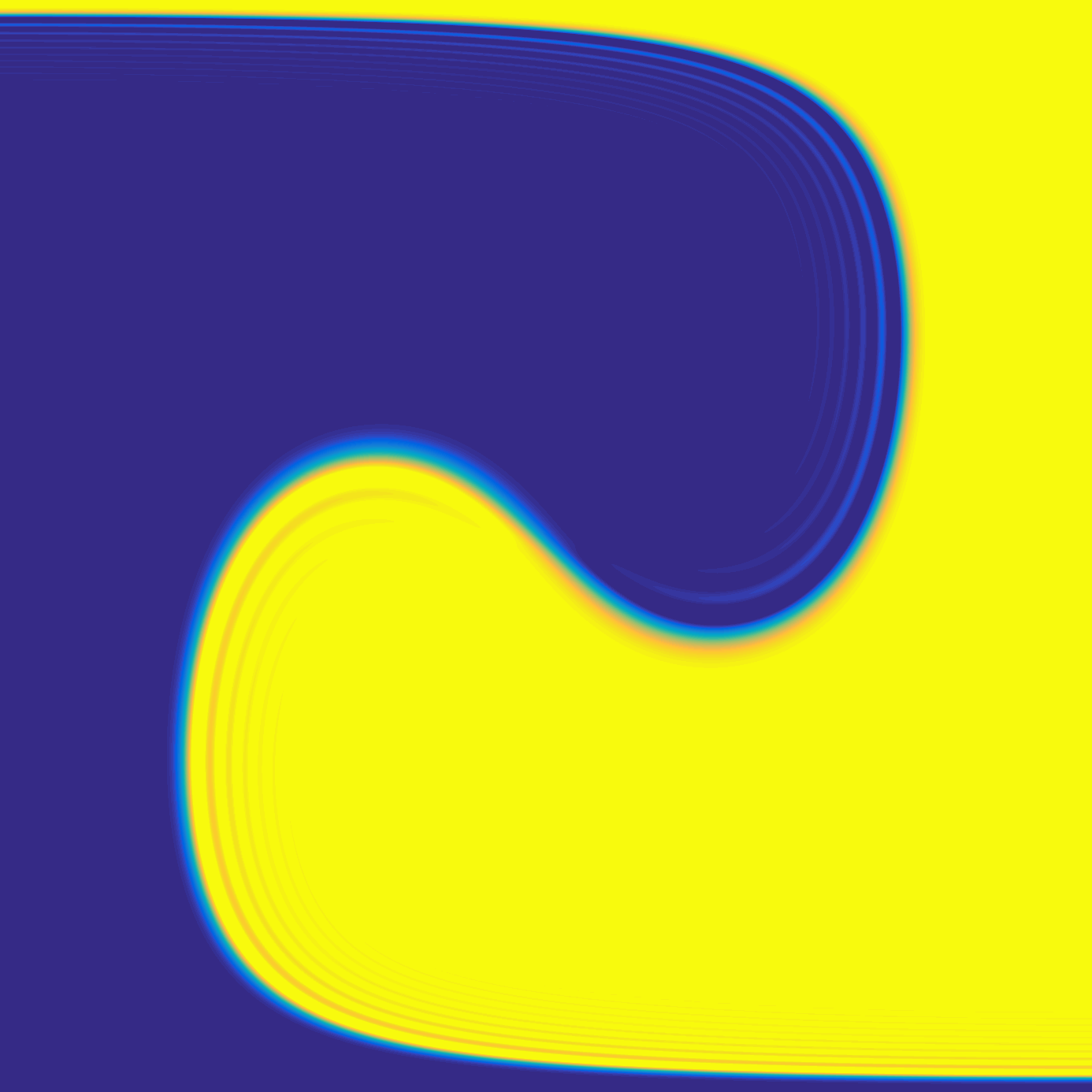}
    \caption{$t=0.4$}
  \end{subfigure}
  \begin{subfigure}[b]{0.16\textwidth}
    \includegraphics[width=\textwidth]{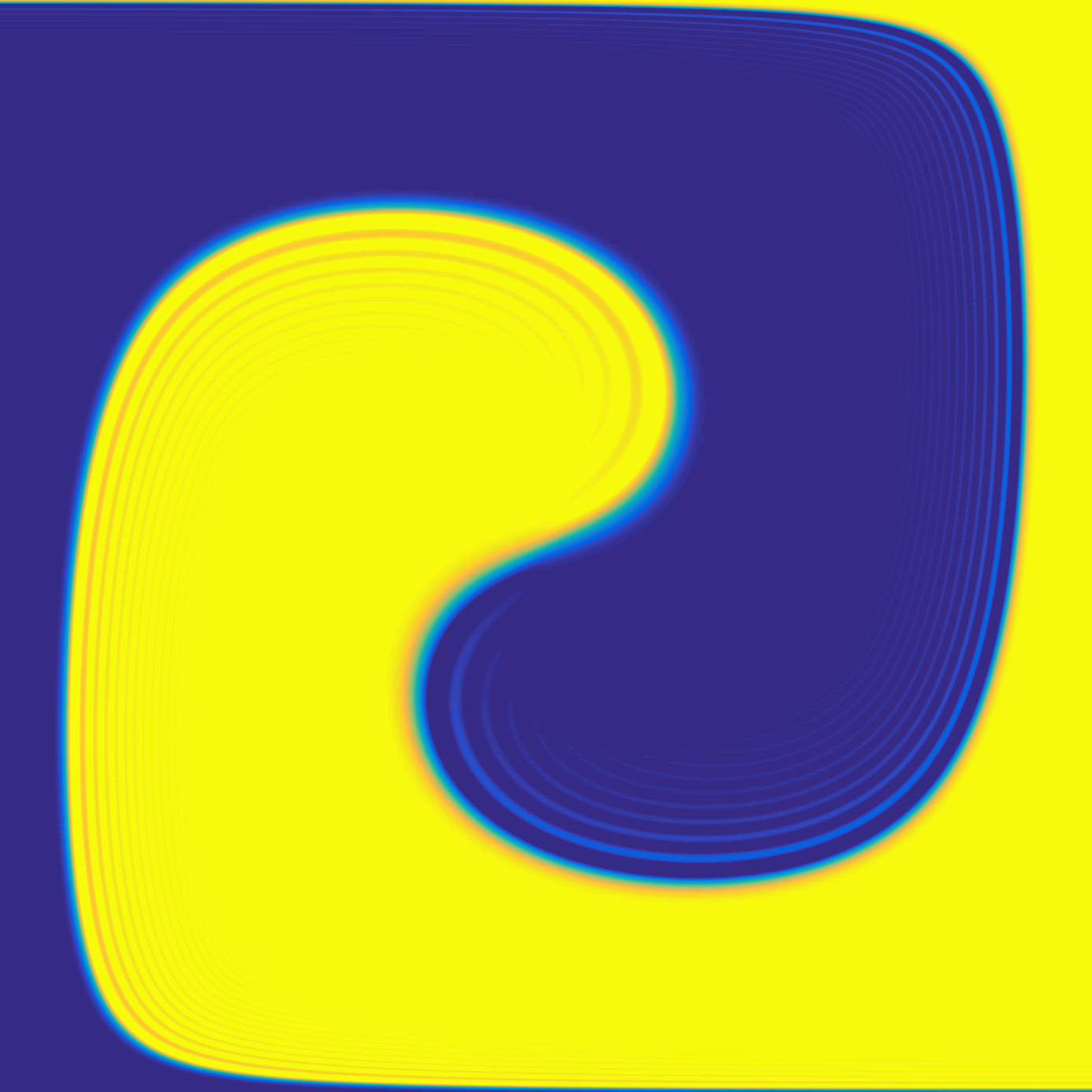}
    \caption{$t=0.6$}
  \end{subfigure}
  \begin{subfigure}[b]{0.16\textwidth}
    \includegraphics[width=\textwidth]{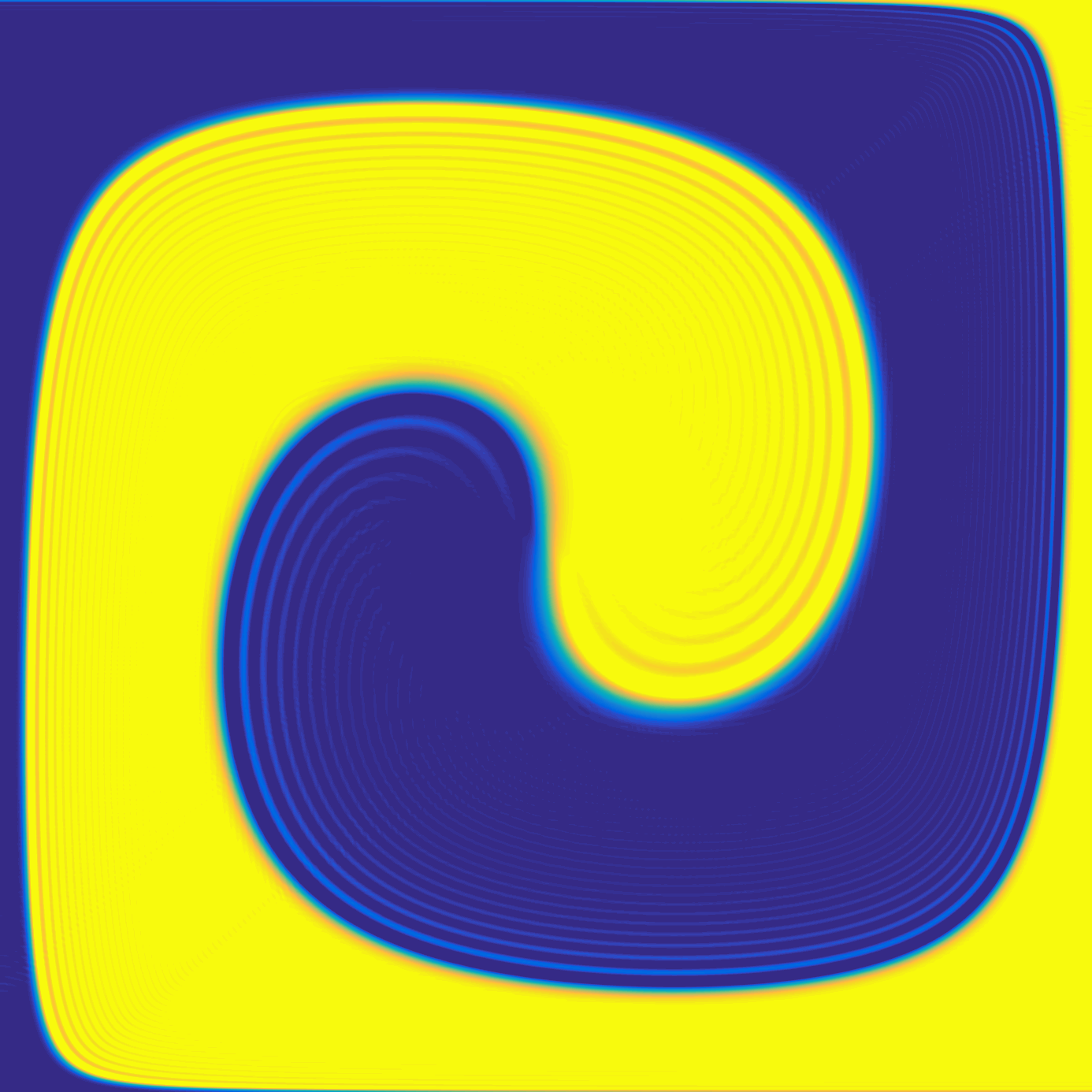}
    \caption{$t=0.8$}
  \end{subfigure}
  \begin{subfigure}[b]{0.16\textwidth}
    \includegraphics[width=\textwidth]{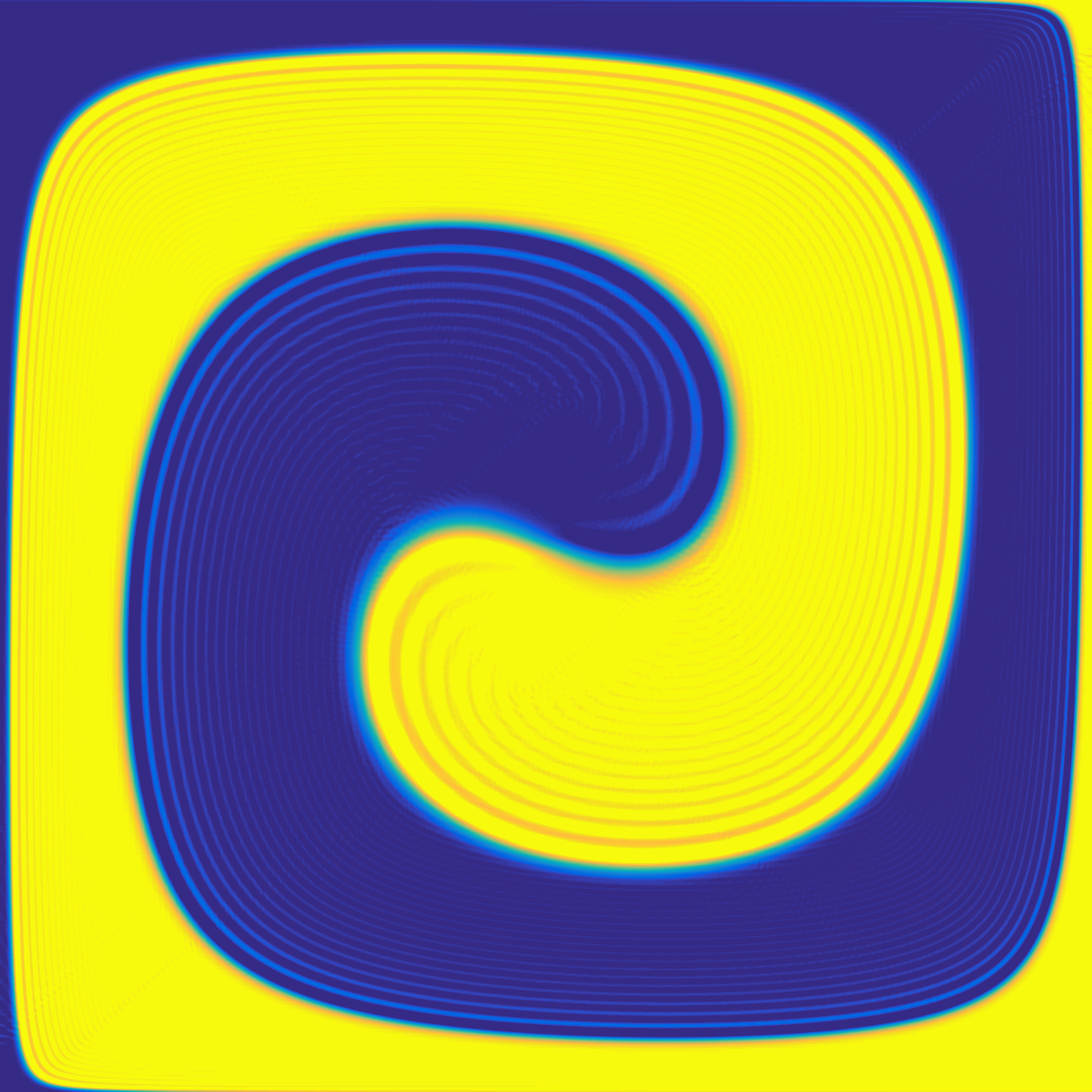}
    \caption{$t=1.0$}
  \end{subfigure}
  \caption{Evolution of $\theta_h$ with cellular flow $\mathbf{b}_1$ and initial data \eqref{eq:theta_init_1}.}
  \label{fig:theta_cellular_1}
\end{figure}

\begin{figure}
  \centering
  \begin{subfigure}[b]{0.16\textwidth}
    \includegraphics[width=\textwidth]{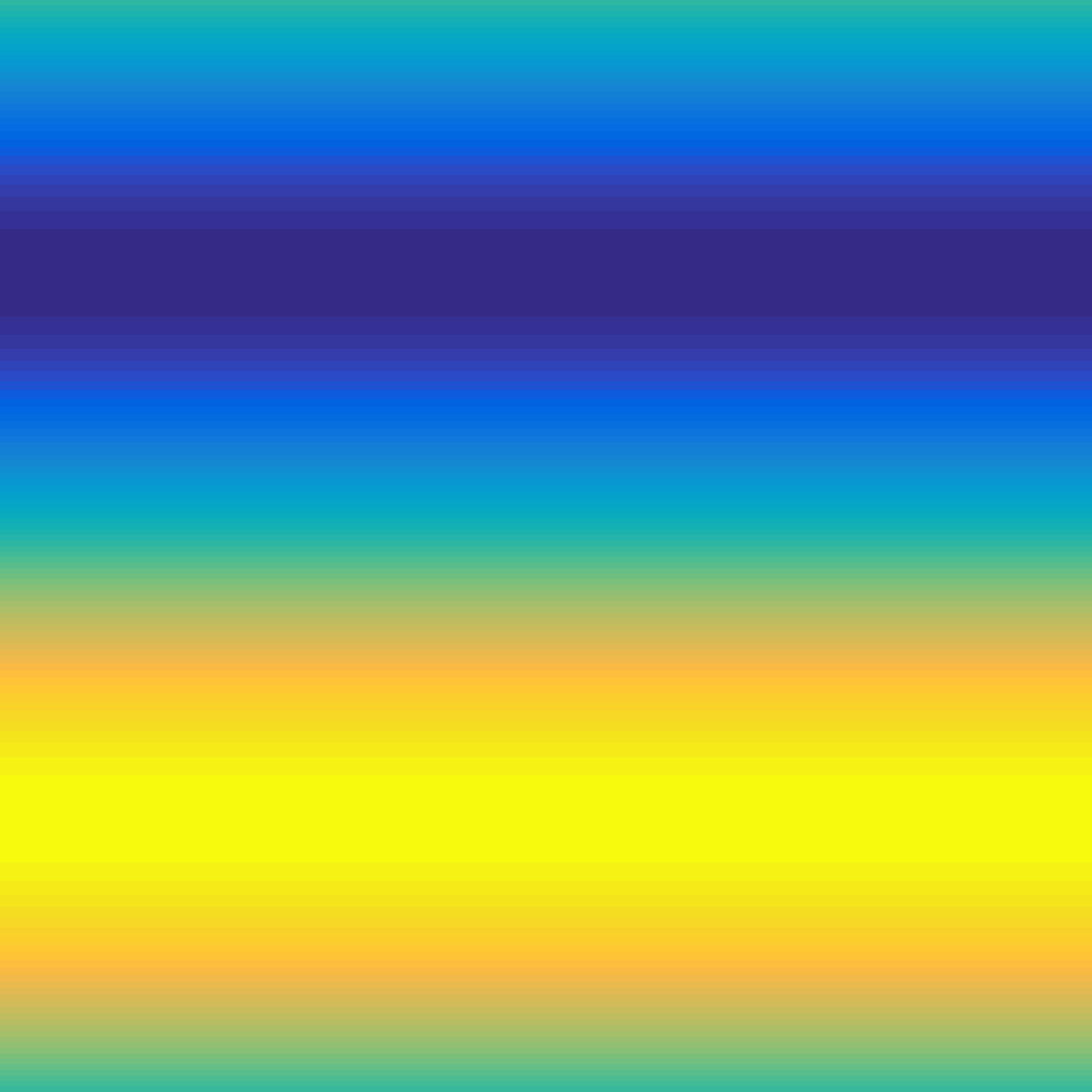}
    \caption{$t=0$}
  \end{subfigure}
  \begin{subfigure}[b]{0.16\textwidth}
    \includegraphics[width=\textwidth]{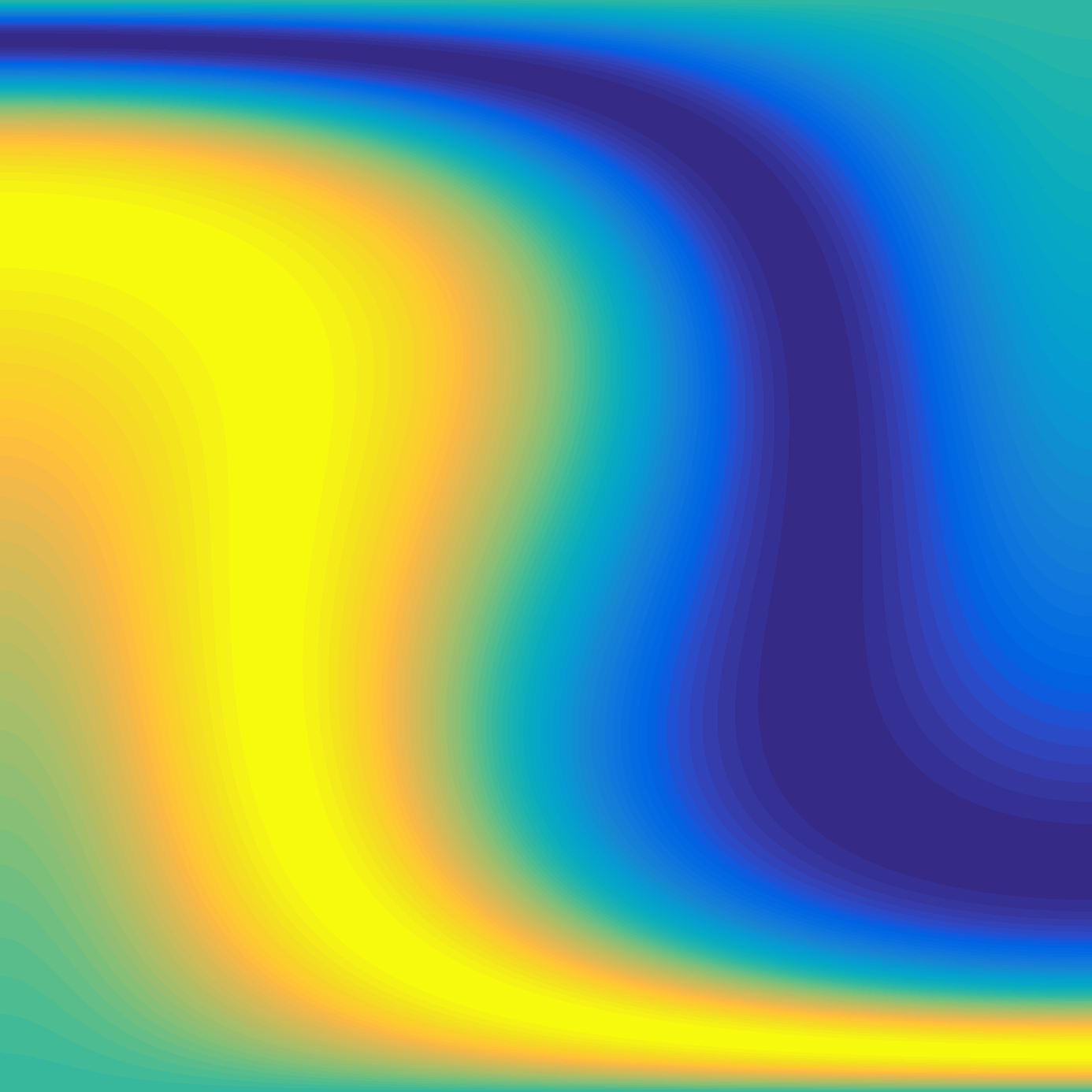}
    \caption{$t=0.2$}
  \end{subfigure}
  \begin{subfigure}[b]{0.16\textwidth}
    \includegraphics[width=\textwidth]{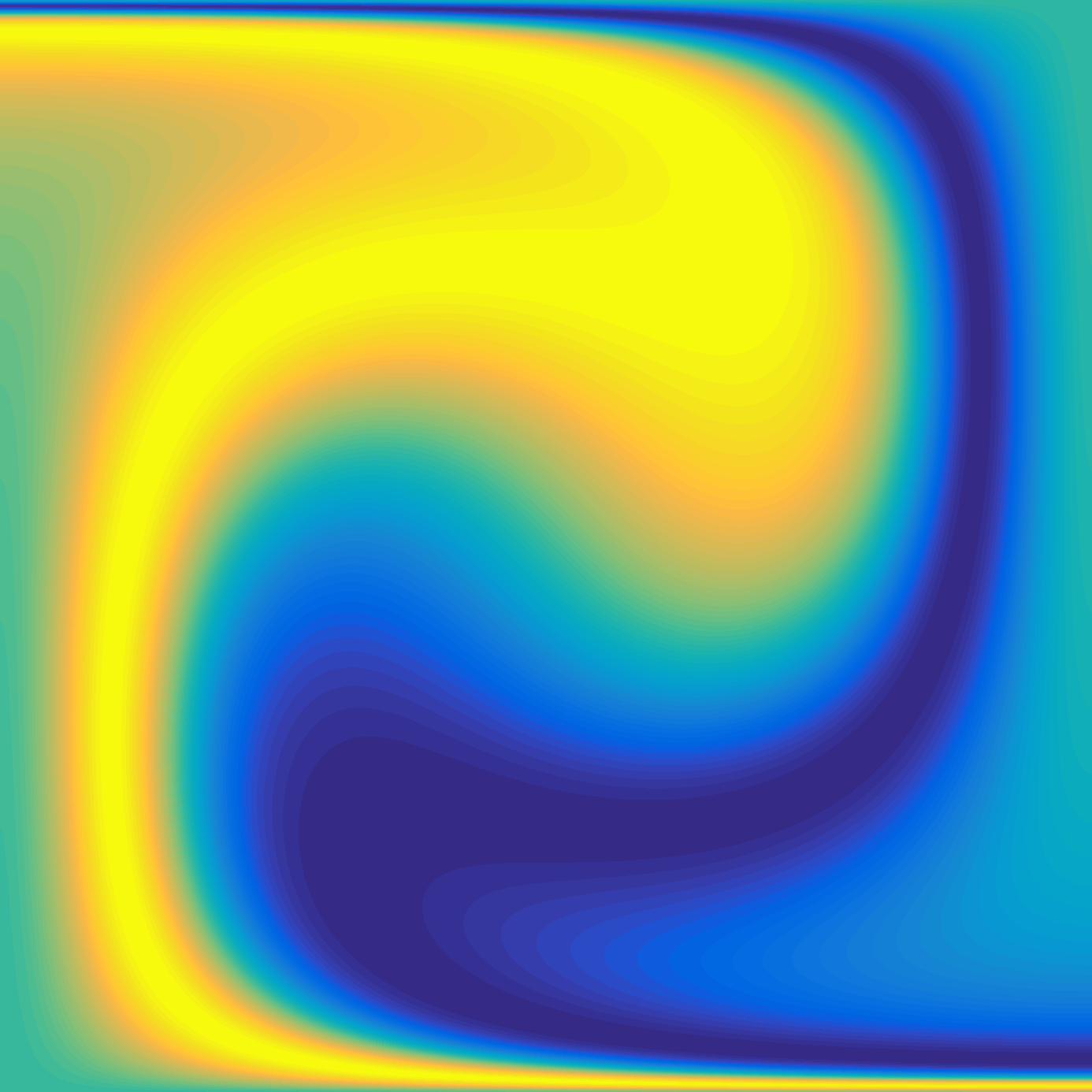}
    \caption{$t=0.4$}
  \end{subfigure}
  \begin{subfigure}[b]{0.16\textwidth}
    \includegraphics[width=\textwidth]{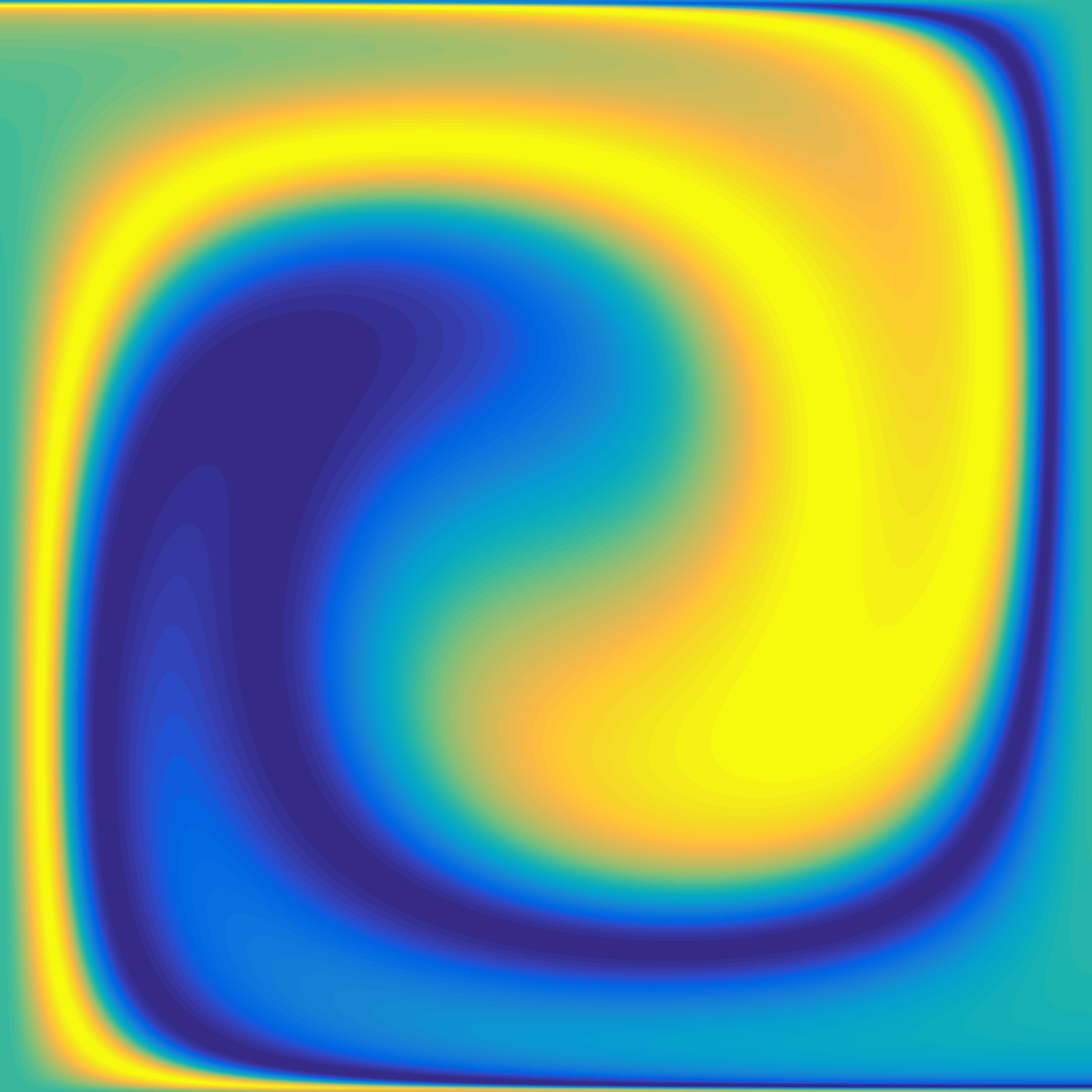}
    \caption{$t=0.6$}
  \end{subfigure}
  \begin{subfigure}[b]{0.16\textwidth}
    \includegraphics[width=\textwidth]{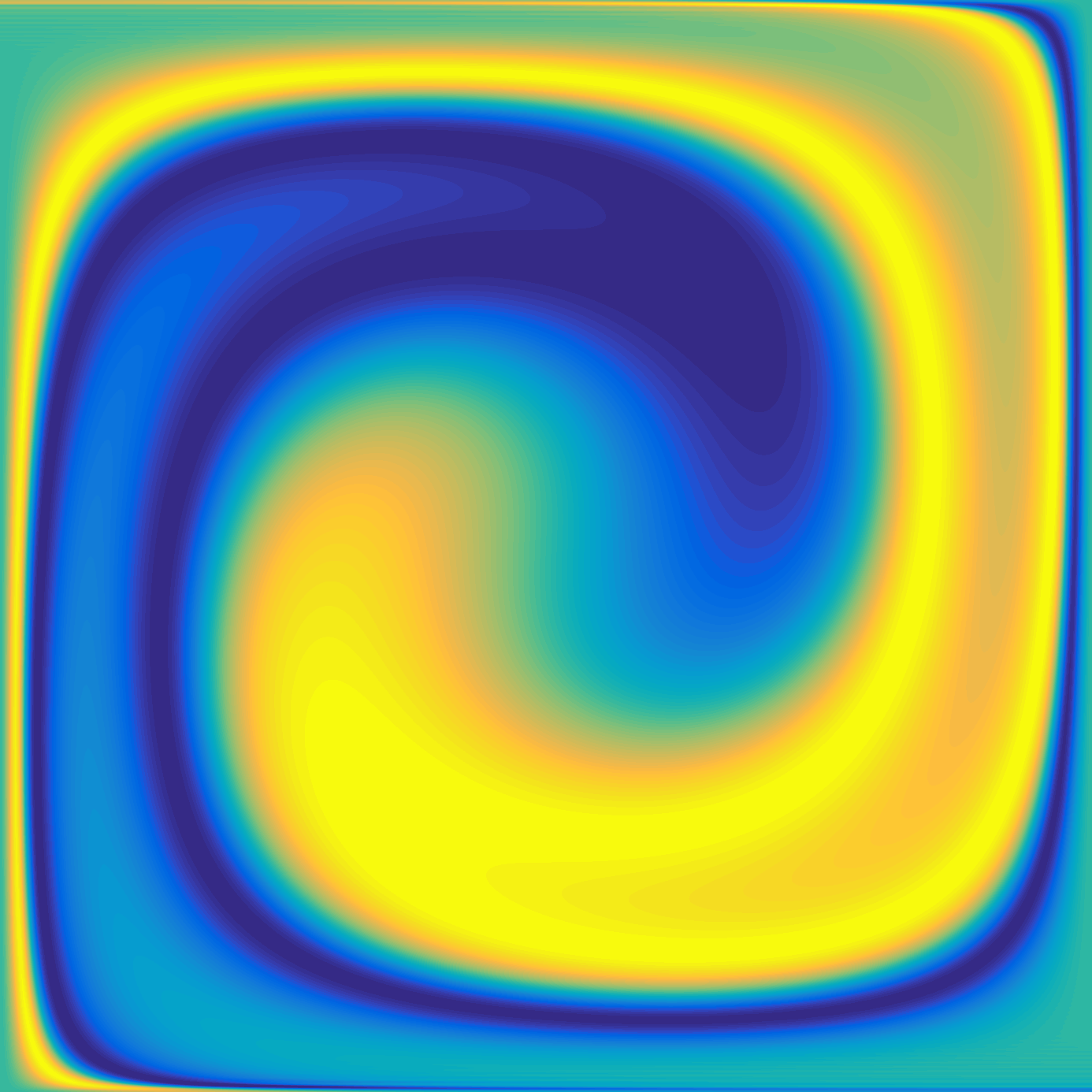}
    \caption{$t=0.8$}
  \end{subfigure}
  \begin{subfigure}[b]{0.16\textwidth}
    \includegraphics[width=\textwidth]{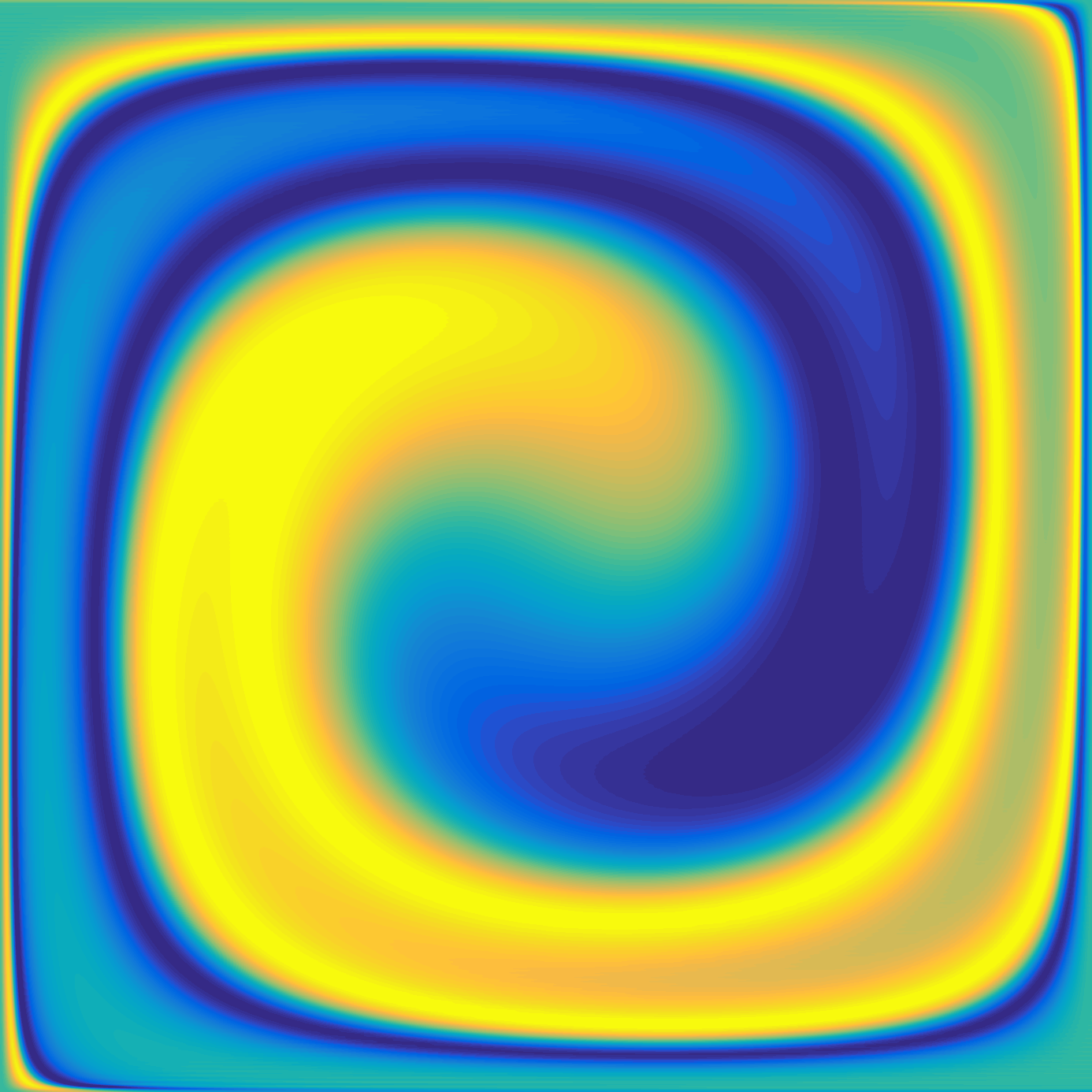}
    \caption{$t=1.0$}
  \end{subfigure}
  \caption{Evolution of $\theta_h$ with cellular flow $\mathbf{b}_1$ and initial data \eqref{eq:theta_init_2}.}
  \label{fig:theta_cellular_2}
\end{figure}

\subsubsection{Optimal control for initial data \eqref{eq:theta_init_1}}
We now apply our optimal control algorithm to the sharp jump initial data  \eqref{eq:theta_init_1}. The control is allowed to use both basis flows $\mathbf{b}_1$ and $\mathbf{b}_2$ in combination with time-dependent coefficients. We perform the optimization in two scenarios, starting from two different initial guesses for the control time-series. The first initial guess is denoted by
\begin{align}\label{eq:u1u2_1}
v_1(t)=1,~v_2(t)=1,~ t\in(0,1),
\end{align}
and the second initial guess is 
\begin{align}\label{eq:u1u2_2}
v_1(t)=\cos(\pi t / 2),
~v_2(t)=\sin(\pi t / 2),~ t\in(0,1).
\end{align}

Figure~\ref{fig:u1u2_mixnorm_init1_1} summarizes the results starting from the zero initial guess \eqref{eq:u1u2_1}. Figure~\ref{fig:u1u2_init1_1} shows the evolution of the control coefficients $v_1(t)$ and $v_2(t)$ for the initial guess \eqref{eq:u1u2_1}. Both controls vary sharply over time, and, over the simulated time interval, the optimized mix-norm $\Vert \theta_h(t) \Vert_{\dot{H}^{-1}(\Omega)}$ exhibits a decay that is well approximated by an exponential with an effective rate of about 2.52 (Figure~\ref{fig:mixnorm_init1_1}). Figure~\ref{fig:property_init1_1} displays the temporal evolution of the following discrete invariants: mass error $|M_h(\theta_h(t))|$, energy error $|E_h(\theta_h(t))-E_h(\theta_h(0))|$, and state-adjoint pairing error 
$\bigl|
    \langle \theta_h(t), \rho_h(t)\rangle_{X_h}   -  
    \langle \theta_h(T), \rho_h(T)\rangle_{X_h}
\bigr|$. Over the entire time interval $[0,1]$, these errors remain around machine precision (mass error $\sim10^{-17}$, energy error $\sim10^{-10}$, state-adjoint pairing error $\sim10^{-14}$), confirming exact conservation of mass, energy, and state--adjoint consistency in practice. The snapshots of $\theta_h$ in Figure~\ref{fig:theta_init1_1} illustrate the scalar field evolution under control. Compared to the baseline in Figure~\ref{fig:mixnorm_cellular}, we see a much more thorough mixing by $t=1$.

\begin{figure}
  \centering
  \begin{subfigure}[b]{0.32\textwidth}
    \includegraphics[width=\textwidth]{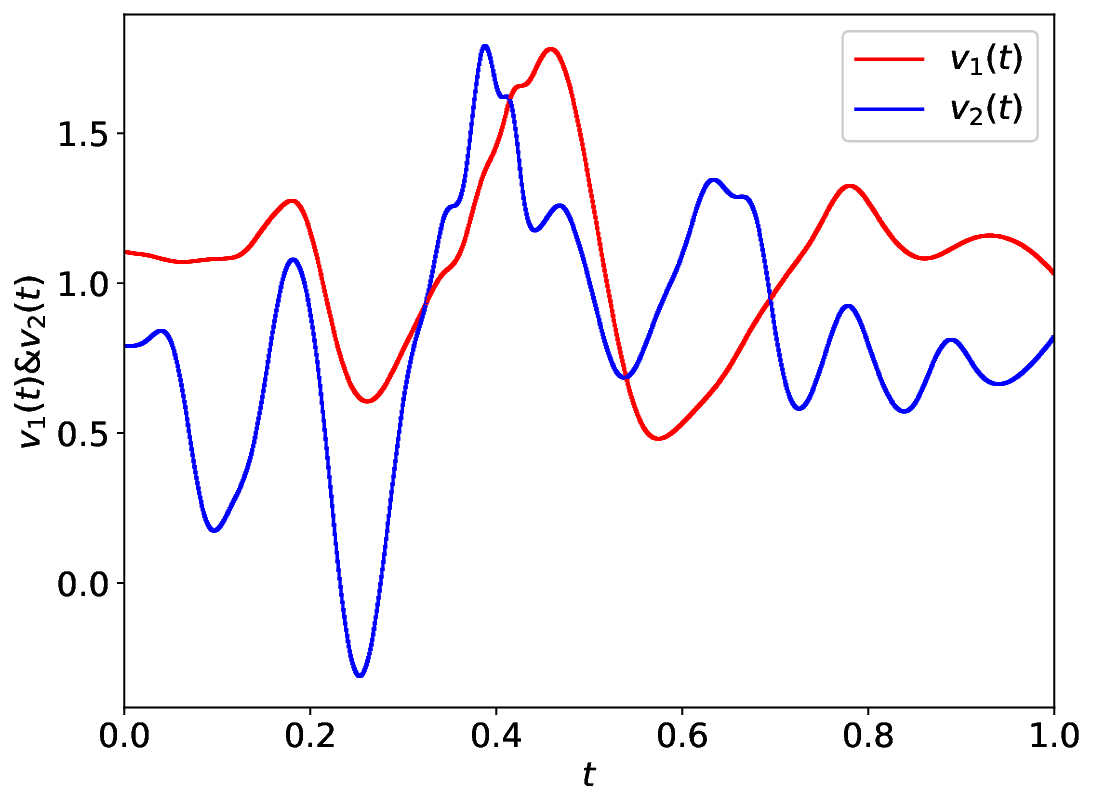}
    \caption{Evolutions of $v_1(t)$, $v_2(t)$.}
    \label{fig:u1u2_init1_1}
  \end{subfigure}
  \begin{subfigure}[b]{0.33\textwidth}
    \includegraphics[width=\textwidth]{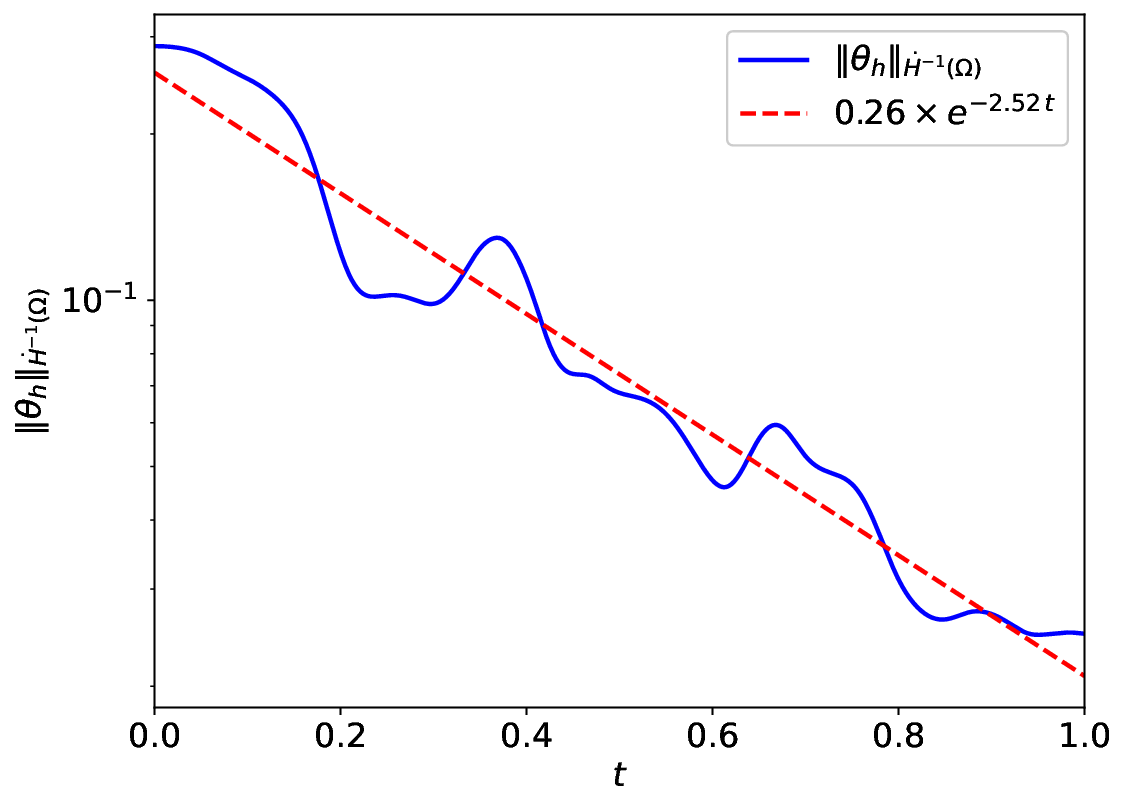}
    \caption{Evolution of $\Vert \theta_h \Vert_{\dot{H}^{-1}(\Omega)}$.}
    \label{fig:mixnorm_init1_1}
  \end{subfigure}
  \begin{subfigure}[b]{0.33\textwidth}
    \includegraphics[width=\textwidth]{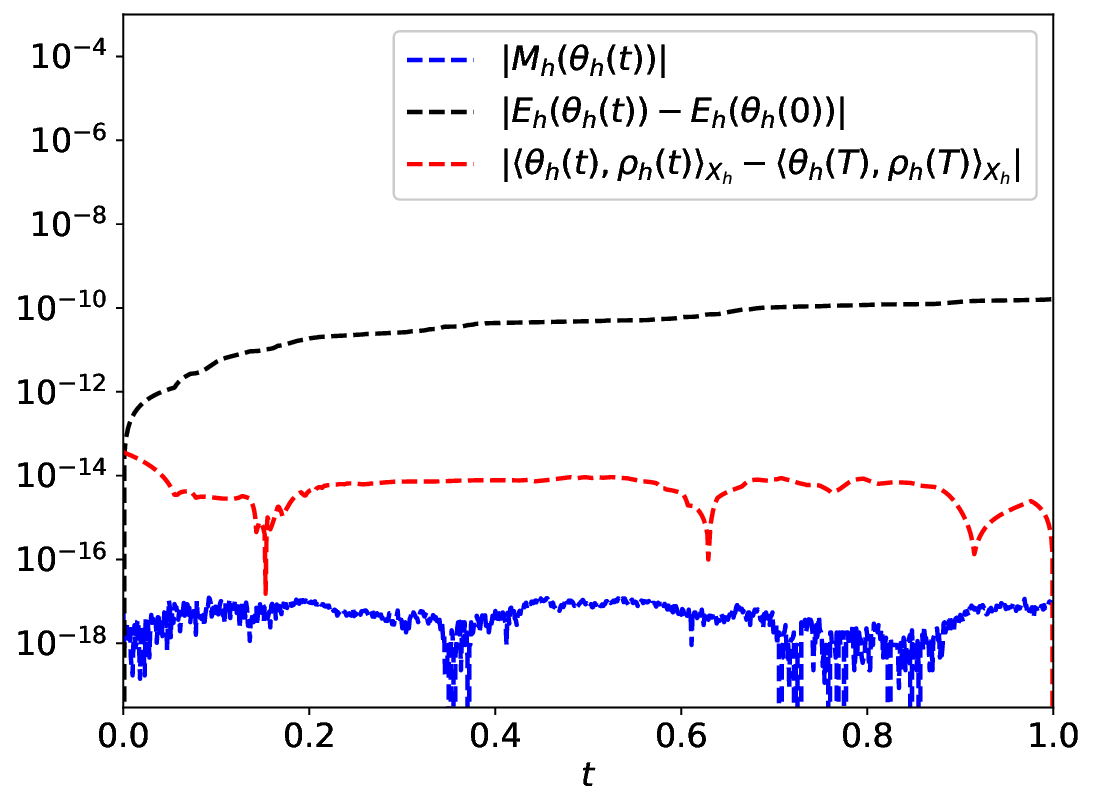}
    \caption{$M_h(\theta_h)$, $E_h(\theta_h)$, $\langle \theta_h, \rho_h \rangle_{X_h}$.}
    \label{fig:property_init1_1}
  \end{subfigure}
    \caption{Evolutions of $v_1(t)$, $v_2(t)$, mix-norm $\Vert \theta_h \Vert_{\dot{H}^{-1}(\Omega)}$, mass $M_h(\theta_h)$, energy $E_h(\theta_h)$ and state-adjoint pairing $\langle \theta_h, \rho_h \rangle_{X_h}$ for $t\in[0, 1]$ with initial control \eqref{eq:u1u2_1} and initial data \eqref{eq:theta_init_1}.}
  \label{fig:u1u2_mixnorm_init1_1}
\end{figure}

\begin{figure}
  \centering
  \begin{subfigure}[b]{0.16\textwidth}
    \includegraphics[width=\textwidth]{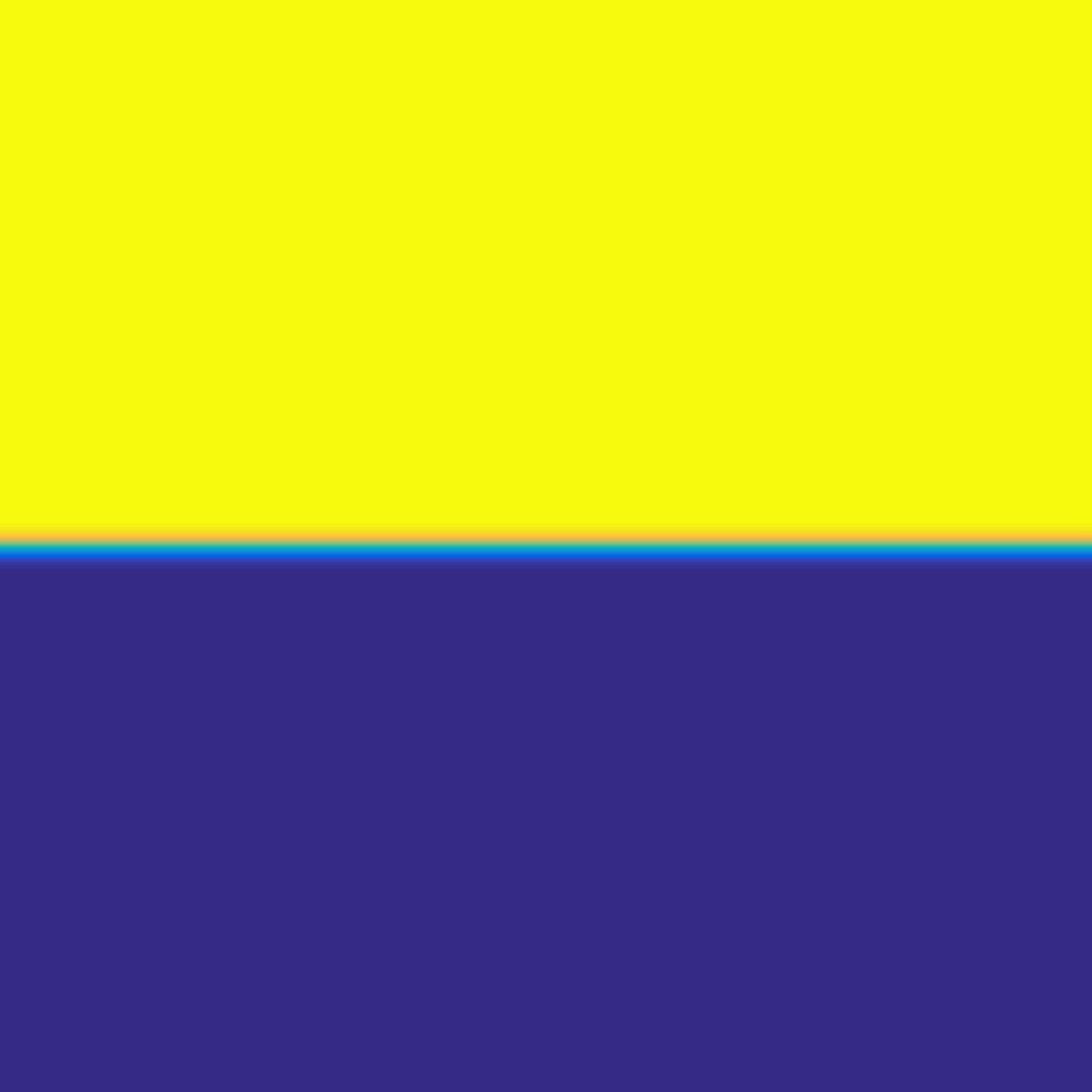}
    \caption{$t=0$}
  \end{subfigure}
  \begin{subfigure}[b]{0.16\textwidth}
    \includegraphics[width=\textwidth]{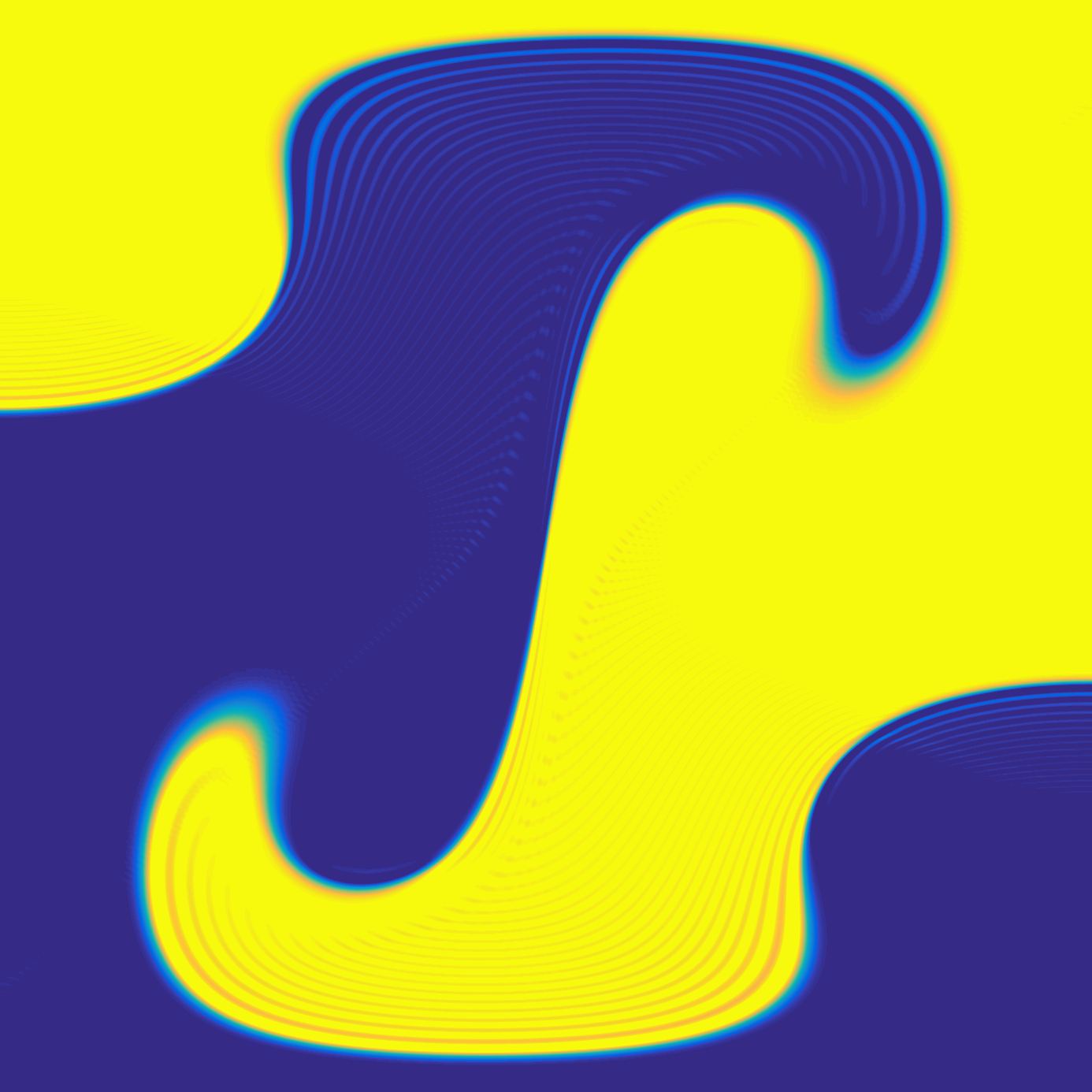}
    \caption{$t=0.2$}
  \end{subfigure}
  \begin{subfigure}[b]{0.16\textwidth}
    \includegraphics[width=\textwidth]{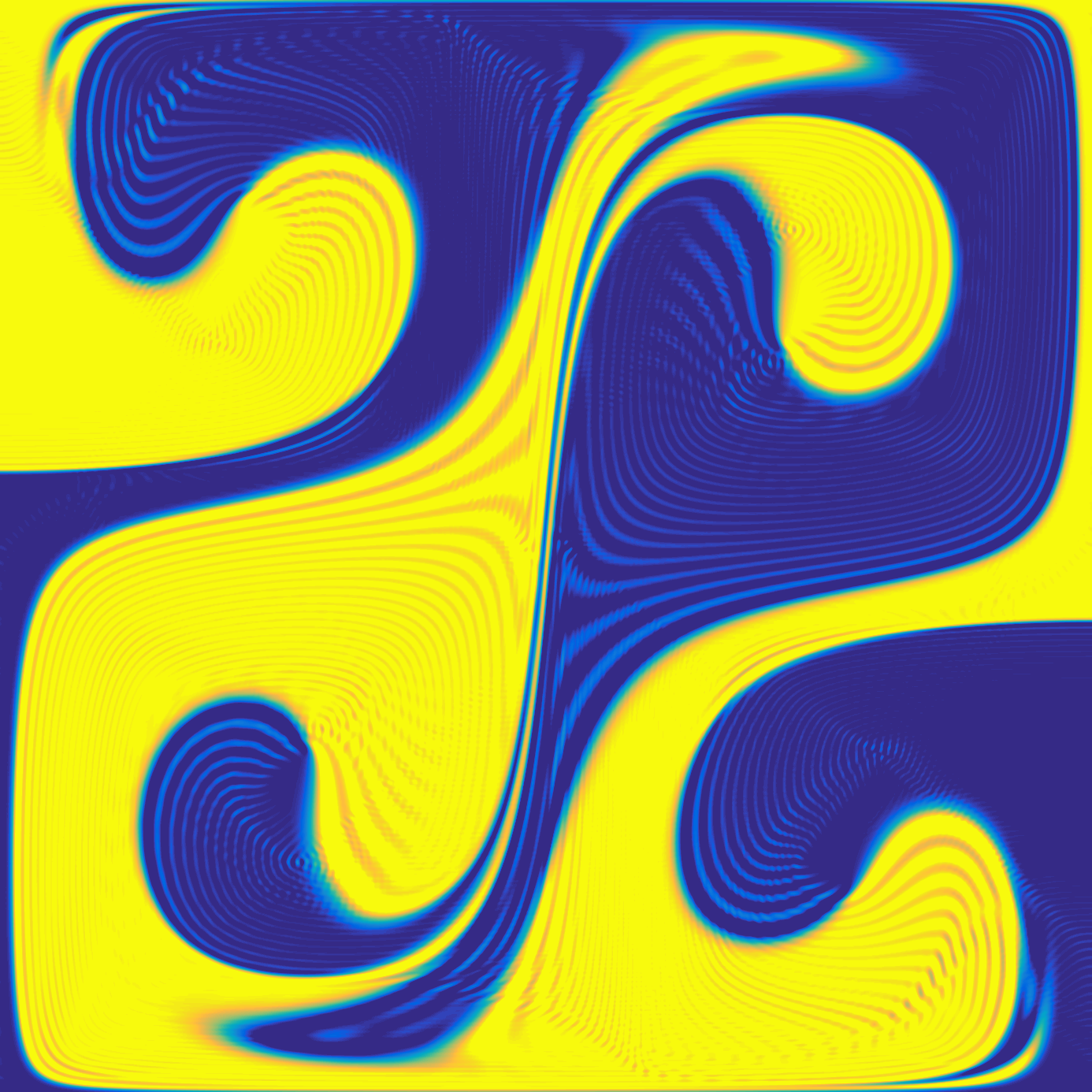}
    \caption{$t=0.4$}
  \end{subfigure}
  \begin{subfigure}[b]{0.16\textwidth}
    \includegraphics[width=\textwidth]{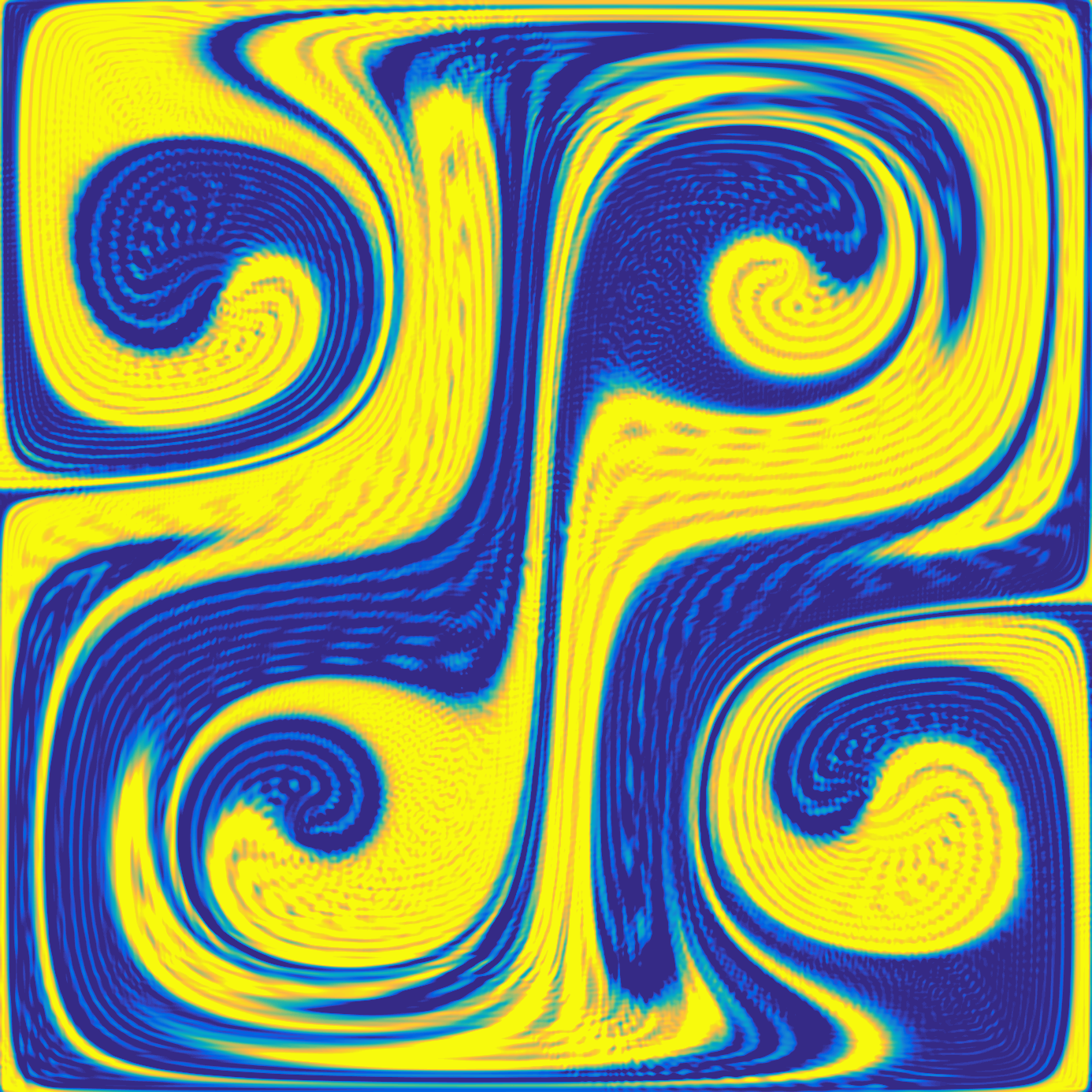}
    \caption{$t=0.6$}
  \end{subfigure}
  \begin{subfigure}[b]{0.16\textwidth}
    \includegraphics[width=\textwidth]{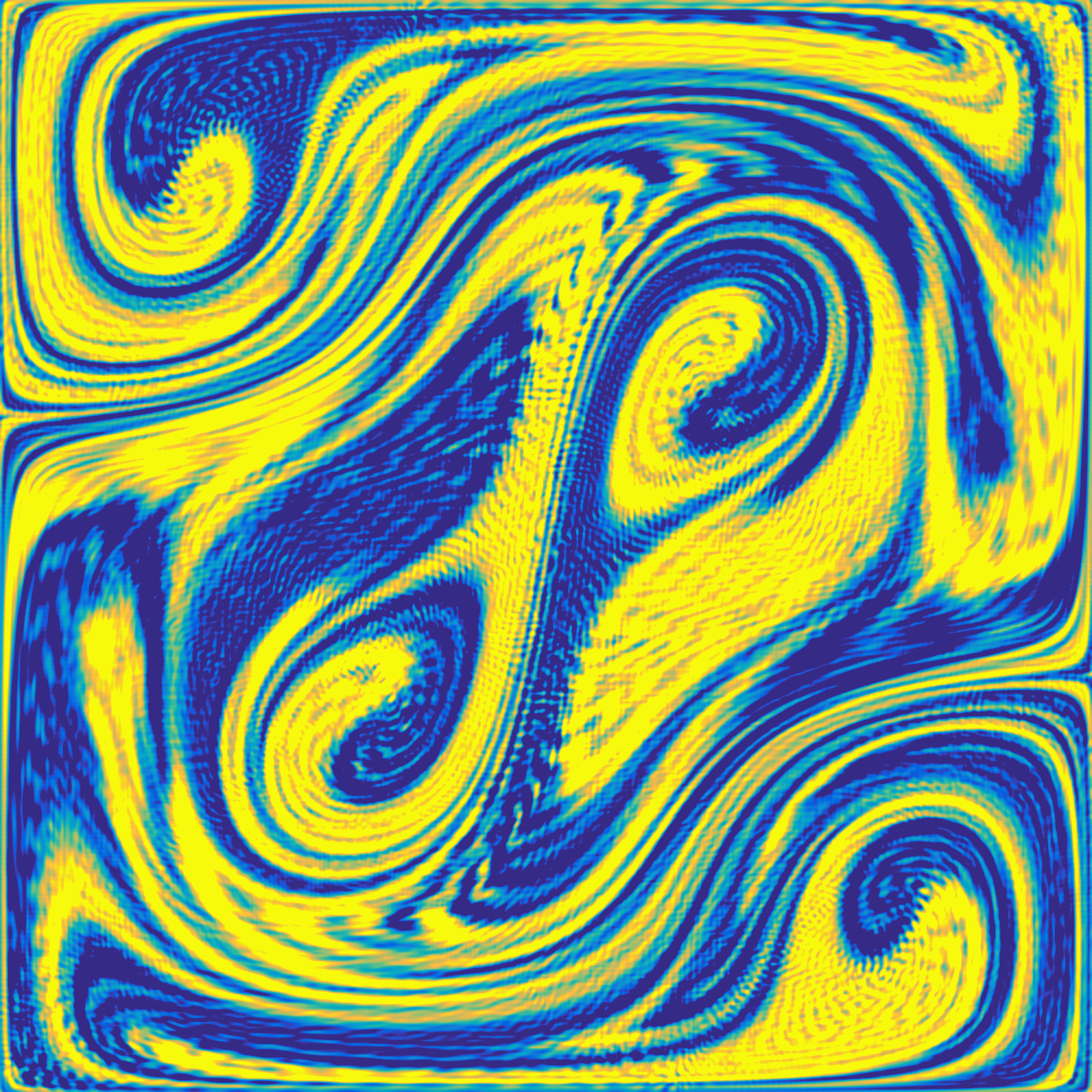}
    \caption{$t=0.8$}
  \end{subfigure}
  \begin{subfigure}[b]{0.16\textwidth}
    \includegraphics[width=\textwidth]{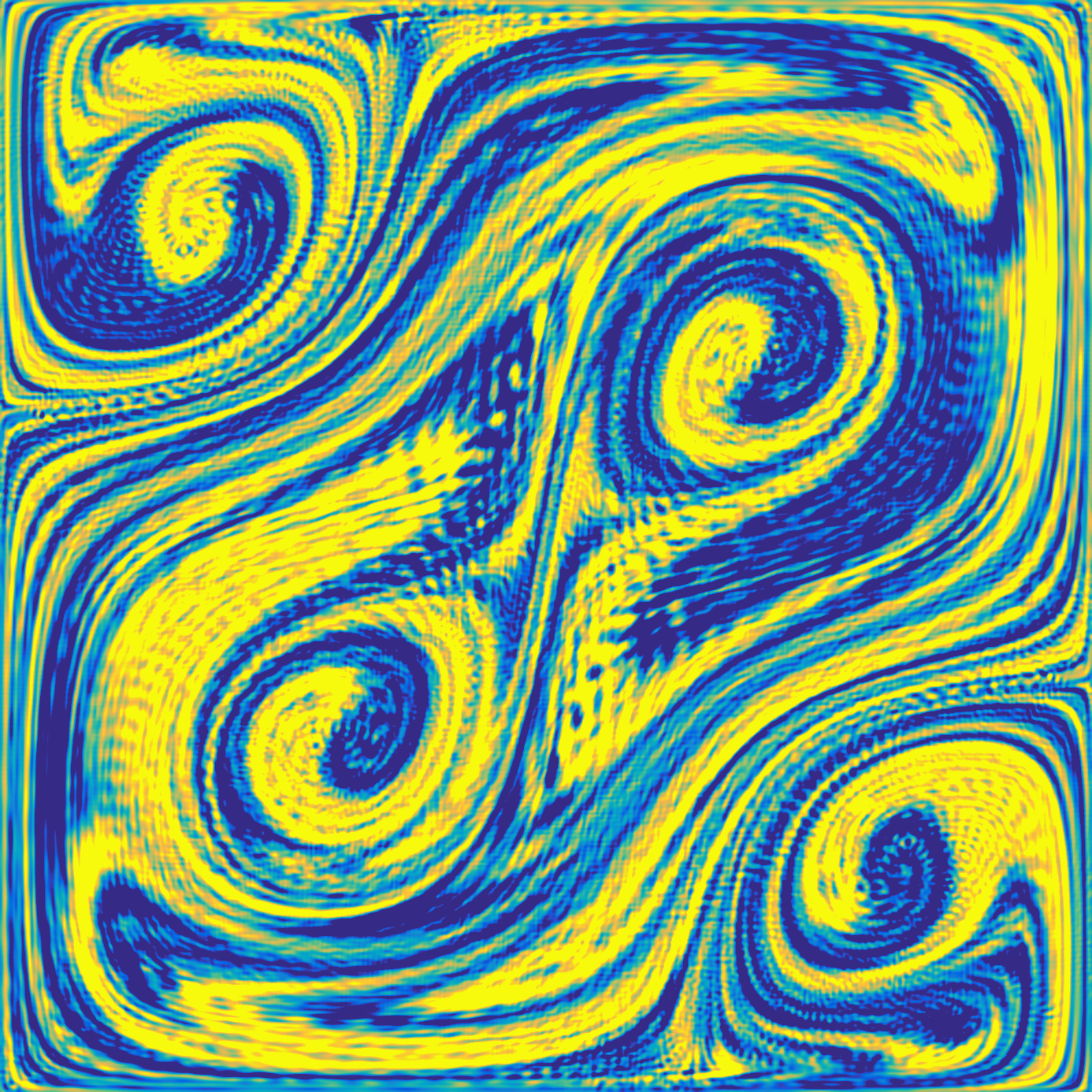}
    \caption{$t=1.0$}
  \end{subfigure}
  \caption{Evolution of $\theta_h$ with initial control \eqref{eq:u1u2_1} and initial data \eqref{eq:theta_init_1}.}
  \label{fig:theta_init1_1}
\end{figure}

Next, we consider the optimization starting from the initial guess  \eqref{eq:u1u2_2}. The optimized control coefficients $v_1(t), v_2(t)$ for this case are shown in Figure~~\ref{fig:u1u2_init1_2}, and the mix-norm decay is shown in Figure~\ref{fig:mixnorm_init1_2}. We again observe a nearly exponential decay of the mix-norm with a rate $1.81$. The invariants (mass, energy, and adjoint consistency) are once more preserved to machine precision, as illustrated in Figure~\ref{fig:property_init1_2}. Snapshots of the scalar field (Figure~\ref{fig:theta_init1_2}) demonstrate that the final state is effectively well-mixed. These results validate that our structure-preserving scheme can achieve efficient mixing in practice: the optimal control steers the flow to accelerate the decay of the mix-norm (relative to any single steady flow), while the discrete conservation laws remain exactly satisfied throughout the process.

\begin{figure}
  \centering
  \begin{subfigure}[b]{0.315\textwidth}
    \includegraphics[width=\textwidth]{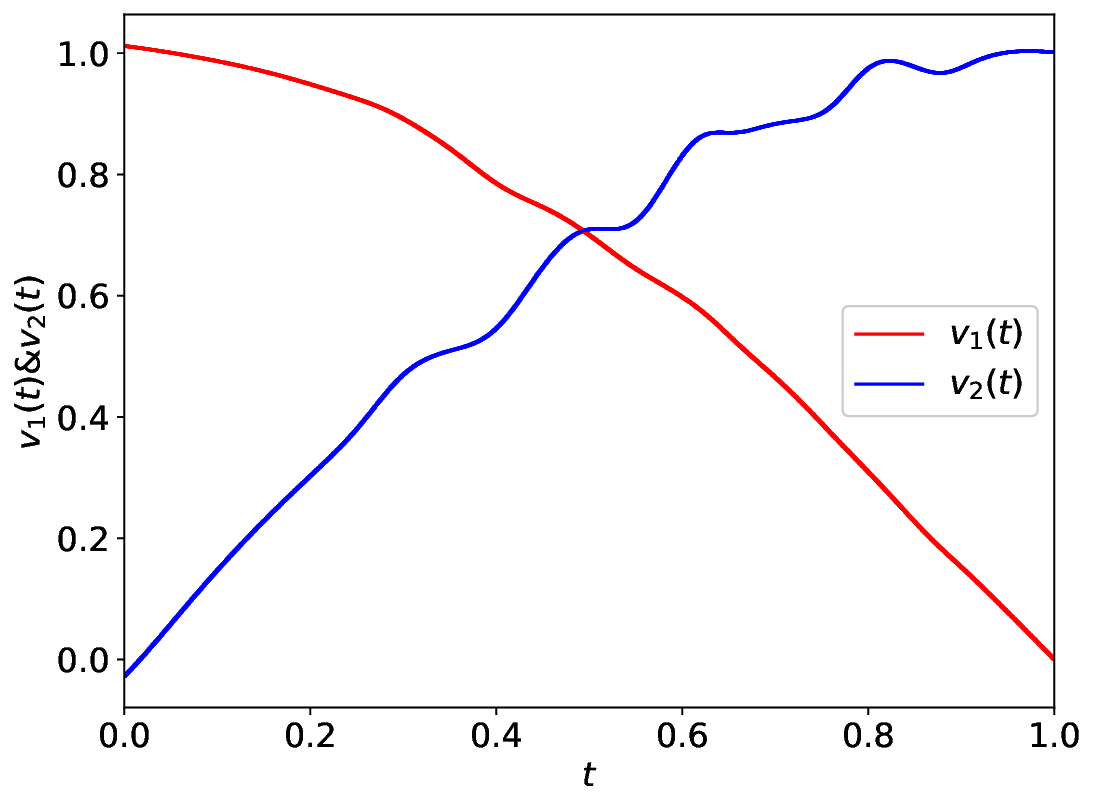}
    \caption{Evolutions of $v_1(t)$, $v_2(t)$.}
    \label{fig:u1u2_init1_2}
  \end{subfigure}
  \begin{subfigure}[b]{0.335\textwidth}
    \includegraphics[width=\textwidth]{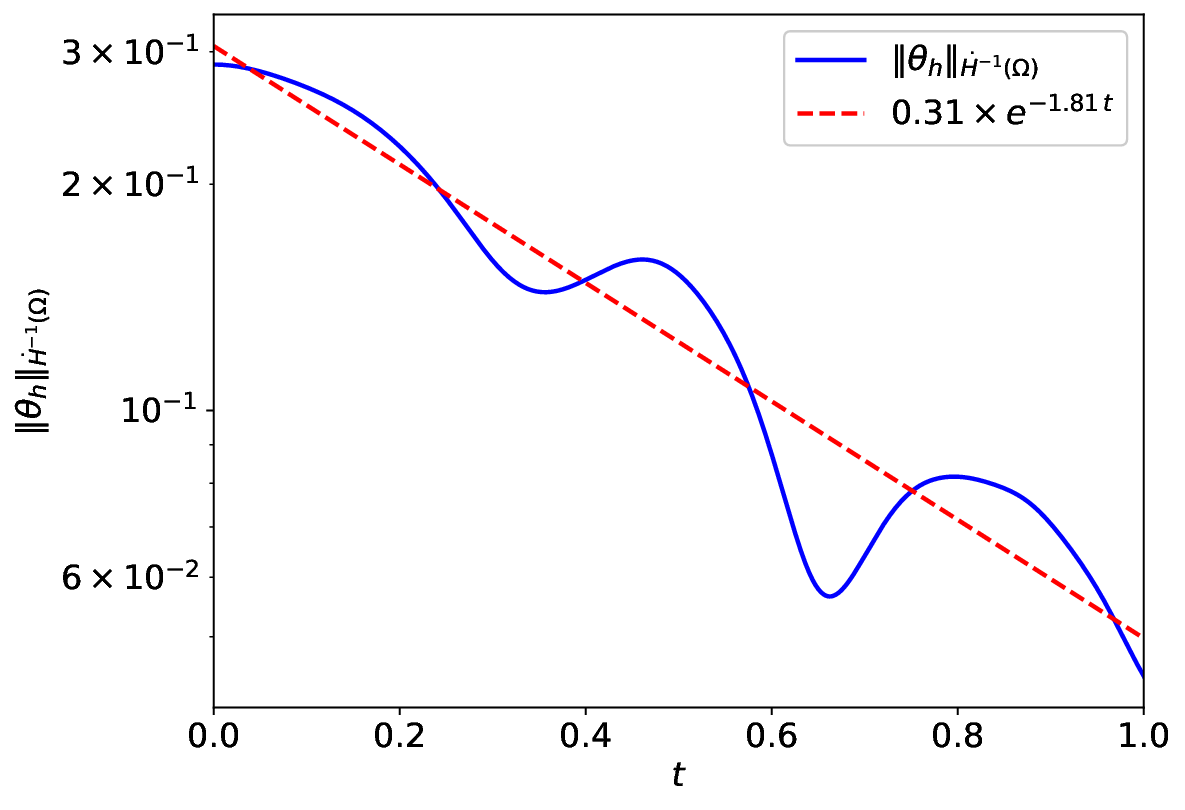}
    \caption{Evolution of $\Vert \theta_h \Vert_{\dot{H}^{-1}(\Omega)}$.}
    \label{fig:mixnorm_init1_2}
  \end{subfigure}
  \begin{subfigure}[b]{0.33\textwidth}
    \includegraphics[width=\textwidth]{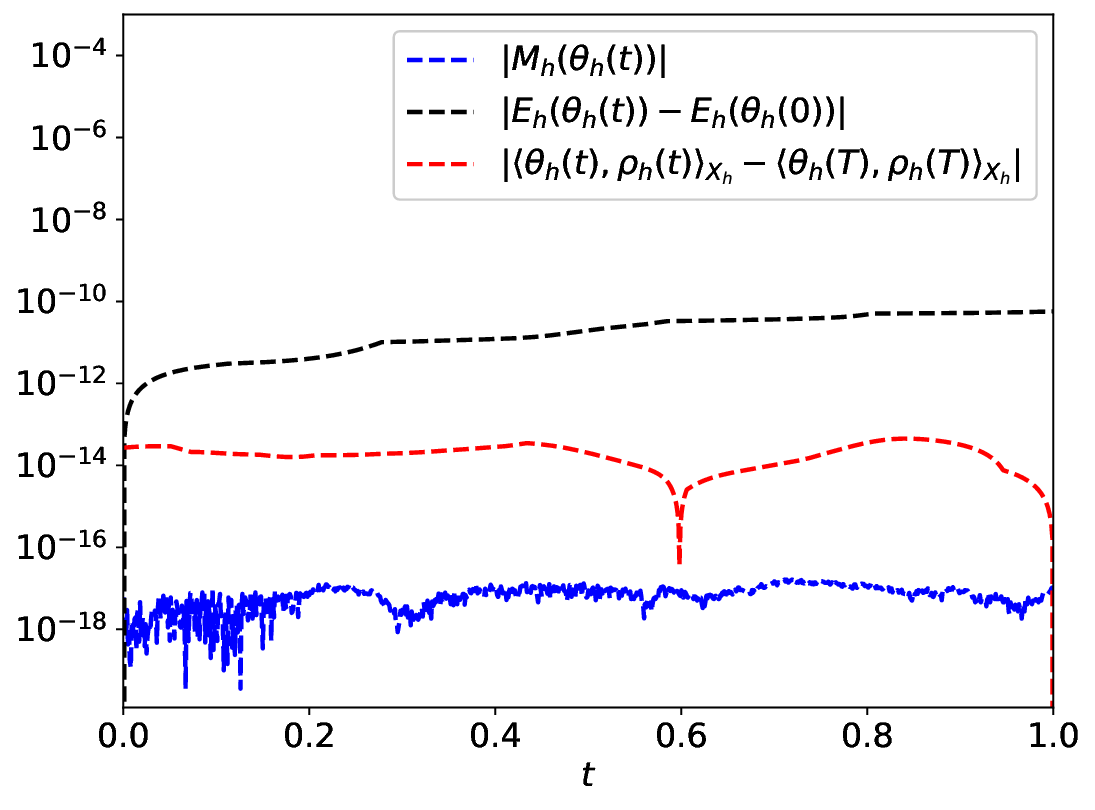}
    \caption{$M_h(\theta_h)$, $E_h(\theta_h)$, $\langle \theta_h, \rho_h \rangle_{X_h}$.}
    \label{fig:property_init1_2}
  \end{subfigure}
    \caption{Evolutions of $v_1(t)$, $v_2(t)$, mix-norm $\Vert \theta_h \Vert_{\dot{H}^{-1}(\Omega)}$, mass $M_h(\theta_h)$, energy $E_h(\theta_h)$ and state-adjoint pairing $\langle \theta_h, \rho_h \rangle_{X_h}$ for $t\in[0, 1]$ with initial control \eqref{eq:u1u2_2} and initial data \eqref{eq:theta_init_1}.}
  \label{fig:u1u2_mixnorm_init1_2}
\end{figure}

\begin{figure}
  \centering
  \begin{subfigure}[b]{0.16\textwidth}
    \includegraphics[width=\textwidth]{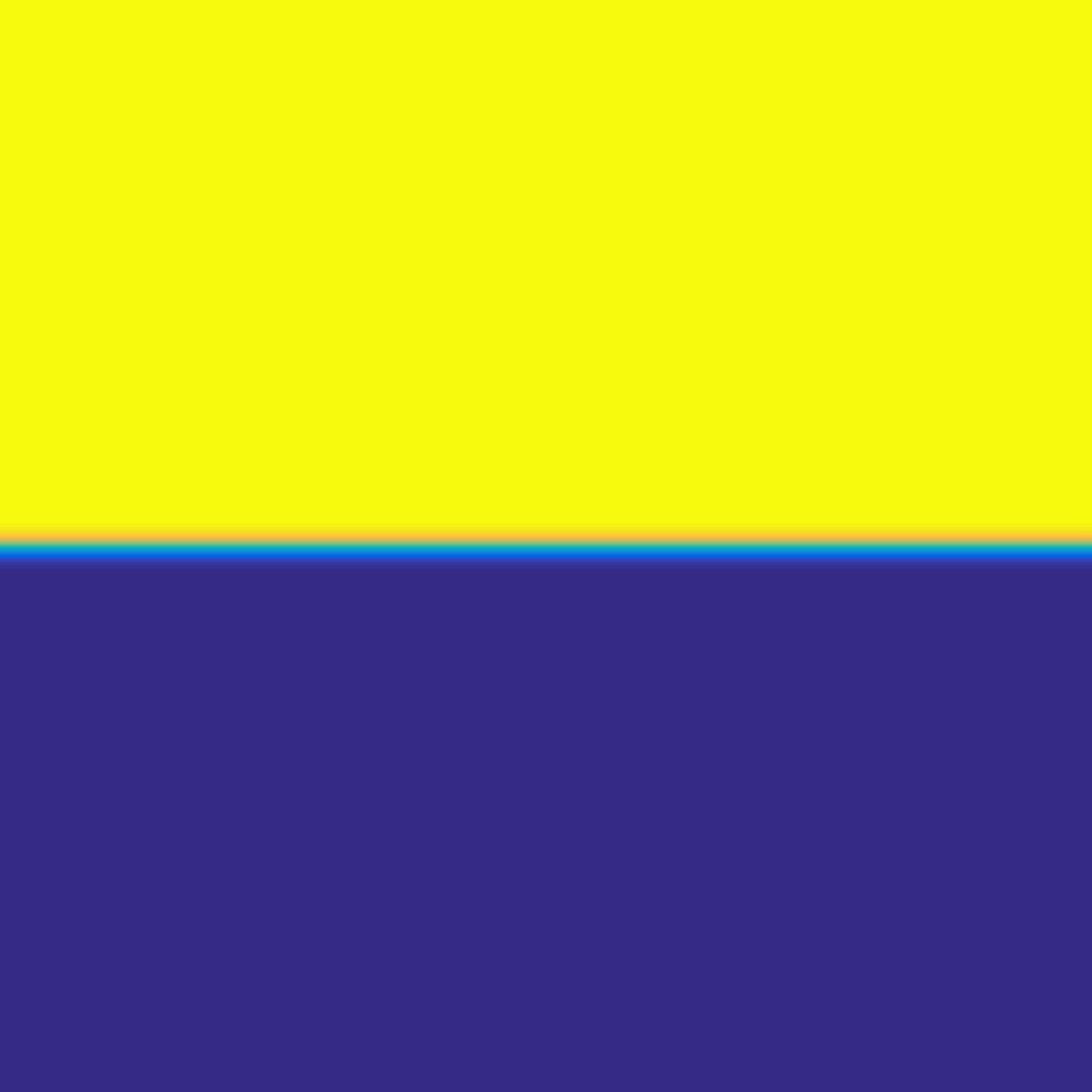}
    \caption{$t=0$}
  \end{subfigure}
  \begin{subfigure}[b]{0.16\textwidth}
    \includegraphics[width=\textwidth]{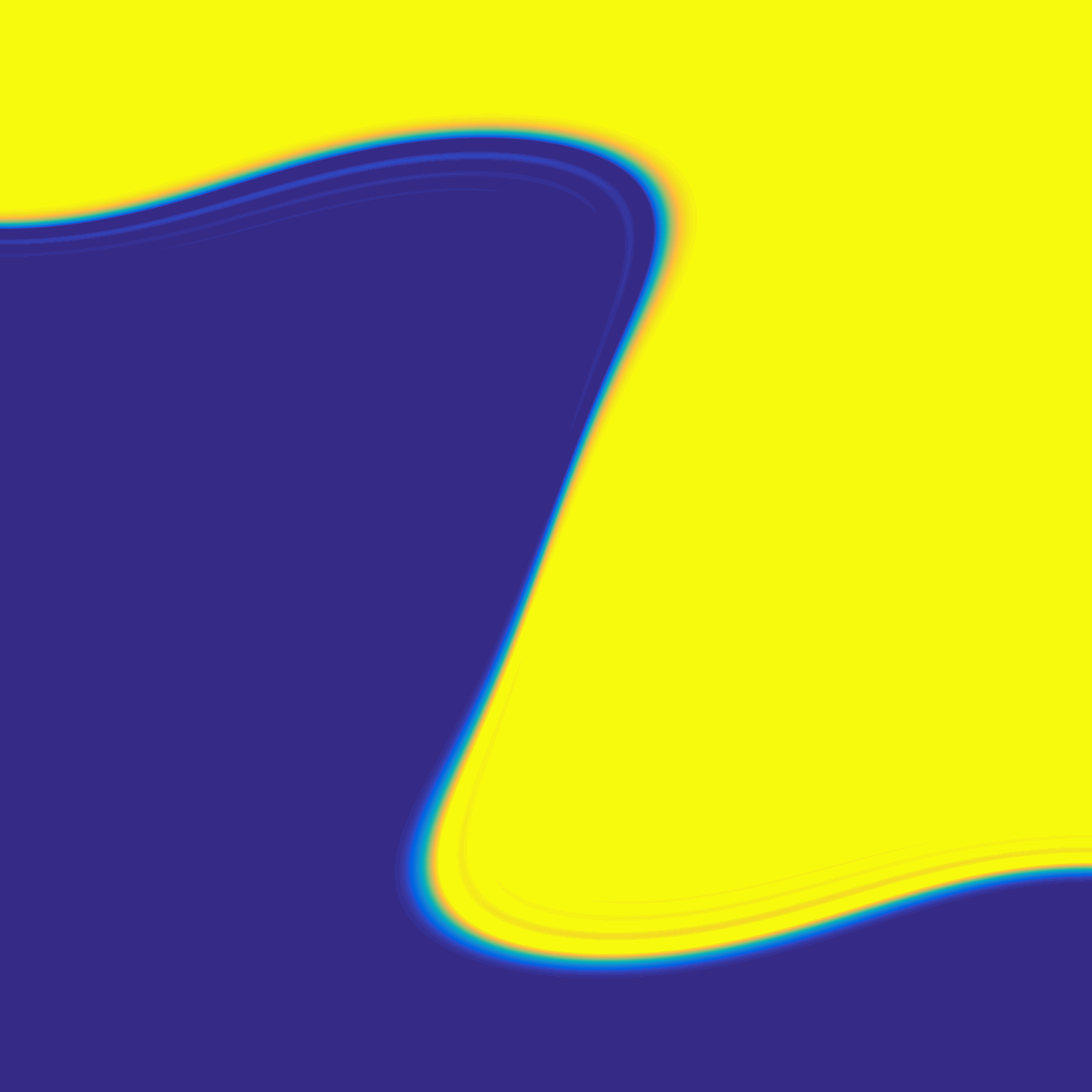}
    \caption{$t=0.2$}
  \end{subfigure}
  \begin{subfigure}[b]{0.16\textwidth}
    \includegraphics[width=\textwidth]{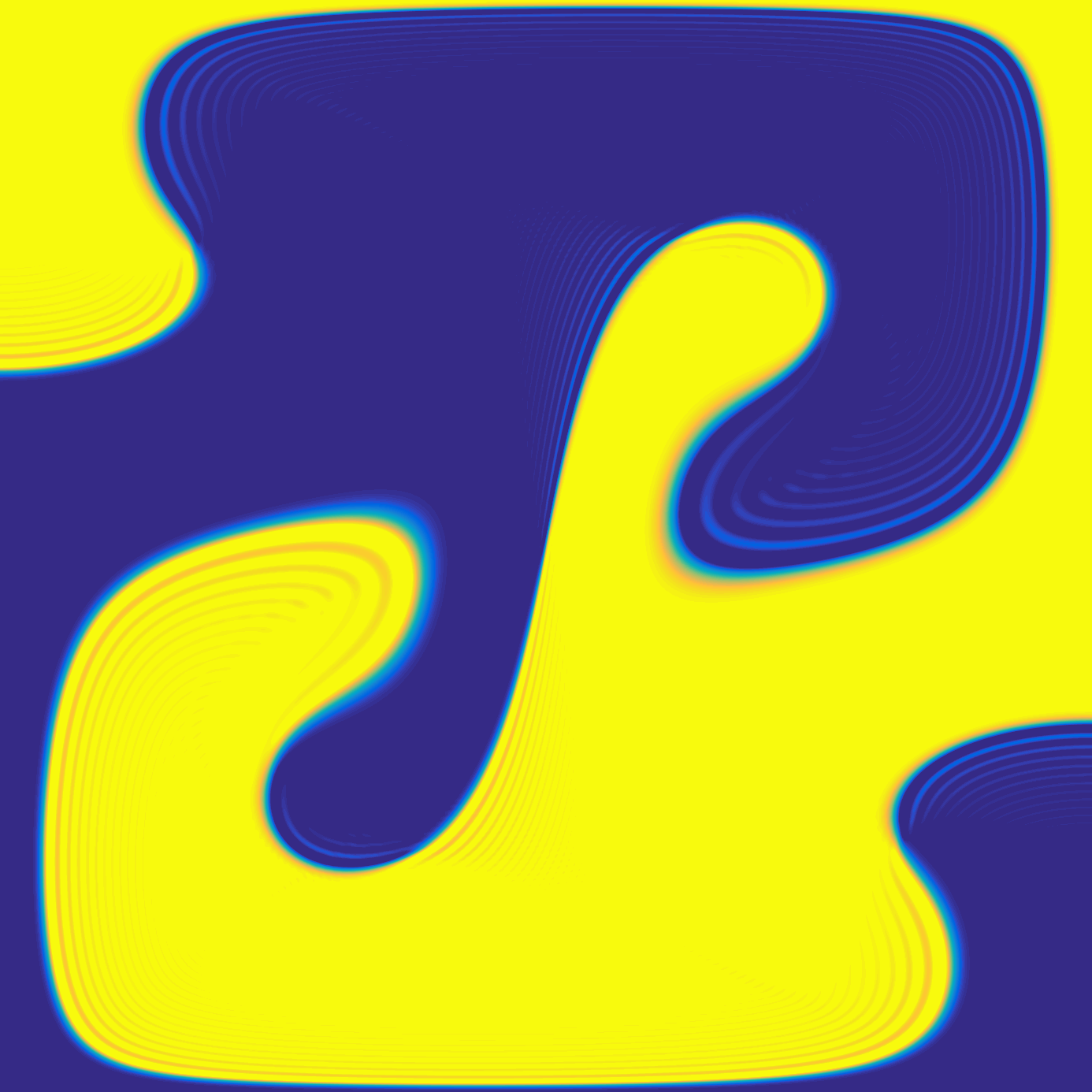}
    \caption{$t=0.4$}
  \end{subfigure}
  \begin{subfigure}[b]{0.16\textwidth}
    \includegraphics[width=\textwidth]{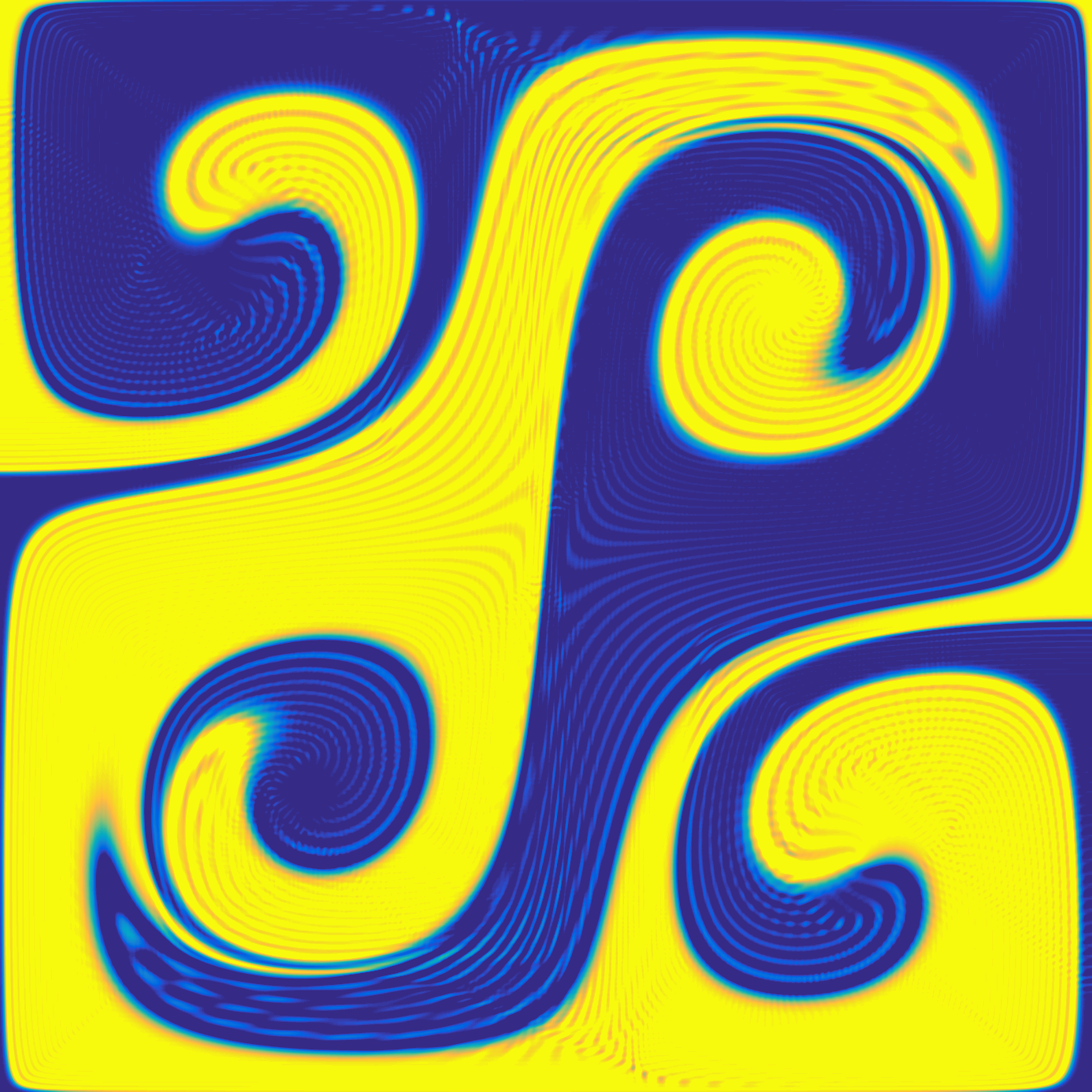}
    \caption{$t=0.6$}
  \end{subfigure}
  \begin{subfigure}[b]{0.16\textwidth}
    \includegraphics[width=\textwidth]{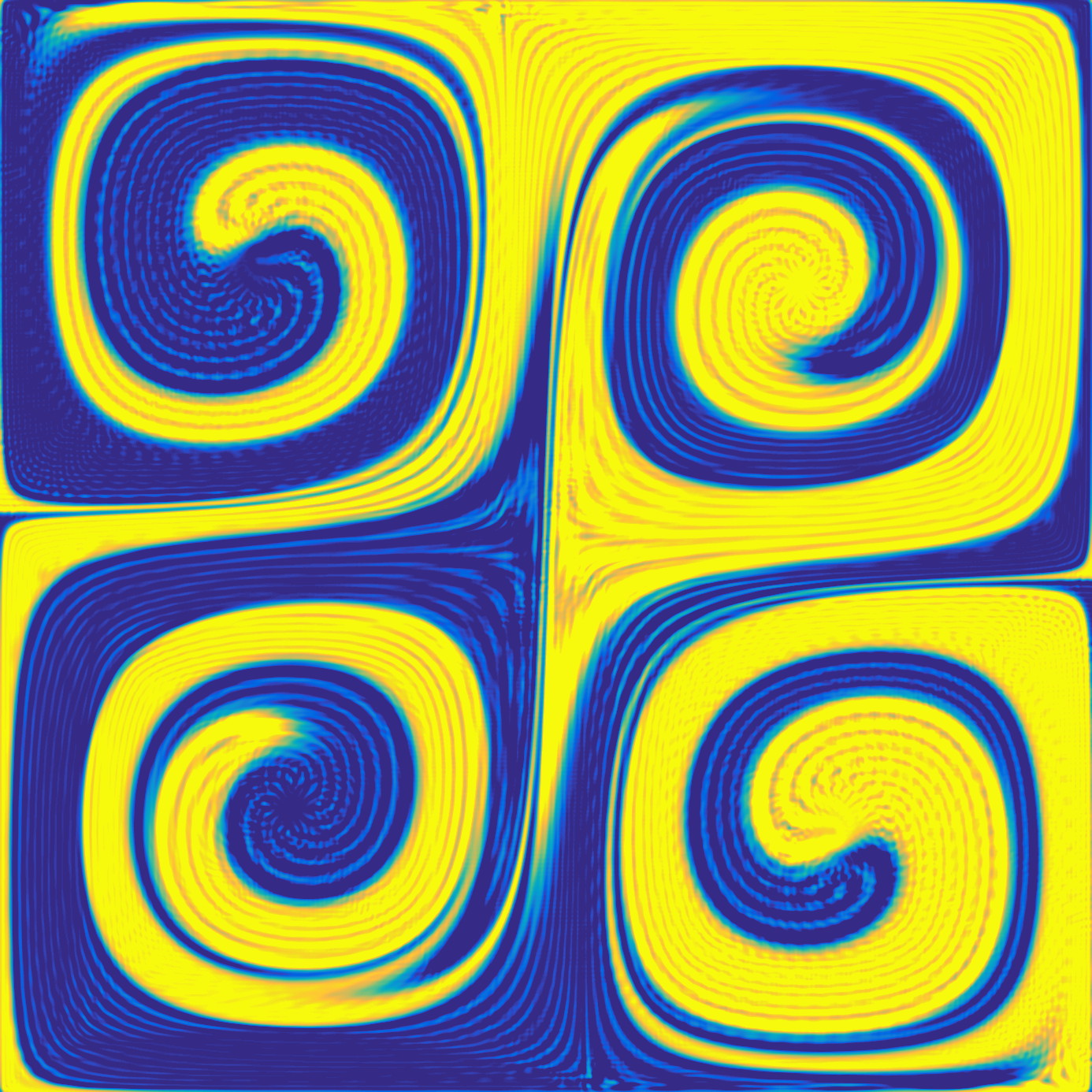}
    \caption{$t=0.8$}
  \end{subfigure}
  \begin{subfigure}[b]{0.16\textwidth}
    \includegraphics[width=\textwidth]{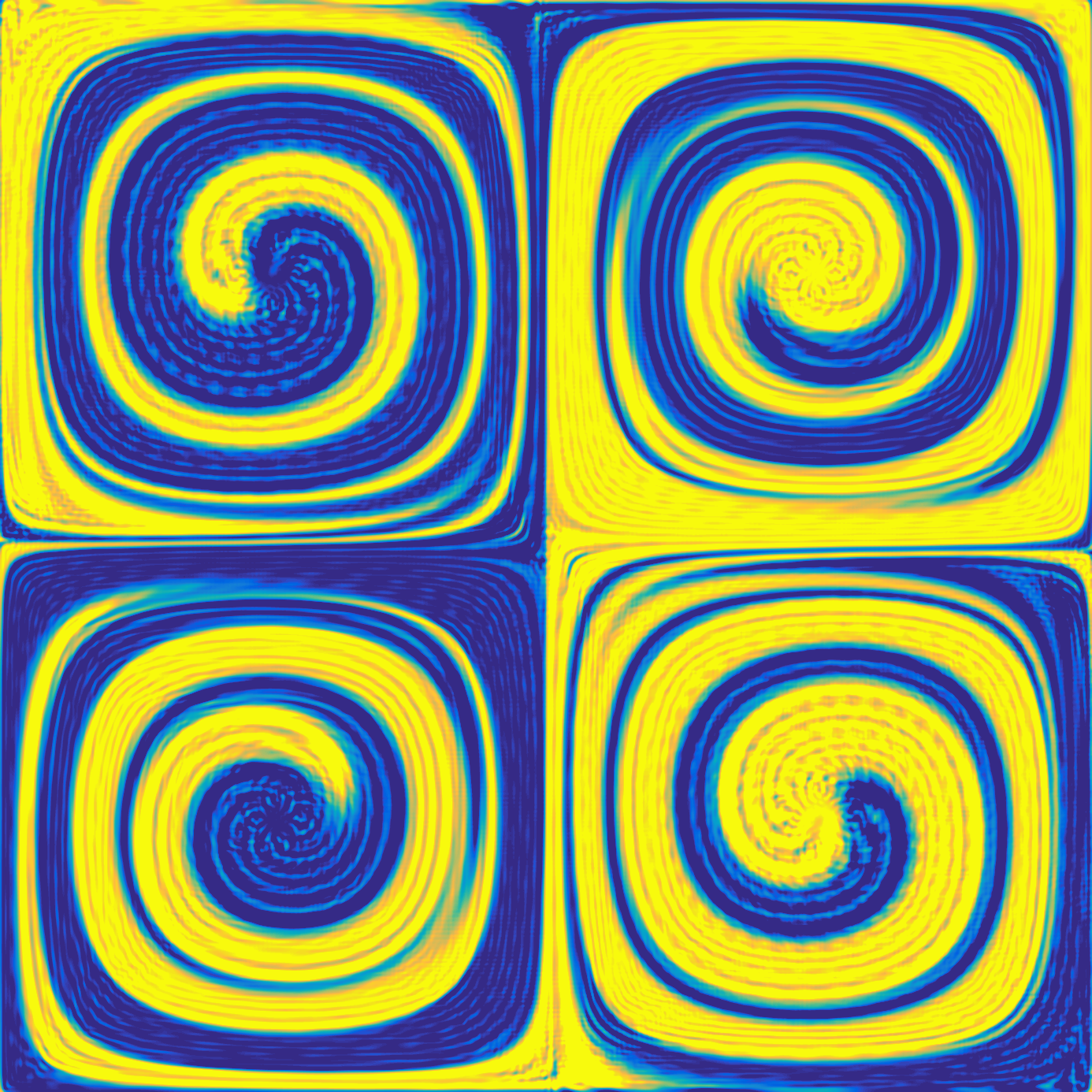}
    \caption{$t=1.0$}
  \end{subfigure}
  \caption{Evolution of $\theta_h$ for $t\in[0, 1]$ with initial control \eqref{eq:u1u2_2} and initial data \eqref{eq:theta_init_1}.}
  \label{fig:theta_init1_2}
\end{figure}

\subsubsection{Optimal control for initial data \eqref{eq:theta_init_2}}
We now turn to the smoother trigonometric initial data \eqref{eq:theta_init_2}. We repeat the optimal control experiments for this case, again using both $\mathbf{b}_1$ and $\mathbf{b}_2$ as control directions and test two different initial guesses for the control. 

Using the initial guess \eqref{eq:u1u2_1}, Figure~\ref{fig:u1u2_init2_1} shows the evolution of the optimized coefficients $v_1(t)$ and $v_2(t)$. The corresponding mix-norm decay is plotted in Figure~\ref{fig:mixnorm_init2_1}: the curve suggests an almost exponential decay with an approximate rate of $1.19$. The discrete mass, energy, and state--adjoint pairing are again conserved (Figure~\ref{fig:mixnorm_init2_1}) with errors on the order of machine precision. Snapshots of the scalar field are given in Figure~\ref{fig:theta_init2_1}. We note that in this optimized control, one of the coefficients ($v_2(t)$) changes sign during the time interval, indicating that the flow reverses its circulation at some point to enhance mixing effectiveness.

\begin{figure}
  \centering
  \begin{subfigure}[b]{0.315\textwidth}
    \includegraphics[width=\textwidth]{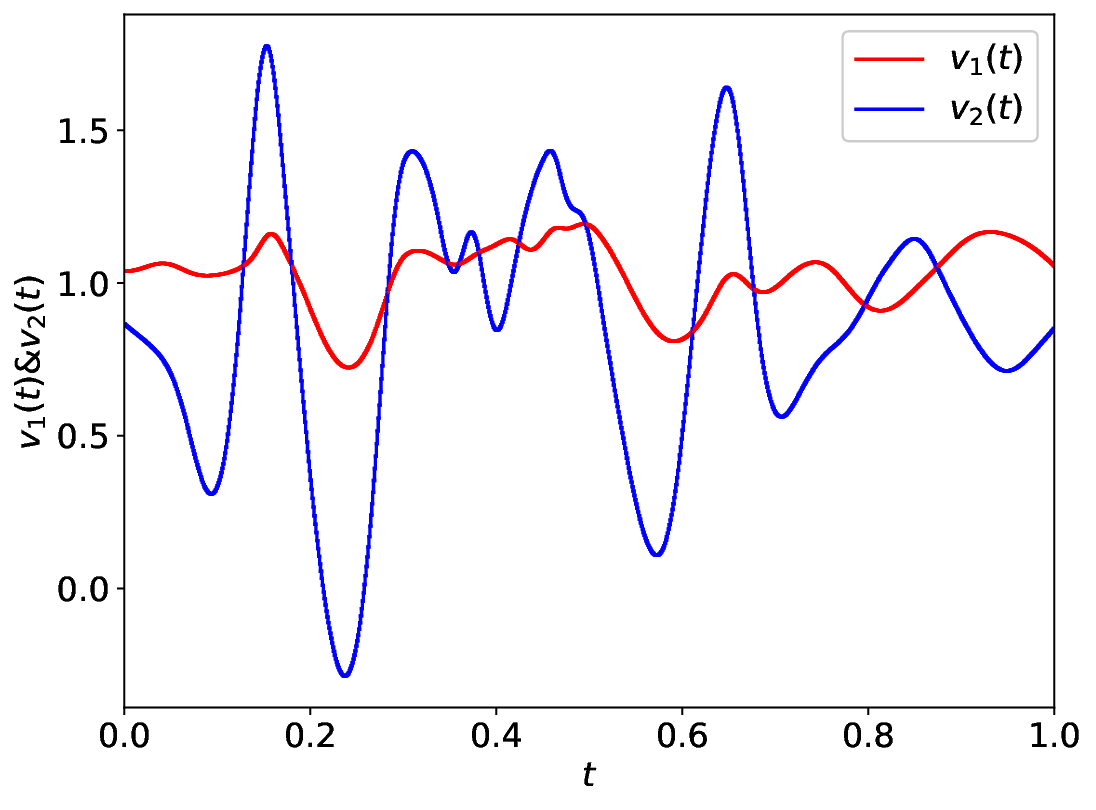}
    \caption{Evolutions of $v_1(t)$, $v_2(t)$.}
    \label{fig:u1u2_init2_1}
  \end{subfigure}
  \begin{subfigure}[b]{0.335\textwidth}
    \includegraphics[width=\textwidth]{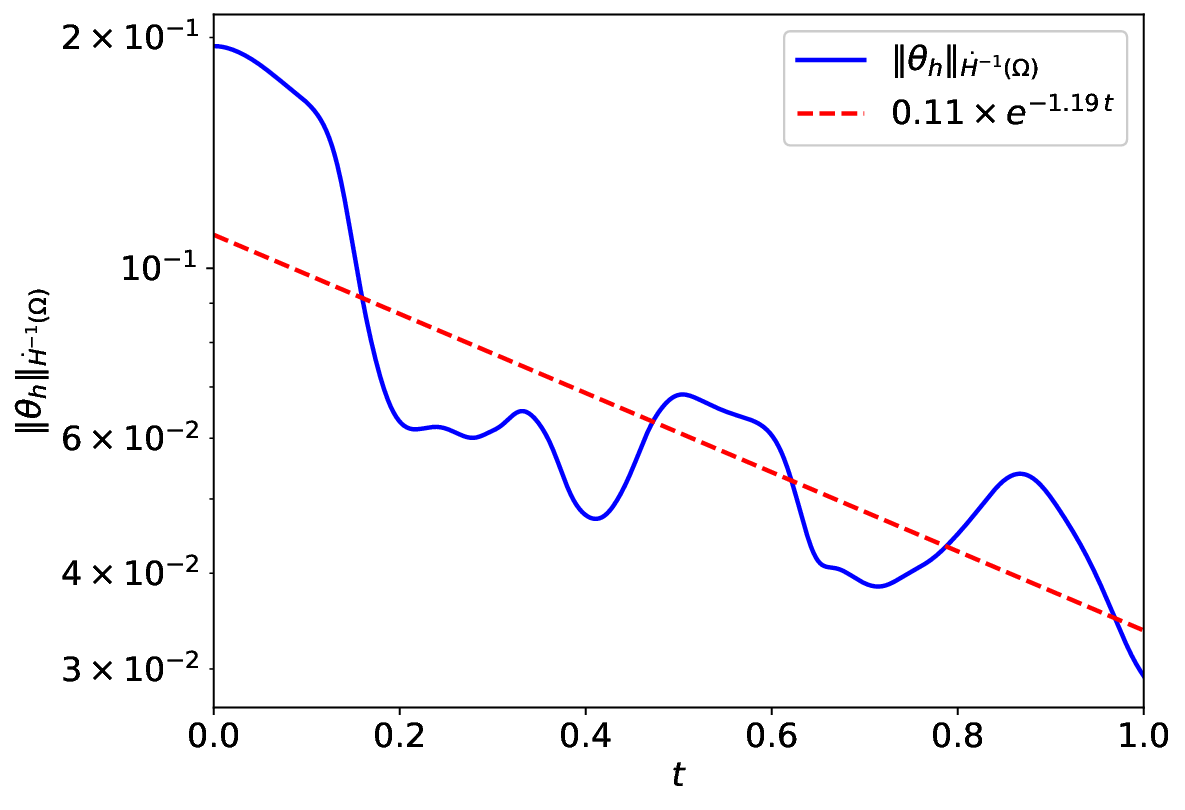}
    \caption{Evolution of $\Vert \theta_h \Vert_{\dot{H}^{-1}(\Omega)}$.}
    \label{fig:mixnorm_init2_1}
  \end{subfigure}
  \begin{subfigure}[b]{0.33\textwidth}
    \includegraphics[width=\textwidth]{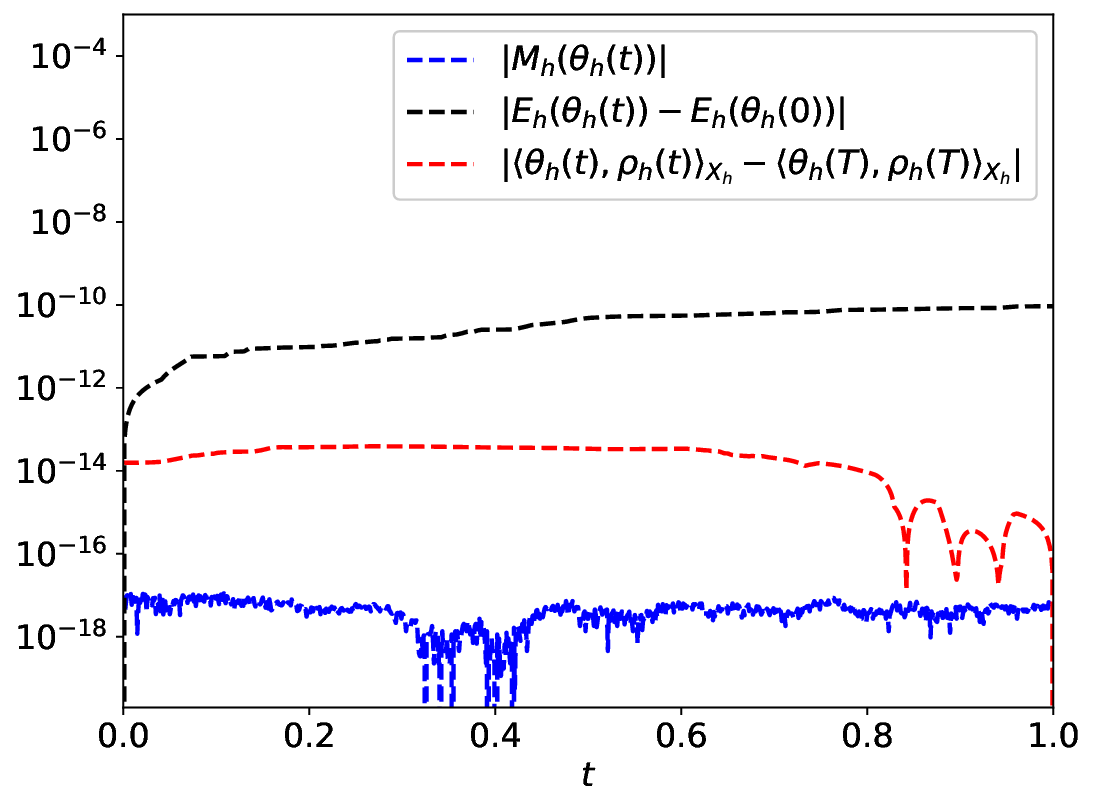}
    \caption{$M_h(\theta_h)$, $E_h(\theta_h)$, $\langle \theta_h, \rho_h \rangle_{X_h}$.}
    \label{fig:property_init2_1}
  \end{subfigure}
    \caption{Evolutions of $v_1(t)$, $v_2(t)$, mix-norm $\Vert \theta_h \Vert_{\dot{H}^{-1}(\Omega)}$, mass $M_h(\theta_h)$, energy $E_h(\theta_h)$ and state-adjoint pairing $\langle \theta_h, \rho_h \rangle_{X_h}$ for $t\in[0, 1]$ with initial control \eqref{eq:u1u2_1} and initial data \eqref{eq:theta_init_2}.}
  \label{fig:u1u2_mixnorm_init2_1}
\end{figure}

\begin{figure}
  \centering
  \begin{subfigure}[b]{0.16\textwidth}
    \includegraphics[width=\textwidth]{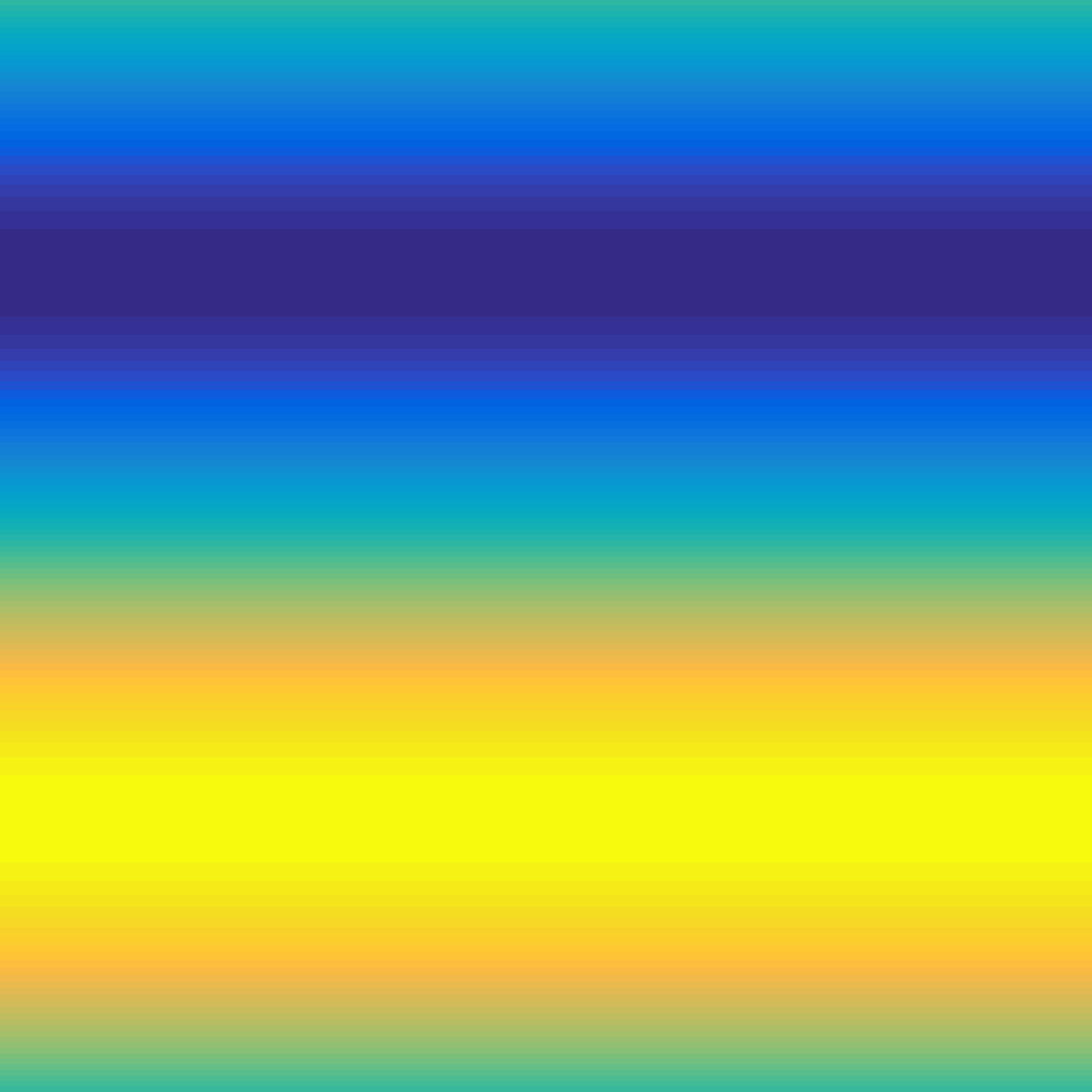}
    \caption{$t=0$}
  \end{subfigure}
  \begin{subfigure}[b]{0.16\textwidth}
    \includegraphics[width=\textwidth]{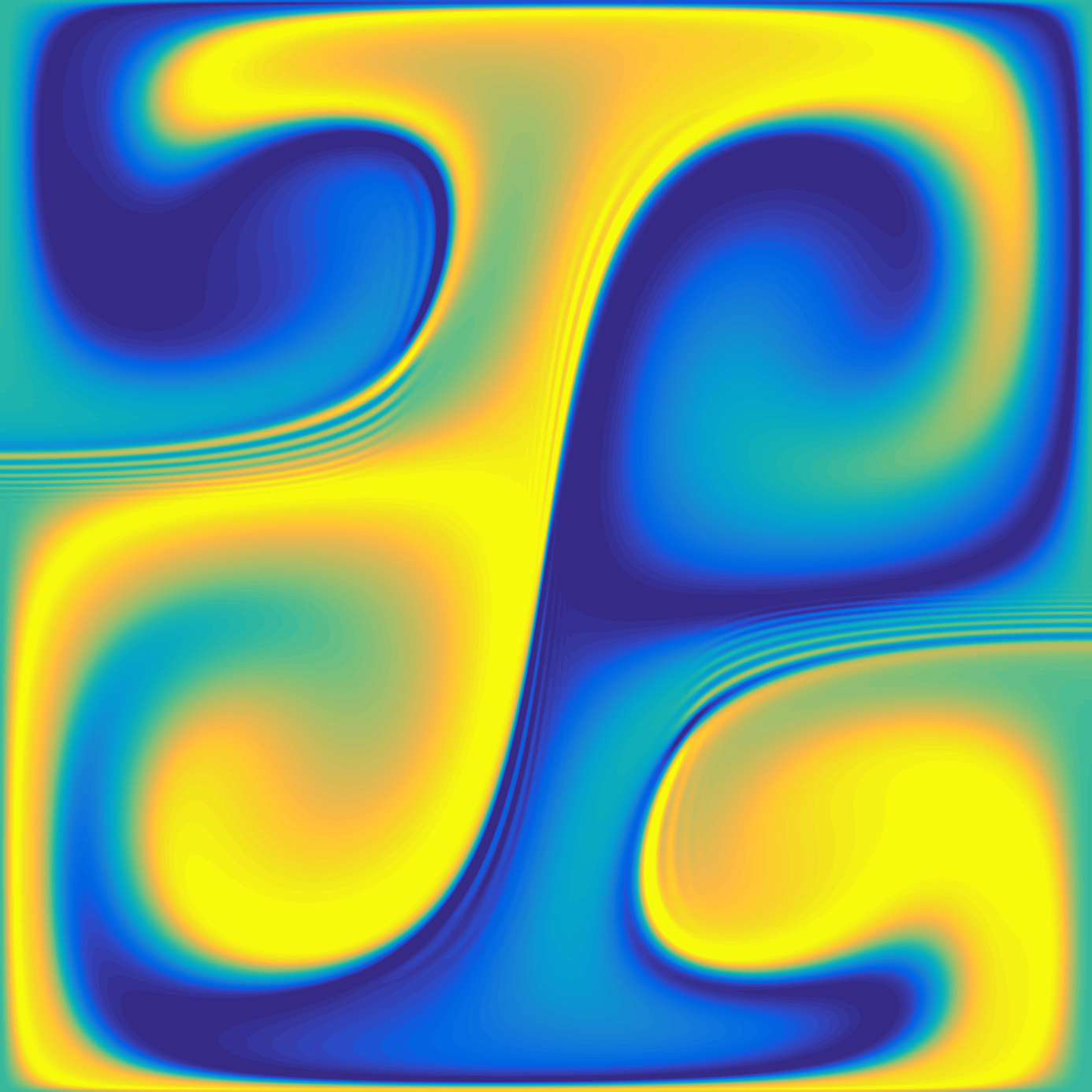}
    \caption{$t=0.2$}
  \end{subfigure}
  \begin{subfigure}[b]{0.16\textwidth}
    \includegraphics[width=\textwidth]{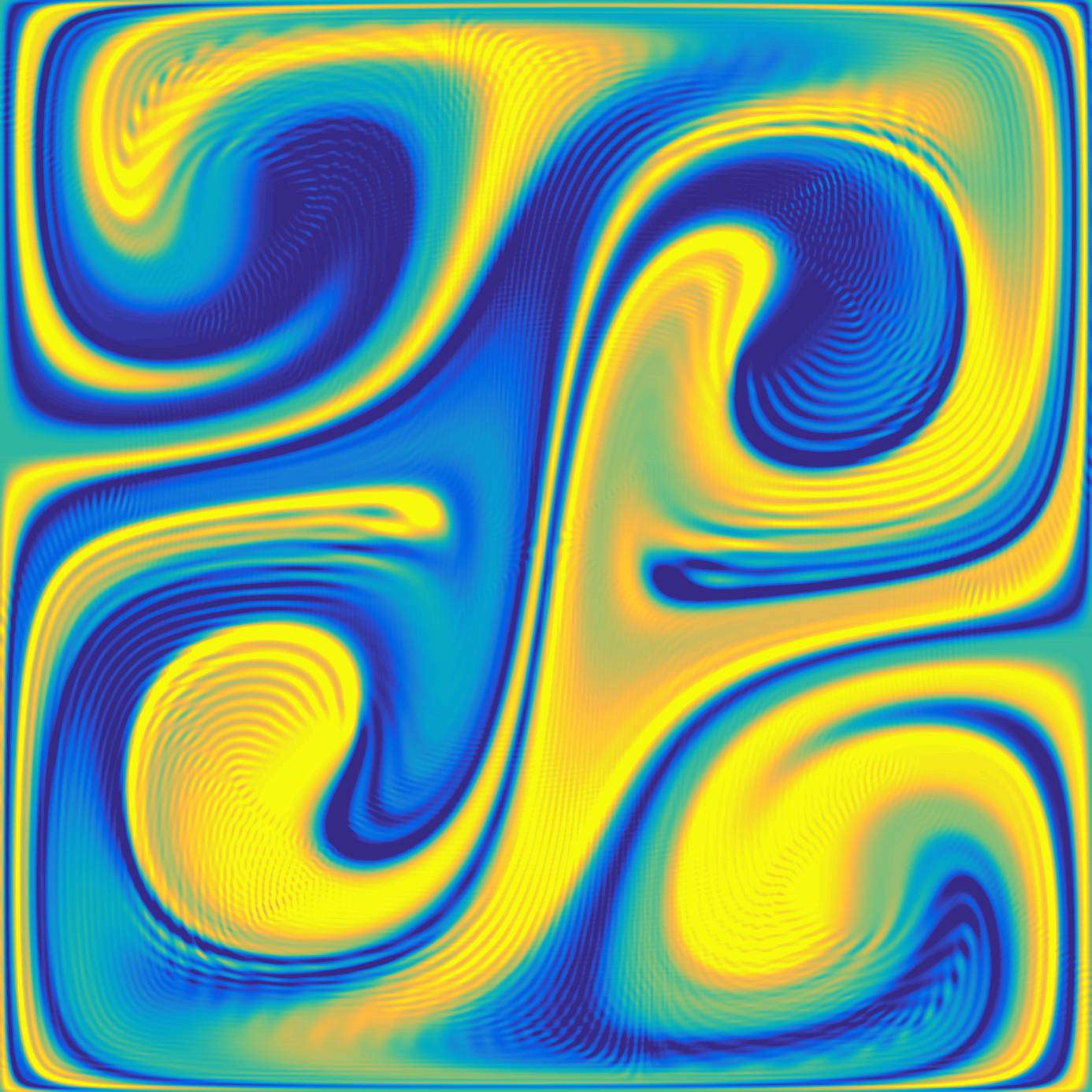}
    \caption{$t=0.4$}
  \end{subfigure}
  \begin{subfigure}[b]{0.16\textwidth}
    \includegraphics[width=\textwidth]{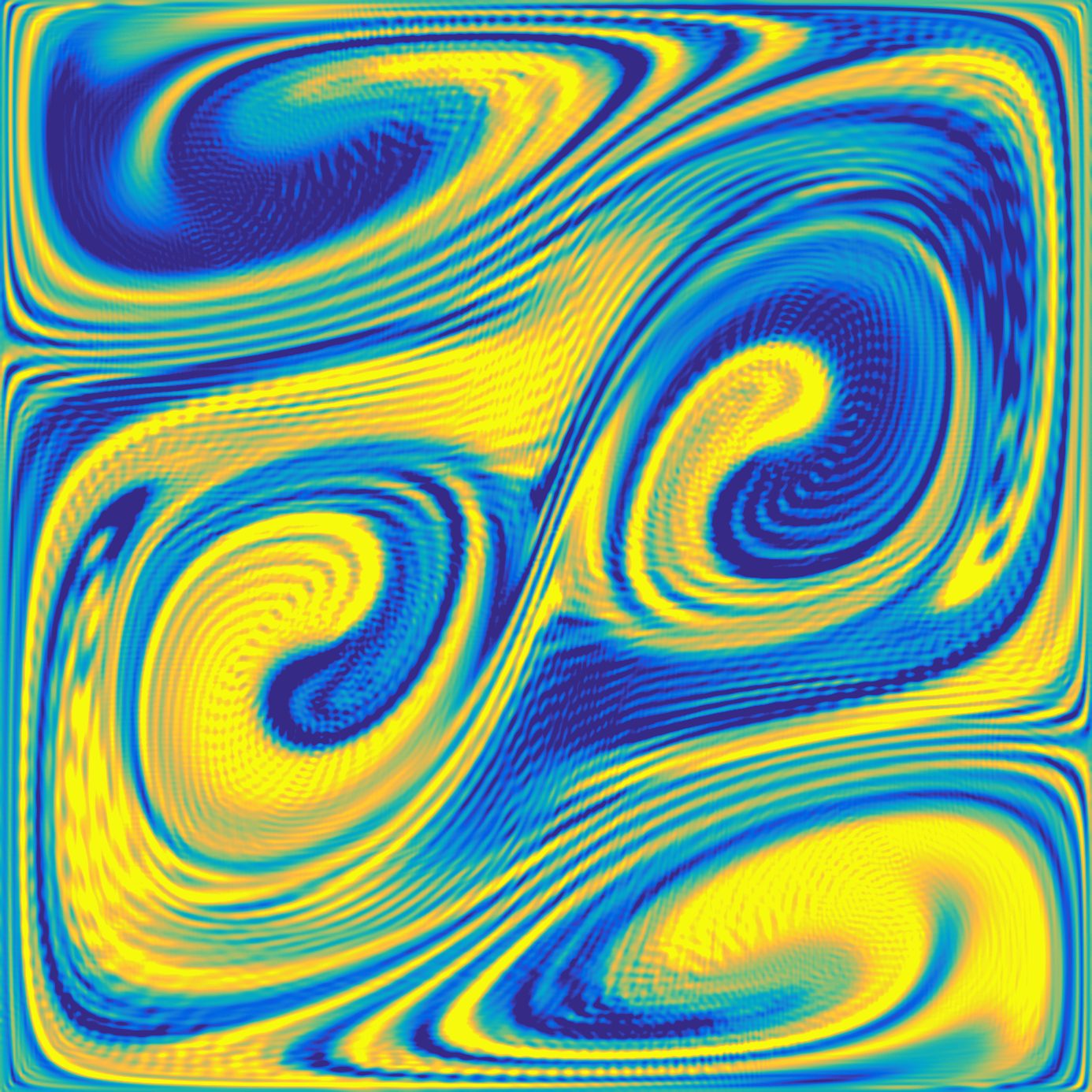}
    \caption{$t=0.6$}
  \end{subfigure}
  \begin{subfigure}[b]{0.16\textwidth}
    \includegraphics[width=\textwidth]{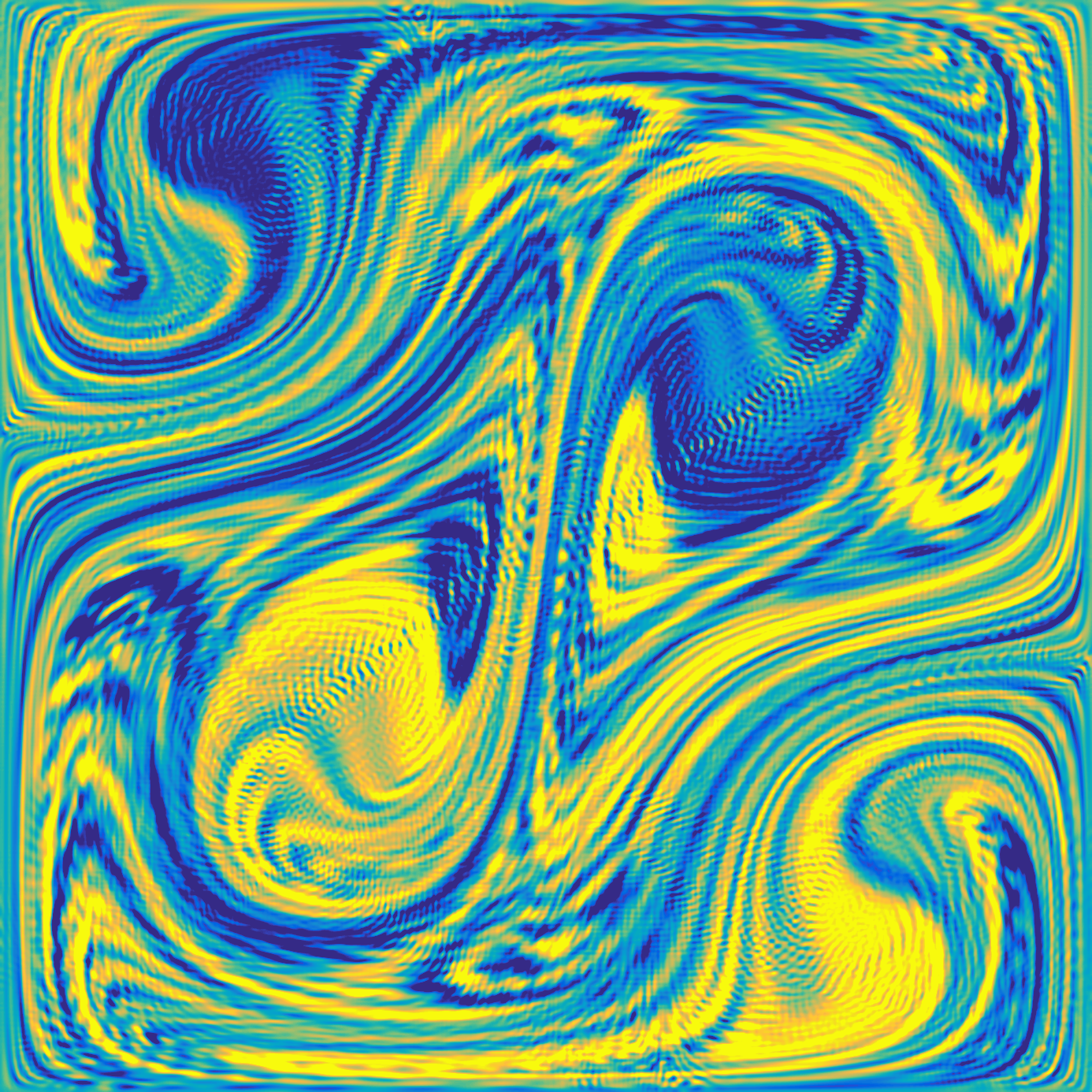}
    \caption{$t=0.8$}
  \end{subfigure}
  \begin{subfigure}[b]{0.16\textwidth}
    \includegraphics[width=\textwidth]{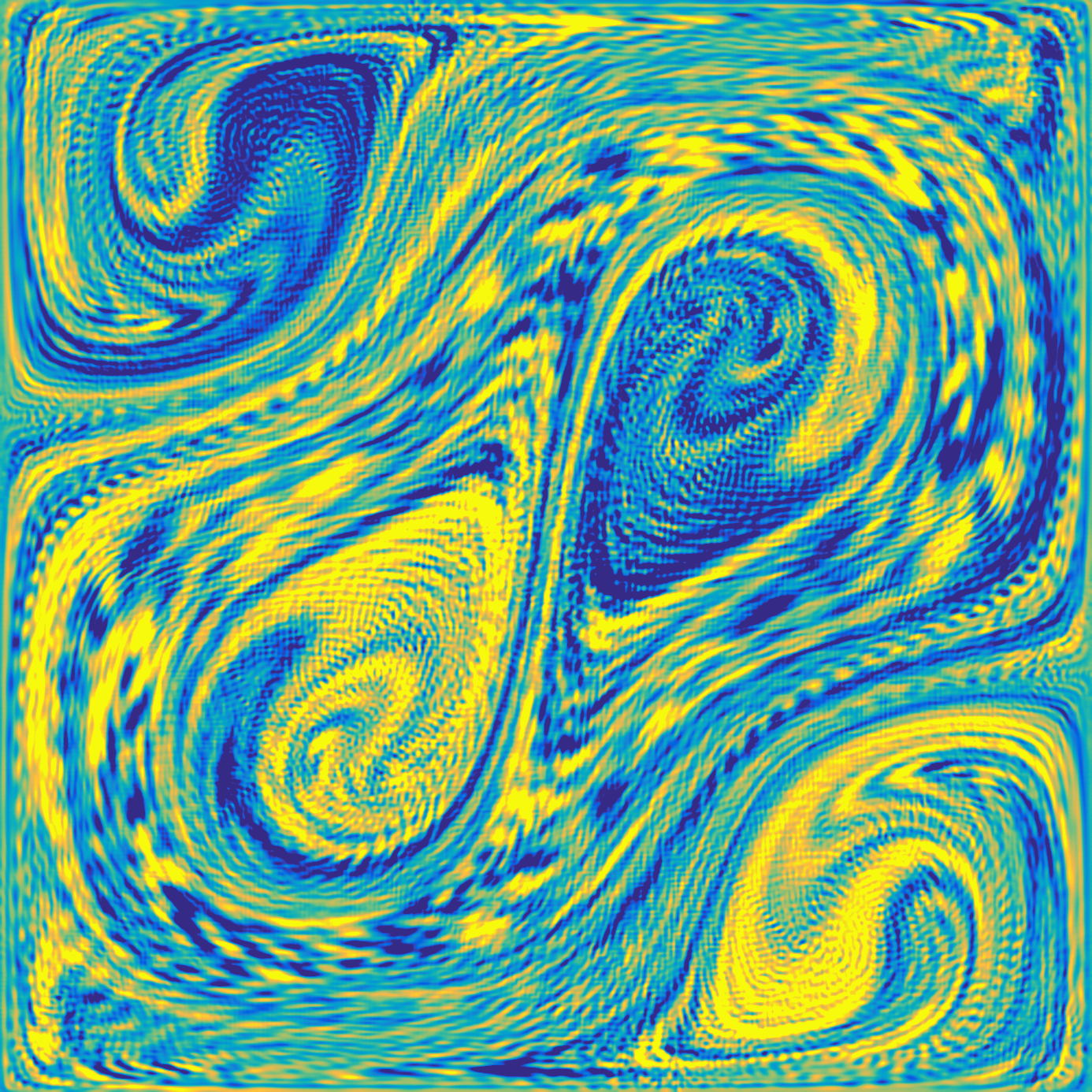}
    \caption{$t=1.0$}
  \end{subfigure}
  \caption{Evolution of $\theta_h$ for $t\in[0, 1]$ with initial control \eqref{eq:u1u2_1} and initial data \eqref{eq:theta_init_2}.}
  \label{fig:theta_init2_1}
\end{figure}

We also perform the optimization starting from the control guess \eqref{eq:u1u2_2}. The optimized control in this scenario is shown in Figure~\ref{fig:u1u2_init2_2}. In this scenario, the mix-norm again decays rapidly, with behavior that is well approximated by an exponential with a rate of about 2.17 (Figure~\ref{fig:mixnorm_init2_2}). Figure~\ref{fig:property_init2_2} again confirms invariant preservation to machine precision, and Figure~\ref{fig:theta_init2_2} shows snapshots of the scalar field.

Overall, the optimally controlled flows for both types of initial data \eqref{eq:theta_init_1} and \eqref{eq:theta_init_2} achieve substantially faster mixing (in terms of mix-norm decay) than any single basis flow could produce. The nearly exponential decay rates obtained underline the efficacy of the control in stirring the scalar. At the same time, the structure-preserving discretization ensures that the fundamental physical invariants of the system are never violated, lending credence to the physical fidelity of the simulations.

\begin{figure}
  \centering
  \begin{subfigure}[b]{0.315\textwidth}
    \includegraphics[width=\textwidth]{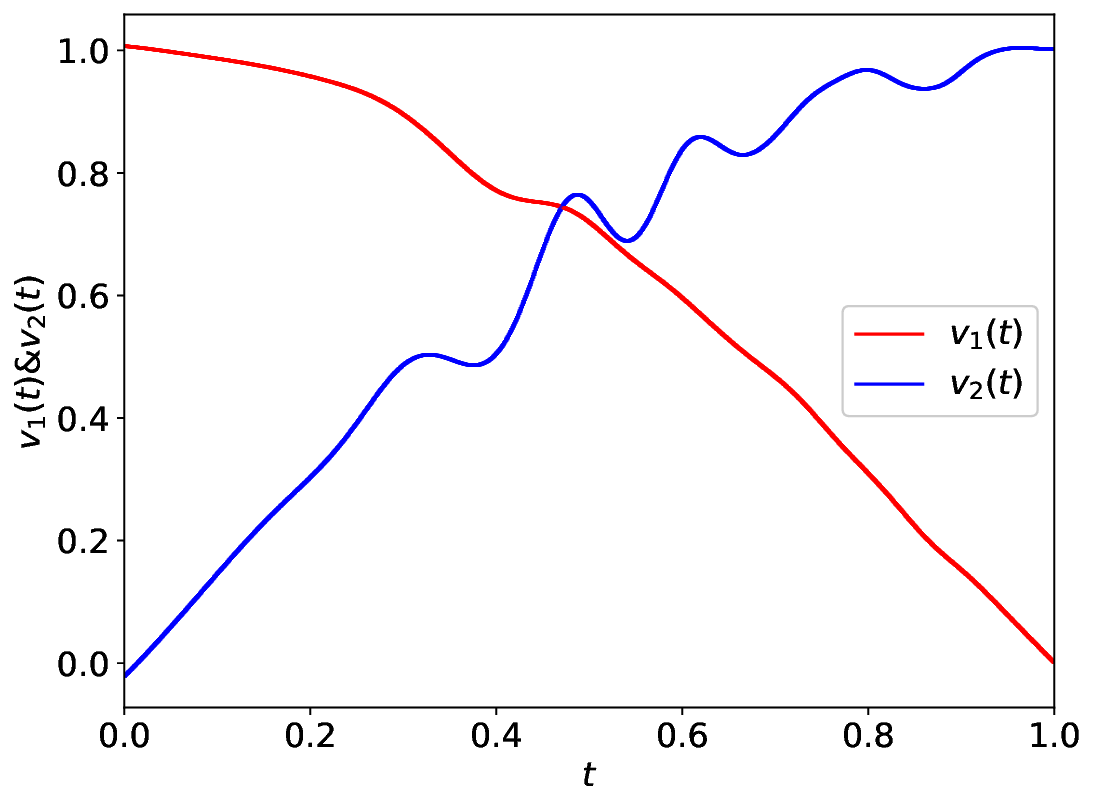}
    \caption{Evolutions of $v_1(t)$, $v_2(t)$.}
    \label{fig:u1u2_init2_2}
  \end{subfigure}
  \begin{subfigure}[b]{0.335\textwidth}
    \includegraphics[width=\textwidth]{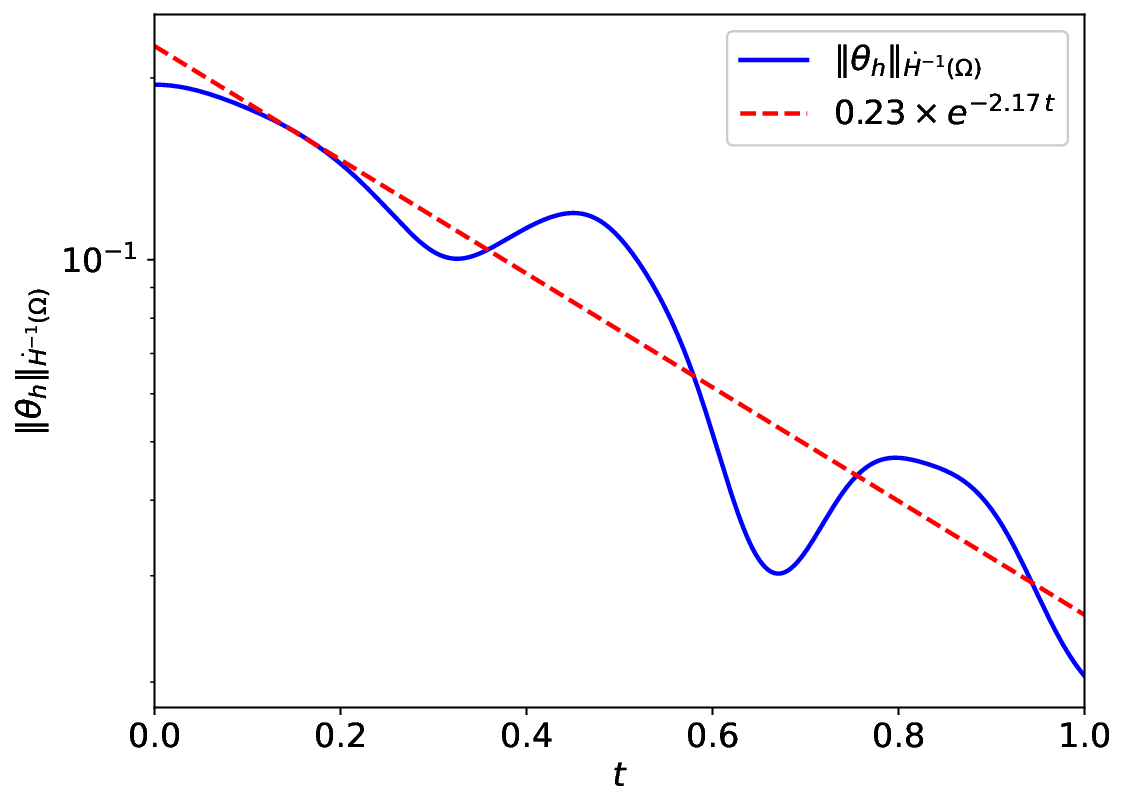}
    \caption{Evolution of $\Vert \theta_h \Vert_{\dot{H}^{-1}(\Omega)}$.}
    \label{fig:mixnorm_init2_2}
  \end{subfigure}
  \begin{subfigure}[b]{0.33\textwidth}
    \includegraphics[width=\textwidth]{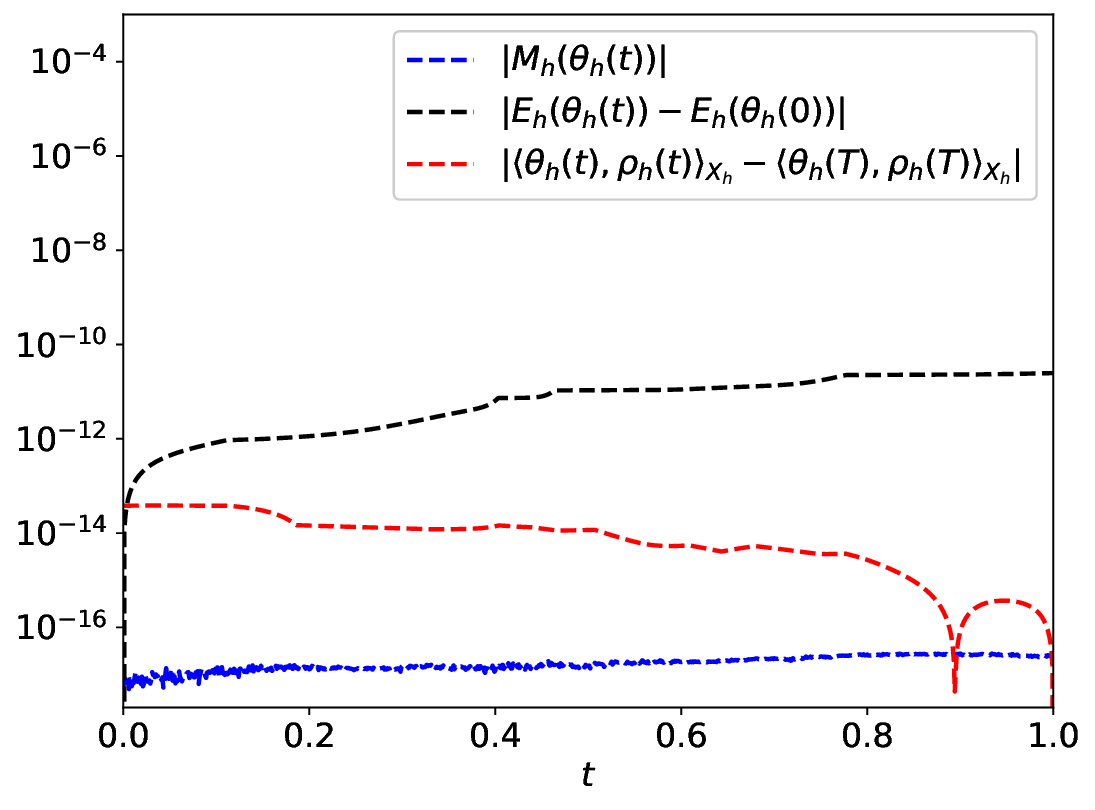}
    \caption{$M_h(\theta_h)$, $E_h(\theta_h)$, $\langle \theta_h, \rho_h \rangle_{X_h}$.}
    \label{fig:property_init2_2}
  \end{subfigure}
    \caption{Evolutions of $v_1(t)$, $v_2(t)$, mix-norm $\Vert \theta_h \Vert_{\dot{H}^{-1}(\Omega)}$, mass $M_h(\theta_h)$, energy $E_h(\theta_h)$ and state-adjoint pairing $\langle \theta_h, \rho_h \rangle_{X_h}$ for $t\in[0, 1]$ with initial control \eqref{eq:u1u2_2} and initial data \eqref{eq:theta_init_2}.}
  \label{fig:u1u2_mixnorm_init2_2}
\end{figure}

\begin{figure}
  \centering
  \begin{subfigure}[b]{0.16\textwidth}
    \includegraphics[width=\textwidth]{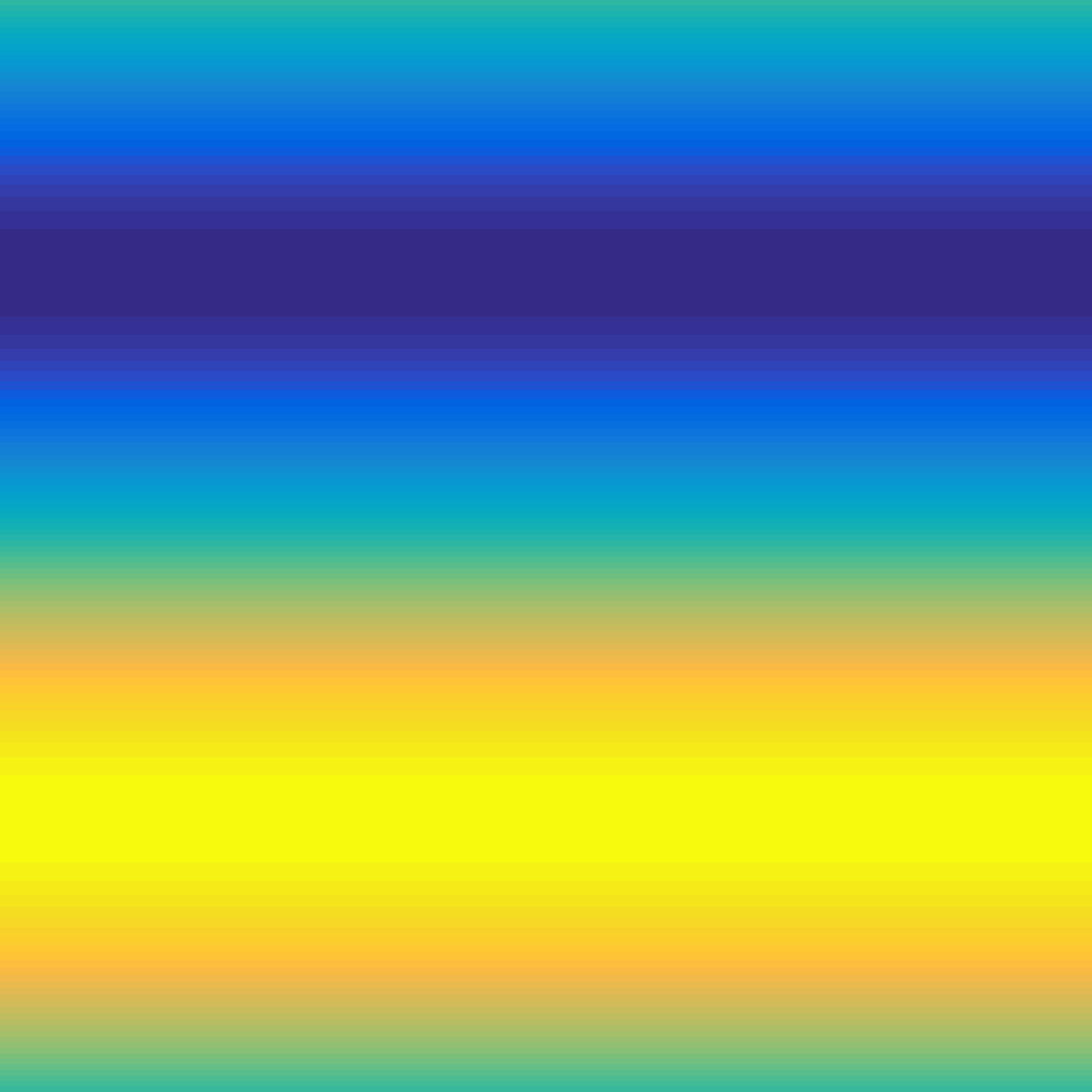}
    \caption{$t=0$}
  \end{subfigure}
  \begin{subfigure}[b]{0.16\textwidth}
    \includegraphics[width=\textwidth]{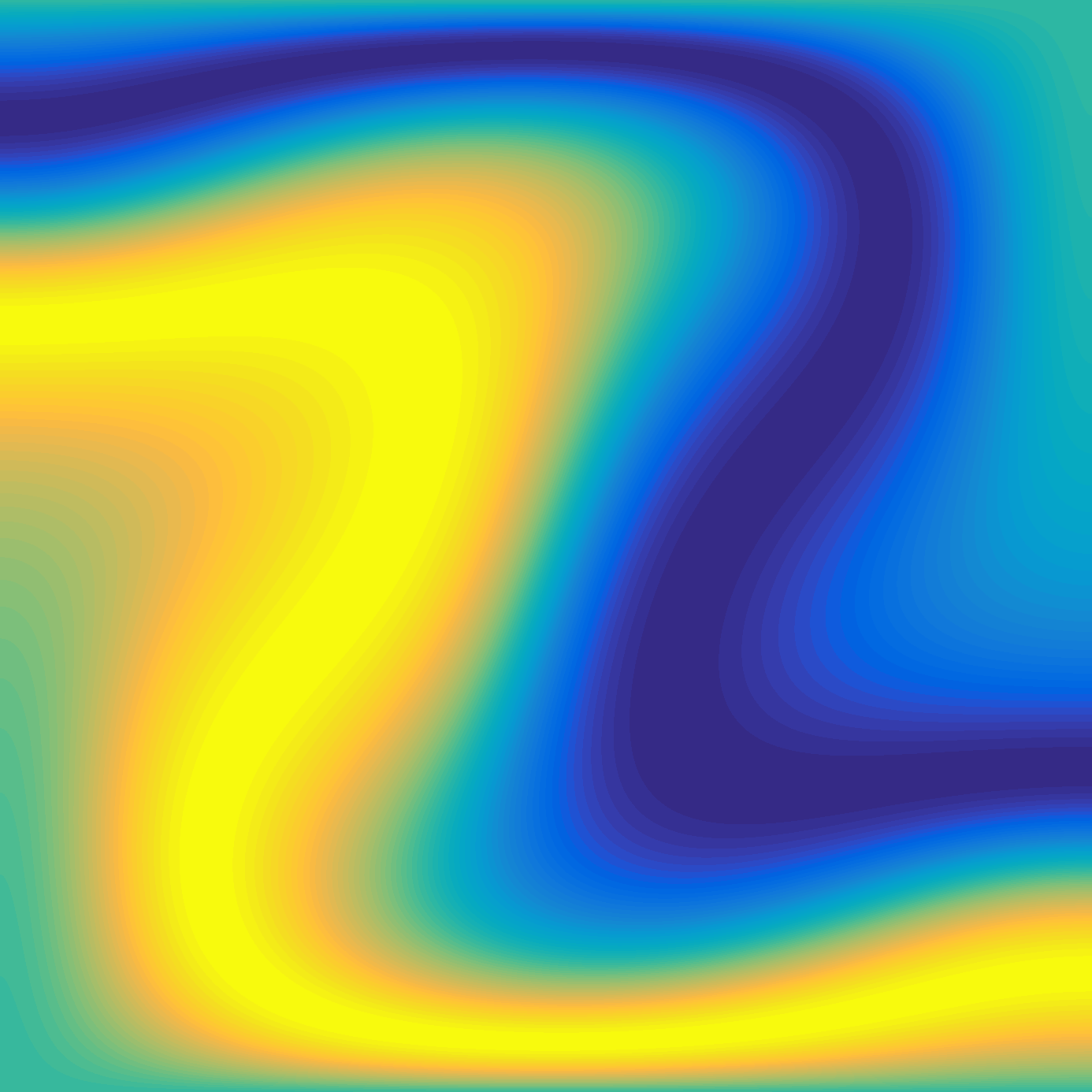}
    \caption{$t=0.2$}
  \end{subfigure}
  \begin{subfigure}[b]{0.16\textwidth}
    \includegraphics[width=\textwidth]{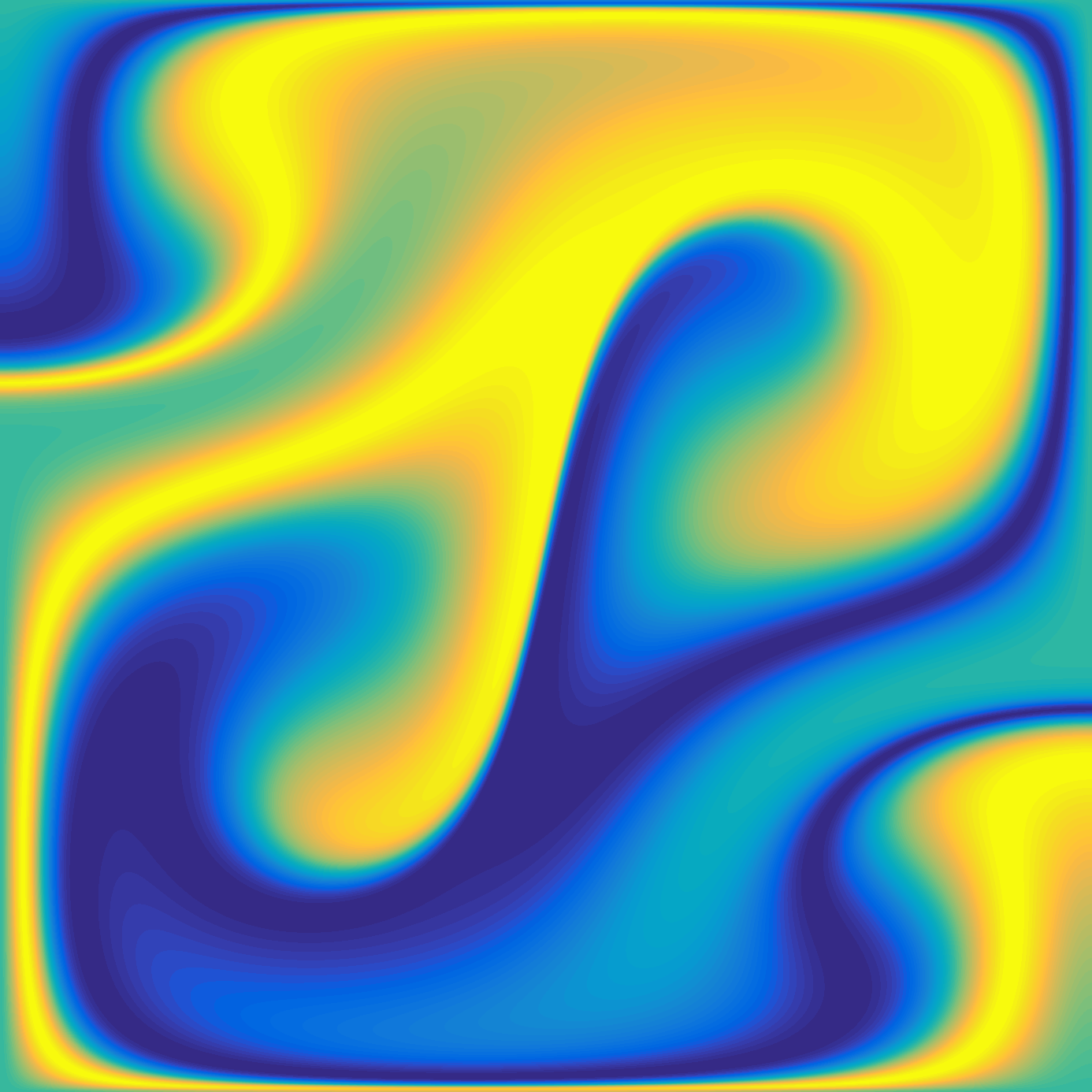}
    \caption{$t=0.4$}
  \end{subfigure}
  \begin{subfigure}[b]{0.16\textwidth}
    \includegraphics[width=\textwidth]{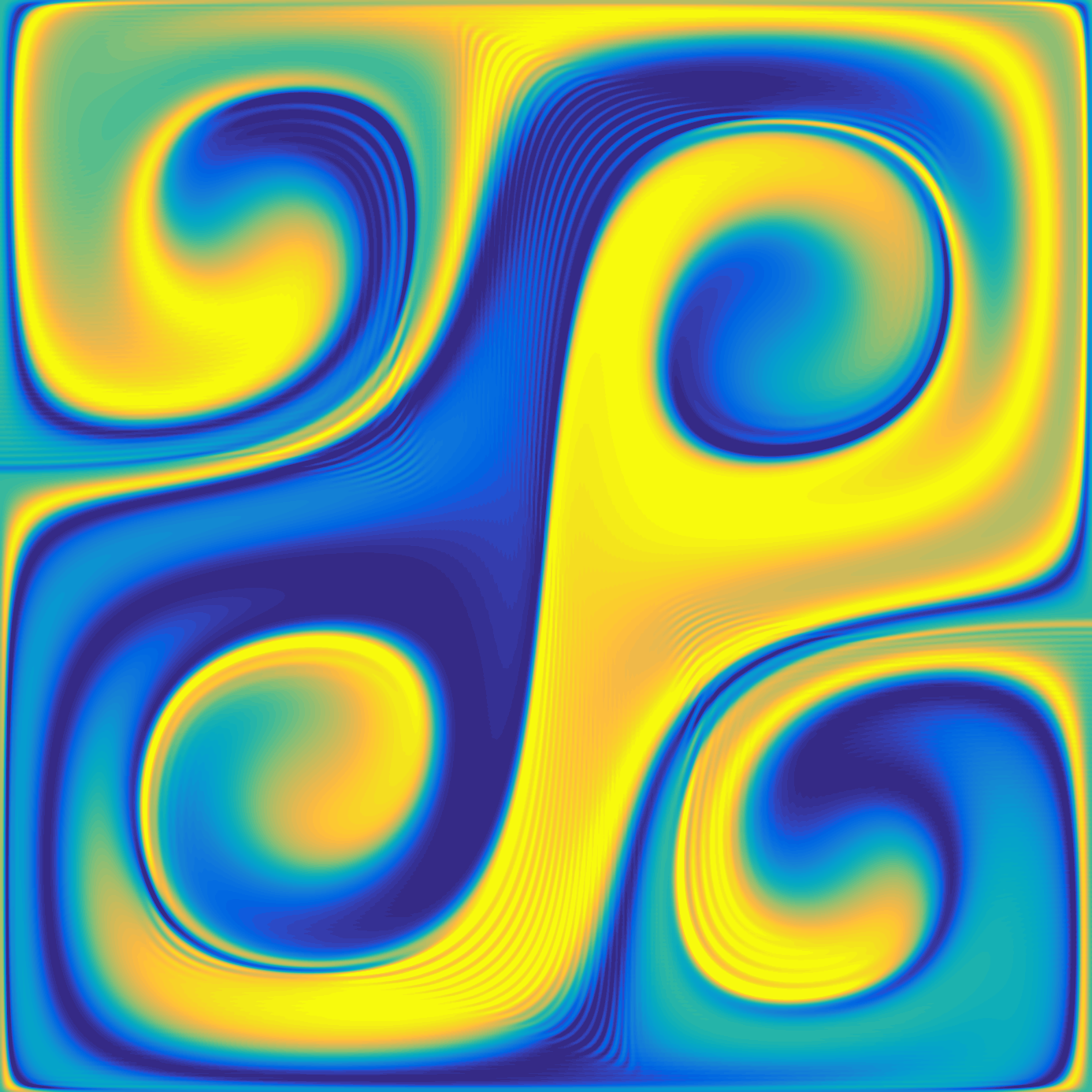}
    \caption{$t=0.6$}
  \end{subfigure}
  \begin{subfigure}[b]{0.16\textwidth}
    \includegraphics[width=\textwidth]{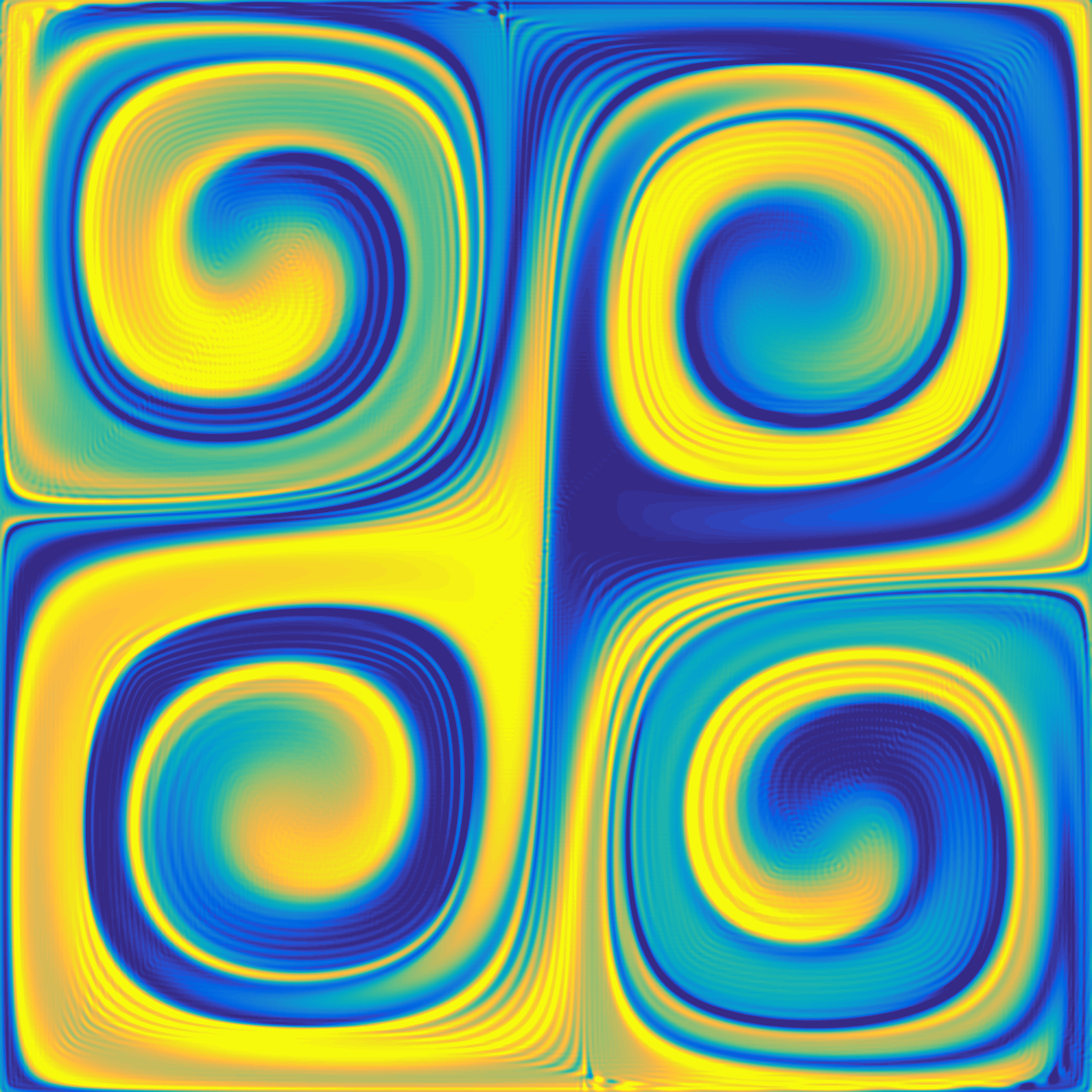}
    \caption{$t=0.8$}
  \end{subfigure}
  \begin{subfigure}[b]{0.16\textwidth}
    \includegraphics[width=\textwidth]{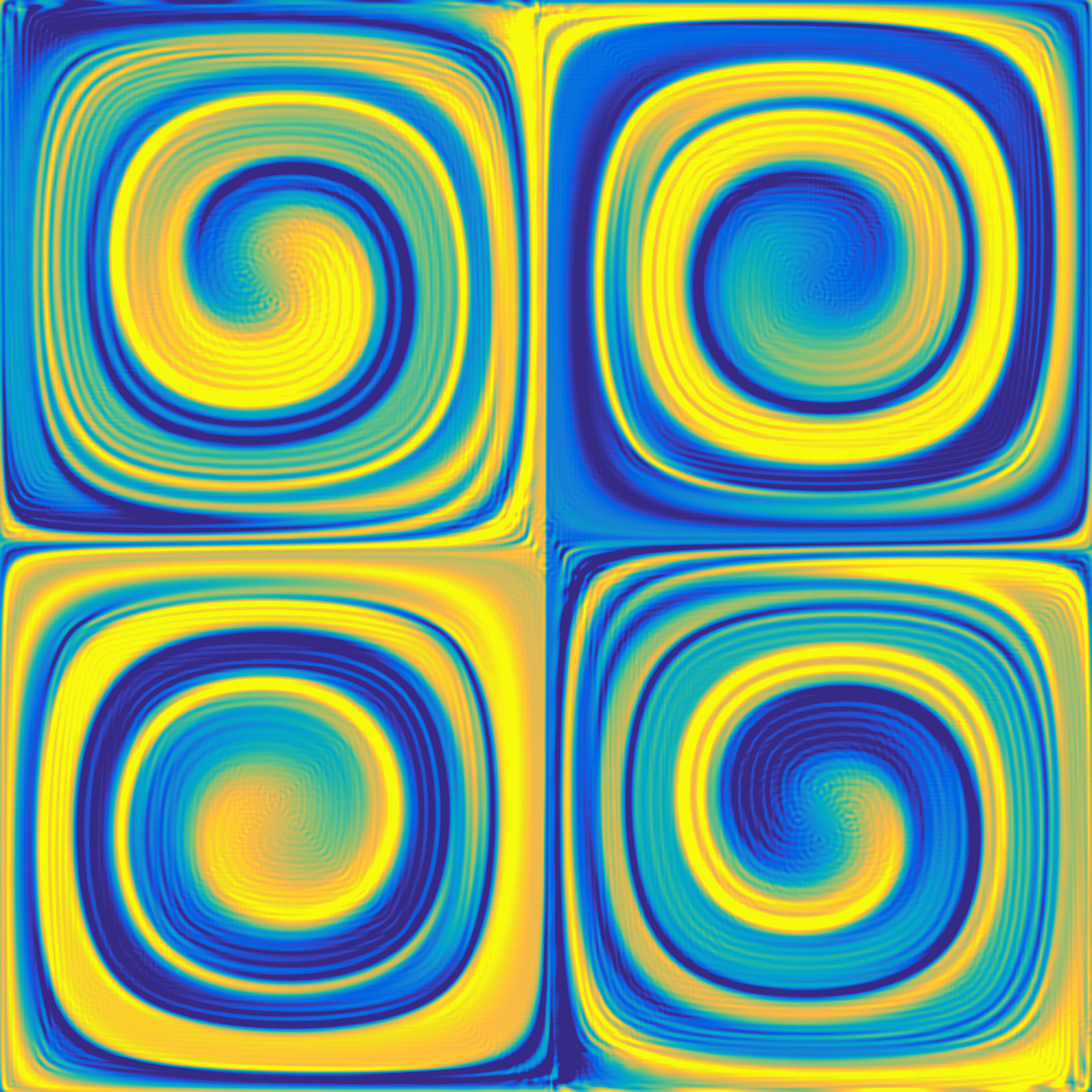}
    \caption{$t=1.0$}
  \end{subfigure}
  \caption{Evolution of $\theta_h$ for $t\in[0, 1]$ with initial control \eqref{eq:u1u2_2} and initial data \eqref{eq:theta_init_2}.}
  \label{fig:theta_init2_2}
\end{figure}

\subsection{Doswell frontogenesis basis}\label{sec:Doswell}
We next consider a more complex flow configuration to demonstrate the flexibility of our framework and to verify that the structure-preserving properties hold on a non-Cartesian mesh. In this section, the domain $\Omega = \{ \mathbf{x} \in \mathbb{R}^2 : |\mathbf{x} - \mathbf{x}_c| < R \}$ with center $\mathbf{x}_c = (0.5, 0.5)$ and radius $R=0.5$. The velocity field is driven by a single steady basis flow defined in polar coordinates $(r, \phi)$ centered at $\mathbf{x}_c$. To handle the circular geometry, we employ a polar grid. We emphasize that the structure-preserving properties of our scheme (Theorem~\ref{20251008-theorem-PerservedStructures}) rely on the skew-symmetry of the discrete divergence operator $D_u$, which in turn relies on the flux cancellation property of the finite-volume discretization.
Since the finite-volume formulation integrates fluxes across cell interfaces, it naturally extends to the curvilinear control volumes defined by the polar grid. By treating the polar mesh as a logical graph of adjacent control volumes, the matrix structure of the discrete operators remains identical to the Cartesian case.
We discretize the domain with $N_r = 500$ radial cells and $N_\phi = 500$ angular cells. The time step is set to $\Delta t = 0.001$ with a final time $T=5$ to allow more mixing in this larger time domain. We use the initial data 
\begin{align}
    \label{eq:init_Doswell}
    \theta^0(x_1,x_2) :=\tanh\!\left(\frac{x_2-0.5}{0.01}\right), ~(x_1,x_2)\in \Omega.
\end{align}

\textbf{Control basis design.}
In the following examples, we consider two prescribed velocity controls on the circular domain
$\Omega$. The first control $\mathbf{b}_1$ is the standard Doswell frontogenesis field defined in~\eqref{eq:v_Doswell}. This control is illustrated in Figure~\ref{fig:doswell1}, with a black circle as the boundary of the domain $\Omega$ and the blue curves as the velocity field. It is incompressible in $\Omega$ and satisfies the no-penetration boundary condition on $\partial\Omega$. 

The second control, $\mathbf{b}_2$, is a five-cell (multi-vortex) Doswell-type frontogenesis field formed by superposing five Doswell vortices supported on five embedded discs within $\Omega$:
four congruent discs arranged around the periphery of $\Omega$ and one disc centered at the middle of $\Omega$. Each disc carries a Doswell vortex centered at its disc center. To enforce no-penetration at each disc boundary and to obtain $C^1$ spatial regularity across the interface when the
field is extended by zero outside the disc, we multiply each vortex by the radial cutoff
\[
C(r)=\bigl(1-(r/R_c)^2\bigr)^3 \quad (0\le r<R_c), 
\qquad C(r)=0 \quad (r\ge R_c),
\]
where $R_c$ denotes the disc radius and $r$ is the distance to the corresponding vortex center. This choice satisfies $C(R_c)=0$ and $C'(R_c)=0$. The resulting control $\mathbf{b}_2$ remains incompressible in $\Omega$ and satisfies $\mathbf{b}_2\cdot\mathbf{n}=0$ on $\partial\Omega$.
Figure~\ref{fig:doswell2} displays $\mathbf{b}_2$ with red dashed lines as the five embedded discs within $\Omega$. This design introduces finer-scale structure than $\mathbf{b}_1$ and is intended to enhance stirring interactions with $\mathbf{b}_1$, thereby promoting more efficient mixing.

\begin{figure}
  \centering
  \begin{subfigure}[b]{0.4\textwidth}
    \includegraphics[width=\textwidth]{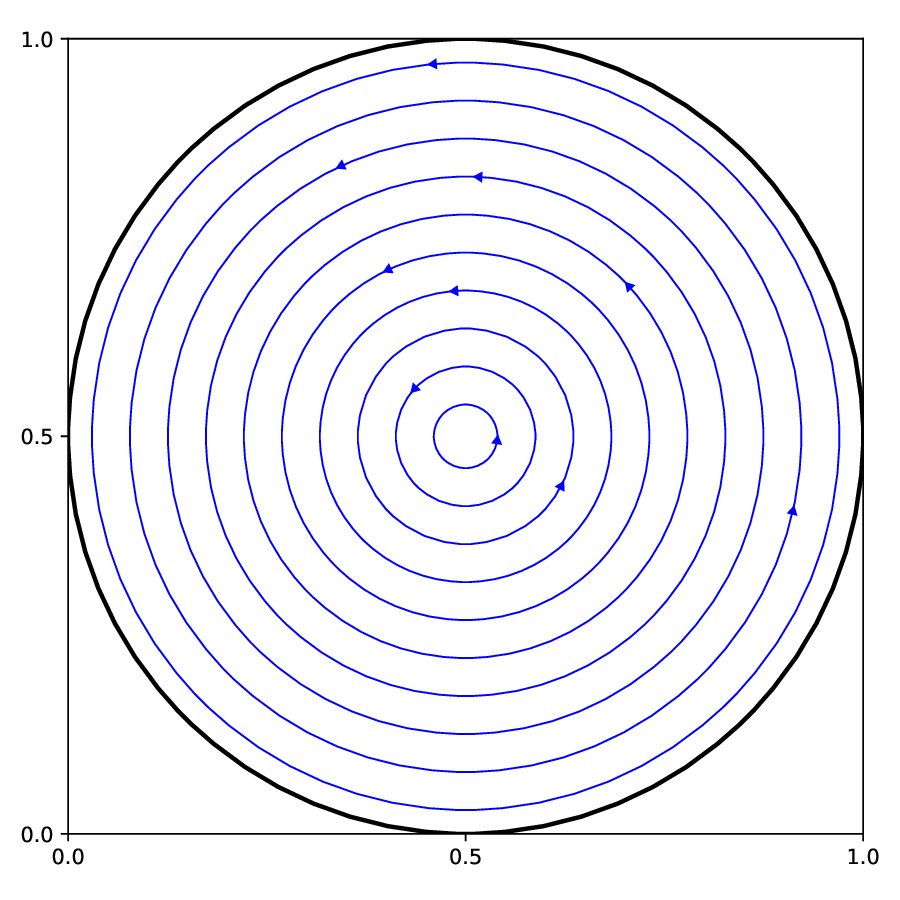}
    \caption{$\mathbf{b}_1$ Doswell basis}
    \label{fig:doswell1}
  \end{subfigure}
  \begin{subfigure}[b]{0.4\textwidth}
    \includegraphics[width=\textwidth]{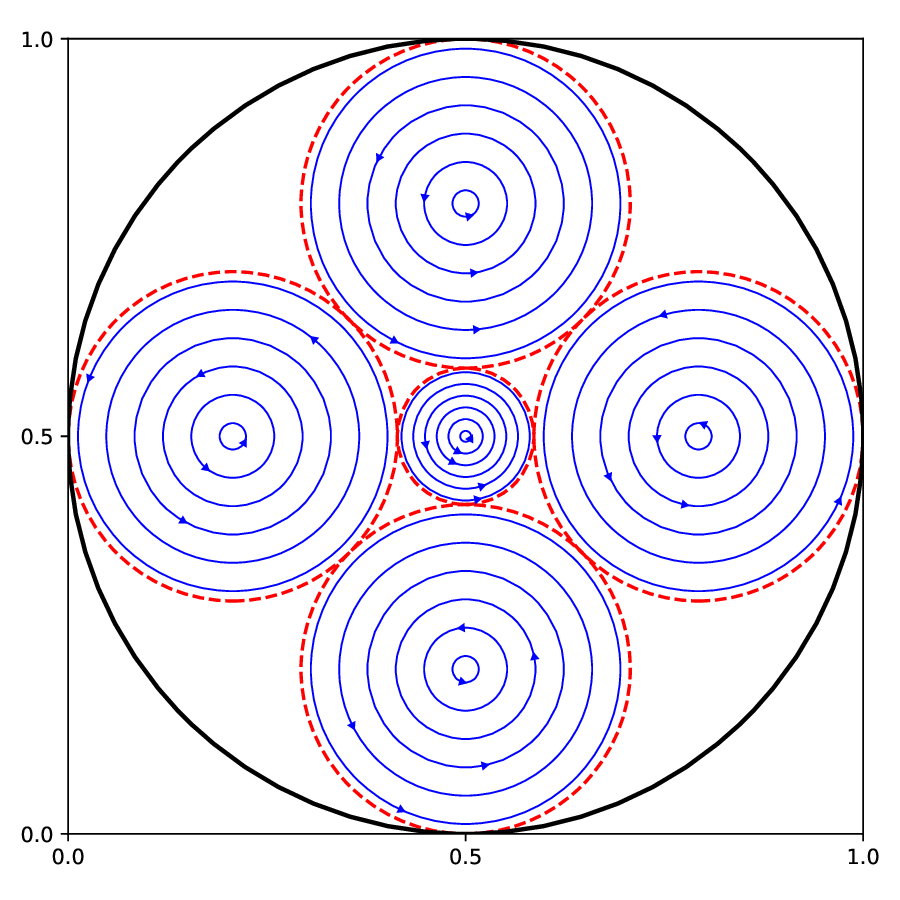}
    \caption{$\mathbf{b}_2$ Doswell basis}
    \label{fig:doswell2}
  \end{subfigure}
  \caption{Doswell frontogenesis basis flows $\mathbf{b}_1$ and $\mathbf{b}_2$.}
  \label{fig:doswell}
\end{figure}

\subsubsection{Single Doswell flow baseline mixing}
We first examine the effect of each basis flow acting alone. We simulate the evolution of $\theta_h$ with the  initial data \eqref{eq:init_Doswell} for two cases that $\mathbf{v}(t,x)=\mathbf{b}_1(x)$ and $\mathbf{v}(t,x)=\mathbf{b}_2(x)$. In both cases, the flow is held steady (no time modulation) to isolate the natural mixing effect of that mode. Figure~\ref{fig:mixnorm_b1_Doswell} plots the mix-norm $\|\theta_h(t)\|_{\dot H^{-1}(\Omega)}$ versus time for these single-flow simulations. We observe that, under $\mathbf{b}_1$ alone, the mix-norm exhibits a slow polynomial decay over time (Figure~\ref{fig:mixnorm_b1_Doswell}). A similar polynomial decay is seen for $\mathbf{b}_2$ alone (Figure~\ref{fig:mixnorm_b2_Doswell}). These baseline results indicate that neither $\mathbf{b}_1$ nor $\mathbf{b}_2$ by itself can mix the scalar efficiently in this scenario. Figures~\ref{fig:theta_Doswell_1} and \ref{fig:theta_Doswell_2} show the evolution of the scalar field under $\mathbf{b}_1$ and $\mathbf{b}_2$, respectively, confirming that while some distortion of the initial concentrations occurs, large unmixed regions remain even at $t=5$ when only a single mode is active.

\begin{figure}
  \centering
  \begin{subfigure}[b]{0.4\textwidth}
    \includegraphics[width=\textwidth]{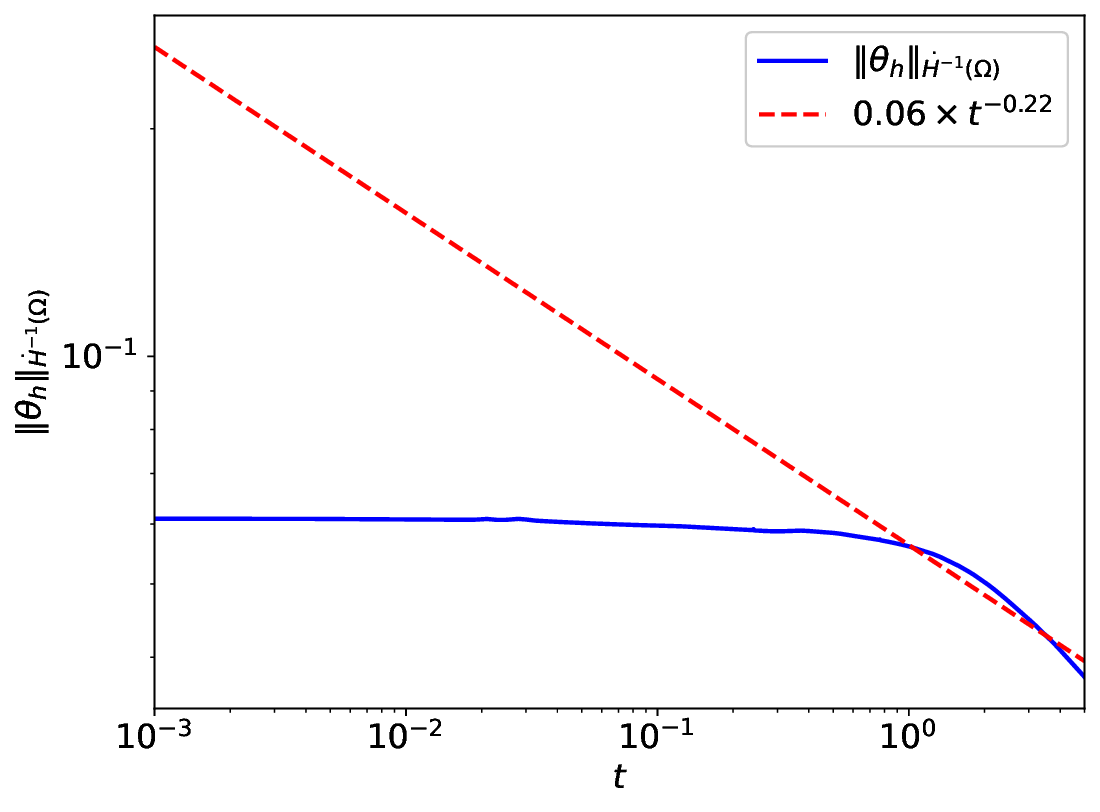}
    \caption{Evolution of $\Vert \theta_h \Vert_{\dot{H}^{-1}(\Omega)}$ with Doswell basis $\mathbf{b}_1$}
    \label{fig:mixnorm_b1_Doswell}
  \end{subfigure}
  \begin{subfigure}[b]{0.4\textwidth}
    \includegraphics[width=\textwidth]{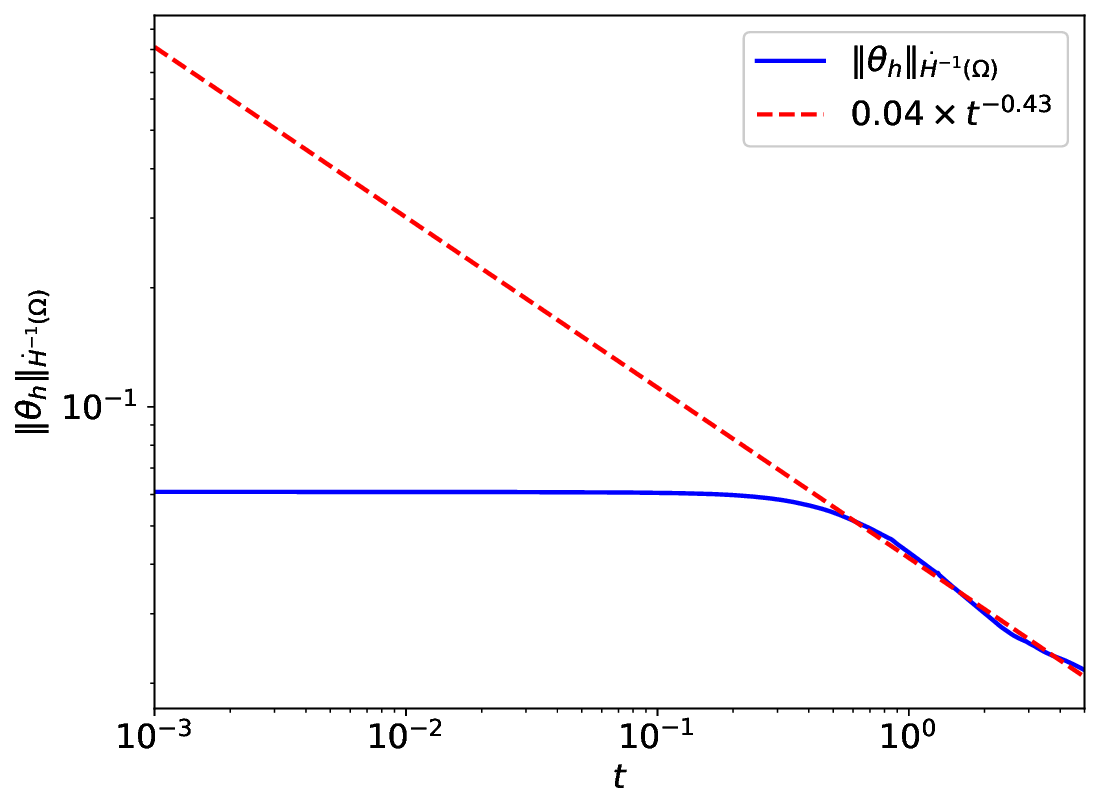}
    \caption{Evolution of $\Vert \theta_h \Vert_{\dot{H}^{-1}(\Omega)}$ with Doswell basis $\mathbf{b}_2$}
    \label{fig:mixnorm_b2_Doswell}
  \end{subfigure}
  \caption{Evolution of $\Vert \theta_h \Vert_{\dot{H}^{-1}(\Omega)}$ with Doswell basis $\mathbf{b}_1$ and $\mathbf{b}_2$.}
\end{figure}

\begin{figure}
  \centering
  \begin{subfigure}[b]{0.16\textwidth}
    \includegraphics[width=\textwidth]{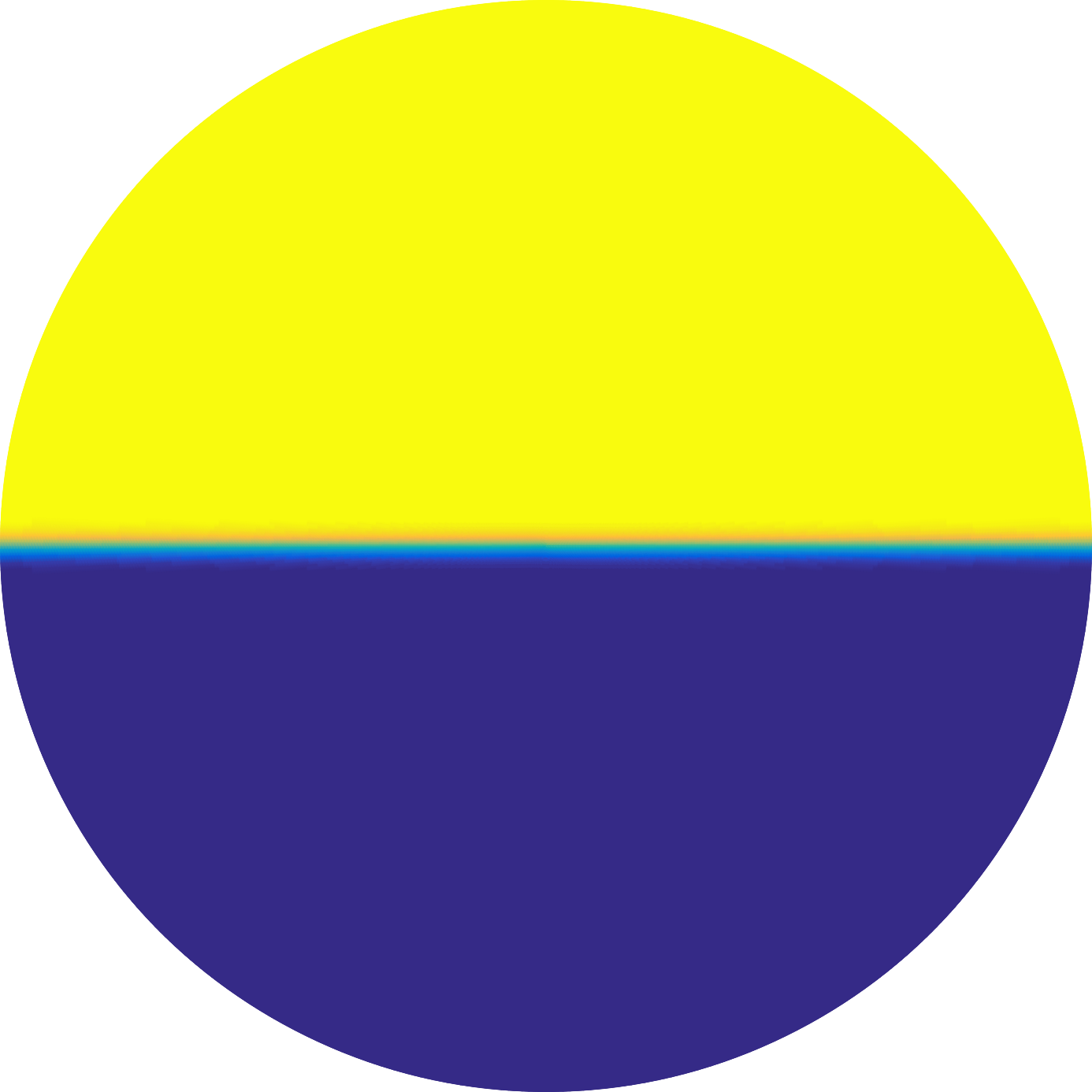}
    \caption{$t=0$}
  \end{subfigure}
  \begin{subfigure}[b]{0.16\textwidth}
    \includegraphics[width=\textwidth]{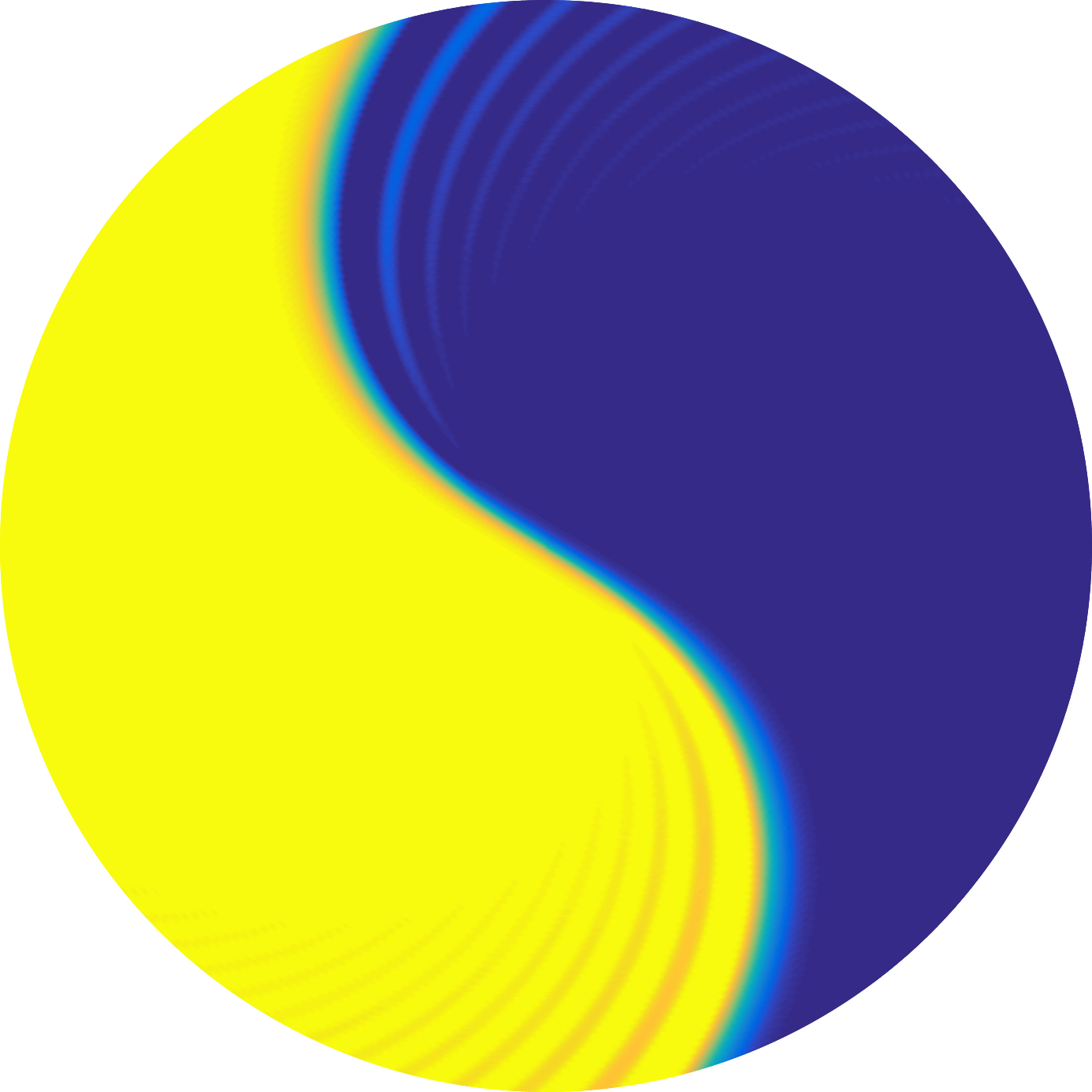}
    \caption{$t=1$}
  \end{subfigure}
  \begin{subfigure}[b]{0.16\textwidth}
    \includegraphics[width=\textwidth]{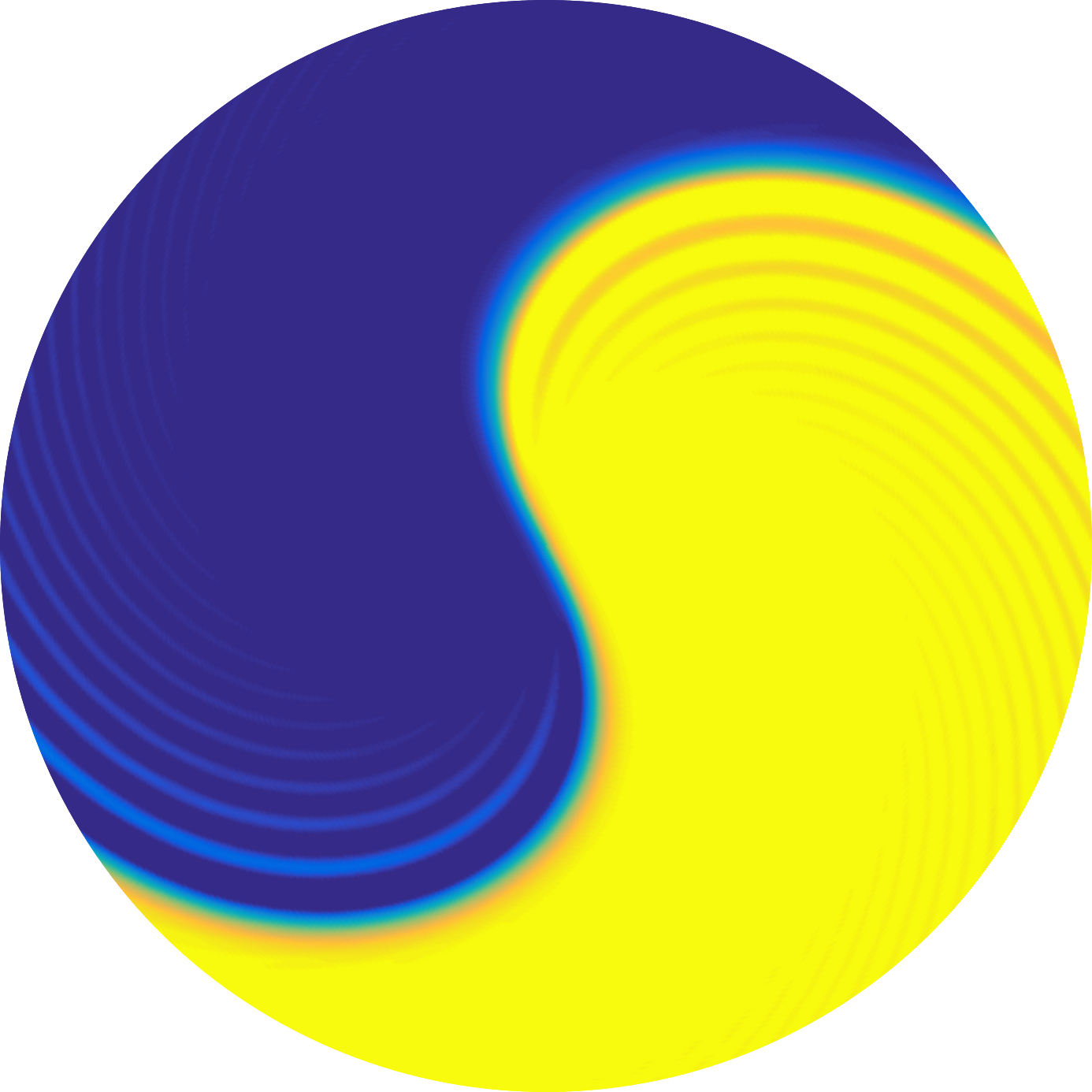}
    \caption{$t=2$}
  \end{subfigure}
  \begin{subfigure}[b]{0.16\textwidth}
    \includegraphics[width=\textwidth]{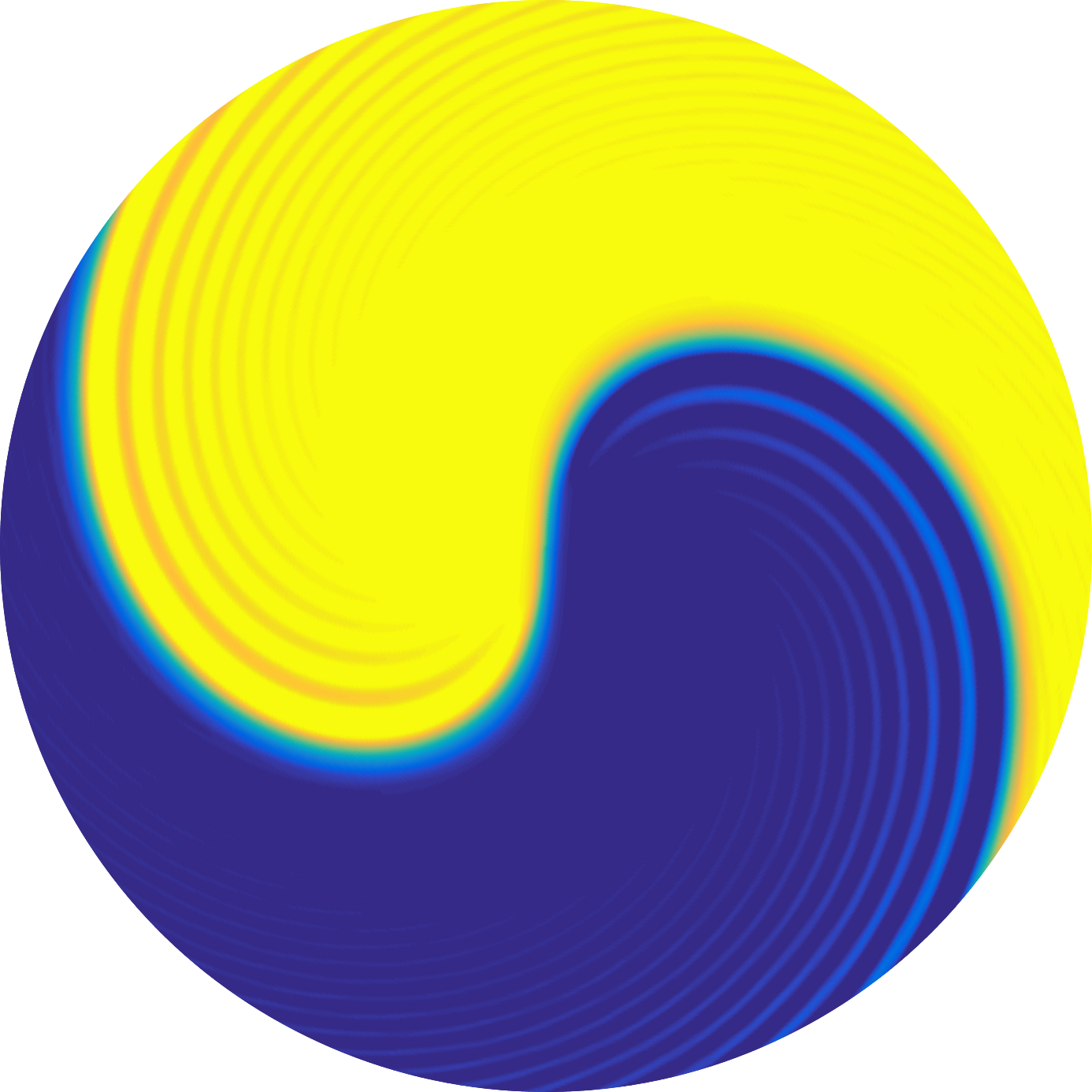}
    \caption{$t=3$}
  \end{subfigure}
  \begin{subfigure}[b]{0.16\textwidth}
    \includegraphics[width=\textwidth]{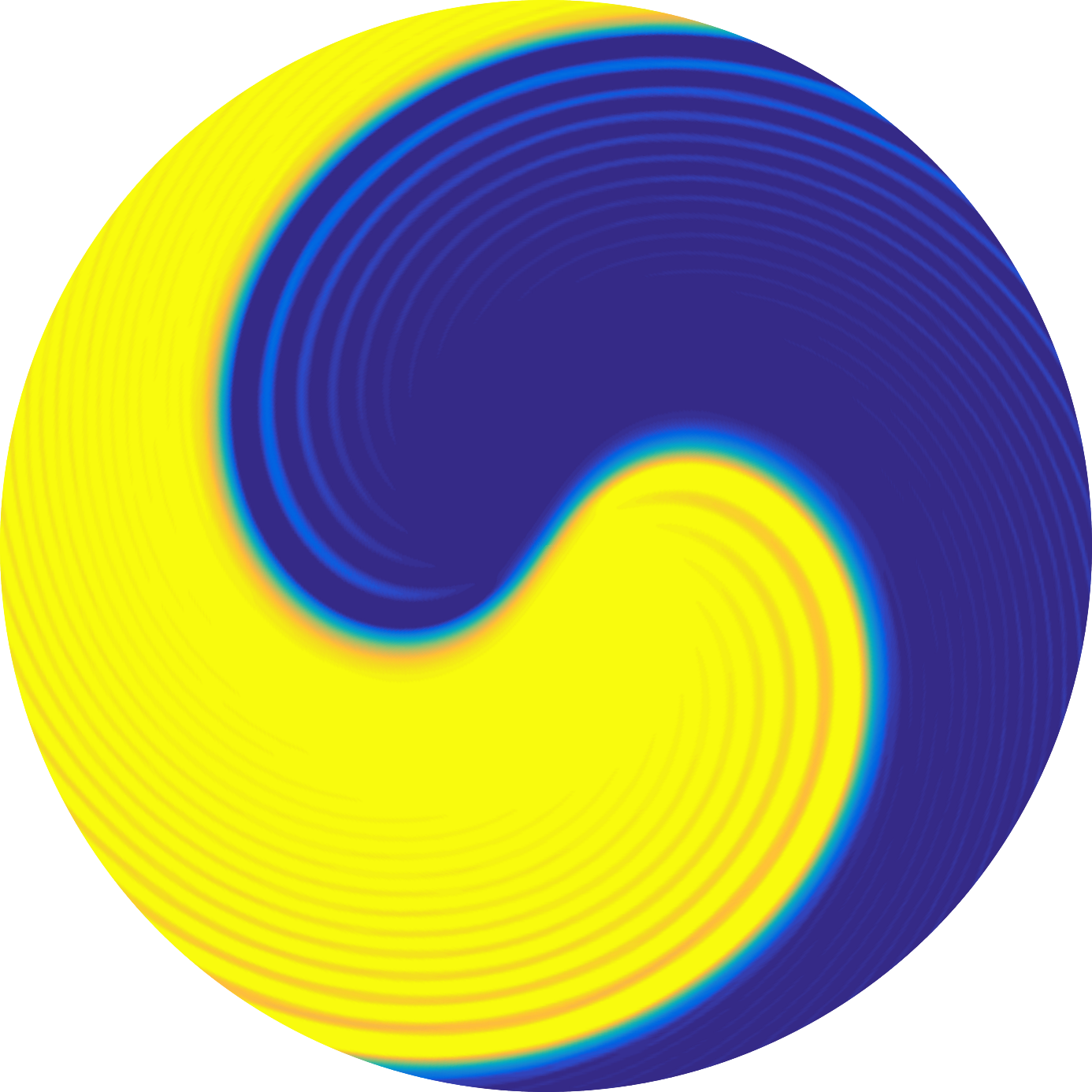}
    \caption{$t=4$}
  \end{subfigure}
  \begin{subfigure}[b]{0.16\textwidth}
    \includegraphics[width=\textwidth]{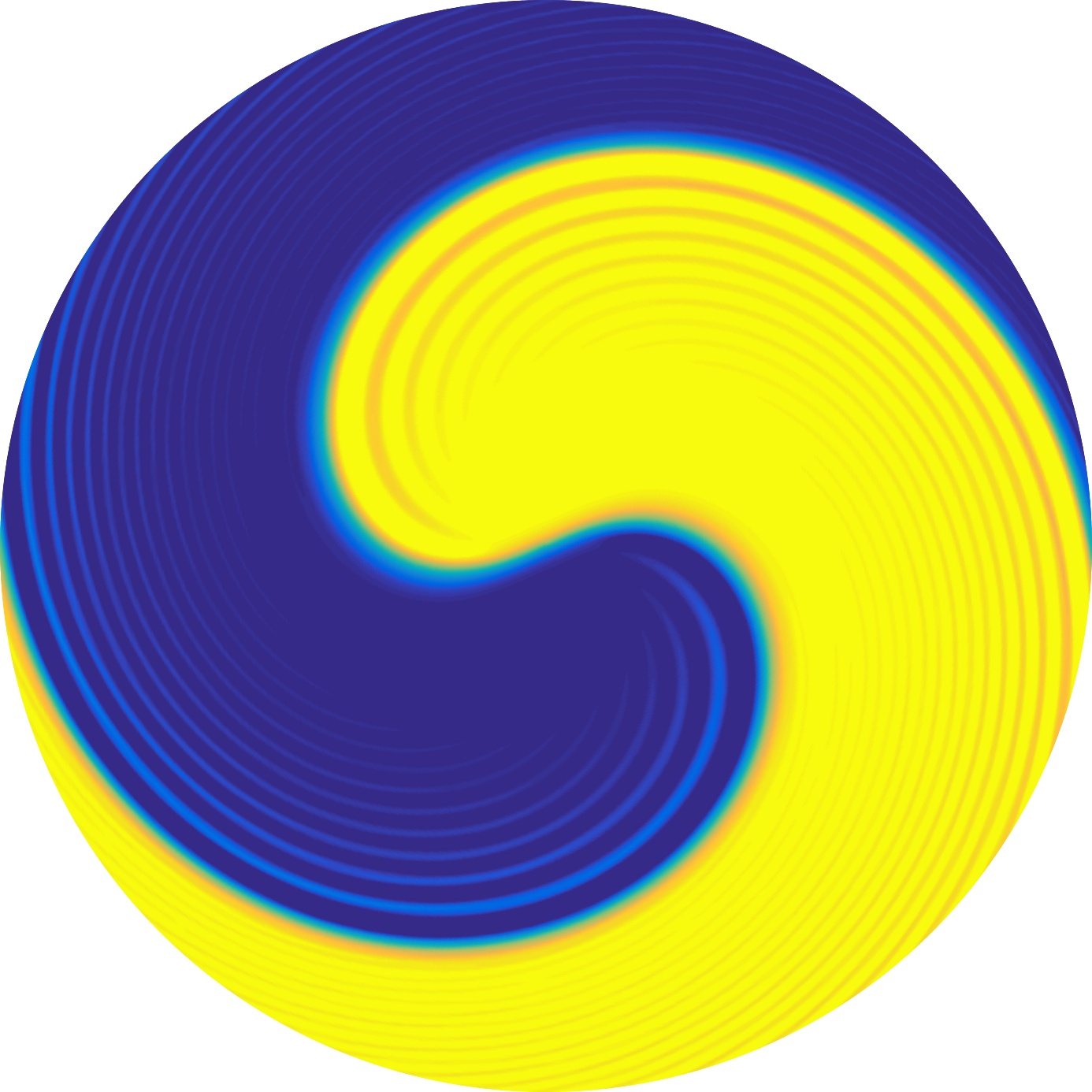}
    \caption{$t=5$}
  \end{subfigure}
  \caption{Evolution of $\theta_h$ for $t\in[0, 5]$ with Doswell basis $\mathbf{b}_1$ and initial data \eqref{eq:init_Doswell}.}
  \label{fig:theta_Doswell_1}
\end{figure}

\begin{figure}
  \centering
  \begin{subfigure}[b]{0.16\textwidth}
    \includegraphics[width=\textwidth]{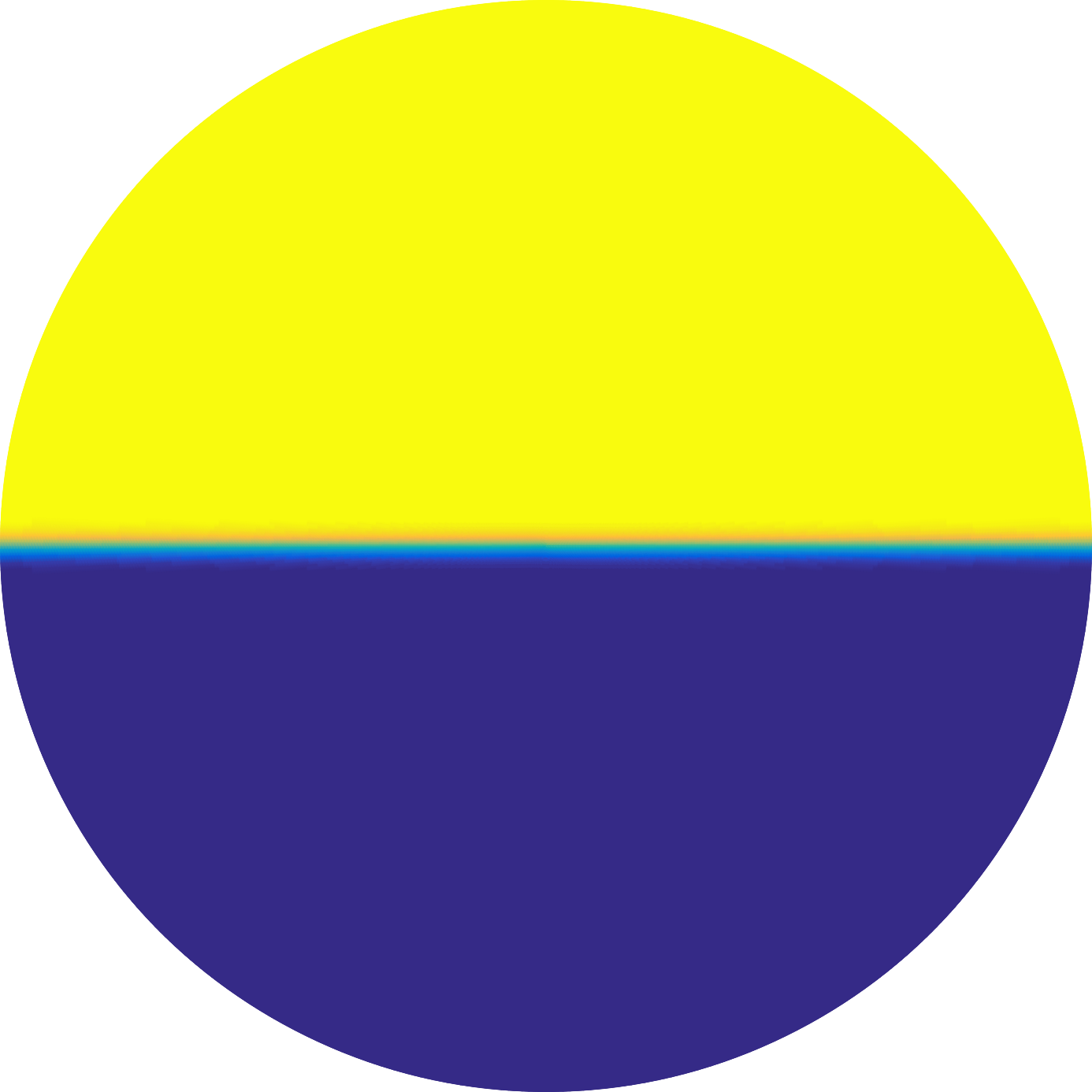}
    \caption{$t=0$}
  \end{subfigure}
  \begin{subfigure}[b]{0.16\textwidth}
    \includegraphics[width=\textwidth]{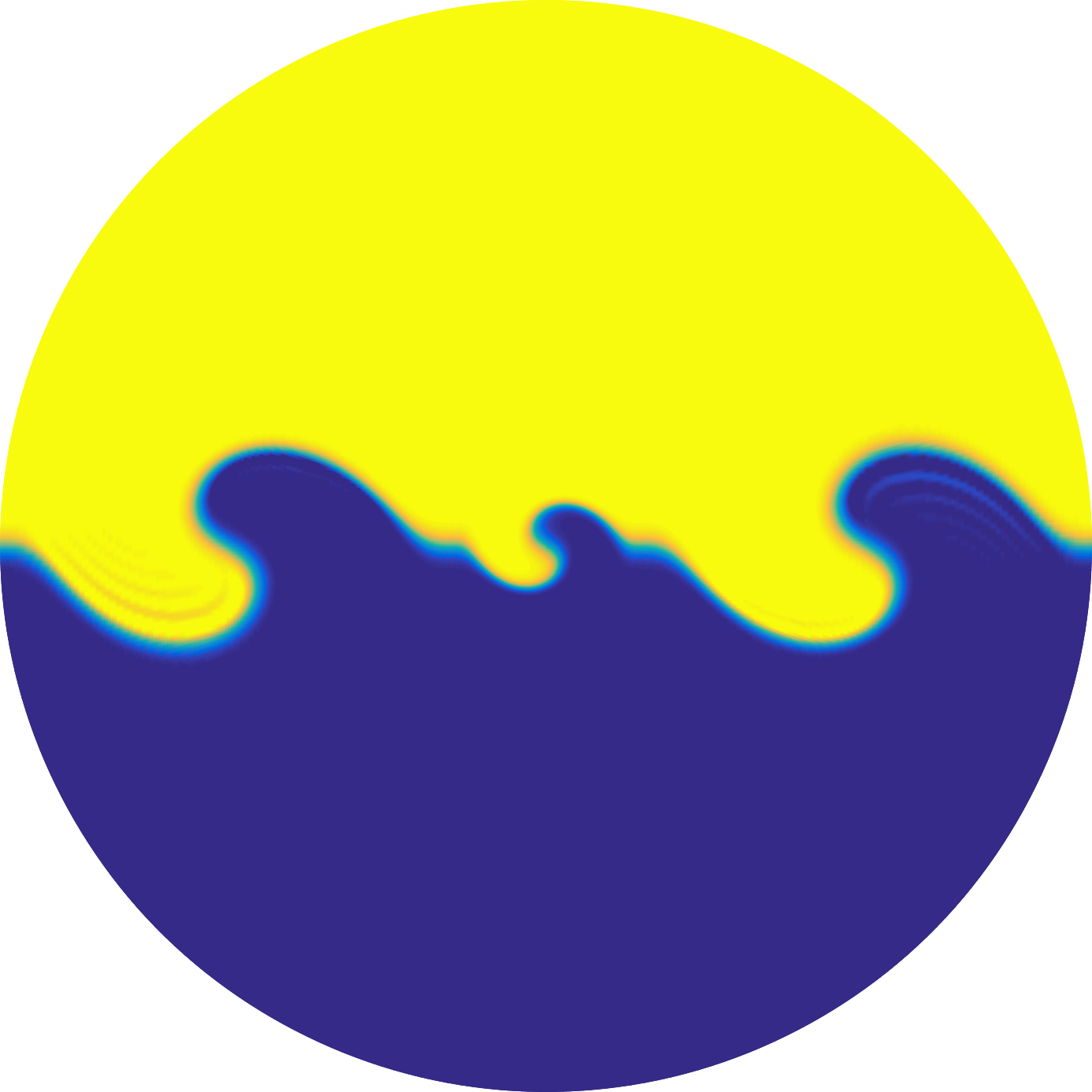}
    \caption{$t=1$}
  \end{subfigure}
  \begin{subfigure}[b]{0.16\textwidth}
    \includegraphics[width=\textwidth]{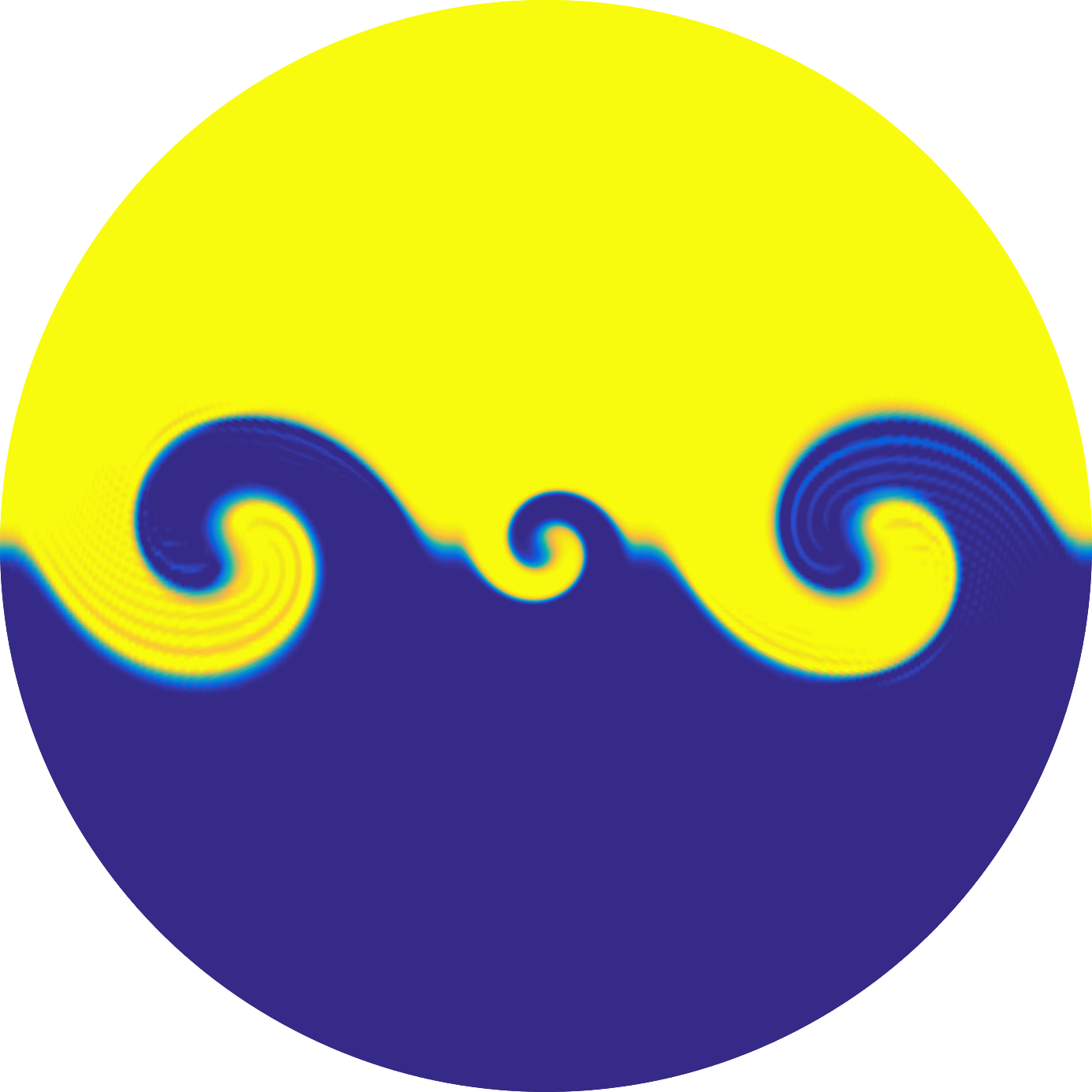}
    \caption{$t=2$}
  \end{subfigure}
  \begin{subfigure}[b]{0.16\textwidth}
    \includegraphics[width=\textwidth]{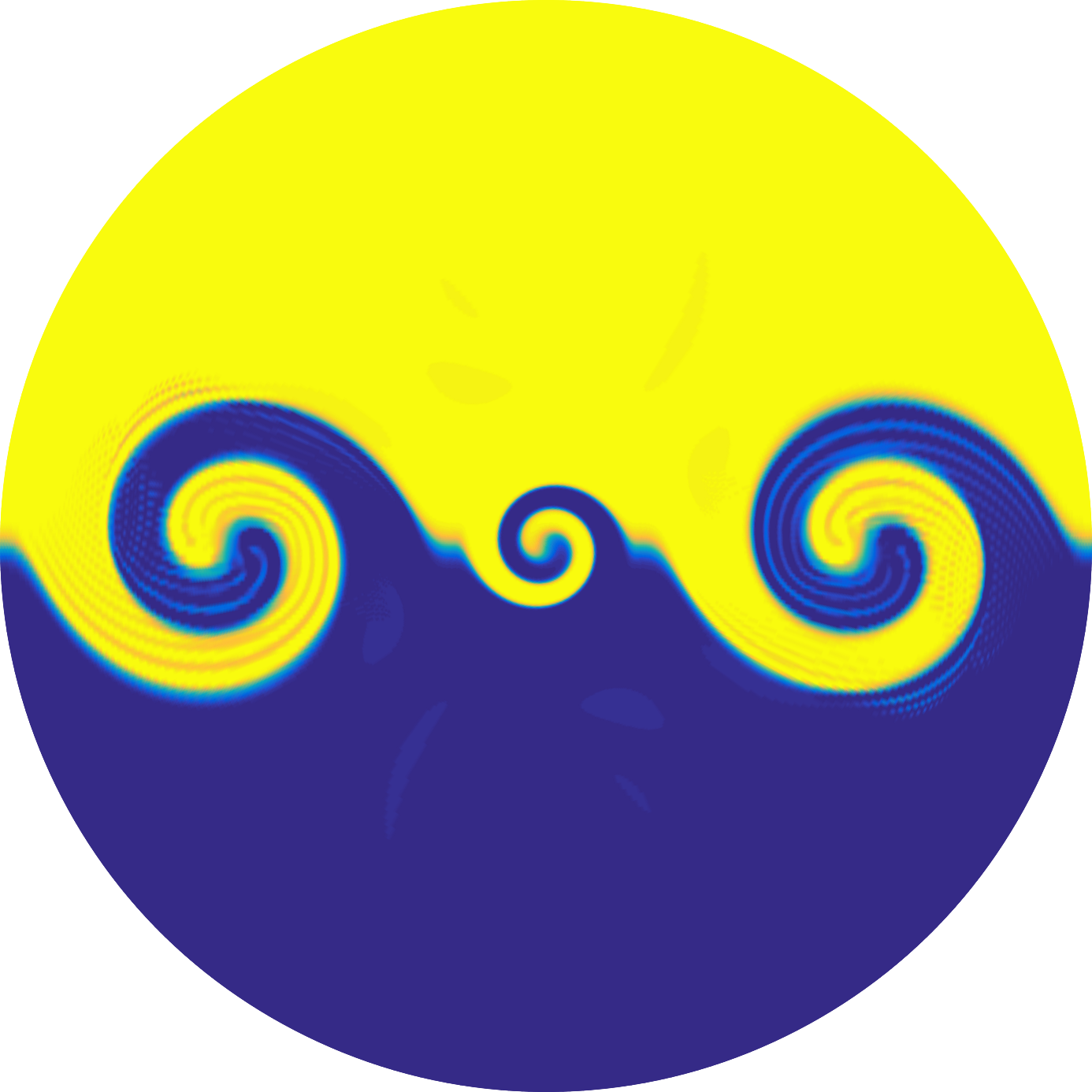}
    \caption{$t=3$}
  \end{subfigure}
  \begin{subfigure}[b]{0.16\textwidth}
    \includegraphics[width=\textwidth]{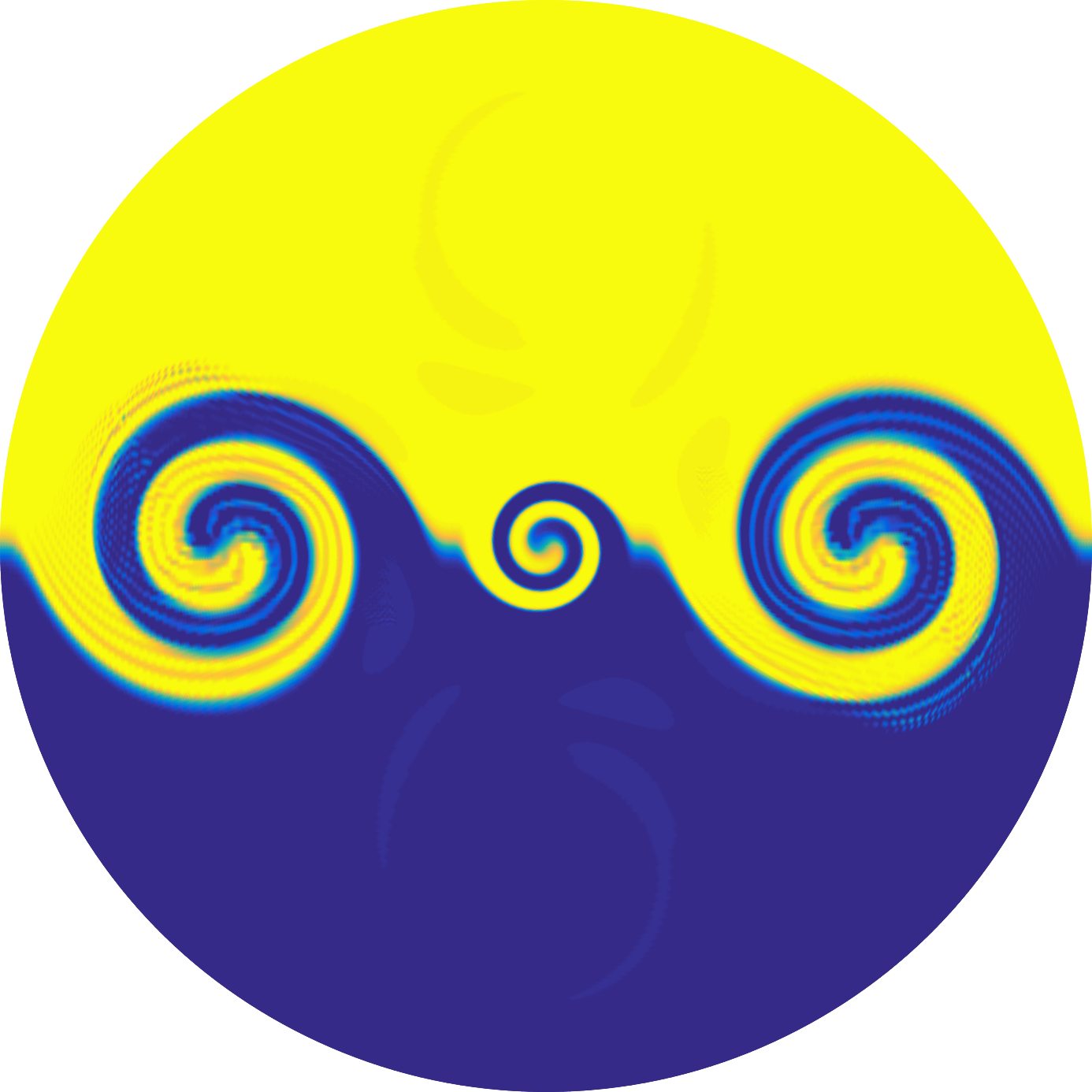}
    \caption{$t=4$}
  \end{subfigure}
  \begin{subfigure}[b]{0.16\textwidth}
    \includegraphics[width=\textwidth]{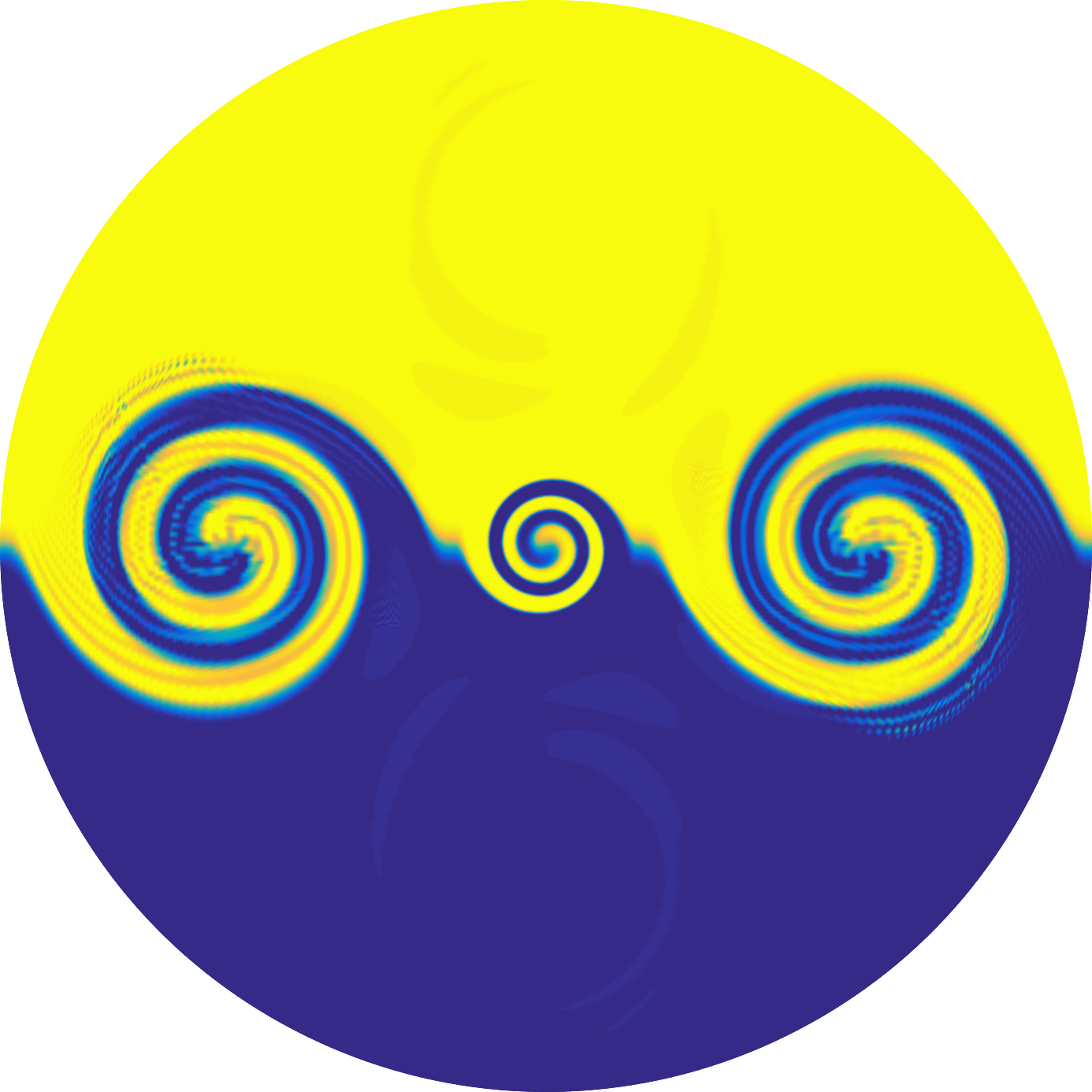}
    \caption{$t=5$}
  \end{subfigure}
  \caption{Evolution of $\theta_h$ for $t\in[0, 5]$ with Doswell basis $\mathbf{b}_2$ and initial data \eqref{eq:init_Doswell}.}
  \label{fig:theta_Doswell_2}
\end{figure}

\subsubsection{Optimal control for Doswell frontogenesis design}
We now apply the optimal control algorithm using the combination \eqref{eq:v_finite} to mix the scalar in $\Omega$. We perform the optimization with two different initial guesses for the control coefficients. The first is the constant guess
\begin{align}
\label{eq:u1_Doswell_1}
    v_1(t)=1,~v_2(t)=1,~ t\in(0,5),
\end{align}
and the second initial guess is 
\begin{align}\label{eq:u1_Doswell_2}
v_1(t)=\cos(\pi t / 2),
~v_2(t)=\sin(\pi t / 2),~ t\in(0,5).
\end{align}
Starting from the first guess \eqref{eq:u1_Doswell_1}, the algorithm converges to an optimal control solution that induces a rapid mixing process. Figure~\ref{fig:u1u2_Doswell_1} shows the optimized control coefficients $v_1(t)$ and $v_2(t)$. Notably, one coefficient diminishes while the other intensifies at certain periods, suggesting that the algorithm finds an optimal schedule for activating $\mathbf{b}_1$ and $\mathbf{b}_2$ in sequence to achieve the best mixing. As shown in Figure~\ref{fig:mixnorm_Doswell_1}, the mix-norm under the optimized control decays rapidly, in fact displaying a near-exponential decrease. We estimate an exponential decay rate of approximately $0.30$ in this case. This decay is much faster than the polynomial decay observed for either $\mathbf{b}_1$ or $\mathbf{b}_2$ alone, underscoring the benefit of using an optimized combination of flows. Furthermore, Figure~\ref{fig:property_Doswell_1} verifies that the conservation properties of our numerical scheme hold on the curvilinear mesh just as they did on the Cartesian grid. All invariants stay within machine precision over the entire simulation. This confirms that the structure-preserving attributes of the scheme are robust to the change of domain geometry and mesh structure. Finally, Figure~\ref{fig:theta_Doswell_optimal_1} depicts snapshots of the scalar field $\theta_h(t,x)$ at selected times under the optimized control.

\begin{figure}[H]
  \centering
  \begin{subfigure}[b]{0.315\textwidth}
    \includegraphics[width=\textwidth]{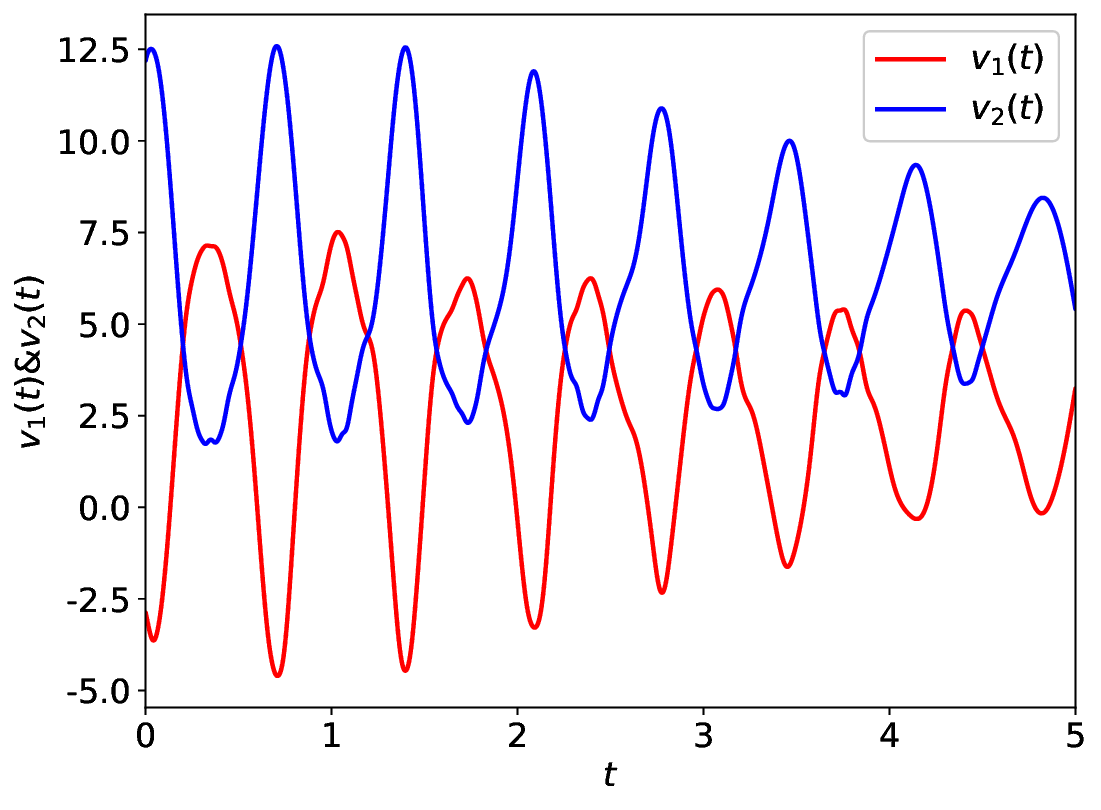}
    \caption{Evolutions of $v_1(t)$, $v_2(t)$.}
    \label{fig:u1u2_Doswell_1}
  \end{subfigure}
  \begin{subfigure}[b]{0.335\textwidth}
    \includegraphics[width=\textwidth]{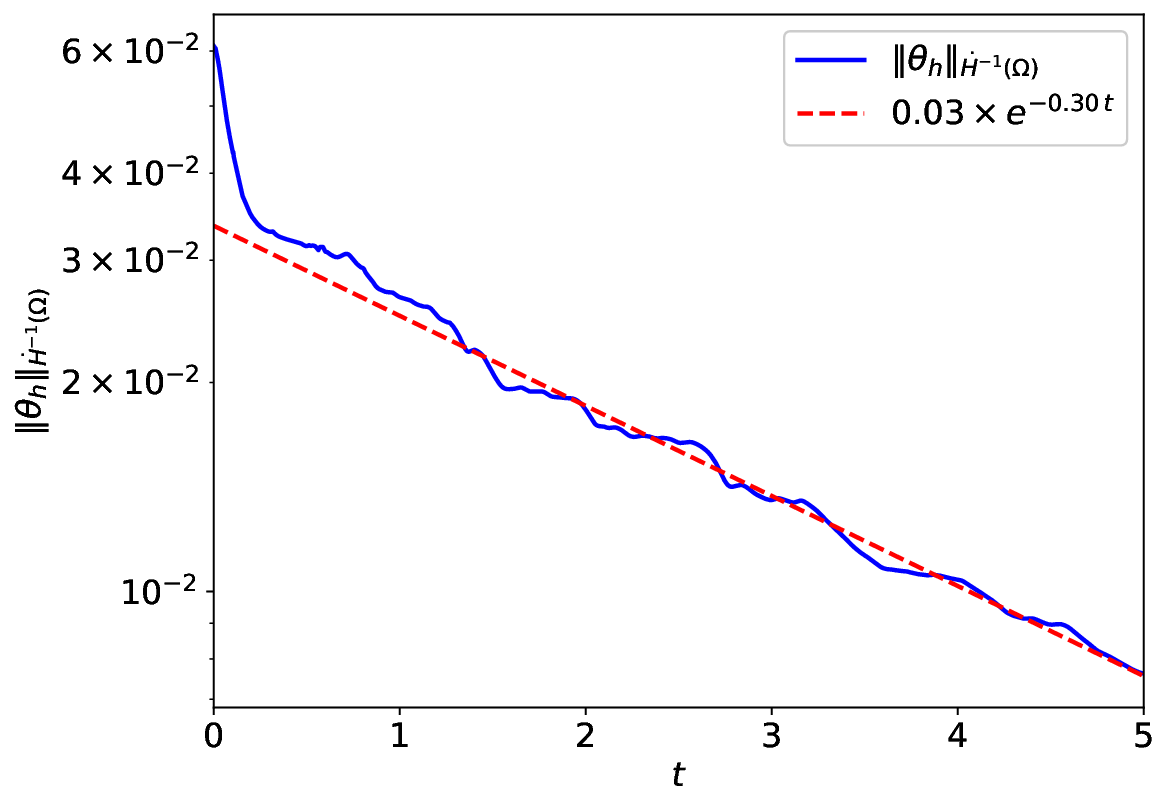}
    \caption{Evolution of $\Vert \theta_h \Vert_{\dot{H}^{-1}(\Omega)}$.}
    \label{fig:mixnorm_Doswell_1}
  \end{subfigure}
  \begin{subfigure}[b]{0.33\textwidth}
    \includegraphics[width=\textwidth]{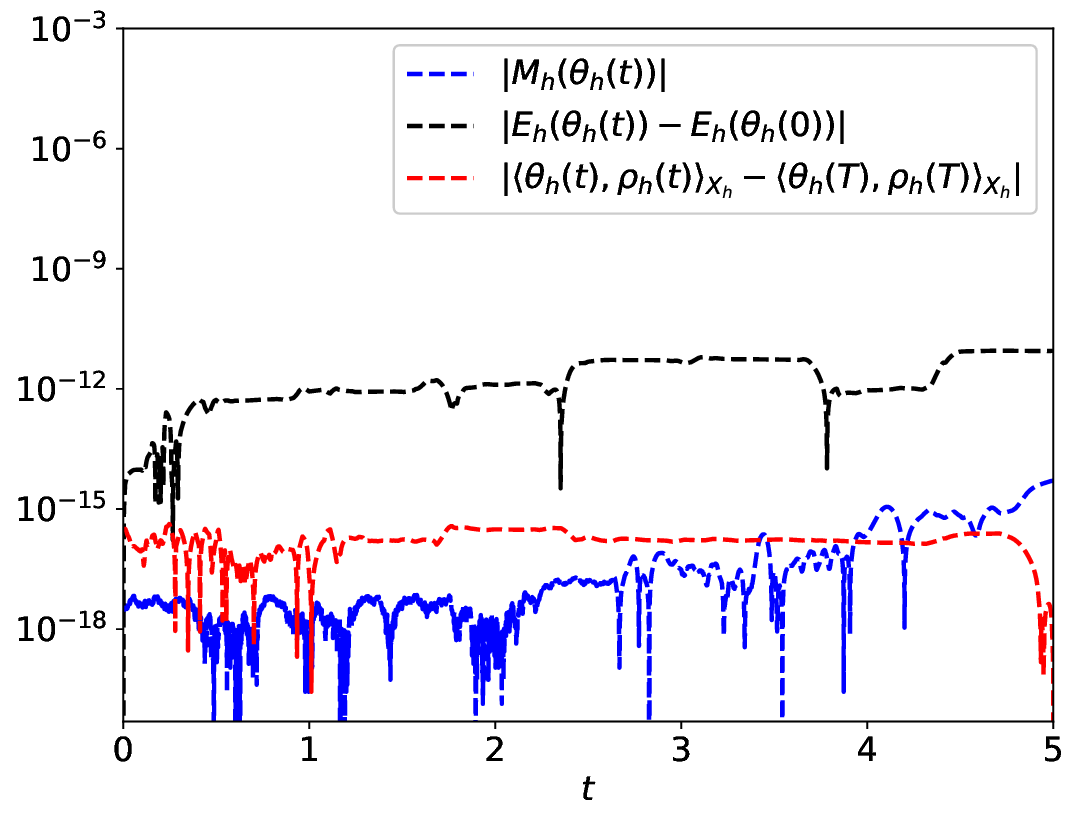}
    \caption{$M_h(\theta_h)$, $E_h(\theta_h)$, $\langle \theta_h, \rho_h \rangle_{X_h}$.}
    \label{fig:property_Doswell_1}
  \end{subfigure}
    \caption{Evolutions of $v(t)$, mix-norm $\Vert \theta_h \Vert_{\dot{H}^{-1}(\Omega)}$, mass $M_h(\theta_h)$, energy $E_h(\theta_h)$ and state-adjoint pairing $\langle \theta_h, \rho_h \rangle_{X_h}$ for $t\in[0, 5]$ with initial control \eqref{eq:u1_Doswell_1} and initial data \eqref{eq:init_Doswell}.}
  \label{fig:u1_mixnorm_Doswell_1}
\end{figure}

\begin{figure}[H]
  \centering
  \begin{subfigure}[b]{0.16\textwidth}
    \includegraphics[width=\textwidth]{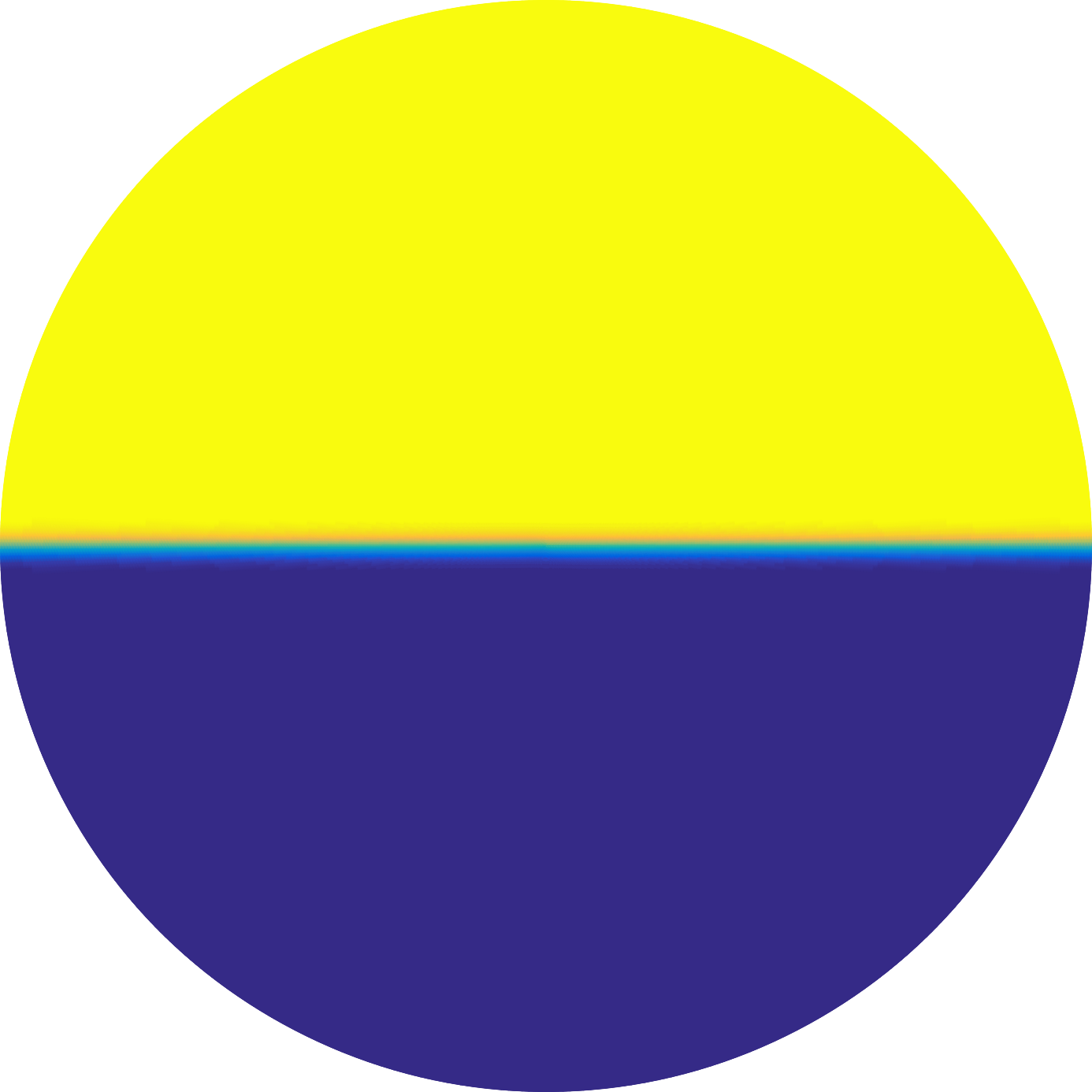}
    \caption{$t=0$}
  \end{subfigure}
  \begin{subfigure}[b]{0.16\textwidth}
    \includegraphics[width=\textwidth]{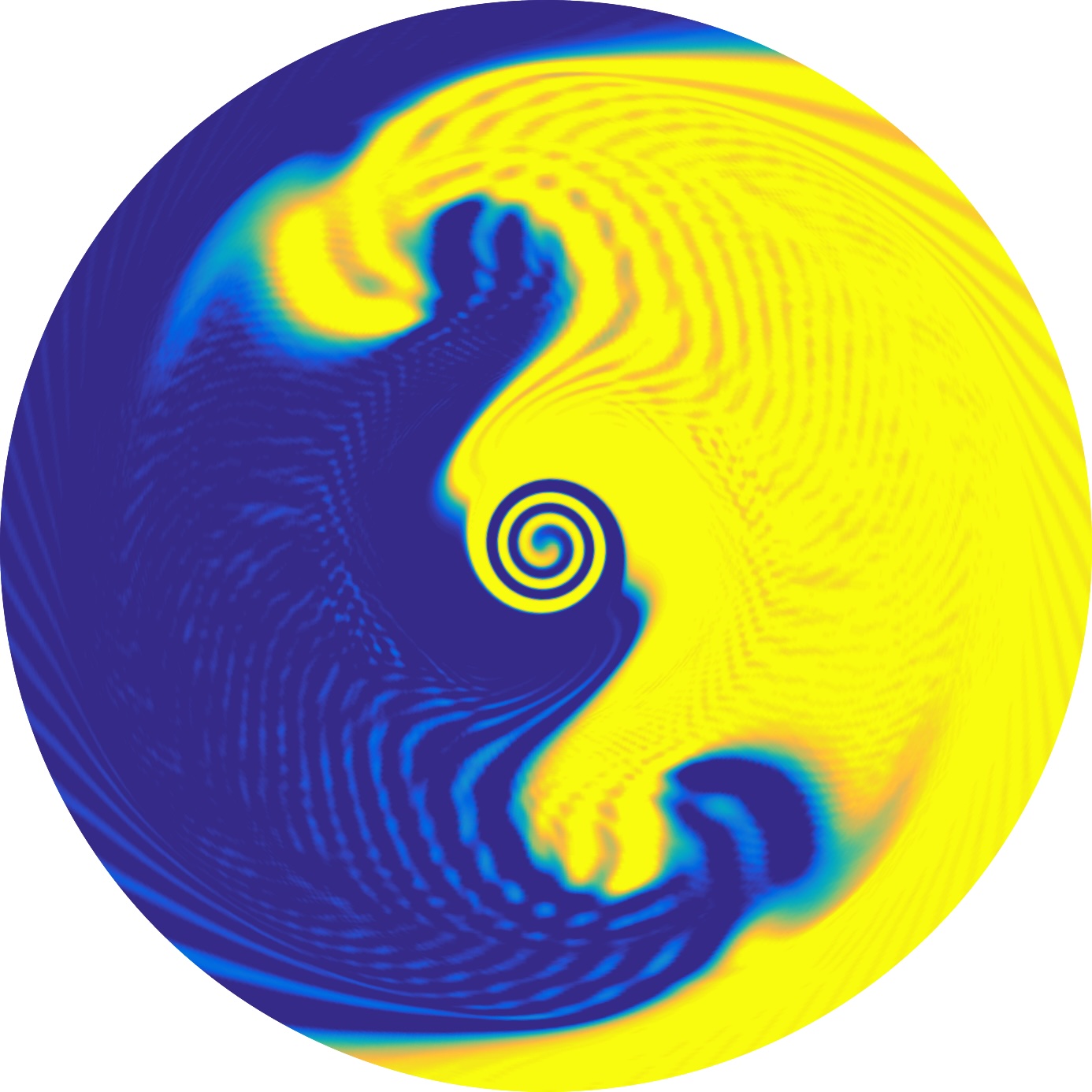}
    \caption{$t=1$}
  \end{subfigure}
  \begin{subfigure}[b]{0.16\textwidth}
    \includegraphics[width=\textwidth]{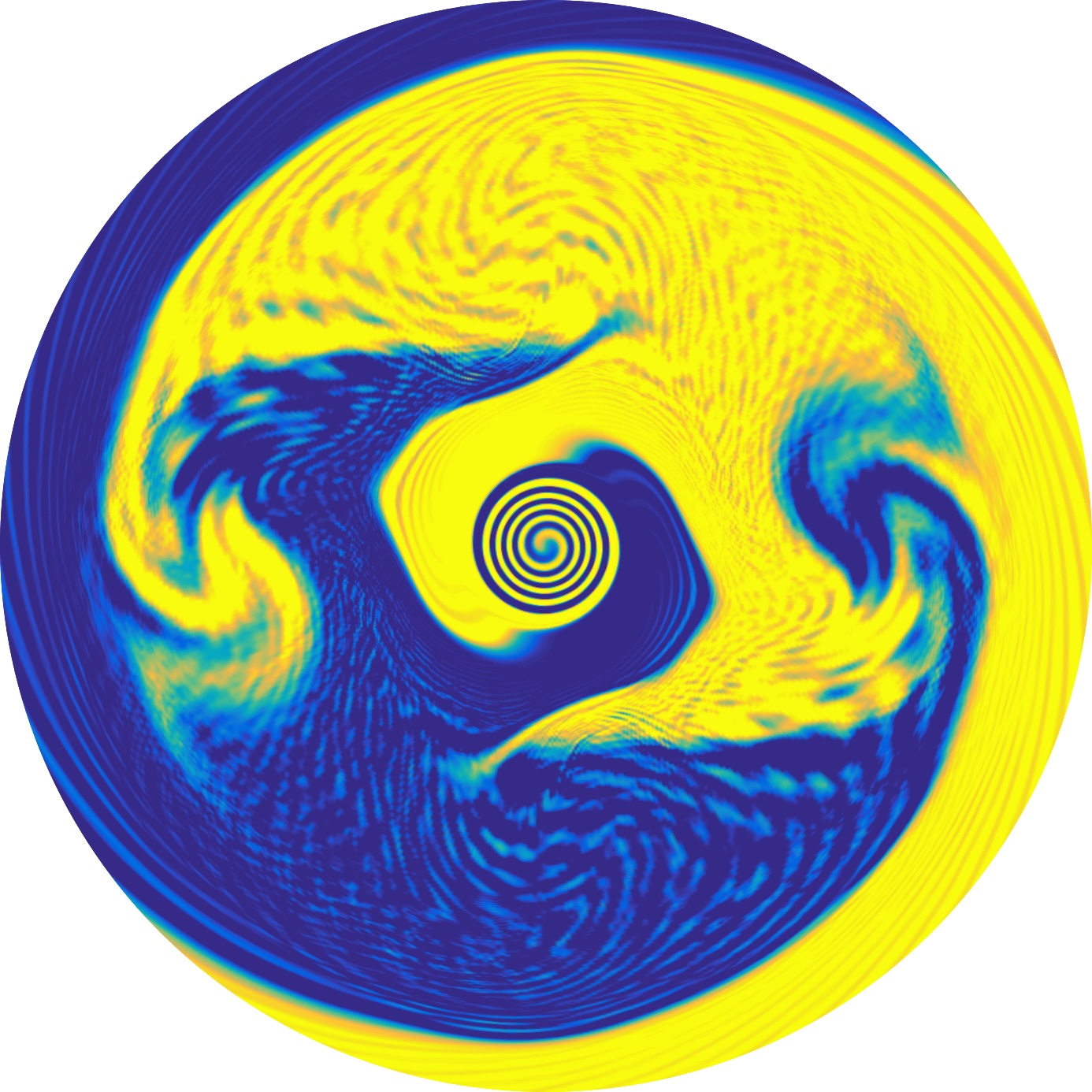}
    \caption{$t=2$}
  \end{subfigure}
  \begin{subfigure}[b]{0.16\textwidth}
    \includegraphics[width=\textwidth]{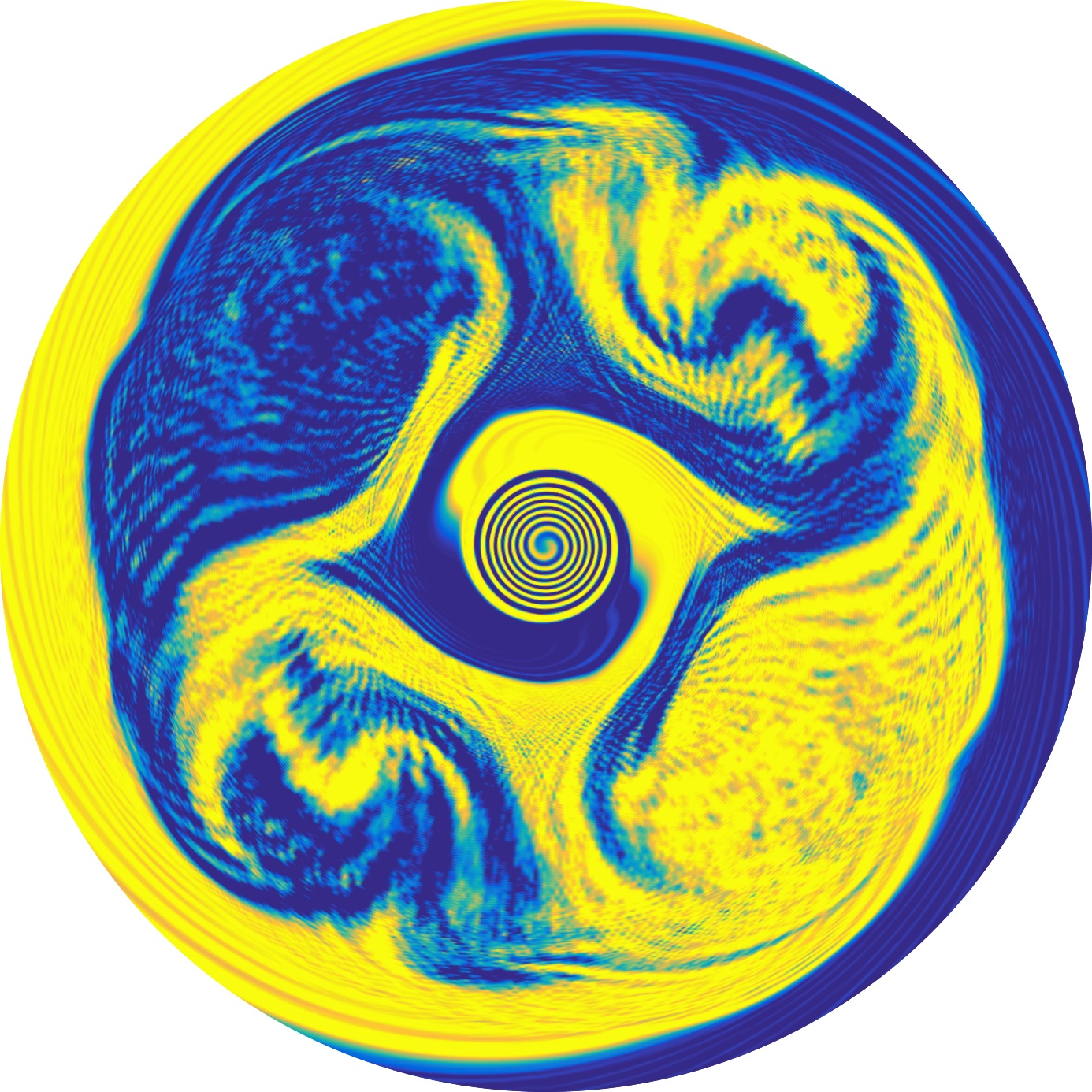}
    \caption{$t=3$}
  \end{subfigure}
  \begin{subfigure}[b]{0.16\textwidth}
    \includegraphics[width=\textwidth]{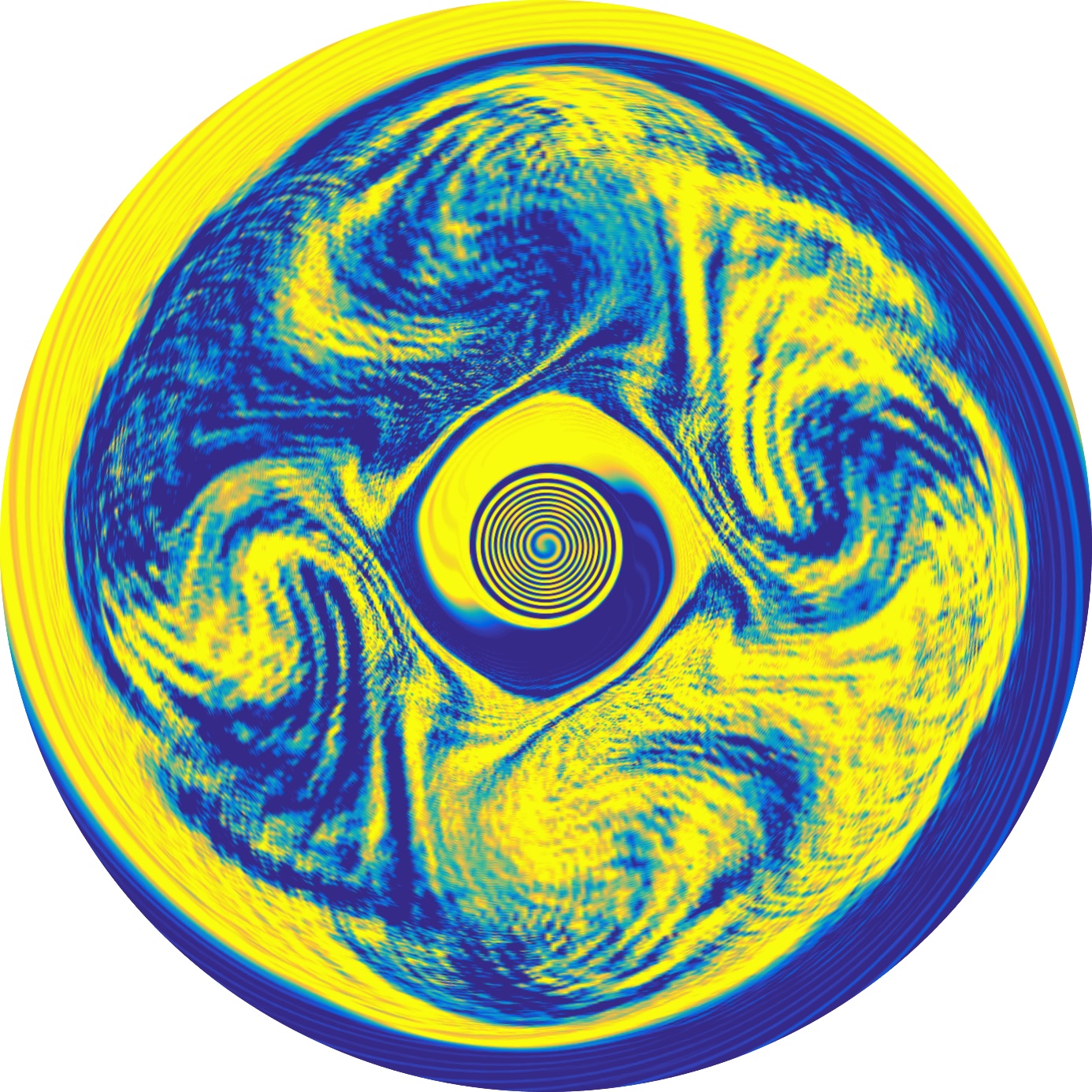}
    \caption{$t=4$}
  \end{subfigure}
  \begin{subfigure}[b]{0.16\textwidth}
    \includegraphics[width=\textwidth]{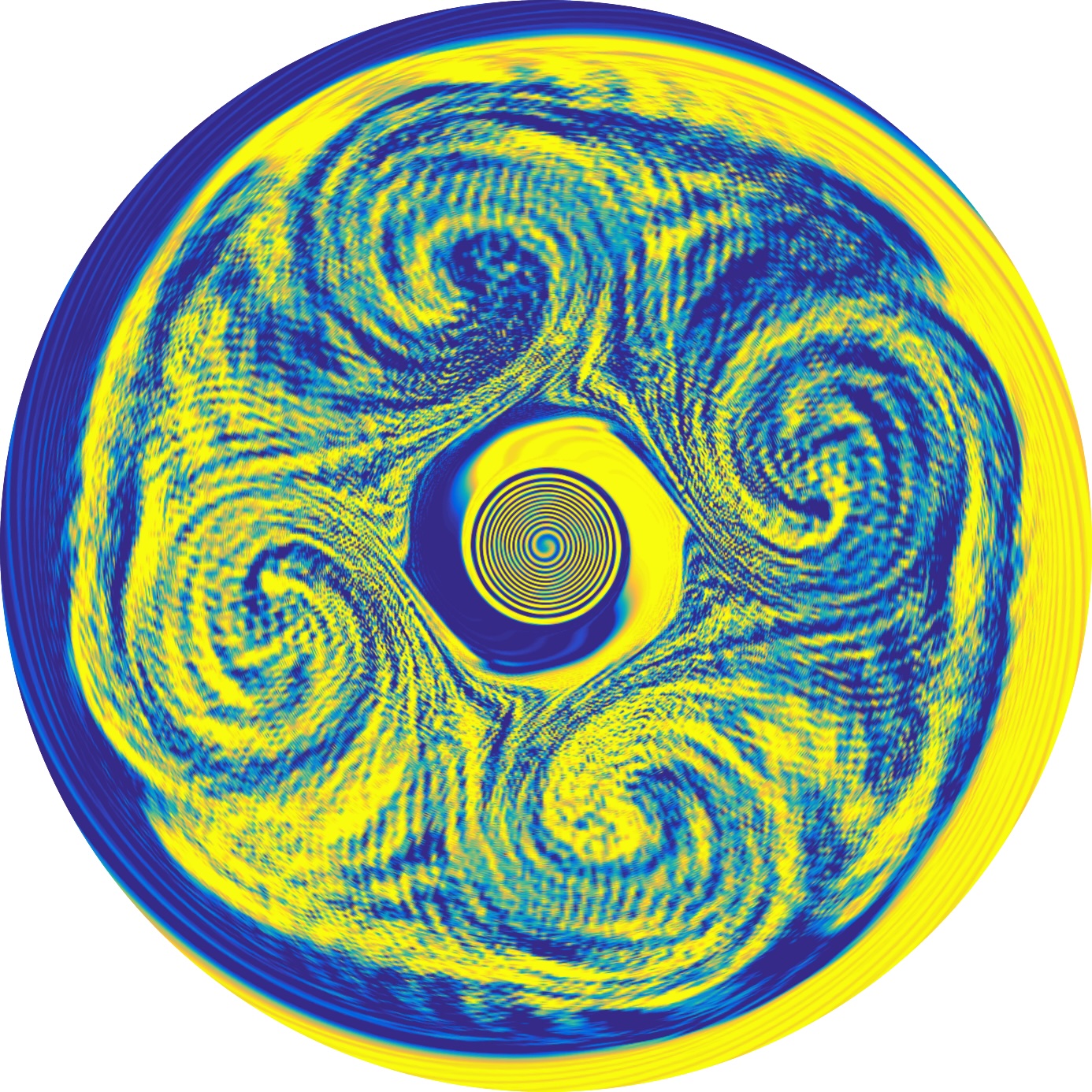}
    \caption{$t=5$}
  \end{subfigure}
  \caption{Evolution of $\theta_h$ for $t\in[0, 5]$ with initial control \eqref{eq:u1_Doswell_1} and initial data \eqref{eq:init_Doswell}.}
  \label{fig:theta_Doswell_optimal_1}
\end{figure}

We repeat the optimization with the initial guess~\eqref{eq:u1_Doswell_2}. Figure~\ref{fig:u1u2_Doswell_2} shows the resulting optimized coefficients. The mix-norm again decays in a near-exponential fashion (Figure~\ref{fig:mixnorm_Doswell_2}), with a rate of approximately $0.25$. The conservation of the properties in Figure~\ref{fig:property_Doswell_2} remains. Figure~\ref{fig:theta_Doswell_optimal_2} shows snapshots of the scalar field.

In summary, the Doswell-flow experiments demonstrate that our optimal control approach can effectively handle more complex flow patterns and domains. The optimized control strategy yields a substantially faster reduction in the mix-norm (almost exponential in time) compared to any single-flow strategy, even on a non-Cartesian domain. At the same time, the strict preservation of mass, energy, and state-adjoint consistency of the scheme carries over to this setting, giving confidence in the physical reliability of the numerical solutions.

\begin{figure}[H]
  \centering
  \begin{subfigure}[b]{0.315\textwidth}
    \includegraphics[width=\textwidth]{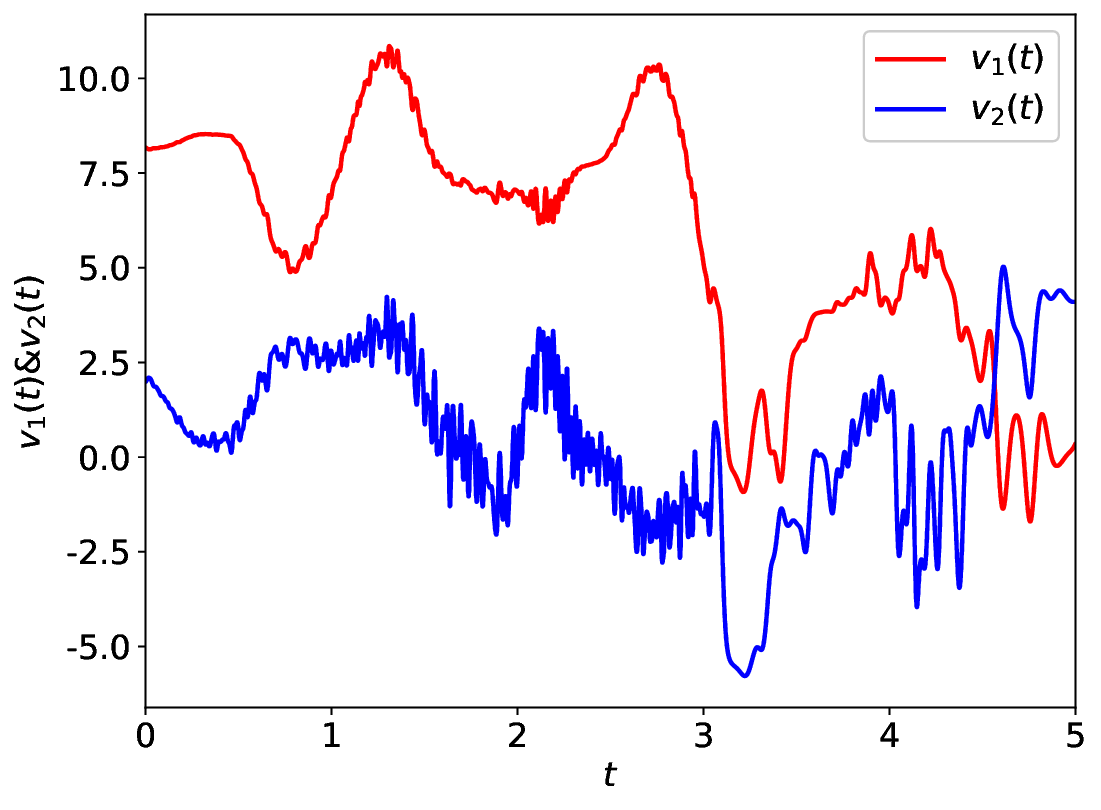}
    \caption{Evolutions of $v_1(t)$, $v_2(t)$.}
    \label{fig:u1u2_Doswell_2}
  \end{subfigure}
  \begin{subfigure}[b]{0.335\textwidth}
    \includegraphics[width=\textwidth]{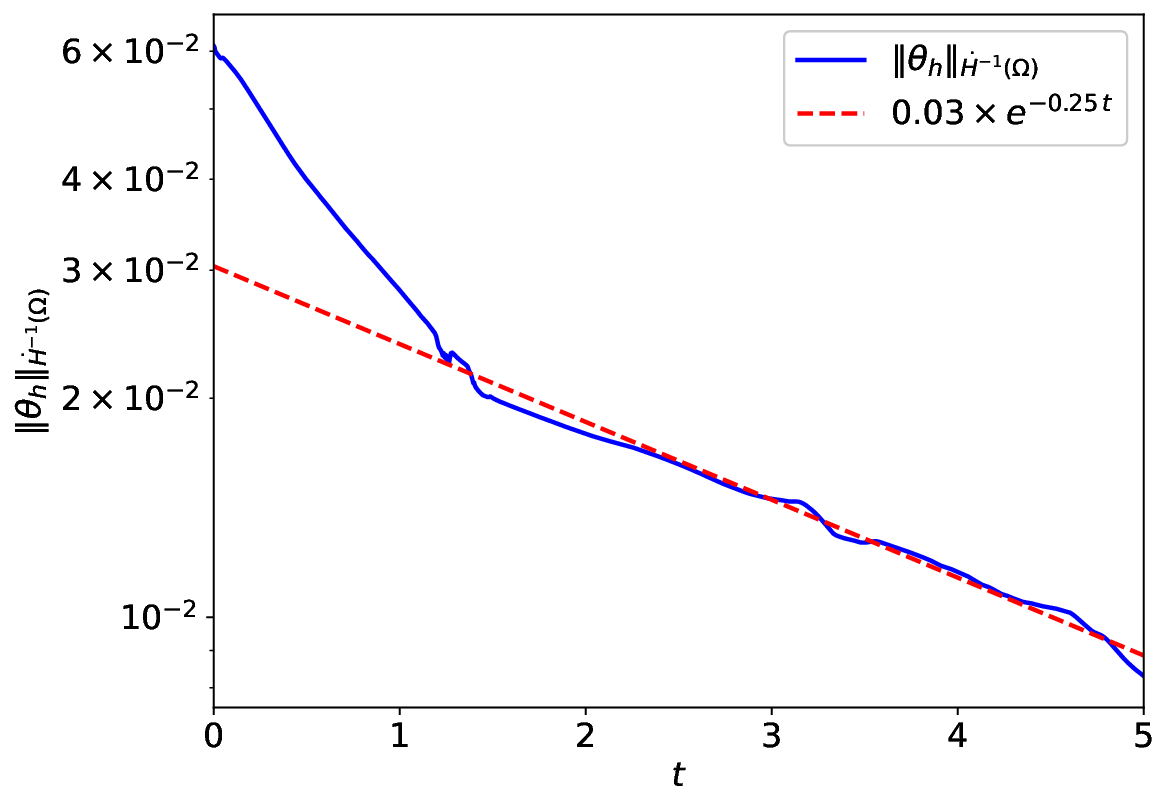}
    \caption{Evolution of $\Vert \theta_h \Vert_{\dot{H}^{-1}(\Omega)}$.}
    \label{fig:mixnorm_Doswell_2}
  \end{subfigure}
  \begin{subfigure}[b]{0.33\textwidth}
    \includegraphics[width=\textwidth]{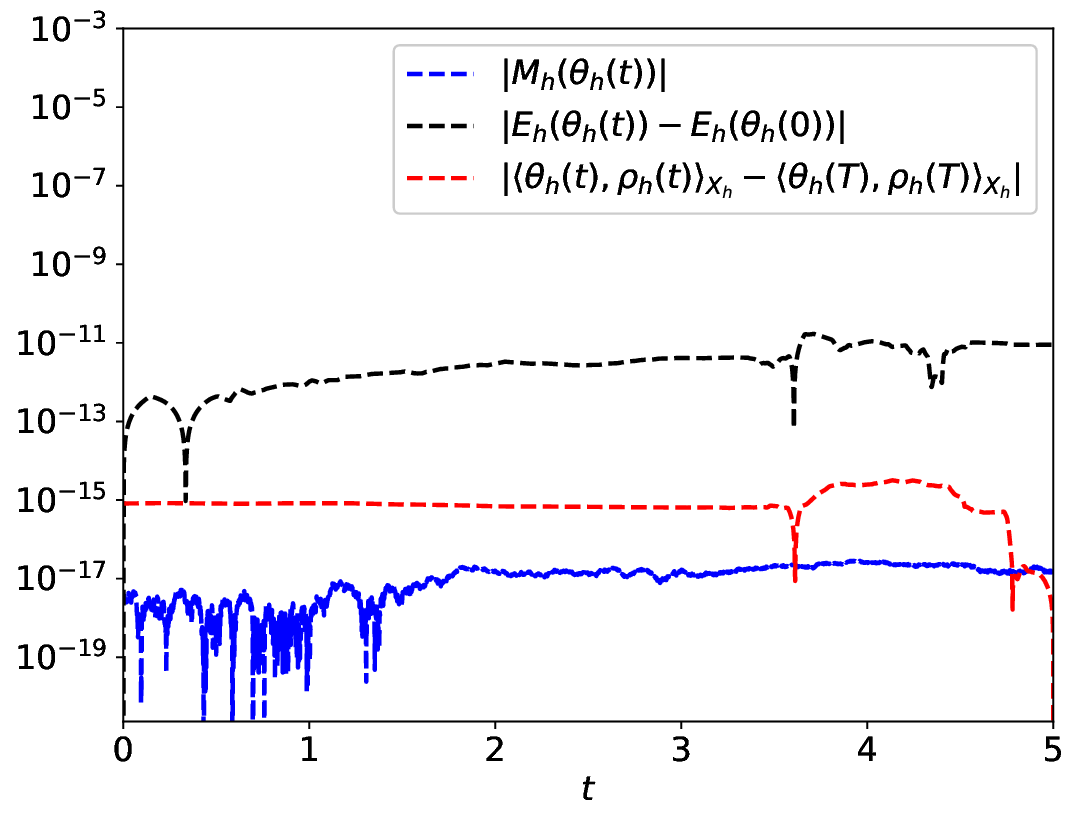}
    \caption{$M_h(\theta_h)$, $E_h(\theta_h)$, $\langle \theta_h, \rho_h \rangle_{X_h}$.}
    \label{fig:property_Doswell_2}
  \end{subfigure}
    \caption{Evolutions of $v(t)$, mix-norm $\Vert \theta_h \Vert_{\dot{H}^{-1}(\Omega)}$, mass $M_h(\theta_h)$, energy $E_h(\theta_h)$ and state-adjoint pairing $\langle \theta_h, \rho_h \rangle_{X_h}$ for $t\in[0, 5]$ with initial control \eqref{eq:u1_Doswell_2} and initial data \eqref{eq:init_Doswell}.}
  \label{fig:u1_mixnorm_Doswell_optimal_2}
\end{figure}

\begin{figure}[H]
  \centering
  \begin{subfigure}[b]{0.16\textwidth}
    \includegraphics[width=\textwidth]{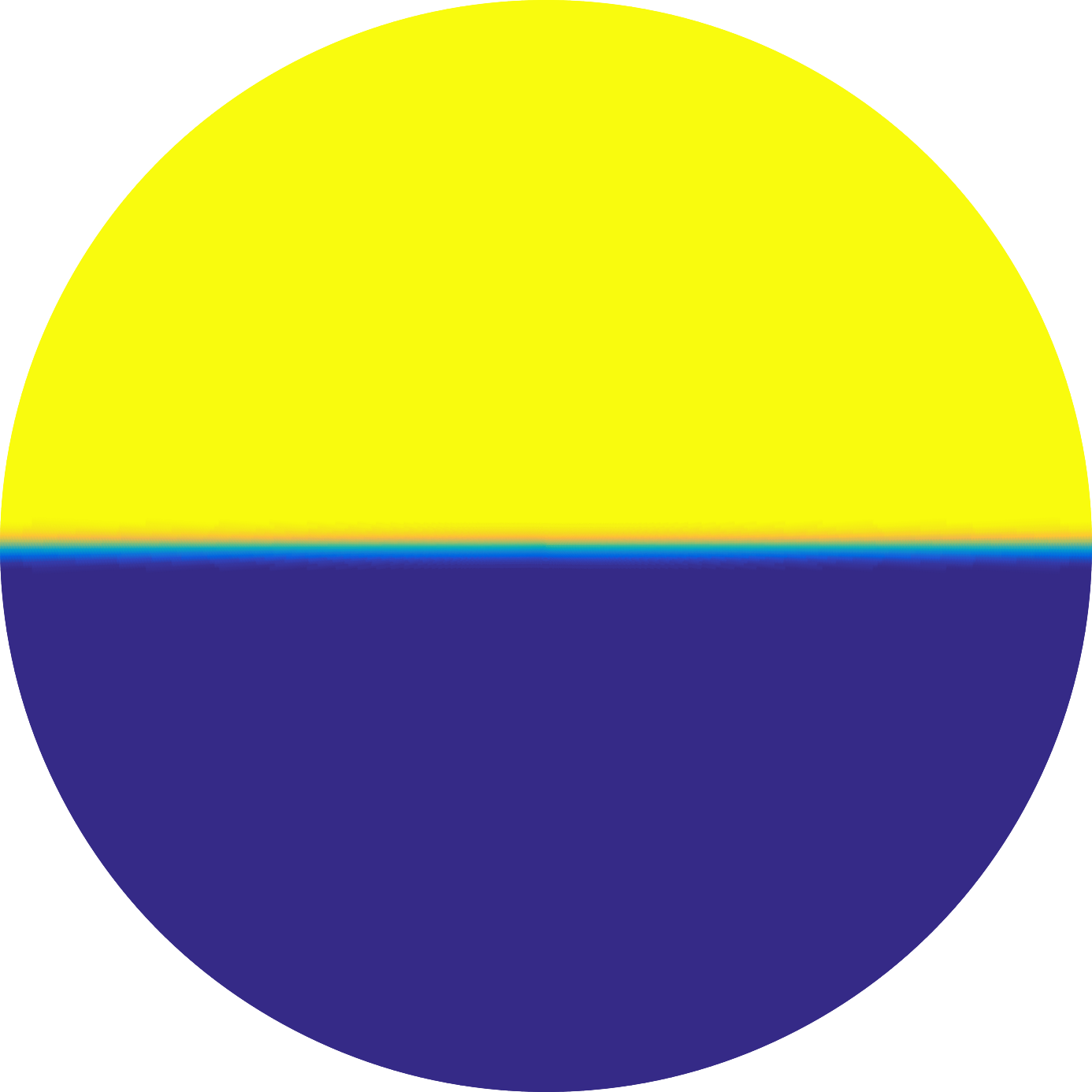}
    \caption{$t=0$}
  \end{subfigure}
  \begin{subfigure}[b]{0.16\textwidth}
    \includegraphics[width=\textwidth]{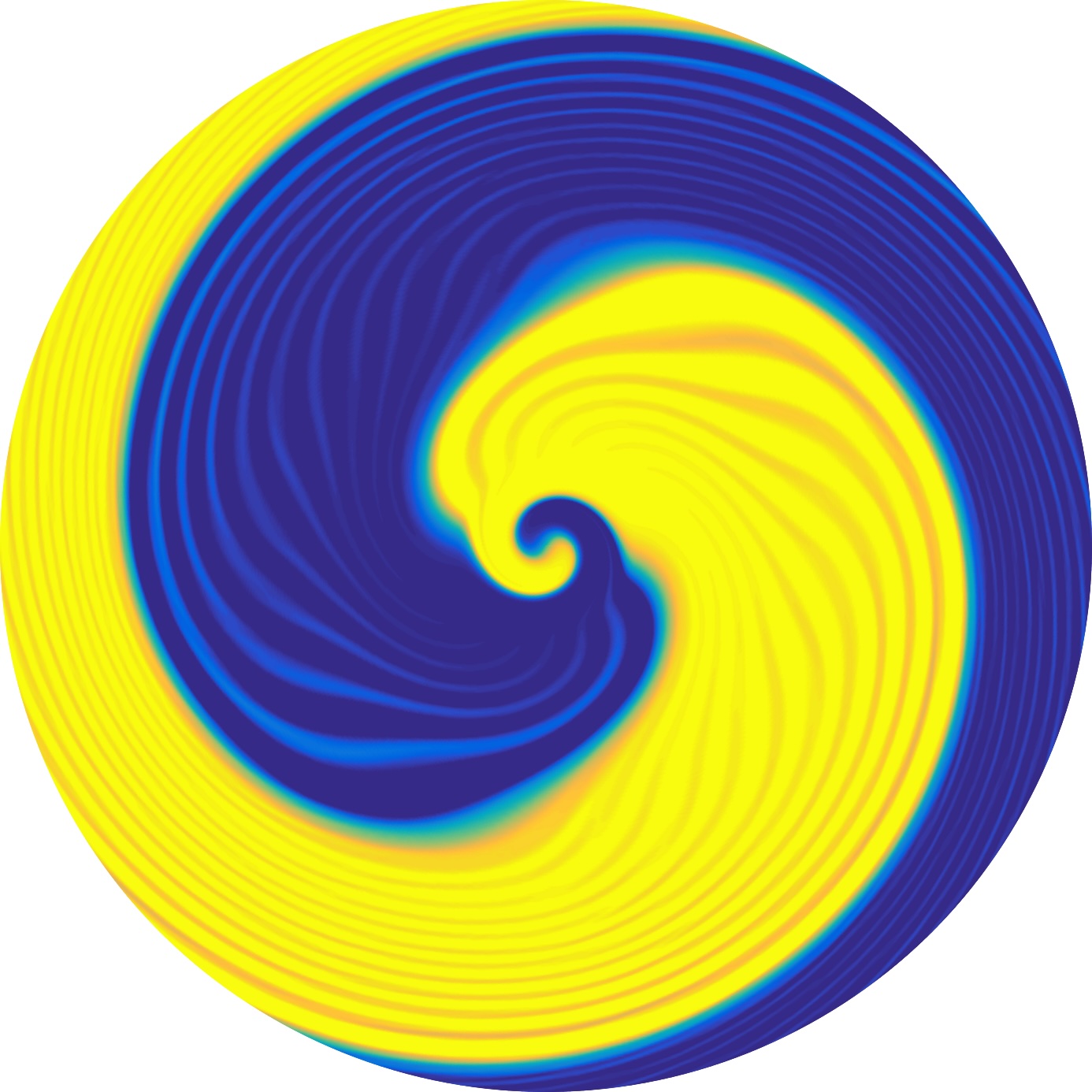}
    \caption{$t=1$}
  \end{subfigure}
  \begin{subfigure}[b]{0.16\textwidth}
    \includegraphics[width=\textwidth]{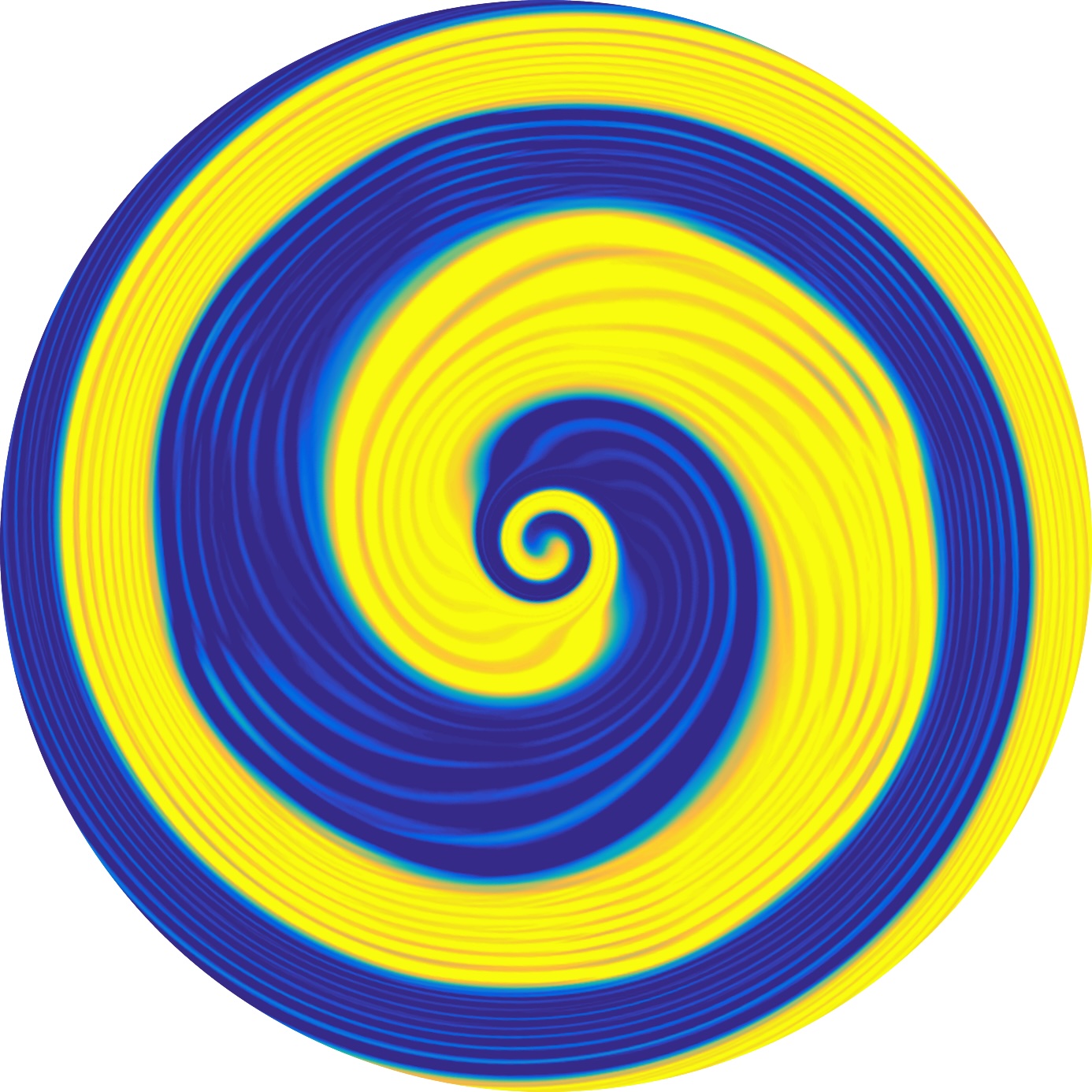}
    \caption{$t=2$}
  \end{subfigure}
  \begin{subfigure}[b]{0.16\textwidth}
    \includegraphics[width=\textwidth]{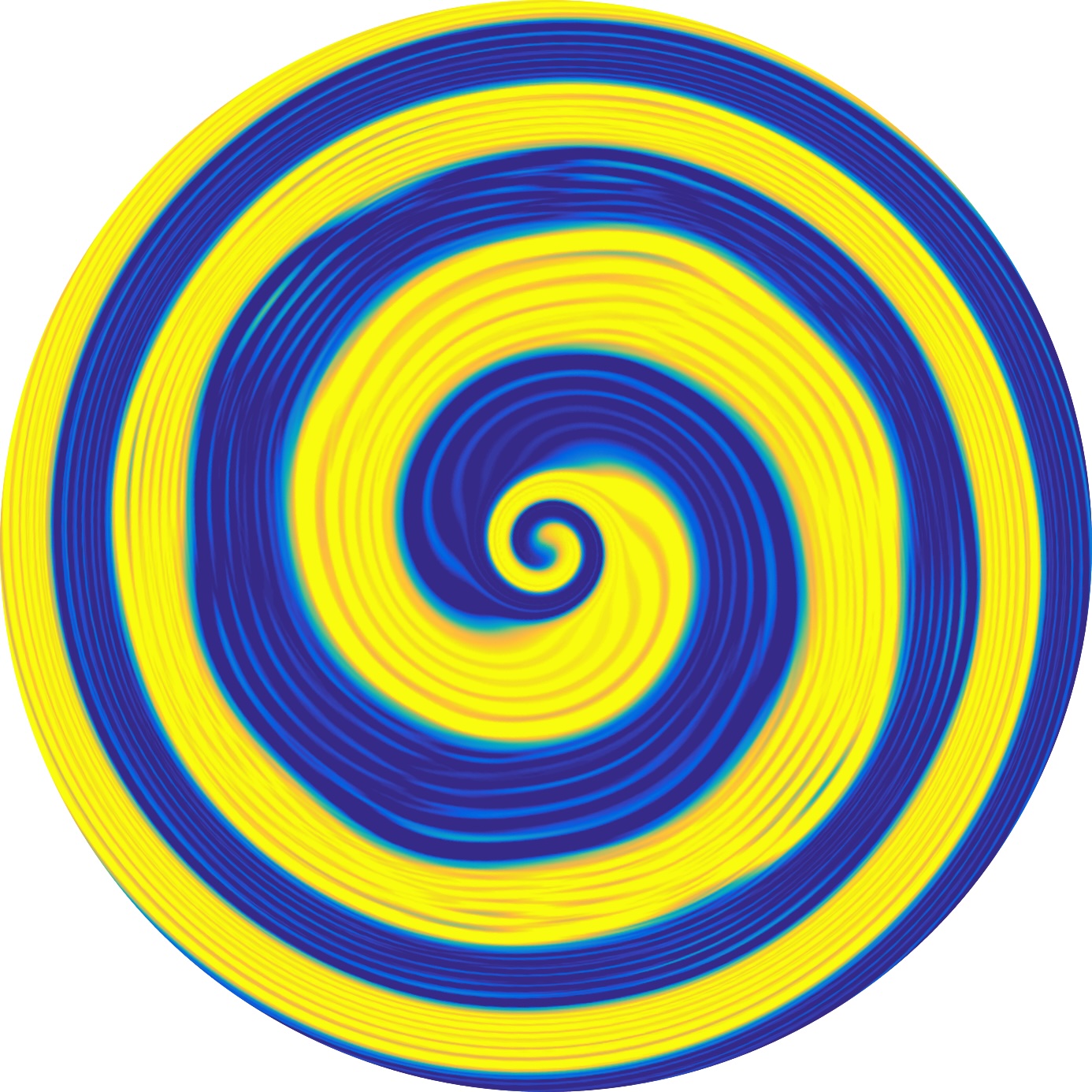}
    \caption{$t=3$}
  \end{subfigure}
  \begin{subfigure}[b]{0.16\textwidth}
    \includegraphics[width=\textwidth]{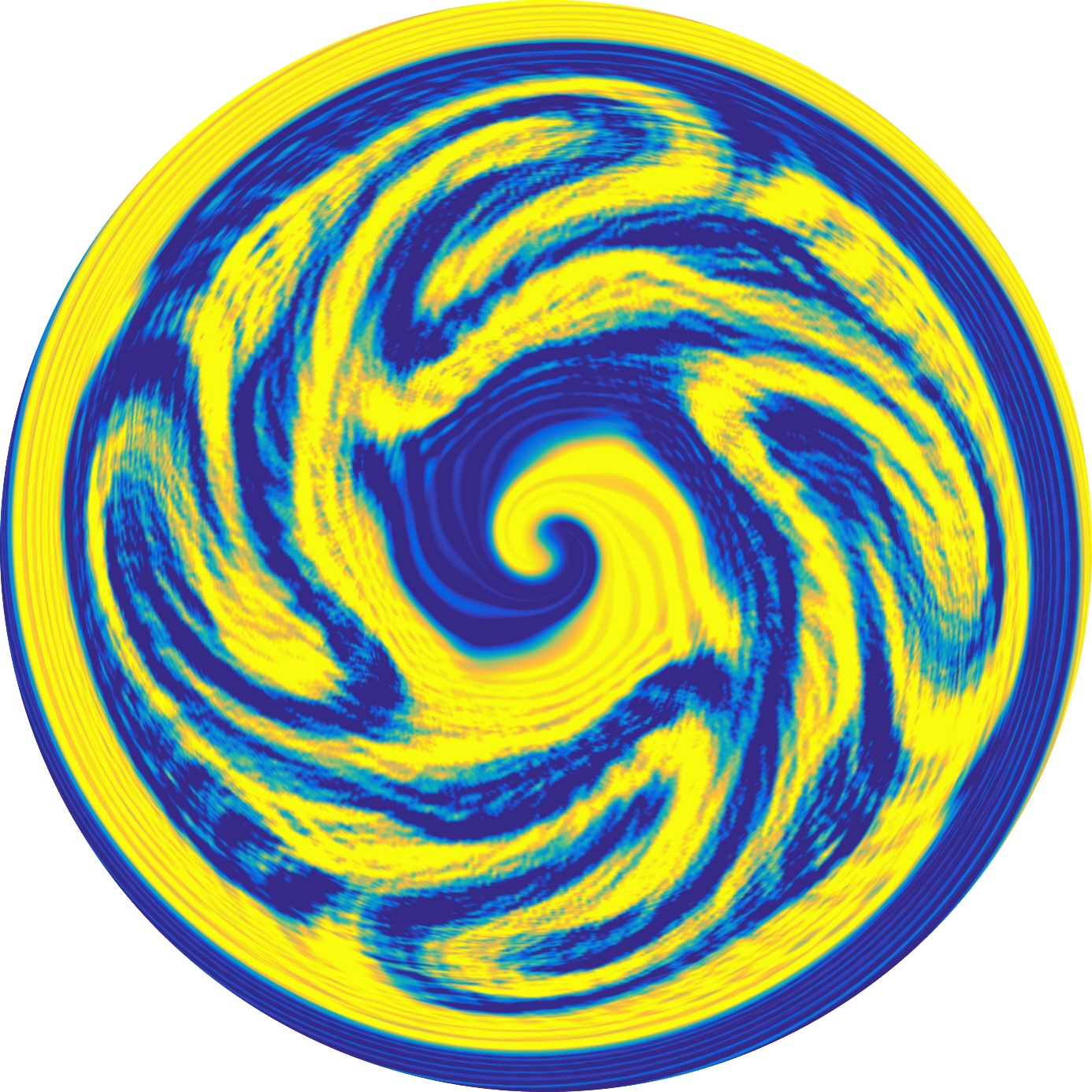}
    \caption{$t=4$}
  \end{subfigure}
  \begin{subfigure}[b]{0.16\textwidth}
    \includegraphics[width=\textwidth]{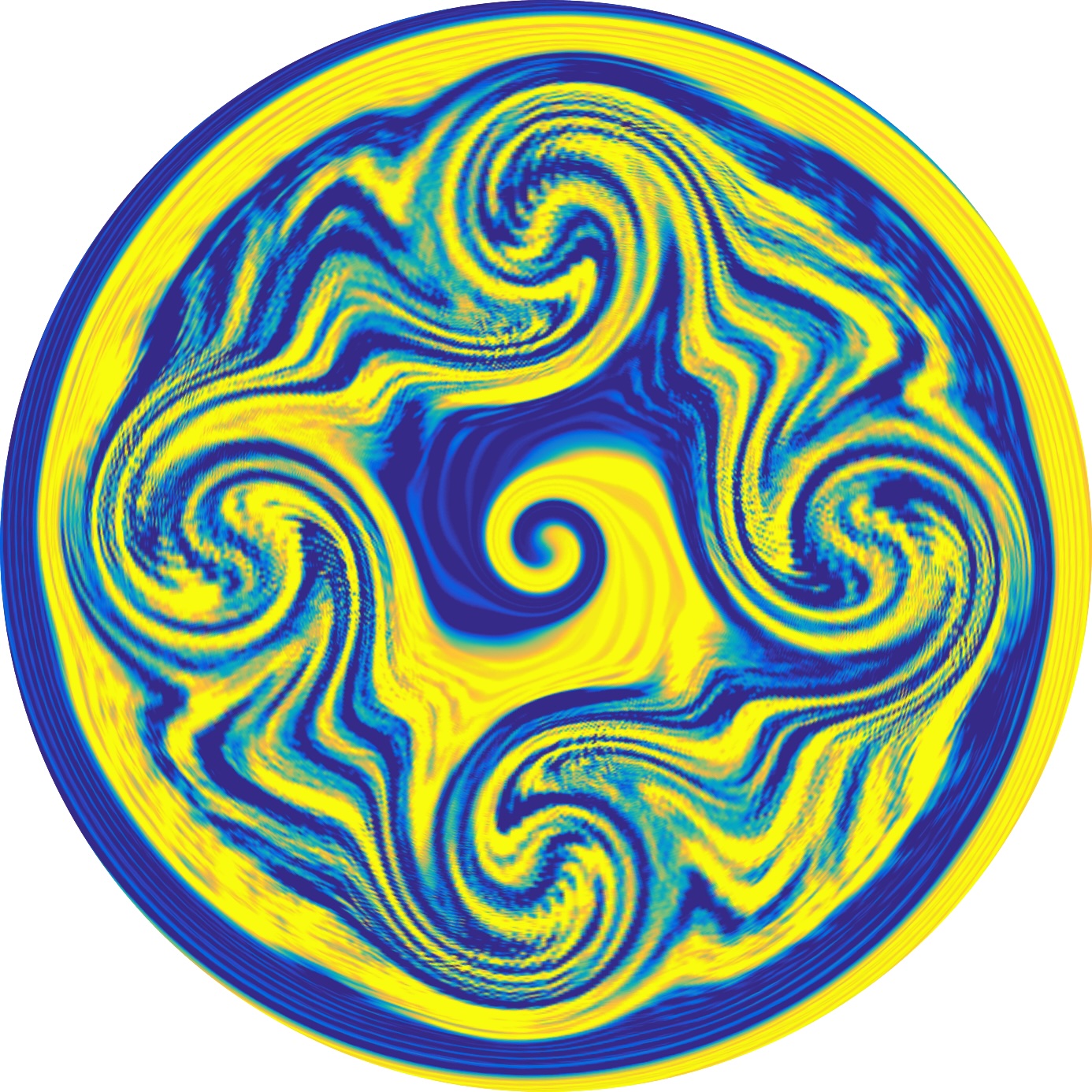}
    \caption{$t=5$}
  \end{subfigure}
  \caption{Evolution of $\theta_h$ for $t\in (0, 5)$ with initial control \eqref{eq:u1_Doswell_2} and initial data \eqref{eq:init_Doswell}.}
  \label{fig:theta_Doswell_optimal_2}
\end{figure}

\section{Conclusions and outlook}\label{conclusions}
We proposed a structure-preserving numerical scheme for optimal control of mixing in incompressible transport. The central innovation is the integration of a conservative-centered finite-volume discretization and time-symmetric Crank--Nicolson stepping within an adjoint-based optimization loop in a way that preserves, at the fully discrete level, (i) mass and $L^2$ invariants and (ii) exact state--adjoint consistency. This yields a reduced-gradient method whose discrete gradient is consistent with the discrete objective (optimize--then--discretize and discretize--then--optimize agree), which is a key requirement for reliable PDE-constrained optimization.

Numerically, the proposed scheme preserves the invariants to machine precision and yields a stable, reliable adjoint-based optimization method. Across the tested geometries and initializations, the resulting optimal controls deliver a significantly faster decay of the mix-norm than steady single-mode mixers, demonstrating both substantial mixing enhancement and robust optimization performance.

From a practical mixing viewpoint, the experiments point to a clear takeaway: with a small, actuator-limited set of stirring modes, any steady actuation remains a weak mixer (at best, polynomial decay), while an optimized time-dependent scheduling of the same modes can produce mix-norm decay consistent with near-exponential rates over the tested horizons.

Several extensions are natural and relevant for future work. First, higher-order spatial discretizations that retain skew-adjointness and time symmetry would improve performance in strongly filamented regimes without compromising gradient fidelity. Second, incorporating realistic actuation constraints (amplitude bounds, enstrophy/energy budgets, switching costs) would quantify performance--cost trade-offs and improve physical relevance. Third, extending the framework to advection--diffusion and Navier--Stokes settings, including three-dimensional configurations, would test the robustness of the observed control mechanisms and enable large-scale studies aligned with practical mixing applications.

\section{Appendix}\label{appendix}

\subsection{Properties of  (\ref{eq:state}) and (\ref{eq:adjoint})}
\label{Appendix-RepresentationOfState}

The properties presented here are classical and stated with short proofs for the sake of completeness of this paper. 

First of all,  the well-posedness of equation (\ref{eq:state}) is stated in the following lemma.

\begin{lem}\label{20251209-lemma-WellPosednessOfTransport}
Let $\theta^0 \in L^2(\Omega)$ and let $\mathbf{v} \in L^2((0,T); C^1(\overline{\Omega}; \mathbb R^d) )$ satisfy \eqref{cond:v}. Then, equation (\ref{eq:state}) has a unique weak solution $\theta \in L^2((0,T); L^2(\Omega))$ in the following sense:
\begin{align*}
  \int_0^T \int_{\Omega}  \theta (-\partial_t \varphi - \mathbf{v} \cdot \nabla \varphi) dx dt
= \int_{\Omega}  \theta^0(x) \varphi(0,x) dx,
~\forall\, \varphi \in C_0^{\infty}([0,T) \times \Omega). 
\end{align*}
Furthermore, it is in $C([0,T]; L^2(\Omega))$ with the  representation
\begin{align}\label{20251209-PDEandODE}
    \theta(t) = \theta^0( \psi_{0,t}(\cdot))
    \text{ in }  L^2(\Omega), ~ t \in [0,T],
\end{align}
where $\{ \psi_{s,t} \}_{s,t \in [0,T]}  \subset C^1(\overline{\Omega}; \overline{\Omega})$ is the measure-preserving flow of the following ODE: 
\begin{align*}
  \frac{d}{ds}  \psi_{s,t}(x) = \mathbf{v}( s, \psi_{s,t}(x) ), ~ s\in [0,T];
  ~~ \psi_{s,t}(x) |_{s=t} =  x \in \overline{\Omega}. 
\end{align*}
\end{lem}

\begin{proof}
First of all, one can check that the function on the right-hand side of (\ref{20251209-PDEandODE}) is a weak solution, which gives the existence. 
Next, we show the uniqueness. Let $\theta_1,\theta_2$ be any two weak solutions. Then, we have
\begin{align}\label{20251209-UniquenessOfStateSolutions}
  \int_0^T \int_{\Omega}  (\theta_1 - \theta_2) (-\partial_t \varphi - \mathbf{v} \cdot \nabla \varphi) dx dt
= 0,
~\forall\, \varphi \in C_0^{\infty}([0,T) \times \Omega). 
\end{align}
Arbitrarily take a function $f \in C_0^{\infty}( (0,T) \times \Omega)$. Define a function
\begin{align*}
    \varphi_f(t,x) := \int_t^T f(s, \psi_{s,t}(x)) ds, ~ t\in [0,T],
    ~ x \in \Omega,
\end{align*}
which is in $H^1((0,T)\times \Omega)$ and satisfies
\begin{align*}
    \partial_t \varphi_f + \mathbf{v} \cdot \nabla \varphi_f = - f 
\text{ in } (0,T) \times \Omega;
~~\varphi_f|_{t=T} =0.
\end{align*}
 By density arguments,  we see from (\ref{20251209-UniquenessOfStateSolutions}) that
\begin{align*}
\int_0^T \int_{\Omega}  (\theta_1 - \theta_2) f dx dt
 = \int_0^T \int_{\Omega}  (\theta_1 - \theta_2) (-\partial_t \varphi_f - \mathbf{v} \cdot \nabla \varphi_f) dx dt
= 0.  
\end{align*}
Since $f \in C_0^{\infty}( (0,T) \times \Omega)$ was arbitrarily taken, the above implies $\theta_1=\theta_2$, which proves the desired uniqueness.

Finally,  each element in $\{ \psi_{s,t} \}_{s,t \in [0,T]}$ is a measure-preserving map due to    Liouville's formula and the incompressibility condition div\,$\mathbf{v}=0$ in \eqref{cond:v}. This completes the proof.
\end{proof}

Next, we show the conservation of $L^p$-norm of solutions to equation \eqref{eq:state}, i.e., for any solution $\theta$ to equation \eqref{eq:state} with $\theta^0 \in L^\infty(\Omega)$, 
\begin{align}\label{20251209-ConservationOfLpNorm}
    \int_{\Omega}  |\theta(t)|^p dx \equiv \text{constant},~  t \in[0,T],
    ~p \in [1,+\infty]. 
\end{align}
Indeed, we obtain from \eqref{20251209-PDEandODE} (where $\{ \psi_{0,t} \}_{t\in[0,T]}$ is measure-preserving) that
\begin{align*}
    \int_{\Omega}  |\theta(t)|^p dx = \int_{\Omega}  | \theta^0( \psi_{0,t}(x) ) |^p dx
    = \int_{\Omega}  | \theta^0( x ) |^p dx,
    ~ t \in [0,T], ~p \in [1,+\infty], 
\end{align*}
which leads to \eqref{20251209-ConservationOfLpNorm}. 

Finally, we show that when $\theta(0), \rho(T) \in L^2(\Omega)$,
\begin{align}\label{20251209-ForwardBackwardPair}
    \int_{\Omega}  \theta(t) \rho(t) dx \equiv \text{constant},~  t \in[0,T].
\end{align}
Because equations (\ref{eq:state}) and (\ref{eq:adjoint}) are the same, the above can be directly checked from \eqref{20251209-PDEandODE} (where $\{ \psi_{0,t} \}_{t\in[0,T]}$ is measure-preserving).

\subsection{Existence of optimal controls}
\label{Appendix-ExistenceOfOptimalControl}
Since $J$ is nonnegative, we can take a minimizing sequence $( \mathbf{v}_k )_{k\geq 1} \subset L^2((0,T);U)$ of the cost functional $J$ so that
\begin{align*}
    \lim_{k\rightarrow +\infty}   J( \mathbf{v}_k)  =  \inf_{\mathbf{v} \in L^2((0,T);U)}  J(\mathbf{v}).
\end{align*}
Then, by \eqref{eq:cost}, this sequence is bounded in $L^2((0,T);U)$. Thus, there is a subsequence, still denoted by the same way, and a function $\mathbf{v}^* \in L^2((0,T); U)$ so that
\begin{align}\label{20251227-MinimizingSequence}
    \mathbf{v}_k \rightharpoonup  \mathbf{v}^*
      \text{ weakly in } L^2((0,T);U)
    \text{ as }  k\rightarrow +\infty.
\end{align}
Furthermore, because of \eqref{eq:v_finite} and \eqref{20251004-ControlConstraint}, we obtain
\begin{align}\label{20251227-UniformC1Bound}
    \sup_{k\geq 1} \| \mathbf{v}_k \|_{ L^2((0,T);C^1(\overline{\Omega}; \mathbb R^d) ) }
    <  + \infty.    
\end{align}

Next, we claim that the sequence $( \theta(\cdot; \mathbf{v}_k) )_{k\geq 1}$ is relatively compact in the space $C([0,T]; L^2(\Omega))$. The key point of the proof is to use the expression \eqref{20251209-PDEandODE}. Write $\{ \psi_{s,t}^{ \mathbf{u} } \}_{s,t\in[0,T]}$ for the flow generated by a vector field $\mathbf{u}$ in $L^2((0,T); U)$. By computations, we determine from \eqref{20251227-UniformC1Bound} that for some $C>0$, 
\begin{align*}
  \sup_{k\geq 1}  | \psi^{ \mathbf{v}_k }_{0,t_1}(x_1)   - \psi^{ \mathbf{v}_k }_{0,t_2}(x_2) |
    \leq  C(|t_1-t_2|^{ \frac{1}{2} } + |x_1-x_2|),
    ~\forall\, x_1,x_2 \in \Omega,
    ~\forall\,t_1,t_2\in[0,T].
\end{align*}
From this and \eqref{20251209-PDEandODE}, one can check that the sequence $( \theta(\cdot; \mathbf{v}_k) )_{k\geq 1}$ is equicontinuous in $C([0,T]; L^2(\Omega))$ and   relatively compact in $L^2(\Omega)$ when taking values at each time $t$ (here the ``$L^2$-version'' of the Ascoli-Arzel\`{a} theorem is used, see \cite[Theorem \@4.26]{brezis2010functional}). Then, by the Ascoli-Arzel\`{a} theorem, this sequence is relatively compact in $C([0,T]; L^2(\Omega))$, which proves the claim.

Now, write $f^*$ for the limit of a subsequence of the sequence $( \theta(\cdot; \mathbf{v}_k) )_{k\geq 1}$ (still denoted by the same way)  in $C([0,T]; L^2(\Omega))$. From this and \eqref{20251227-MinimizingSequence}, we determine that for each $\varphi \in C_0^{\infty}([0,T) \times \Omega)$,  
\begin{align*}
  \int_0^T \int_{\Omega}  f^* (\partial_t \varphi + \mathbf{v}^* \cdot \nabla \varphi) dx dt
=& \lim_{k \rightarrow +\infty}  \int_0^T \int_{\Omega}  \theta(t; \mathbf{v}_k) (\partial_t \varphi + \mathbf{v}_{k} \cdot \nabla \varphi) dx dt
\nonumber\\
=& - \int_{\Omega}  \theta^0(x) \varphi(0,x) dx,
\end{align*}
which means that $f^*$ is a weak solution to equation \eqref{eq:state} with the same initial data and velocity field as $\theta(\cdot;\mathbf{v}^*)$. Thus,
\begin{align*}
    \theta(\cdot;\mathbf{v}^*) = f^* = \lim_{k \rightarrow +\infty}
    \theta(\cdot; \mathbf{v}_k) 
\text{ in } C([0,T]; L^2(\Omega)).
\end{align*}
Then, due to weak lower semicontinuity of $J$, we find
\begin{align*}
    J( \mathbf{v}^* )  \leq   \liminf_{k \rightarrow +\infty} J( \mathbf{v}_k)  =  \inf_{\mathbf{v} \in L^2((0,T);U)}  J(\mathbf{v}),
\end{align*}
which implies that $\mathbf{v}^*$ is an optimal control of $J$. This ends the proof.

\subsection{Well-posedness of (\ref{eq:eta})}
\label{Appendix-WellPosenessOfEllpiticEquation}

A function $\eta \in H^1(\Omega)$ with $\int_{\Omega} \eta dx =0$ is said to be a solution to equation \eqref{eq:eta} if  $\eta$ satisfies
\begin{align}\label{20251206-WellPosedness}
\langle \nabla \eta, \nabla \varphi \rangle_{L^2(\Omega; \mathbb R^d)}
=  \langle \theta(T;\mathbf{v}), \varphi \rangle_{L^2(\Omega)}
 ,~\forall\, \varphi \in  H^1(\Omega). 
\end{align}
Write $\dot{H}^{1}(\Omega)$ for the space of all the zero-mean functions in $H^1(\Omega)$.  First, we aim to show equation \eqref{eq:eta} has a solution in $\dot{H}^{1}(\Omega)$. To this end, define 
\begin{align*}
    a(f,g) := \langle \nabla f, \nabla g \rangle_{L^2(\Omega; \mathbb R^d)},~f,g \in \dot{H}^{1}(\Omega). 
\end{align*}
This functional is bilinear, symmetric, and bounded over $\dot{H}^{1}(\Omega)$. Furthermore, by the Poincar\'{e}-Wirtinger inequality (see \cite[Page\@ 312]{brezis2010functional}), it is also coercive over $\dot{H}^{1}(\Omega)$. Then, by the Lax-Milgram theorem (see \cite[Corollary\@ 5.8]{brezis2010functional}), there is a unique function $\eta^* \in \dot{H}^{1}(\Omega)$ so that
\begin{align*}
\langle \nabla \eta^*, \nabla \varphi \rangle_{L^2(\Omega; \mathbb R^d)}
= a(\eta^*, \varphi)
=  \langle \theta(T;\mathbf{v}), \varphi \rangle_{L^2(\Omega)}
 ,~\forall\, \varphi \in  \dot{H}^1(\Omega). 
\end{align*}
Since $\int_{\Omega} \theta(T;\mathbf{v}) dx = \int_{\Omega} \theta^0 dx =0$ (by \eqref{20251209-PDEandODE}), the above implies $\eta^*$ satisfies equation \eqref{20251206-WellPosedness}, which proves the existence of solutions to equation \eqref{eq:eta}. 

Next, we show the uniqueness of the solution to equation \eqref{eq:eta}. For this purpose, let $\eta_1, \eta_2 \in \dot{H}^1(\Omega)$ be any two solutions. Then, by \eqref{20251206-WellPosedness}, we have
\begin{align*}
\langle \nabla \eta_1, \nabla \varphi \rangle_{L^2(\Omega; \mathbb R^d)}
= \langle \nabla \eta_2, \nabla \varphi \rangle_{L^2(\Omega; \mathbb R^d)}
=\langle \theta(T;\mathbf{v}), \varphi \rangle_{L^2(\Omega)}
 ,~\forall\, \varphi \in  H^1(\Omega), 
\end{align*}
which implies
\begin{align*}
\langle \nabla (\eta_1 - \eta_2), \nabla \varphi \rangle_{L^2(\Omega; \mathbb R^d)}
= 0 ,~\forall\, \varphi \in  H^1(\Omega). 
\end{align*}
Since $\int_{\Omega} \eta_i dx = 0$ ($i=1,2$), the above indicates that $\eta_1 = \eta_2$ over $\Omega$, which proves the desired uniqueness.  This completes the proof.

\subsection{Characterization (\ref{20251010-AnotherFormOfCostFunctional})}
\label{Appendix-NewFormOfJ}

To prove \eqref{20251010-AnotherFormOfCostFunctional}, it suffices to show 
\begin{align}\label{20251124-Estiamte-3}
     \| f \|^2_{ \dot{H}^{-1}(\Omega) } = 
    \langle f, \eta(f)  \rangle_{L^2(\Omega)}
\text{ for each }   f\in L^2(\Omega)
\text{ with } \int_{\Omega} f dx =0,
\end{align}
where $\eta(f)$ the solution to equation \eqref{eq:eta} with $\theta$ replaced by $f$. To this end, take a function $f\in L^2(\Omega)$ with $\text{ with } \int_{\Omega} f dx =0$. By \eqref{eq:eta} (see also \eqref{20251206-WellPosedness}), we find
\begin{align}\label{20251124-Estiamte-2}
 \langle f, \varphi \rangle_{L^2(\Omega)}
 = \langle \nabla \eta(f), \nabla \varphi \rangle_{L^2(\Omega; \mathbb R^d)}
 ,~\forall\, \varphi \in  H^1(\Omega). 
\end{align}
This, together with  \eqref{neg_norm}, implies
\begin{align}\label{20251124-Estiamte-1}
   \| f \|_{ \dot{H}^{-1}(\Omega) } 
   \leq& \sup_{ \|\nabla\varphi \|_{L^2(\Omega; \mathbb R^d)} \leq 1} 
   \langle f, \varphi \rangle_{L^2(\Omega)}
   \nonumber\\
   =& \sup_{ \|\nabla\varphi \|_{L^2(\Omega; \mathbb R^d)} \leq 1}
   \langle \nabla \eta(f), \nabla \varphi \rangle_{L^2(\Omega; \mathbb R^d)}
\leq \| \nabla \eta(f) \|_{L^2(\Omega; \mathbb R^d)}. 
\end{align}
At the same time, it follows from \eqref{20251124-Estiamte-2} (with $\varphi=\eta(f)$) and \eqref{neg_norm} that
\begin{align*}
    \|\nabla \eta(f) \|_{L^2(\Omega; \mathbb R^d)}^2
= \langle f, \eta(f)  \rangle_{L^2(\Omega)} 
\leq  \| f \|_{ \dot{H}^{-1}(\Omega) } \| \nabla \eta(f) \|_{L^2(\Omega; \mathbb R^d)}.
\end{align*}
From this and \eqref{20251124-Estiamte-1}, we determine that
\begin{align*}
    \| f \|_{ \dot{H}^{-1}(\Omega) }= \| \nabla \eta(f) \|_{L^2(\Omega; \mathbb R^d)},
\end{align*}
which, along with \eqref{20251124-Estiamte-2} (with $\varphi=\eta(f)$) again, leads  \eqref{20251124-Estiamte-3}. This ends the proof of \eqref{20251010-AnotherFormOfCostFunctional}.

\subsection{Gradient computation (\ref{20251006-GraidentOfJ})}
\label{Appendix-Gradient}

Let $\mathbf{v}  \in   L^2( (0,T); U)$. We will prove the existence of  $J'(\mathbf{v})$ and compute it below. For this purpose, fix another velocity field $\mathbf{w}  \in   L^2( (0,T); U)$ and define
\begin{align}\label{20251208-NewPerturbatedControl}
    \mathbf{v}_{\varepsilon} := \mathbf{v} +  \varepsilon \mathbf{w},
    ~\varepsilon \in \mathbb R.
\end{align}
Write $\theta_{\varepsilon}$ for the solution to equation \eqref{eq:state} with the same initial data $\theta^0$ and the velocity field $\mathbf{v}$ replaced by $\mathbf{v}_{\varepsilon}$. Denote by  $\eta_{\varepsilon}$ the solution to equation \eqref{eq:eta} with $\theta$ replaced by $\theta_{\varepsilon}$. For simplicity, write
\begin{align*}
    \delta\theta_{\varepsilon} := \frac{1}{\varepsilon} (\theta_{\varepsilon} - \theta)
   ~\text{and}~
\delta\eta_{\varepsilon} := \frac{1}{\varepsilon} (\eta_{\varepsilon} - \eta), ~\varepsilon \in \mathbb R. 
\end{align*}
By the definition of $J$ in (OP), we obtain that for each $\varepsilon \in \mathbb R$, 
\begin{align}\label{20251124-GradientOfJ}
\frac{1}{\varepsilon} \big(
    J(\mathbf{v}_{\varepsilon})  -  J(\mathbf{v})
\big)
=&  
\frac{1}{2} \big[
\big\langle \theta(T), \delta\eta_{\varepsilon} \big\rangle_{L^2(\Omega)}
+ \big\langle \delta\theta_{\varepsilon}(T), \eta \big\rangle_{L^2(\Omega)}
+ \varepsilon \big\langle \delta\theta_{\varepsilon} (T), \delta\eta_{\varepsilon} \big\rangle_{L^2(\Omega)}
\big]
\nonumber\\
&+  \gamma \langle \mathbf{v}, \mathbf{w} \rangle_{L^2( (0,T); U)}
+ \frac{\gamma}{2}  \varepsilon\| \mathbf{w} \|_{L^2( (0,T); U)}^2
\nonumber\\
=& \big\langle \delta\theta_{\varepsilon}(T), \eta \big\rangle_{L^2(\Omega)}
+ \gamma \langle \mathbf{v}, \mathbf{w} \rangle_{L^2( (0,T); U)}
\nonumber\\
& + \frac{\varepsilon}{2} \big\langle \delta\theta_{\varepsilon} (T), \delta\eta_{\varepsilon} \big\rangle_{L^2(\Omega)}
+ \frac{\gamma}{2}  \varepsilon\| \mathbf{w} \|_{L^2( (0,T); U)}^2
. 
\end{align}
Here, we used the fact that $\big\langle \theta(T), \delta\eta_{\varepsilon} \big\rangle_{L^2(\Omega)}
= \big\langle \delta\theta_{\varepsilon}(T), \eta \big\rangle_{L^2(\Omega)}$ (deduced from \eqref{20251206-WellPosedness} with $(\theta,\varphi)=(\theta, \delta\eta_{\varepsilon})$ and $(\theta,\varphi)=(\delta\theta_{\varepsilon}, \eta)$ applied twice).

Next, we claim
\begin{align}\label{20251124-WeakConvergenceOfStates}
    \theta_{\varepsilon}  \rightharpoonup
    \theta  \text{ weakly in } L^2((0,T);L^2(\Omega))
    \text{ as }  \varepsilon \rightarrow 0.
\end{align}
To this end, arbitrarily take a subsequence $\{ \varepsilon_k \}_{k\geq 1}$ with $\lim_{k\rightarrow+\infty} \varepsilon_k =0$. Then, by 
\eqref{20251209-PDEandODE} and \eqref{20260105-AssumptionsOnInitialData}, 
the sequence $\{ \theta_{\varepsilon_k} \}_{k\geq 1}$ is bounded in $L^2((0,T);L^2(\Omega))$. Hence, there is a subsequence of $\{ \varepsilon_k \}_{k\geq 1}$, still denoted by the same way, and a function $\hat\theta \in L^2((0,T);L^2(\Omega))$ so that 
\begin{align*}
    \theta_{\varepsilon_k}  \rightharpoonup
    \hat\theta  \text{ weakly in } L^2((0,T);L^2(\Omega))
    \text{ as }  k \rightarrow +\infty.
\end{align*}
Since $\theta_{\varepsilon_k}$ verifies equation \eqref{eq:state} with $\mathbf{v}$ replaced by $\mathbf{v}_{\varepsilon_k}$, the above weak convergence, together with the strong convergence of $\{ \mathbf{v}_{\varepsilon_k} \}_{k\geq1}$ (see \eqref{20251208-NewPerturbatedControl}), yields that for each $\varphi \in C_0^{\infty}([0, T) \times \Omega)$,  
\begin{align*}
  \int_0^T \int_{\Omega}  \hat\theta (\partial_t \varphi + \mathbf{v} \cdot \nabla \varphi) dx dt
= \lim_{k \rightarrow +\infty}  \int_0^T \int_{\Omega}  \theta_{\varepsilon_k} (\partial_t \varphi + \mathbf{v}_{\varepsilon_k} \cdot \nabla \varphi) dx dt
= - \int_{\Omega}  \theta^0(x) \varphi(0,x) dx,
\end{align*}
which means that $\hat\theta$ is a weak solution to equation \eqref{eq:state} with the same initial data and velocity field as $\theta$. Thus, $\hat\theta = \theta$  (by Lemma \ref{20251209-lemma-WellPosednessOfTransport}). 
Because the subsequence $\{ \varepsilon_k \}_{k\geq 1}$ was arbitrarily taken, we conclude that \eqref{20251124-WeakConvergenceOfStates} holds.

Now, we study both $\delta\theta_{\varepsilon}(T)$ and $\delta\eta_{\varepsilon} $. Since both $\theta_{\varepsilon}$ and $\theta$ verify equation \eqref{eq:state}, it is clear that $\delta\theta_{\varepsilon}$  satisfies the following equation
\begin{align}\label{20251124-DifferenceEquation}
    \partial_t \delta\theta_{\varepsilon}     + \text{div}( \mathbf{v} \delta\theta_{\varepsilon} ) 
= - \text{div}(\mathbf{w} \theta_{\varepsilon}),
~ t\in (0,T);
~~\delta\theta_{\varepsilon}|_{t=0} = 0.
\end{align}
For each $z \in H^1(\Omega)$, denote by $\rho(\cdot;z)$ the solution to the adjoint equation \eqref{eq:adjoint} with $\eta$ replaced by $z$. By an analog of \eqref{20251209-PDEandODE} for equation \eqref{eq:adjoint}, there is a  $C>0$ so that
\begin{align}\label{20251124-H1Regularity}
\| \rho(\cdot;z) \|_{C([0,T];H^1(\Omega))}  \leq  C \|z\|_{H^1(\Omega)},
~\forall\, z \in H^1(\Omega). 
\end{align}
This, together with  \eqref{20251124-DifferenceEquation}, yields that for each $z \in H^1(\Omega)$,
\begin{align*}
& \langle \delta\theta_{\varepsilon}(T), z \rangle_{L^2(\Omega)}
= \langle \delta\theta_{\varepsilon}(T), \rho(T;z) \rangle_{L^2(\Omega)}
- \langle \delta\theta_{\varepsilon}(0), \rho(0;z) \rangle_{L^2(\Omega)}
\nonumber\\
=&  \int_0^T \big\langle
    - \text{div}(\mathbf{w} \theta_{\varepsilon}), \rho(t;z)
\big\rangle_{(H^1(\Omega))', H^1(\Omega)} dt
= \int_0^T 
    \theta_{\varepsilon} \nabla \rho(t;z) \cdot \mathbf{w}(t)
dt.
\end{align*}
Here, we used the no-penetration boundary condition of $\mathbf{w}$ (as well as a density argument on $\theta_{\varepsilon}$) in the last step. Then, we obtain from \eqref{20251124-WeakConvergenceOfStates} and \eqref{20251124-H1Regularity} that the sequence $\{ \delta\theta_{\varepsilon}(T) \}_{\varepsilon \in \mathbb R}$ is bounded in $(H^1(\Omega))'$, and that 
\begin{align}\label{20251124-ConvergenceOfPair}
    \lim_{\varepsilon \rightarrow 0}
    \langle \delta\theta_{\varepsilon}(T), \eta \rangle_{L^2(\Omega)}
= \int_0^T 
    \theta \nabla \rho(t;\eta) \cdot \mathbf{w}(t)
dt.
\end{align}
Furthermore, since each function $\delta\eta_{\varepsilon}=\eta(\delta\theta_{\varepsilon}(T))$ satisfies \eqref{eq:eta} (as well as \eqref{20251206-WellPosedness}), the boundedness of the sequence $\{ \delta\eta_{\varepsilon}(T) \}_{\varepsilon \in \mathbb R}$ in $H^1(\Omega)$ follows from that of $\{ \delta\theta_{\varepsilon}(T) \}_{\varepsilon \in \mathbb R}$  in $(H^1(\Omega))'$.

Finally, since $\mathbf{w}  \in   L^2( (0,T); U)$ was arbitrarily taken, we deduce from \eqref{20251124-GradientOfJ} and \eqref{20251124-ConvergenceOfPair} (as well as the boundedness of $\{ \delta\theta_{\varepsilon}(T) \}_{\varepsilon \in \mathbb R}$  and $\{ \delta\eta_{\varepsilon} \}_{\varepsilon \in \mathbb R}$) that $J'(\mathbf{v})$ exists in the G\^{a}teaux sense. Furthermore, it holds that for each $\mathbf{w}  \in   L^2( (0,T); U)$, 
\begin{align*}
 \langle J'(\mathbf{v}), \mathbf{w}  \rangle_{L^2( (0,T); U)}
=&  
\int_0^T \big\langle  
    \theta \nabla \rho(t;\eta), \mathbf{w}(t)
\big\rangle_{L^2(\Omega;\mathbb R^d)} dt
+  \gamma \langle \mathbf{v}, \mathbf{w} \rangle_{L^2( (0,T); U)}
\nonumber\\
=& \int_0^T \big\langle 
    \gamma \mathbf{v}(t) + \theta \nabla \rho(t;\eta), \mathbf{w}(t)
\big\rangle_{L^2(\Omega;\mathbb R^d)} dt.
\end{align*}
This leads to \eqref{20251006-GraidentOfJ} and completes the proof.

\subsection{Properties of the discrete Laplacian $L_h$}
\label{Appendix-DiscreteEllipticEquation}

\begin{lem}\label{20251209-lemma-PropertiesOfLh}
The operator $L_h$ in \eqref{20251011-DiscreteLaplacian} is self-adjoint and invertible over the finite dimension space $X_h^0$ (given by \eqref{20251208-SubspaceOfXh}). 
\end{lem}

\begin{proof}
We just give the sketch of the proof. By the definition of $L_h$ in \eqref{20251011-DiscreteLaplacian}, one can directly check that it maps $X_h^0$ to $X_h^0$ and is self-adjoint. The invertibility of $L_h$ can be proved by similar arguments as the proof of \cite[Lemma\@ 10.1]{EymardFVM} (there another kind of meshes is applied), which is equivalent to the injectivity of $L_h$. Indeed, we let $L_h \eta_h =0$ for some $\eta_h \in X_h^0$, and then see from \eqref{20251011-DiscreteLaplacian} (as well as \eqref{20251008-InnerProductOfXh}) that 
\begin{align*}
   0 = \langle L_h \eta_h, \eta_h \rangle_{X_h} =   \frac{1}{2}
   \sum_{K \in \mathcal T_h} 
\sum_{\Gamma \in \mathcal F_K, \Gamma \not\subset \partial\Omega} 
\frac{|\Gamma|}{ d_{K, L(K,\Gamma)} }  ( \eta_{h, L(K,\Gamma)} - \eta_{h, K} )^2, 
\end{align*}
which implies $\eta_h \in X_h^0$ is constant and so is zero due to its zero-mean. Therefore, $L_h$ is invertible over $X_h^0$. This completes the proof. 
\end{proof}

\subsection{Proof of Lemma \ref{20251008-lemma-AntisymmetricDiscreteDivergence}}
\label{sec:proof-lemma-AntisymmetricDiscreteDivergence}

\begin{proof}
First, we aim to show that $D_{ \mathbf{u} }$ is antisymmetric over $X_h$. For this purpose, take two arbitrary vectors $\mathbf{y}_h := (y_K)_{K\in \mathcal T_h}$ and  $ \mathbf{z}_h := (z_K)_{K\in \mathcal T_h}$ from $X_h$. By the definition of $D_{ \mathbf{u} }$ in \eqref{20251007-DiscreteDivergence} as well as \eqref{20251008-InnerProductOfXh} and \eqref{20250930-FluxOnFace}, we obtain
\begin{align}\label{20251008-ComputationsForDiscreteDivergence}
& \big\langle 
   D_{ \mathbf{u} } \mathbf{y}_h,  \mathbf{z}_h
\big\rangle_{X_h}
= \sum_{ K \in \mathcal T_h }   
\left(
\sum_{ \Gamma \in \mathcal F_K,  \,\Gamma \not\subset \partial\Omega }     F_{K,\Gamma}( \mathbf{u},  \mathbf{y}_h  )
\right)     z_K
\nonumber\\
=&  \sum_{ K \in \mathcal T_h }   
\sum_{ \Gamma \in \mathcal F_K,  \,\Gamma \not\subset \partial\Omega } 
\frac{1}{2}  \big( z_K  y_K + z_K  y_{L(K,\Gamma)} \big)
    \int_{ \Gamma }  \mathbf{u} \cdot \mathbf{n}_K d \sigma
    \\
=& \frac{1}{2} \sum_{ K \in \mathcal T_h }   z_K  y_K
\left(
\sum_{ \Gamma \in \mathcal F_K,  \,\Gamma \not\subset \partial\Omega } 
    \int_{ \Gamma }  \mathbf{u} \cdot \mathbf{n}_K d \sigma
\right)
 - \frac{1}{2} \sum_{ K \in \mathcal T_h }   
\sum_{ \Gamma \in \mathcal F_K,  \,\Gamma \not\subset \partial\Omega } 
   z_K  y_{L(K,\Gamma)} 
    \int_{ \Gamma }  \mathbf{u} \cdot \mathbf{n}_{L(K,\Gamma)} d \sigma. 
    \nonumber
\end{align}
Meanwhile, since $\mathbf{u} \in U$, we see from \eqref{20251004-ControlConstraint} that $\mathbf{u} \cdot \mathbf{n}_{\Omega} = 0$ on $\partial\Omega$ and div$\,\mathbf{u} = 0$ in $\Omega$, which yields that for each $K \in \mathcal T_h$, 
\begin{align}\label{20251009-ComputationsForDiscreteDivergence-3}
\sum_{ \Gamma \in \mathcal F_K,  \,\Gamma \not\subset \partial\Omega } 
    \int_{ \Gamma }  \mathbf{u} \cdot \mathbf{n}_K d \sigma
= \sum_{ \Gamma \in \mathcal F_K } 
    \int_{ \Gamma }  \mathbf{u} \cdot \mathbf{n}_K d \sigma
= \int_{ \partial K }  \mathbf{u} \cdot \mathbf{n}_K d \sigma
= \int_K \text{div}\,\mathbf{u}  \,dx
= 0.
\end{align}
This, together with \eqref{20251008-ComputationsForDiscreteDivergence}, implies
\begin{align}\label{20251009-ComputationsForDiscreteDivergence-2}
\big\langle 
    D_{ \mathbf{u} } \mathbf{y}_h,  \mathbf{z}_h
\big\rangle_{X_h}
= - \frac{1}{2} \sum_{ K \in \mathcal T_h }   
\sum_{ \Gamma \in \mathcal F_K,  \,\Gamma \not\subset \partial\Omega } 
   z_K  y_{L(K,\Gamma)} 
    \int_{ \Gamma }  \mathbf{u} \cdot \mathbf{n}_{L(K,\Gamma)} d \sigma. 
\end{align}
At the same time, for each control volume $L \in \mathcal T_h$ and each face $\Gamma \in \mathcal F_L $ with $\Gamma \not\subset \partial\Omega$, there is a unique $K\in \mathcal T_h$ so that 
\begin{align*}
    \Gamma \in \mathcal F_K     ~\text{and}~
    L = L(K,\Gamma)   ~(\text{i.e., } K=L(L,\Gamma) ),
\end{align*}
which yields 
\begin{align*}
  \sum_{ K \in \mathcal T_h }   
\sum_{ \Gamma \in \mathcal F_K,  \,\Gamma \not\subset \partial\Omega } 
z_K  y_{L(K,\Gamma)} 
    \int_{ \Gamma }  \mathbf{u} \cdot \mathbf{n}_{L(K,\Gamma)} d \sigma
= \sum_{ L \in \mathcal T_h }   
\sum_{ \Gamma \in \mathcal F_L,  \,\Gamma \not\subset \partial\Omega } 
z_{L(L,\Gamma)}  y_{L} 
    \int_{ \Gamma }  \mathbf{u} \cdot \mathbf{n}_L d \sigma.
\end{align*}
This, together with \eqref{20251009-ComputationsForDiscreteDivergence-2} and \eqref{20251009-ComputationsForDiscreteDivergence-3} (with $K$ replaced by $L$), implies
\begin{align*}
\big\langle 
    D_{ \mathbf{u} } \mathbf{y}_h,  \mathbf{z}_h
\big\rangle_{X_h}
=& - \frac{1}{2} \sum_{ L \in \mathcal T_h }   
\sum_{ \Gamma \in \mathcal F_L,  \,\Gamma \not\subset \partial\Omega } 
z_{L(L,\Gamma)}  y_{L}
    \int_{ \Gamma }  \mathbf{u} \cdot \mathbf{n}_L d \sigma
    \nonumber\\
=& - \sum_{ L \in \mathcal T_h }   
\sum_{ \Gamma \in \mathcal F_L,  \,\Gamma \not\subset \partial\Omega } 
\frac{1}{2}  \big( z_L  y_L + z_{L(L,\Gamma)}  y_{L} \big)
    \int_{ \Gamma }  \mathbf{u} \cdot \mathbf{n}_L d \sigma
\nonumber\\
=& - \sum_{ L \in \mathcal T_h }    y_L
\left(
\sum_{ \Gamma \in \mathcal F_L,  \,\Gamma \not\subset \partial\Omega } 
\frac{z_L   + z_{L(L,\Gamma)}}{2}  
    \int_{ \Gamma }  \mathbf{u} \cdot \mathbf{n}_L d \sigma
\right). 
\end{align*}
Then, by the definition of $D_{ \mathbf{u} }$ in \eqref{20251007-DiscreteDivergence} as well as \eqref{20250930-FluxOnFace} and \eqref{20251008-InnerProductOfXh}, we determine
\begin{align*}
\big\langle 
    D_{ \mathbf{u} } \mathbf{y}_h,  \mathbf{z}_h
\big\rangle_{X_h}
= - \sum_{ L \in \mathcal T_h }    y_L   F_{L,\Gamma}( \mathbf{u},  \mathbf{z}_h  )
= - \big\langle 
    \mathbf{y}_h,  D_{ \mathbf{u} } \mathbf{z}_h
\big\rangle_{X_h}. 
\end{align*}
Since $\mathbf{y}_h, \mathbf{z}_h $ are arbitrary in $X_h$, the above leads to the desired conclusion. 

Next, we verify the invertibility of the linear operator $I_h + D_{\mathbf{u}}$ on $X_h$. Since $X_h$ is finite-dimensional, it suffices to show that $0$ is not an eigenvalue of $I_h + D_{\mathbf{u}}$. 

Arguing by contradiction, assume that $0$ is an eigenvalue. Then there exists an eigenvector $\mathbf{y}_h \in X_h\setminus\{0\}$ such that
\[
    (I_h + D_{\mathbf{u}})\mathbf{y}_h = 0
    \qquad \text{in } X_h.
\]
Using that $D_{\mathbf{u}}$ is antisymmetric, we obtain
\begin{equation*}
\langle \mathbf{y}_h, \mathbf{y}_h \rangle_{X_h}
= \langle \mathbf{y}_h, \mathbf{y}_h \rangle_{X_h}
  + \langle \mathbf{y}_h, D_{\mathbf{u}} \mathbf{y}_h \rangle_{X_h} \\
= \langle \mathbf{y}_h, (I_h + D_{\mathbf{u}})\mathbf{y}_h \rangle_{X_h}
 = 0,
\end{equation*}
which implies $\mathbf{y}_h = 0$ in $X_h$, a contradiction. Hence, $0$ is not an eigenvalue of $I_h + D_{\mathbf{u}}$, and therefore $I_h + D_{\mathbf{u}}$ is invertible on $X_h$. This completes the proof.
\end{proof}

\subsection{Convergence of the numerical scheme}
\label{Appendix-state_convergence}

This subsection presents the convergence of the Crank-Nicolson cell-centered finite-volume scheme \eqref{eq:CN_cell} to the original equation \eqref{eq:state}. 

\begin{prop}
\label{prop:state_convergence}
Let $T>0$ and  $\mathbf v\in L^2((0,T); U)$. 
Let $h>0$ and $\mathcal T_h$ be an admissible mesh as in Section~\ref{sec:space_discrete}. Let  $\Delta t := T/N_t$ with $N_t \in \mathbb N^+$  and  let $\{(\theta_K^{(n)})_{K \in \mathcal T_h} \}_{n=0}^{N_t}\subset X_h$ be the solution of equation \eqref{eq:CN_cell}  where   
\begin{align}\label{20260107-VelocityCoefficients}
    v_i^{(n)} := \fint_{n \Delta t}^{(n+1) \Delta t} \langle \mathbf{v}(s), \mathbf{b}_i\rangle_{L^2(\Omega; \mathbb R^d)}  ds,
    ~i=1,\ldots,m,
    ~n=0,\ldots,N_t-1.
\end{align}
Then, when $h \to 0^+$ and $\Delta t \to 0^+$, the following piecewise-constant time--space reconstruction
\begin{align}\label{20260107-ReconstructionFromDiscretization}
    \bar\theta_{h,\Delta t}(t,x):= \sum_{n=0}^{N_t} \sum_{K \in \mathcal T_h} 
\chi_{[n\Delta t, \,(n+1)\Delta t)}(t)  \,\chi_K(x)\, \theta^{(n)}_K,
~t\in [0,T],  ~x\in \Omega
\end{align}
converges to $\theta(T;\mathbf v)$ strongly in $L^2((0,T)\times\Omega)$.  
Furthermore, it holds that
\begin{align}\label{20260107-StrongConvergenceOfFinalState}
    \bar\theta_{h,\Delta t}(T,\cdot) \to \theta(T;\mathbf v)
\text{ strongly in }  L^2(\Omega)
\text{ as }  h,\Delta t\to0, 
\end{align}
\end{prop}

\begin{proof}
The strong convergence of $\{ \bar\theta_{h,\Delta t} \}$ in $L^2((0,T)\times\Omega)$ will be done by the weak convergence, plus the convergence of the norms, of this sequence. 
First,  we know from \eqref{20251008-MassConservation} that
\begin{align*}
    \|\theta_h^{(n)}\|_{X_h}=\|\theta_h^{(0)}\|_{X_h},
~ n=0,\dots,N_t.
\end{align*}
By this and \eqref{20260107-ReconstructionFromDiscretization}, the sequence $\{ \bar\theta_{h,\Delta t} \}$ satisfies that for each $t \in [0,T]$, 
\begin{align}\label{20260107-ConvergenceOfNorm}
    \|\bar\theta_{h,\Delta t} (t)  \|_{X_h}
    = \|\theta_h^{(0)}\|_{X_h}
    \rightarrow  \| \theta^0 \|_{L^2(\Omega)}
    \text{ as } h \rightarrow 0^+. 
\end{align}
In particular, it is bounded in the space $L^2((0,T)\times\Omega)$. 

Next, we claim that  
\begin{align}\label{20260107-WeakConvergence}
    \bar\theta_{h,\Delta t}  \rightharpoonup  \theta(T;\mathbf v)
\text{ weakly in }  L^2((0,T)\times\Omega)
\text{ as }  h,\Delta t \rightarrow 0^+. 
\end{align}
To this end, arbitrarily take a subsequence with a weak limit  in $L^2((0,T)\times\Omega)$, denoted by $\theta^*$. Arbitrarily take a function $\varphi\in C_0^\infty([0,T)\times\Omega)$ and define
\begin{align}\label{20260107-DiscreteTestFunction}
\varphi_h^{(n)} := 
\bigg( \frac{1}{|K|}\int_K \varphi(n \Delta t,x) dx
\bigg)_{K\in\mathcal T_h}
\in X_h,
~ n=0,\dots,N_t.
\end{align}
From the anti-symmetry of $D_{\mathbf{b}_i}$ (see Lemma \ref{20251008-lemma-AntisymmetricDiscreteDivergence}), we rewrite equation \eqref{eq:CN_cell} in the following discrete weak form:
\begin{align}\label{20260107-DiscreteWeakForm}
- \big\langle 
    \theta_h^{(0)}, \varphi_h^{(0)} 
\big\rangle_{X_h}
=& \sum_{n=1}^{N_t-1} \Delta t   \Big\langle \theta_h^{(n)},   
      \frac{ 
        \varphi_h^{(n)} - \varphi_h^{(n-1)} 
    }{ \Delta t}
\Big\rangle_{X_h} 
\nonumber\\
& \quad\quad\quad\quad
+  \sum_{n=0}^{N_t-1} \Delta t
 \sum_{i=1}^m  v_i^{(n)} 
\Big\langle
    \frac{ \theta_h^{(n)} + \theta_h^{(n+1)}  }{2},
  D_{\mathbf{b}_i} \varphi_h^{(n)} 
\Big\rangle_{X_h}.   
\end{align}
At the same time, by computations as well as \eqref{20251111-PropertiesOfMesh}, there is a constant $C>0$ (independent of $h$) so that
\begin{align}\label{20260107-UsefulEstiamte}
\sup_{ K \in \mathcal T_h }
\Big| D_{\mathbf{b}_i} f_h - 
  \fint_K \text{div}( 
        \mathbf{b}_i f 
    )  dx
\Big|
\leq C h \| f \|_{C^2( \overline{\Omega})},
~\forall\,  f \in C^2( \overline{\Omega}),
\end{align}
where $f_h$ is defined by the way as in \eqref{20260107-DiscreteTestFunction} with $\varphi$ replaced by $f$. 
This, together with \eqref{20260107-DiscreteWeakForm}, \eqref{20260107-ReconstructionFromDiscretization} and \eqref{eq:v_finite}, yields
\begin{align*}
 -\int_\Omega \theta^0(x)  \varphi(0,x) dx =  
    \int_0^T \int_\Omega \theta^* \big( 
    \partial_t\varphi + \text{div}( 
        \mathbf{v} \varphi
    )
\big)
dx dt.
\end{align*}
Since $\varphi\in C_0^\infty([0,T)\times\Omega)$ was arbitrarily taken, the above implies that $\theta^*$ is the weak solution to equation \eqref{eq:state}, which is $\theta(\cdot; \mathbf{v})$. Note that $\theta^*$ is any accumulation point of the bounded sequence $\{ \bar\theta_{h,\Delta t} \}$ in the weak topology of $L^2((0,T)\times\Omega)$. 
Therefore, \eqref{20260107-WeakConvergence} is true. 

Now, because  the space $L^2((0,T)\times\Omega)$ is uniformly convex, by \cite[Proposition 3.32]{brezis2010functional}, the desired strong convergence of $\{ \bar\theta_{h,\Delta t} \}$ in $L^2((0,T)\times\Omega)$ follows from \eqref{20260107-WeakConvergence} and 
\eqref{20260107-ConvergenceOfNorm}. 

Finally, we show \eqref{20260107-StrongConvergenceOfFinalState} by a similar idea. For this purpose, let $\psi \in C_0^{\infty}(\Omega)$ and define $\psi_h$ in a similar way as in \eqref{20260107-DiscreteTestFunction} with $\varphi$ replaced by $\psi$. 
From  the anti-symmetry of $D_{\mathbf{b}_i}$ (see Lemma \ref{20251008-lemma-AntisymmetricDiscreteDivergence}), we rewrite equation \eqref{eq:CN_cell} in another discrete weak form:
\begin{align*}
\big\langle 
    \theta_h^{(N_t)}, \psi_h
\big\rangle_{X_h}
- \big\langle 
    \theta_h^{(0)}, \psi_h
\big\rangle_{X_h}
=& \sum_{n=0}^{N_t-1} \Delta t  
 \sum_{i=1}^m  v_i^{(n)} 
\Big\langle
    \frac{ \theta_h^{(n)} + \theta_h^{(n+1)}  }{2},
  D_{\mathbf{b}_i} \psi_h
\Big\rangle_{X_h}.
\end{align*}
Because of the strong convergence of $\{ \bar\theta_{h,\Delta t} \}$ in $L^2((0,T)\times\Omega)$, the above, together with \eqref{20260107-ReconstructionFromDiscretization}, \eqref{20260107-UsefulEstiamte} and \eqref{eq:v_finite}, yields that any weak accumulation point $\theta^*_T$ of the sequence $\{ \bar\theta_{h,\Delta t}(T,\cdot) \}$ in $L^2(\Omega)$ satisfies
\begin{align*}
\int_\Omega \theta^*_T \psi dx
-\int_\Omega \theta^0 \psi dx =  
    \int_0^T \int_\Omega \theta(t;\mathbf{v})  
     \text{div}( 
        \mathbf{v} \psi
    )
dx dt.
\end{align*}
Since $\psi \in C_0^{\infty}(\Omega)$ was arbitrarily chosen, the last equality implies that  $\theta^*_T$ must be equal to $\theta(T;\mathbf{v})$. This, along with \eqref{20260107-ConvergenceOfNorm} (with $t=T$) and  \cite[Proposition 3.32]{brezis2010functional}, yields \eqref{20260107-StrongConvergenceOfFinalState}. 
This completes the proof.

\end{proof}

\subsection{Convergence of discrete cost functionals (\ref{20251008-DiscreteFunctional})}
\label{Appendix-ConvDiscreteObjective}

This subsection is devoted to the convergence of discrete cost functional, where extra assumptions on the mesh are required by the two-point flux approximation
finite-volume Laplacian \eqref{20251011-DiscreteLaplacian} (see  \cite[Definitions 10.1 and 9.1]{EymardFVM}). 

\begin{prop}\label{prop:ConvDiscreteObjective}
Use the same notations as in Proposition \ref{prop:state_convergence}. 
Suppose that the mesh $\mathcal T_h$ satisfies the following extra conditions:
\begin{itemize}
\item a family of points $\{x_K\}_{K\in\mathcal T_h}$ with $x_K\in K$ and a family of points
      $\{x_\Gamma\}_{\Gamma\in\mathcal E_h}$ with $x_\Gamma\in \Gamma$ such that
      \begin{itemize}
      \item if $\Gamma=K|L$ is an interior face, then the segment $[x_K,x_L]$ is orthogonal to $\Gamma$
            and intersects $\Gamma$ at $x_\Gamma$,
      \item if $\Gamma\subset\partial\Omega$ is a boundary face of $K$, then the segment $[x_K,x_\Gamma]$
            is orthogonal to $\Gamma$.
      \end{itemize}
\end{itemize}
Then,  it holds that $\lim_{h,\Delta t\to0}
\widetilde J_h(\mathbf v_h) = J(\mathbf v)$.
\end{prop}

\begin{proof}
By similar arguments as in \cite[Thm.~10.3]{EymardFVM}, we find from \eqref{20260107-StrongConvergenceOfFinalState} that the solution $\eta_h = ( \eta_{h,K})_{K \in \mathcal T_h} $ to the finite-volume Neumann scheme
\eqref{20251011-DiscreteLaplacian} satisfies
\begin{align*}
\lim_{\Delta t \rightarrow 0^+}
    \lim_{h \rightarrow 0^+}
\sum_{K \in \mathcal T_h}  \eta_{h,K} \chi_K
= \eta( \theta(T; \mathbf{v}) ) 
\text{ strongly in }  L^2(\Omega). 
\end{align*}
This implies
\begin{align*}
  \lim_{h,\Delta t \rightarrow 0^+}  \langle \theta_h^{N_t}, \eta_h \rangle_{X_h}
  = \big\langle \theta(T; \mathbf{v}), \eta( \theta(T; \mathbf{v})) \big\rangle_{L^2(\Omega)}.
\end{align*}
At the same time, by \eqref{20260107-VelocityCoefficients} and \eqref{eq:v_finite}, it is clear that
\begin{align*}
    \lim_{\Delta t \rightarrow 0^+}
    \Delta t  \sum_{n=0}^{N_t-1} \sum_{i=1}^m  | v_i^{(n)}|^2
= \int_0^T  \sum_{i=1}^m \langle \mathbf{v}(s), \mathbf{b}_i\rangle_{L^2(\Omega; \mathbb R^d)}^2  dt
= \int_0^T \| \mathbf{v} \|_{L^2(\Omega; \mathbb R^d)}^2 dt.
\end{align*}
The last two equalities, together with \eqref{20251008-DiscreteFunctional} and \eqref{20251010-AnotherFormOfCostFunctional}, leads to the desired conclusion. This completes the proof.

\end{proof}

\section*{Acknowledgment}
The first author is partially supported by NSF DMS-2111486, DMS-2205117 and AFOSR FA9550-23-1-0675. This work was initiated during the author's visit at the Chair for Dynamics, Control, Machine Learning and Numerics, Friedrich-Alexander-Universit\" at Erlangen N\"urnberg, Germany, under the support of Humboldt Research Fellowship for Experienced Researchers program from the Alexander von Humboldt Foundation.

The second author is supported by the Alexander von Humboldt Professorship program and the Deutsche Forschungsgemeinschaft (DFG, German Research Foundation) under project C07 of the Sonderforschungsbereich/Transregio 154 ``Mathematical Modelling, Simulation and Optimization Using the Example of Gas Networks" (project ID: 239904186).

The third author is supported by the National Natural Science Foundation of China under grants 12171359 and the Humboldt Research Fellowship for Experienced Researchers program from the Alexander von Humboldt Foundation.

The last author has been funded by the ERC Advanced Grant CoDeFeL (ERC-2022-ADG-101096251), the grants TED2021-131390B-I00-DasEl of MINECO, and PID2023-146872OB-I00-DyCMaMod of MICIU (Spain); the Alexander von Humboldt Professorship program; the European Union's Horizon Europe MSCA project ModConFlex (HORI\-ZON-MSCA-2021-DN-01(project 101073558); the Transregio 154 Project ``Mathematical Modelling, Simulation and Optimization Using the Example of Gas Networks" of the DFG; the AFOSR 24IOE027 project; and the Madrid Government-UAM Agreement for the Excellence of the University Research Staff in the context of the V PRICIT (Regional Programme of Research and Technological Innovation).

\bibliographystyle{abbrv}
\bibliography{refs_mixing}

\end{document}